\documentclass[11pt]{amsart}

\usepackage{epsf,amssymb,times,overpic,amscd,diagrams}
\usepackage{xcolor}
\usepackage[all]{xy}

\usepackage{amsfonts}
\usepackage{amsmath}
\usepackage{graphicx}
\usepackage{epsfig}
\usepackage{epstopdf}
\usepackage{hyperref}

\newcommand{\ba}{\begin{array}}
\newcommand{\ea}{\end{array}}
\newcommand{\be}{\begin{enumerate}}
\newcommand{\ee}{\end{enumerate}}

\counterwithin*{equation}{section}
\counterwithin*{equation}{subsection}
\counterwithin*{equation}{subsubsection}

\renewcommand{\theequation}{%
	\thesection.%
	\ifnum\value{subsection}>0 \arabic{subsection}.\fi
	\ifnum\value{subsubsection}>0 \arabic{subsubsection}.\fi
	\arabic{equation}%
}

\newtheorem{thm}[equation]{Theorem}
\newtheorem{prop}[equation]{Proposition}
\newtheorem{lemma}[equation]{Lemma}
\newtheorem{cor}[equation]{Corollary}

\newtheorem{claim}[equation]{Claim}
\newtheorem*{claim2}{Claim}
\newtheorem*{claimA}{Claim A}
\newtheorem*{claimB}{Claim B}
\newtheorem*{claimC}{Claim C}

\theoremstyle{definition}

\newtheorem{defn}[equation]{Definition}

\theoremstyle{remark}

\newtheorem{rmk}[equation]{Remark}
\newtheorem{example}[equation]{Example}

\newtheorem{convention}[equation]{Convention}
\newtheorem{notation}[equation]{Notation}

\newcommand{\R}{\mathbb{R}}
\newcommand{\Z}{\mathbb{Z}}

\newcommand{\F}{\mathbb{F}_2}
\newcommand{\bdry}{\partial}
\newcommand{\s}{\vskip.1in}
\newcommand{\n}{\noindent}

\newcommand{\Hom}{\operatorname{Hom}}

\newcommand{\ra}{\rightarrow}

\newcommand{\xra}{\xrightarrow}
\newcommand{\mf}{\mathbf}
\newcommand{\op}{\operatorname}
\newcommand{\cal}{\mathcal}
\newcommand{\es}{\emptyset}
\newcommand{\lan}{\langle}
\newcommand{\ran}{\rangle}
\newcommand{\ul}{\underline}

\newcommand{\g}{\Gamma}
\newcommand{\cf}{\mathcal{F}}
\newcommand{\cs}{\mathcal{S}}
\newcommand{\csg}{\mathcal{S}(\Gamma)}
\newcommand{\cstc}{\mathcal{S}_{\widetilde{\mathcal{C}}}}
\newcommand{\cstd}{\mathcal{S}_{\widetilde{\mathcal{D}}}}
\newcommand{\bne}{B_{n,e}}
\newcommand{\cne}{\mathcal{C}_{n,e}}
\newcommand{\tcne}{\widetilde{\mathcal{C}}_{n,e}}
\newcommand{\dne}{\mathcal{D}_{n,e}}
\newcommand{\tdne}{\widetilde{\mathcal{D}}_{n,e}}
\newcommand{\fne}{\mathcal{F}_{n,e}}
\newcommand{\tfne}{\widetilde{\mathcal{F}}_{n,e}}
\newcommand{\tf}{\widetilde{\mathcal{F}}}
\newcommand{\rne}{R_{n,e}}
\newcommand{\qne}{Q_{n,e}}
\newcommand{\rnea}{R_{n,e}^{\operatorname{alg}}}

\newcommand{\rg}{R_+(\Gamma)}
\newcommand{\vz}{\mathbb{V}}
\newcommand{\tpv}{V^+}
\newcommand{\vnb}{V^+_{nb}}
\newcommand{\gv}{\Gamma_{\mathbf{v}}}
\newcommand{\gw}{\Gamma_{\mathbf{w}}}
\newcommand{\gu}{\Gamma_{\mathbf{u}}}
\newcommand{\guv}{\Gamma_{\underline{\mathbf{b}}}}
\newcommand{\gov}{\Gamma_{\overline{\mathbf{b}}}}
\newcommand{\uv}{\underline{\mathbf{b}}}
\newcommand{\ov}{\overline{\mathbf{b}}}
\newcommand{\ob}{\overline{\beta}}

\newcommand{\oi}{OI}
\newcommand{\mi}{\mathbf{i}}
\newcommand{\mj}{\mathbf{j}}
\newcommand{\mv}{\mathbf{v}}
\newcommand{\mw}{\mathbf{w}}
\newcommand{\iv}{\mathbf{i}_{\mathbf{v}}}
\newcommand{\nv}{NV}
\newcommand{\dnv}{DNV}
\newcommand{\slv}{SLV}
\newcommand{\shv}{SHV}
\newcommand{\sv}{SV}
\newcommand{\vi}{\mathbf{v} | \mathbf{i}}

\newcommand{\bi}{\beta(\mathbf{i})}
\newcommand{\wb}{\mathbf{w}(\beta)}
\newcommand{\wh}{\widehat}
\newcommand{\whg}{\widehat{\Gamma}}
\newcommand{\whv}{\widehat{\mathbf{v}}}
\newcommand{\whi}{\widehat{\mathbf{i}}}
\newcommand{\wt}{\widetilde}
\newcommand{\tb}{\widetilde{\beta}}
\newcommand{\tg}{\widetilde{\Gamma}}

\begin{document}
%%%%%%%%%%%%%%%%%%%%%%%%%%%%%%%%%%%%%%%%%%%%%%%%%%%%%%%%%
\title[Contact categories of disks]
{Contact categories of disks}

\author{Ko Honda}
\address{University of California, Los Angeles, Los Angeles, CA 90095}
\email{honda@math.ucla.edu} \urladdr{http://www.math.ucla.edu/\char126 honda}

\author{Yin Tian}
\address{Yau Mathematical Sciences Center, Tsinghua University, Beijing 100084, China}
\email{yintian@tsinghua.edu.cn} \urladdr{}

%\date{\today}

\keywords{Contact structures, categorification, Heegaard Floer homology}

\subjclass[2010]{Primary 53D10; Secondary 53D40.}

\thanks{KH supported by NSF Grants DMS-0805352, DMS-1105432, DMS-1406564, and DMS-154914. YT supported by NSFC 11601256 and 11971256.}

\begin{abstract}
In the first part of the paper we associate a pre-additive category $\mathcal{C}(\Sigma)$ to a closed oriented surface $\Sigma$, called the {\em contact category} and constructed from contact structures on $\Sigma\times[0,1]$.  There are also $\mathcal{C}(\Sigma,F)$, where $\Sigma$ is a compact oriented surface with boundary and $F\subset \bdry\Sigma$ is a finite oriented set of points which bounds a submanifold of $\bdry\Sigma$, and universal covers $\widetilde{\mathcal{C}}(\Sigma)$ and $\widetilde{\mathcal{C}}(\Sigma,F)$ of $\mathcal{C}(\Sigma)$ and $\mathcal{C}(\Sigma,F)$.  In the second part of the paper we prove that the universal cover of the contact category of a disk admits an embedding into its ``triangulated envelope.''
\end{abstract}

\maketitle
%%%%%%%%%%%%%%%%%%%%%%%%%%%%%%%%%%%%%%%%%%%%%%%%%%%%%%%%%

\tableofcontents

\section{Introduction}

%\subsection{Main results}

The goal of this paper is twofold. The first goal is to associate a pre-additive category $\mathcal{C}(\Sigma)$ to a closed oriented surface $\Sigma$, called the {\em contact category} and constructed from contact structures on $\Sigma\times[0,1]$.\footnote{The contact category was discovered by the first author around 2007, but never written up systematically.  We hope that this is the first in a series of papers which develops the theory of contact categories.  The idea of constructing a contact category was also pursued by Kevin Walker (unpublished) at around the same time.}  There are also $\mathcal{C}(\Sigma,F)$, where $\Sigma$ is a compact oriented surface with boundary and $F\subset \bdry\Sigma$ is a finite oriented set of points which bounds a submanifold of $\bdry\Sigma$, and universal covers $\widetilde{\mathcal{C}}(\Sigma)$ and $\widetilde{\mathcal{C}}(\Sigma,F)$ of $\mathcal{C}(\Sigma)$ and $\mathcal{C}(\Sigma,F)$. The contact category $\mathcal{C}(\Sigma)$ admits a decomposition
$$\mathcal{C}(\Sigma)=\coprod_{i\in\Z} \mathcal{C}(\Sigma,i)$$
into connected components, where $i$ is the Euler class (= first Chern class) of the contact structure evaluated on $\Sigma$.

The contact categories, a priori, have no reason to satisfy any nontrivial axioms of a triangulated category.  In spite of such an inauspicious start, the contact categories partially satisfy the axioms of a triangulated category,  and, in particular, have distinguished triangles that we call the {\em bypass exact triangles}.

The second goal of this paper is to study the universal covers of contact categories of a disk in more detail.  When $\Sigma=D^2$, $\# F=2n+2$, and we are in the component where the Euler class is $n-2e$, we abbreviate
$$\widetilde{\mathcal{C}}_{n,e}:= \widetilde{\mathcal{C}}(D^2,F; n-2e).$$
We prove that $\tcne$ admits an embedding into its ``triangulated envelope"; more precisely, we have:

\begin{thm} \label{thm embed}
There exist a family of triangulated categories $\tdne$ and additive functors $\tfne: \tcne \ra \tdne$ such that
$\tfne$ are fully faithful and images of $\tfne$ generate $\tdne$ under taking iterated cones.  Moreover, $\tfne$ is exact, i.e., takes bypass exact triangles to distinguished triangles.
\end{thm}

In this paper we take $\tcne$ and $\tdne$ to be $\F$-linear, where $\F$ is the field of two elements.
We believe that the analogue of Theorem \ref{thm embed} holds for any ground field, but technically difficult to keep track of signs.

The category $\tdne$ is the homotopy category of bounded chain complexes of finitely generated left projective $\rne$-modules, where $\rne$ is isomorphic to the homology of a strands algebra over a disk (cf.\ \cite{LOT} and \cite{Za}).
The contact categories and their relation to Heegaard Floer homology have been extensively studied by Mathews in a series of papers \cite{M1, M2, M3, M4}.

%\subsection{Triangulated envelope}

In~\cite{Co}, Cooper defines the {\em formal contact category} ${\cal Ko}_+(\Sigma,F)$, which can be interpreted as an abstractly constructed ``triangulated envelope'' of $\widetilde{\mathcal{C}}(\Sigma,F)$.  (Strictly speaking, the version in \cite{Co} is ungraded, but it is expected that his construction works in the graded case; we are referring to the graded version as ${\cal Ko}_+(\Sigma,F)$.)  He also proves the equivalence of ${\cal Ko}_+(D^2,F)$ and $\tdne$: the functor ${\cal Ko}_+(D^2,F)\to \tdne$ is defined using the universal property of ${\cal Ko}_+(D^2,F)$ and the functor $\tdne \to {\cal Ko}_+(D^2,F)$ is constructed using the work of Zarev~\cite{Za}.  Combining Theorem~\ref{thm embed} and (the extension of) Cooper's work gives:

\begin{thm}
$\tcne$ embeds in the triangulated envelope ${\cal Ko}_+(D^2,F)$.
\end{thm}

It is a very interesting problem to understand whether the contact category for a general surface embeds in its triangulated envelope.
For algebraic applications, the contact categories of rectangles and annuli were used by the second author to give categorifications of the quantized Lie superalgebra $\mathfrak{sl}(1|1))$ and the Clifford algebras \cite{T1, T2, T3, T4}.

\s\n
{\em Index of notation.} We have provided an index of notation at the end of the paper, which we  hope the reader will find useful starting with Section~\ref{section: algebraic description}.

\s\n{\em Organization of the paper.}
In Section~\ref{section: contact category} we define the contact categories and their universal covers for general surfaces $\Sigma$ and $(\Sigma,F)$; from Section~\ref{section: contact category of disk} we restrict to contact categories $\cne$ and $\tcne$ over disks.  In Section~\ref{section: contact category of disk} we introduce the Serre functors of $\cne$ and $\tcne$ which provide essential simplifications in the proof of Theorem~\ref{thm embed}. In Section~\ref{section: algebraic description} we introduce notation to algebraically describe $\cne$ and $\tcne$ and in Section~\ref{section: defn of algebra} we define a family of triangulated categories $\tdne$.
In Section~\ref{section: defn of functor} we construct a family of functors $\fne: \cne \ra \dne$ of additive categories and in Section~\ref{section: functors on universal cover} we extend $\fne$ to $\tfne: \tcne \ra \tdne$ and show that the $\tfne$ preserve the shift functors and distinguished triangles.
Finally in Section~\ref{section: triangulated envelope} we show that the $\tfne$ are fully faithful and the images of $\tfne$ generate $\tdne$ under taking iterated cones.

\s\n {\em Acknowledgements.}   The first author thanks Takashi Tsuboi and the University of Tokyo (in 2007), as well as MSRI and the organizers of the Symplectic and Contact Geometry and Topology Program for their hospitality (in 2009-2010). The first author also thanks Will Kazez and Gordana Mati\'c on collaborations leading up to this work, Thomas Geisser for discussions on category theory, and Paolo Ghiggini, Toshitake Kohno and Shigeyuki Morita for helpful discussions. The second author thanks the Simons Center for an excellent research environment. Finally we thank the referee for a careful reading of the paper and a large list of comments.

\section{The contact category} \label{section: contact category}

The goal of this section is to define the contact categories $\mathcal{C}(\Sigma)$ and $\mathcal{C}(\Sigma,F)$.  We first recall some properties of bypasses and contact
structures.

\subsection{Contact structures and bypasses}

For more details, the reader is referred to \cite{H1}.

\subsubsection{Convex surfaces}

Let $\Sigma$ be a compact oriented surface and $\Gamma\subset \Sigma$ be an
oriented, properly embedded $1$-manifold (i.e., a multicurve) which
divides $\Sigma-\Gamma$ into alternating positive and negative
regions in the sense that the sign changes every time $\Gamma$ is
crossed once transversely.  The positive region (resp.\ negative
region) will be denoted $R_+(\Gamma)$ (resp.\ $R_-(\Gamma)$), and
the orientation of $\Gamma$ and the boundary orientation of
$R_+(\Gamma)$ agree. Such a $1$-manifold $\Gamma$ is called a {\em
dividing set} of $\Sigma$.

Recall that an oriented embedded surface $\Sigma$ in a contact $3$-manifold $(M,\xi)$ is {\em $\xi$-convex} (or simply {\em convex}) if there is a contact vector field $X$ which is positively transverse to $\Sigma$. The {\em dividing set of $\Sigma$ with respect to $(\xi,X)$} is the locus
$$\Gamma=\{x\in\Sigma~|~X(x)\in \xi(x)\}.$$
In this paper, our convex surfaces are either closed or compact with
Legendrian boundary. In such cases, $\Gamma$ is an oriented, properly
embedded $1$-manifold and its isotopy class is independent of the
choice of $X$. The positive (resp.\ negative) region is the set of
points $x\in\Sigma$ for which the orientation induced by $X(x)$ on
$\xi(x)$ coincides with (resp.\ is opposite of) the orientation on
$\xi(x)$. A dividing set $\Gamma$ on $\Sigma$ will usually be viewed
as the dividing set $\Gamma_{\Sigma_0}$ with respect to a
$[-\varepsilon,\varepsilon]$-invariant contact structure
$\xi_\Gamma$ on $\Sigma\times[-\varepsilon,\varepsilon]$, where we
write $\Sigma_t=\Sigma\times\{t\}$; in other words, we are locally taking $X=\bdry_t$.

Let $\Sigma$ be a convex surface with dividing set $\Gamma$. According to a criterion of Giroux~\cite{Gi1}, $\Sigma$ has a tight neighborhood if and only if either (i) $\Sigma=S^2$ and $\Gamma=S^1$, or (ii) $\Sigma\not=S^2$ and $\Gamma$ has no homotopically trivial closed  component.

\subsubsection{Bypasses} \label{subsubsection: bypasses}

An embedded Legendrian arc $\delta\subset \Sigma$ is an {\em arc of attachment of a bypass} if $\delta$ is transverse to $\Gamma$ and 
%\marginpar{\cg \tiny made things a little more precise; chose a specific model for a bypass and made dependence on $U$ and $\phi$ explicit}
has exactly three intersections with $\Gamma$, two of which are points of
$\bdry\delta$.  We also write $\delta=\delta_+\cup \delta_-$, where $\delta_\pm$ is the closure of $\delta\cap R_\pm(\Gamma)$.
Let $U\subset \Sigma$ be a disk neighborhood of $\delta$, which transversely intersects $\Gamma$ along three arcs and whose boundary is Legendrian. Consider an overtwisted disk $(\{(r,\theta)~|~r\leq 1\},\zeta)$, where $(r,\theta)$ are polar coordinates and $\zeta$ is the germ of a contact structure such that its characteristic foliation is of ``wheel-and-spokes" type with leaves $r=1$ and $\theta=const$. Then a {\em bypass} $D$ is $(\{r\leq 1, 0\leq \theta\leq \pi\}, \zeta|_{0\leq \theta\leq \pi})$, i.e.,
one-half of an overtwisted disk.  

We now attach $D$ to the invariant contact structure $(\Sigma\times[0,\varepsilon],\xi_\Gamma)$ along $\Sigma_\varepsilon$ (resp.\ $(\Sigma\times[-\varepsilon,0],\xi_\Gamma)$ along $\Sigma_{-\varepsilon}$) so that the diameter of $D$ is glued to $\delta\times\{\varepsilon\}$ (resp.\ $\delta\times \{-\varepsilon\}$). If $D$ is attached to $\Sigma_\varepsilon$ (resp.\ $\Sigma_{-\varepsilon}$), then we say the bypass is attached {\em from the front} (resp.\ {\em from the back}).  When $D$ is attached from the front, a small one-sided neighborhood of $(\Sigma\times[0,\varepsilon])\cup D$ can be viewed as $(\Sigma\times[0,1],\xi)$, where the dividing set $\Gamma_{\Sigma_0}$ is
$\Gamma$ and $\Gamma_{\Sigma_1}$ is obtained from $\Gamma$ by
performing the local operation on $U$ as in Figure~\ref{bypass}.   More specifically, in the rest of this paper, we:
\be
\item[(B1)] fix a model contact structure $(D^2\times[0,1],\zeta)$ such that $D^2\times\{0,1\}$ is convex with Legendrian boundary, $\zeta$ is $\bdry_t$-invariant on $D^2\times[0,\varepsilon]$ and a neighborhood of $(D^2\times\{1\})\cup (\bdry D^2\times [0,1])$, the dividing sets $\Gamma_{D^2\times\{0\}}$, $\Gamma_{D^2\times\{1\}}$ and the Legendrian arc $\delta_0$ are as given in Figure~\ref{bypass}, and $\zeta$ is contactomorphic to a small one-sided neighborhood of $(D^2\times[0,\varepsilon])\cup_{\delta_0}D$, where the contactomorphism is the identity on $D^2\times [0,\varepsilon]$;
\item[(B2)] choose an identification $\phi:U\stackrel\sim\to D^2$ such that $\phi(\Gamma_{U\times\{0\}})=\Gamma_{D^2\times\{0\}}$ and $\phi(\delta)=\delta_0$; and
\item[(B3)] let $\xi|_{(\Sigma-U)\times[0,1]}$ be $t$-invariant with dividing sets $\Gamma_{(\Sigma-U)\times\{t\}}=(\Gamma-U)\times\{t\}$ and $\xi|_{U\times[0,1]} = (\phi\times id_{[0,1]})^* \zeta$.
\ee
In other words, a bypass attachment depends on the choices of $U\supset \delta$ and $\phi:U\stackrel\sim \to D^2$. We remark that topologically $\Gamma_{\Sigma_1}$ is obtained from $\Gamma=\Gamma_{\Sigma_0}$ by applying two band sums in succession.

\begin{figure}[ht]
\begin{overpic}[width=8cm]{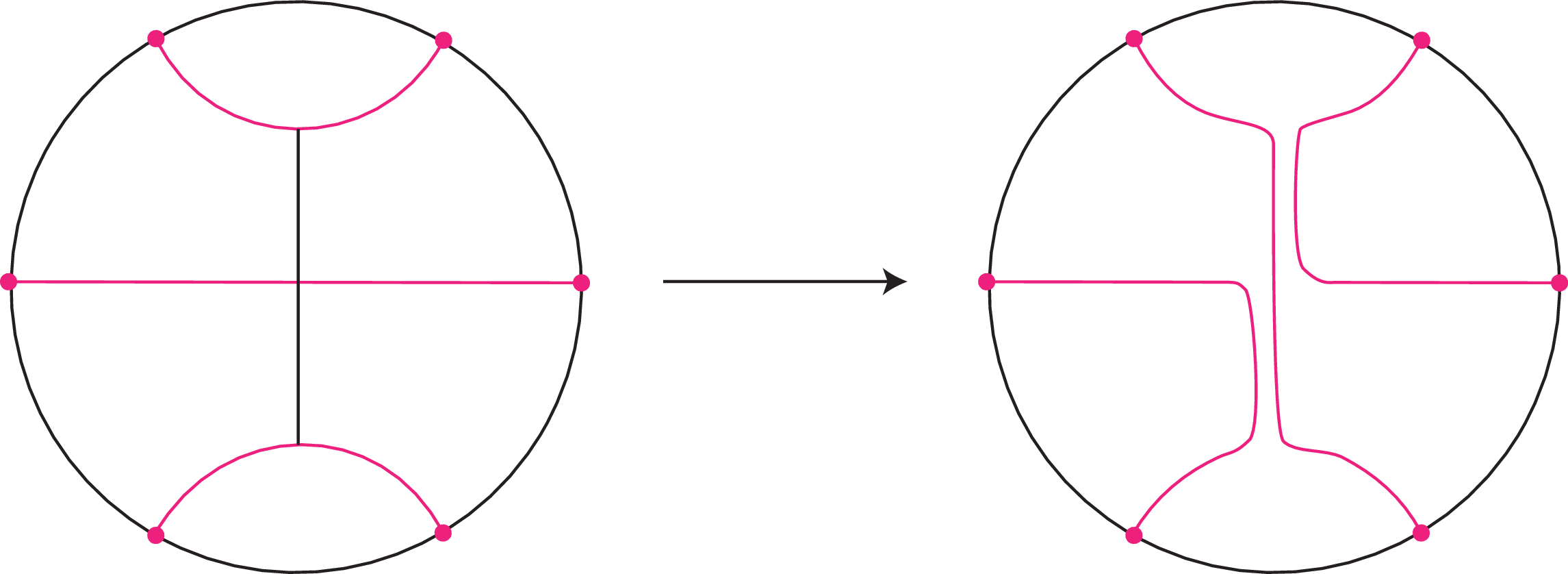}
\put(20,22){\tiny $\delta_0$} \put(0,30){\tiny $D^2$}
\end{overpic}
\caption{Effect of a bypass attachment along $\delta_0$ from the front.  The left-hand side is $D^2\times\{0\}$ and the right-hand side is $D^2\times\{1\}$.  The red arcs are the dividing curves.} \label{bypass}
\end{figure}

Suppose that $(\Sigma\times[0,1],\xi)$ is obtained by attaching a bypass along
$\delta$ to $(\Sigma_0,\Gamma_0)$ from the front. If $U_0\subset
\Sigma_0$ is the disk neighborhood of $\delta$ and $U_1$ is the
corresponding disk on $\Sigma_1$ (in particular,
$\Gamma_0|_{\Sigma_0-U_0}=\Gamma_1|_{\Sigma_1-U_1}$), then there is
a bypass arc of attachment $\delta'\subset U_1$ which gives (a
contact manifold isotopic to) $(\Sigma\times[0,1],\xi)$ when
attached to $(\Sigma_1,\Gamma_1)$ from the back.  We will call
$\delta'$ the {\em anti-bypass} arc of the bypass arc $\delta$.

\begin{convention}
If we do not explicitly mention from which side the bypass is attached, we always assume the bypass is attached from the front. 
\end{convention}

A bypass is {\em overtwisted} (resp.\ {\em trivial}) if there exists a disk neighborhood $U\subset \Sigma$ of the arc of attachment $\delta$ such that:
\be
\item $\Gamma\pitchfork U$ and $\Gamma|_U$ consists of two arcs;
\item if $\Gamma'$ is the result of attaching the bypass, then $\Gamma'|_U$ has a homotopically trivial component (resp.\ $\Gamma'|_U$ is homotopic to $\Gamma|_U$).
\ee
See Figure~\ref{ot-trivial-bypass}.

\begin{figure}[ht]
\begin{overpic}[width=6cm]{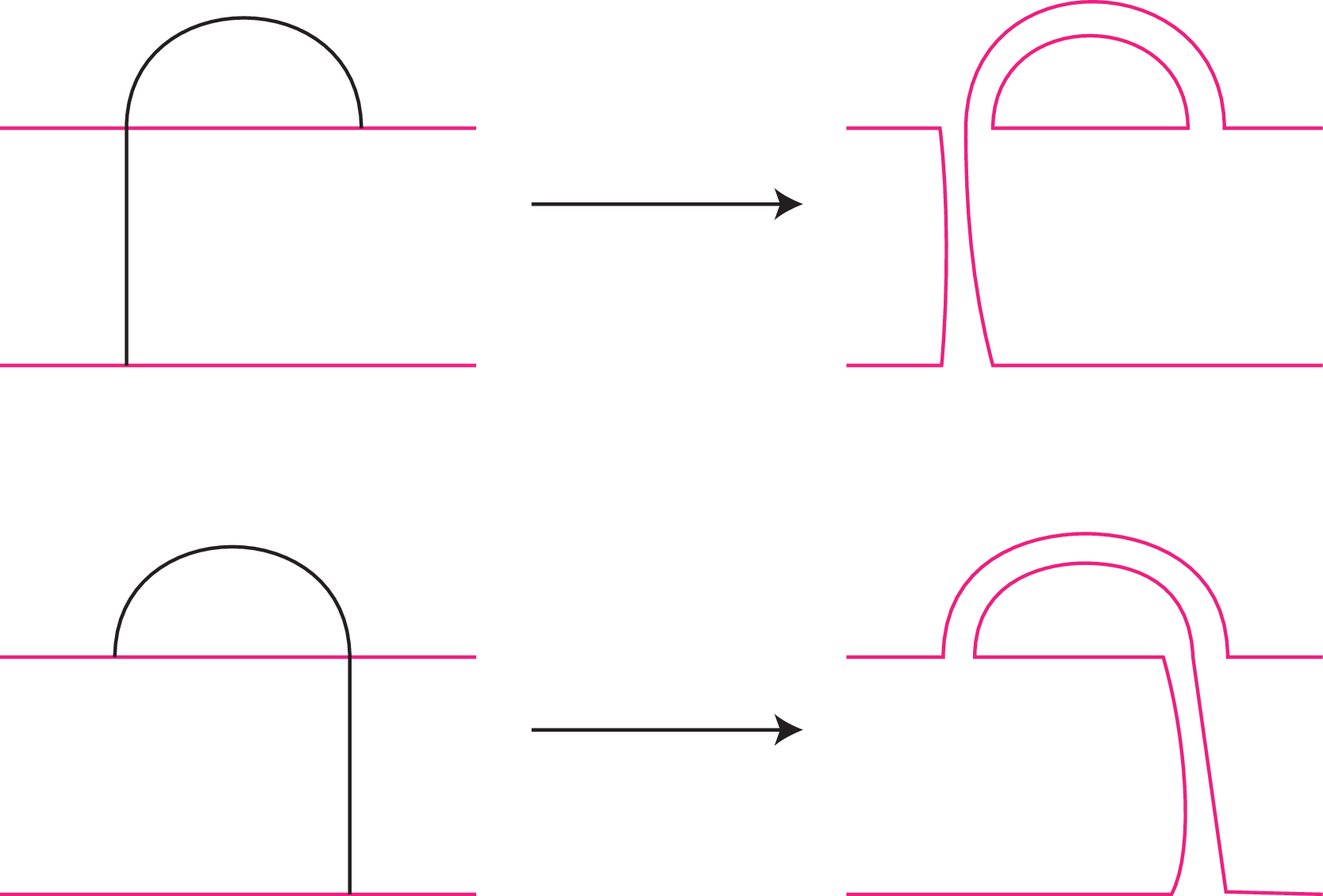}
%\put(20,22){\tiny $\delta$} \put(0,30){\tiny $U$}
\end{overpic}
\caption{Overtwisted bypass (top) and trivial bypass (bottom).} \label{ot-trivial-bypass}
\end{figure}

\subsubsection{Bypass rotation}

We will now discuss {\em bypass rotation}, which was introduced in \cite{HKM}. Let $\Sigma$ be a convex surface with
dividing set $\Gamma$.  The ambient contact manifold for $\Sigma$ is
the $[-\varepsilon,\varepsilon]$-invariant contact neighborhood of
$\Sigma=\Sigma_0$.  Let $\delta_0$ and $\delta_1$ be arcs of
attachment as given in Figure~\ref{rotate}.  In particular,
$\delta_0$ is obtained from $\delta_1$ by rotating one endpoint in
the counterclockwise direction. The bypasses are to be attached
``from the front''.  We will call such an operation {\em left
rotation}.

\s
\begin{figure}[ht]
\begin{overpic}[width=3.5cm]{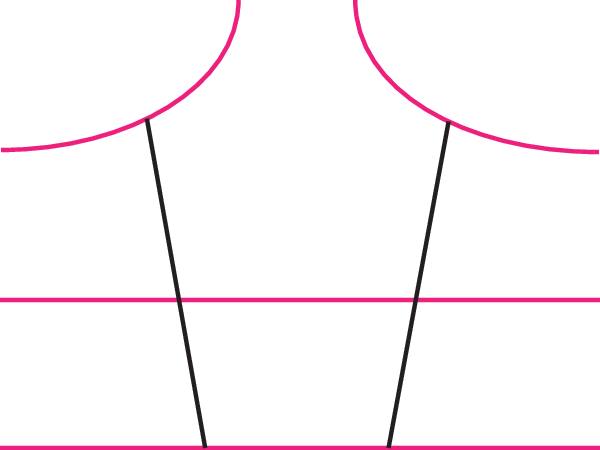}
\put(27.5,44){\tiny $\delta_0$} \put(66,44){\tiny $\delta_1$}
%\put(2,27){\tiny $\gamma_1$} \put(2,2.5){\tiny $\gamma_2$}
\end{overpic}
\caption{The arc $\delta_0$ is to the left of $\delta_1$.}
\label{rotate}
\end{figure}

\begin{lemma}[Bypass Rotation] \label{lemma: bypass rotation}
Let $(\Sigma\times[0,1],\xi_{\delta_1})$ be the contact manifold
obtained from $(\Sigma,\Gamma)$ by attaching a bypass from the front
along $\delta_1$.  If $\delta_0$ is obtained from $\delta_1$ by left
rotation, then there exists a bypass along $\delta_0$ inside
$(\Sigma\times[0,1],\xi_{\delta_1})$.
\end{lemma}

The lemma is completely local, i.e., it is valid when $\Sigma=D^2$
and $\Gamma$ consists of the four arcs given in Figure~\ref{rotate}.

\subsection{The contact category $\mathcal{C}(\Sigma)$}
\label{subsection: rcc}

Let $\Sigma$ be a closed, oriented surface. In this subsection we assign to each $\Sigma$ a category $\mathcal{C}(\Sigma)$, called the {\em contact category}.

%We will later show that the contact category, while not even additive\footnote{We can complete $\mathcal{C}(\Sigma)$ into an additive category by (abstractly) defining products of objects and making Hom sets into $\Z$-modules.}, still satisfies many properties of a triangulated category, including the octahedral axiom. (See for example \cite{We} for category theory and homological algebra.)

\begin{defn}
A surface $\Sigma$ is {\em collared} if it is equipped with auxiliary data $(\Sigma\times[-\varepsilon,\varepsilon], X)$, where:
\begin{enumerate}
\item[(i)] $\Sigma\times[-\varepsilon,\varepsilon]$ is a thickening of $\Sigma$ with coordinates $(x,t)$ so that $\Sigma=\Sigma\times\{0\}$; and
\item[(ii)] $X$ is the nonsingular vector field $\bdry_t$ on $\Sigma\times[-\varepsilon,\varepsilon]$, i.e., the pullback of $\bdry_t$ under the projection $\Sigma\times[-\varepsilon,\varepsilon]\to [-\varepsilon,\varepsilon]$.
\end{enumerate}
The manifold $\Sigma\times[-\varepsilon,\varepsilon]$ is a {\em collar} or {\em collar neighborhood} of $\Sigma$.
\end{defn}

We often write $\Sigma_t=\Sigma\times\{t\}$.

\subsubsection{The category $\mathbf{Cont}(\Sigma)$} \label{subsubsection: category cont}

We first define the category $\mathbf{Cont}(\Sigma)$, where $\Sigma$ is a collared surface with $\varepsilon>0$ small.

The objects of $\mathbf{Cont}(\Sigma)$ are dividing sets $\Gamma$ on $\Sigma$, where a {\em dividing set} $\Gamma$ is an oriented embedded $1$-manifold which is the oriented boundary of an open $2$-dimensional submanifold $R_+(\Gamma)$ of $\Sigma$. The submanifold $R_+(\Gamma)$ has the same orientation as $\Sigma$ and
$R_-(\Gamma)=\Sigma-\overline{R_+(\Gamma)}$ has the opposite orientation as $\Sigma$.   The collection of objects of $\mathbf{Cont}(\Sigma)$ will be denoted by $\mathfrak{ob}(\mathbf{Cont}(\Sigma))$.

\begin{rmk} \label{rmk object isotopy}
We take the objects to be $1$-manifolds, not {\em isotopy classes} of $1$-manifolds. See Section \ref{ssec: isotopy morphism} for more details.
\end{rmk}

Next we define $\Hom_{\mathbf{Cont}(\Sigma)}(\Gamma,\Gamma')$ to be the set of homotopy classes of contact structures $\xi$ on $\Sigma\times[0,1]$ such that:
\begin{enumerate}
\item the boundary $\bdry (\Sigma\times[0,1])=\Sigma_1-\Sigma_0$ is $\xi$-convex;
\item there exists an extension of $\xi$ to $\Sigma\times[-\varepsilon,1+\varepsilon]$ so that $(\Sigma\times[-\varepsilon, \varepsilon],\bdry_t)$ and $(\Sigma\times[1-\varepsilon,1+\varepsilon],\bdry_t)$ are collared neighborhoods of $\Sigma_0$ and $\Sigma_1$ and on which $\bdry_t$ is a contact vector field with dividing sets $\Gamma$ and $\Gamma'$ on $\Sigma_0$ and $\Sigma_1$.
\end{enumerate}
Two contact structures $\xi$ and $\xi'$ are homotopic if there is a path $\{\xi_s\}_{s\in[0,1]}$ of contact structures on $\Sigma\times[0,1]$ from $\xi$ to $\xi'$ satisfying (1) and (2) above.

The identity morphism $\Gamma\stackrel{id}\to \Gamma$ is (homotopy class of) the $[0,1]$-invariant contact structure $\xi$ on $\Sigma\times[0,1]$ with dividing set $\Gamma$ on $\Sigma\times\{t\}$, $t\in[0,1]$.

To take the composition of $[\xi] \in \Hom_{\mathbf{Cont}(\Sigma)}(\Gamma,\Gamma')$ and $[\xi']\in \Hom_{\mathbf{Cont}(\Sigma)}(\Gamma',\Gamma'')$, we choose representatives $\xi$ and $\xi'$ so they agree on collared neighborhoods of $\Sigma$ and then glue.  The composition $[\xi'\circ \xi]$ does not depend on the choices (see Remark~\ref{rmk: contractible set}). The
associativity and unit axioms are easily verified.

\begin{rmk} \label{rmk: contractible set}
The set of contact structures on $\Sigma\times[-\varepsilon,\varepsilon]$ which have dividing set $\Gamma$ with respect to the contact vector field $X=\bdry_t$ is
contractible.  For this reason, we will suppress the collar $\Sigma\times[-\varepsilon,\varepsilon]$ in the rest of the paper.
\end{rmk}

\begin{notation}
In what follows we abuse notation and write $\xi\in\Hom_{\mathbf{Cont}(\Sigma)}(\g,\g')$ to mean the homotopy class of a contact structure $\xi$.
\end{notation}

\subsubsection{Isotopy of dividing curves and the weak identity morphism} \label{ssec: isotopy morphism}

Suppose $\Gamma_0, \Gamma_1\in \mathfrak{ob}(\mathbf{Cont}(\Sigma))$ and $\Gamma_t$, $t\in[0,1]$, is an isotopy of dividing curves from $\Gamma_0$ to $\Gamma_1$. 
\begin{defn}
A contact structure $\xi$ on $\Sigma\times[0,1]$ is a {\em weak identity morphism from $\Gamma_0$ to $\Gamma_1$} if there exists a a contact vector field $X$ for $\xi$ such that $X$ is transverse to all $\Sigma_t=\Sigma\times\{t\}$ and the dividing set of $\Sigma_t$ with respect to $(\xi,X)$ is $\Gamma_t$.
\end{defn}  

A weak identity morphism from $\Gamma_0$ to $\Gamma_1$ gives an {\em isomorphism} between
$\Gamma_0$ and $\Gamma_1$, since one can similarly define its
inverse morphism $\Gamma_1\stackrel{\xi'}\to\Gamma_0$ which is also a weak identity morphism.

The space of dividing curves of a fixed isotopy type has trivial
fundamental group, except when $\Sigma=S^2$ or $T^2$, or when
$\Gamma$ has a homotopically trivial component. Consider the
situation where $\Sigma=T^2=\R^2/\Z^2$. Suppose $\Gamma_0=\Gamma_1$
consists of two parallel, homotopically nontrivial curves of slope
$\infty$.  Let $\xi_0$ be the $[0,1]$-invariant contact structure
with dividing set $\Gamma$ on $\Sigma\times\{t\}$ for all
$t\in[0,1]$. If $\phi: T^2\times[0,1]\stackrel\sim\rightarrow
T^2\times[0,1]$ is the diffeomorphism $(x,y,t)\mapsto (x+t,y,t)$,
then let $\xi_1=\phi^*\xi_0$. The contact structures $\xi_0$ and
$\xi_1$ are not isotopic relative to the boundary.  (However, they
are isotopic when the dividing sets are allowed to move freely.)
Similarly, when $\Gamma$ has a homotopically trivial component, we can take the homotopically trivial component and isotop it around a
nontrivial loop in $\Sigma$.

\subsubsection{Bypass attachment}

The most basic nontrivial morphism comes from a bypass attachment.  When attaching a bypass along $\delta$ to $(\Sigma,\Gamma)$ we need to Legendrian realize $\delta$ and $\bdry U$, where $U\subset \Sigma$ is a disk neighborhood of $\delta$ which transversely intersects $\Gamma$ along three arcs.  This can be done using the Legendrian realization principle of \cite{H1}, which states that there exists a homotopy of contact structures $\{\xi_s\}_{s\in[0,1]}$ on $\Sigma\times[-\varepsilon,\varepsilon]$ such that:
\begin{enumerate}
	\item $\xi_0$ is a given contact structure on $\Sigma\times[-\varepsilon,\varepsilon]$ which is $t$-invariant with dividing sets $\Gamma\times\{t\}$, $t\in[0,1]$,
	\item $\bdry_t$ is a contact vector field on $\Sigma\times[-\varepsilon,\varepsilon]$ with dividing sets $\Gamma\times\{t\}$ for all $\xi_s$, $s\in[0,1]$, and $t\in[-\varepsilon,\varepsilon]$, and 
	\item $\delta$ and $\bdry U$ are Legendrian with respect to $\xi_1$. 
\end{enumerate}
Since we are taking homotopy classes of contact structures in the definition of  $\mathbf{Cont}(\Sigma)$ in Section~\ref{subsubsection: category cont}, we may assume that the Legendrian realization automatically takes place when attaching bypasses. 

A bypass attachment of $D$ from the front along $\delta$ depends on the choices of $U\supset \delta$ and $\phi:U\stackrel\sim\to D^2$ by (B1)--(B3) from Section~\ref{subsubsection: bypasses}.  Two bypass attachments with the same $\delta$ and $U$ are related by a weak identity morphism which is ``supported on'' $U$.

Every morphism $\zeta\in \Hom_{\mathbf{Cont}(\Sigma)}(\Gamma,\Gamma')$ can be written as a composition of bypass attachment morphisms (or {\em bypass morphisms} for short), followed by a weak identity morphism; see \cite{H2}.

\subsubsection{Connected components of $\mathbf{Cont}(\Sigma)$}

Consider the following map $\phi$ which partitions the set of dividing curves
according to their Euler class = first Chern class (we will often refer to $\phi$
as the ``Spin$^c$-map''):
$$ \phi: \mathfrak{ob}(\mathbf{Cont}(\Sigma))\to \Z,$$
$$\Gamma\mapsto \chi(R_+(\Gamma))-\chi(R_-(\Gamma)).$$
Here $R_+(\Gamma)$ is the positively oriented subsurface of $\Sigma$
whose boundary is $\Gamma$; $R_-(\Gamma)$ is the negatively
oriented subsurface which is the complement of $R_+(\Gamma)$ in
$\Sigma$; and $\chi$ is the Euler characteristic.

We leave it to the reader to verify that the set $\phi^{-1}(i)$ is connected, i.e., for any pair $\Gamma$, $\Gamma'$ with the same $\phi$ value, there is a sequence of bypass morphisms from $\Gamma$ to $\Gamma'$. We will write $\mathbf{Cont}(\Sigma,i)$ for the full subcategory of $\mathbf{Cont}(\Sigma)$ whose objects are $\phi^{-1}(i)$. Then we have
$$\mathbf{Cont}(\Sigma)=\coprod_{i\in \Z} \mathbf{Cont}(\Sigma,i).$$
We will often refer to $\mathbf{Cont}(\Sigma,i)$ as a {\em connected component of $\mathbf{Cont}(\Sigma)$.}

\subsubsection{``Zero objects'' and ``zero morphisms''} \label{ssec: zero object}

A ``zero object'' in $\mathbf{Cont}(\Sigma,i)$ is a dividing set $\Gamma$ with an overtwisted neighborhood and a ``zero morphism'' is a homotopy class of overtwisted contact structures. Recall that, according to Giroux~\cite{Gi1}, $\Gamma$ is not a ``zero
object'' if and only if either $\Sigma=S^2$ and $\Gamma$ is connected,
or $\Sigma\not=S^2$ and $\Gamma$ has no homotopically trivial
component.

Recall that, by Eliashberg's theorem~\cite{El1}, there
is a unique overtwisted contact structure in each homotopy class of
$2$-plane field. Hence there are as many ``zero morphisms'' as there
are homotopy classes of $2$-plane fields in each
$\Hom_{\mathbf{Cont}(\Sigma)}(\Gamma,\Gamma')$; this problem is remedied when we pass to the
universal cover of the contact category in Section~\ref{subsection: universal cover}.

\subsubsection{The contact category $\mathcal{C}(\Sigma)$} \label{subsubsection: contact category}

We are now in a position to define the contact category
$$\mathcal{C}(\Sigma)=\coprod_{i\in \Z} \mathcal{C}(\Sigma,i),$$
which is an $\F$-linear (and in particular a pre-additive) category. The objects of $\mathcal{C}(\Sigma,i)$ are the same as those of $\mathbf{Cont}(\Sigma,i)$ and $\Hom_{\mathcal{C}(\Sigma,i)}(\g,\g')$ is the $\F$-vector space generated by the homotopy classes of tight contact structures of $\Hom_{\mathbf{Cont}(\Sigma,i)}(\g,\g')$.

We are identifying all the ``zero morphisms'' in $\Hom_{\mathbf{Cont}(\Sigma,i)}(\g,\g')$ into the unique zero morphism of $\Hom_{\mathcal{C}(\Sigma,i)}(\g,\g')$.
Moreover, the ``zero objects'' in $\mathbf{Cont}(\Sigma,i)$ become genuine zero objects in $\mathcal{C}(\Sigma,i)$.  They are isomorphic to each other; we choose one
zero object and denote it by $\mathbf{0}$.

\subsubsection{The categories $\mathbf{Cont}(\Sigma,F)$ and $\mathcal{C}(\Sigma,F)$}

Let $\Sigma$ be a compact oriented surface with boundary and let $F\subset \bdry\Sigma$ be a finite set of points which divides $\bdry\Sigma$ into alternating positive and
negative regions $R_+(F)$ and $R_-(F)$, i.e., the signs on both sides of any point in $F$ are opposite.  (In other words, $F$ is a set of {\em oriented} points
which is the boundary of a $1$-dimensional submanifold of $\bdry \Sigma$.)

The objects of $\mathbf{Cont}(\Sigma,F)$ are dividing sets $\Gamma$ with endpoints on
$F$, subject to the condition that the signs on $\bdry\Sigma-F$ and
the signs on $\Sigma-\Gamma$ agree. The morphisms
$\Gamma\stackrel\xi\rightarrow \Gamma'$ are homotopy classes of
contact structures on $\Sigma\times[0,1]$ so that the dividing set
on $\Sigma\times\{0\}$ is $\Gamma$, the dividing set on
$\Sigma\times\{1\}$ is $\Gamma'$, and the dividing set on
$\bdry\Sigma\times[0,1]$ is $\overline{F}\times[0,1]$, where
$\overline{F}$ is the set consisting of one point on each component
of $\bdry\Sigma-F$.

The contact category 
$$\mathcal{C}(\Sigma,F)=\coprod_{i\in \Z} \mathcal{C}(\Sigma,F,i)$$ 
is defined in the same way as in Section~\ref{subsubsection: contact category}: the objects of $\mathcal{C}(\Sigma,i)$ are the same as those of $\mathbf{Cont}(\Sigma,F,i)$ and $\Hom_{\mathcal{C}(\Sigma,F, i)}(\g,\g')$ is the $\F$-vector space generated by the tight contact structures of $\Hom_{\mathbf{Cont}(\Sigma,F,i)}(\g,\g')$.

\begin{notation}
From now on, $\Hom$ without subscripts will always mean $\Hom_{\mathcal{C}(\Sigma)}$ or $\Hom_{\mathcal{C}(\Sigma,F)}$.
\end{notation}

\subsubsection{Generators and relations} 

%We now give a description of the generators and relations in $\mathcal{C}(\Sigma)$ or $\mathcal{C}(\Sigma,F)$.  Recall that every contact structure $\Gamma\stackrel\xi\to\Gamma'$ can be written as a composition of bypass morphisms, modulo an isotopy of dividing sets.  The description of the relations is due to Bin Tian.\footnote{Not to be confused with the second author.}
%
%\begin{thm}[Bin Tian]\label{thm: relations in contact category}
%Given a contact structure $\Gamma\stackrel\xi\to \Gamma'$, any
%two sequences of bypass attachments from $\Gamma$ to $\Gamma'$
%which compose to give $\xi$ can be taken to one another via the
%following two types of operations:
%\begin{enumerate}
%\item[(R$_1$)] Far commutativity --- given two disjoint bypass arcs of attachment,
%we can reverse the order in which the bypasses are attached.
%\item[(R$_2$)] Adding a trivial bypass.
%\end{enumerate}
%\end{thm}
%
%Theorem~\ref{thm: relations in contact category} is straightforward to prove for the contact category of a disk when $\xi$ is tight and will be sketched in Section~\ref{section: contact category of disk}.
%
%
%It is not hard to see that relations (R$_1$) and (R$_2$) are equivalent to (R$_1$) and (R$_2'$):
%\begin{enumerate}
%\item[(R$_2'$)]  Bypass rotation --- referring to Figure~\ref{rotate}, if we attach bypasses along $\delta_0$ and then along $\delta_1$, then the resulting contact structure is equivalent to attaching a bypass along $\delta_1$ only.
%\end{enumerate}

We now give a description of the generators and relations in $\mathcal{C}(\Sigma)$ or $\mathcal{C}(\Sigma,F)$. Recall that every $\xi\in \Hom_{\mathbf{Cont}(\Sigma)}(\Gamma,\Gamma')$ can be written as a composition of bypass morphisms and weak identity morphisms.  The description of the relations is due to Bin Tian.\footnote{Not to be confused with the second author.}

\begin{thm}[Bin Tian]\label{thm: relations in contact category}
Given $\xi\in \Hom_{\mathbf{Cont}(\Sigma)}(\Gamma,\Gamma')$, any 
%\marginpar{\cg \tiny modified statement of theorem which refers to weak identity morphisms now}
two sequences of bypass attachments and weak identity morphisms from $\Gamma$ to $\Gamma'$
which compose to give $\xi$ can be taken to one another via the
following two types of operations:
\begin{enumerate}
\item[(R$_1$)] Far commutativity --- given two disjoint bypass arcs of attachment,
we can reverse the order in which the bypasses are attached.
\item[(R$_2$)] Adding a weak identity morphism.
\end{enumerate}
\end{thm}

Theorem~\ref{thm: relations in contact category} is straightforward to prove for the contact category of a disk when $\xi$ is tight and the proof will be sketched in Section~\ref{section: contact category of disk}.

It is easy to see that relations (R$_1$) and (R$_2$) imply (R$_2'$): 
%\marginpar{\cg \tiny switched to easier statement}
\begin{enumerate}
\item[(R$_2'$)]  Bypass rotation --- referring to Figure~\ref{rotate}, if we attach bypasses along $\delta_0$ and then along $\delta_1$, then the resulting contact structure is homotopic to the one obtained from attaching a bypass along $\delta_1$ and followed by a weak identity morphism.
\end{enumerate}

\subsubsection{Opposite category}

The opposite category $\mathcal{C}(\Sigma)^{op}$ of $\mathcal{C}(\Sigma)$ is obtained by
reversing the arrows.  It is not hard to see that
$\mathcal{C}(\Sigma)^{op}$ is equivalent to $\mathcal{C}(-\Sigma)$
via the contravariant functor which sends $\Gamma$ to $-\Gamma$ and
the morphism $\Gamma\stackrel\xi\rightarrow \Gamma'$ to
$-\Gamma\stackrel\xi\leftarrow -\Gamma'$.  Observe that when we
switch from $\Sigma$ to $-\Sigma$, the positive and negative regions
of $\Sigma-\Gamma$ get switched, i.e., $\Gamma$ gets sent to
$-\Gamma$.

\subsection{Bypass exact triangles} A sequence of bypass attachments
gives a triangle, called the {\em bypass exact triangle}, as
follows: Suppose the initial configuration is $\Gamma_1$. Pick an
arc of attachment $\delta\subset \Sigma$ and its neighborhood $U$.
Apply a bypass attachment from the front along $\delta$ to obtain
$\Gamma_2$.  Now, inside $U$, there is a unique arc of attachment
$\delta'$ which intersects all three arcs of $\Gamma_2\cap U$. A
bypass attachment from the front along $\delta'$ yields $\Gamma_3$.
Similarly, a third bypass attachment from the front along $\delta''$
yields $\Gamma_1$. This is summarized in
Figure~\ref{bypass-triangle}. For convenience, we say that the above
bypass exact triangle {\em starts at $(\Sigma,\Gamma_1,\delta)$}.
\begin{figure}[ht]
\begin{overpic}[width=4cm]{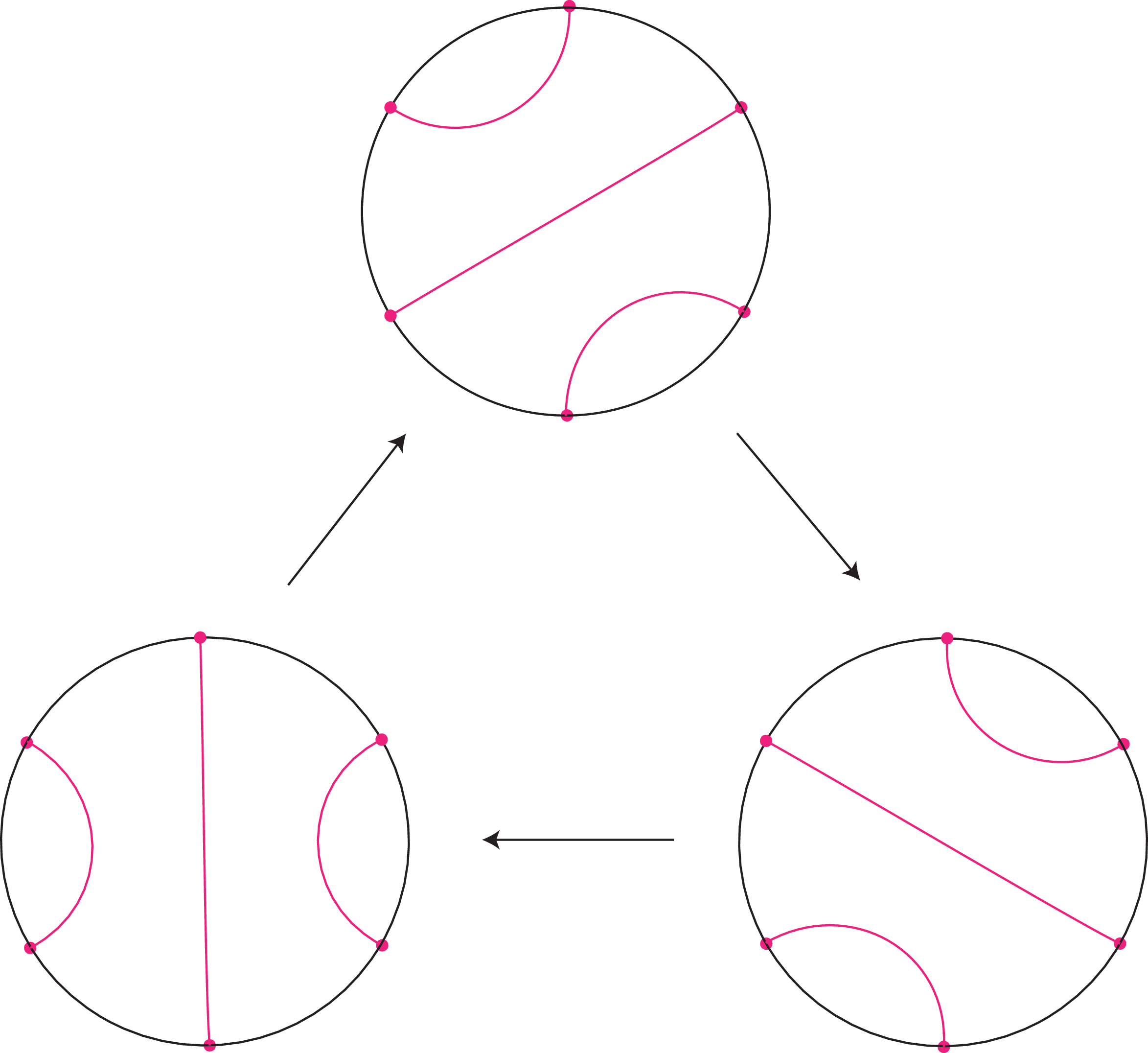}
\end{overpic}
\caption{The top is $\Gamma_1$, the bottom right is $\Gamma_2$, and
the bottom left is $\Gamma_3$.} \label{bypass-triangle}
\end{figure}

We claim that attaching two bypasses in succession inside $U$
creates an overtwisted contact structure. Indeed, $\delta'$ is the
anti-bypass arc of the bypass arc $\delta$.  Hence the bypass along
$\delta'$ from the back and the bypass along $\delta'$ from the
front glue to give an overtwisted disk. Therefore, the bypass
triangle will have the property that the composition of any two
successive edges is the zero morphism.

\s\n {\bf Examples of bypass triangles.}

\s\n (i) (Identity triangle) Consider the morphism
$\Gamma_1=\Gamma\stackrel{id}\rightarrow\Gamma=\Gamma_2$ which is
equivalently obtained by attaching a {\em trivial bypass}. Now,
attaching the next bypass yields $\Gamma_3$ which has a
homotopically trivial component.  Hence $\Gamma_3\cong \mathbf{0}$.
$$\begin{diagram} \label{eqn: identity exact triangle}
\Gamma_1=\Gamma & & \rTo^{id} & & \Gamma_2=\Gamma\\
&\luTo & & \ldTo &\\
&& \Gamma_3\cong \mathbf{0} &&
\end{diagram}$$

\n (ii) (Fold-unfold triangle) Consider the morphism
$\Gamma_1=\Gamma\rightarrow \Gamma'=\Gamma_2$ corresponding to a
bypass of fold type (i.e., a bypass such that $\Gamma_2$ is the disjoint union of $\Gamma_1$ and two parallel homotopically nontrivial curves). The next bypass attachment is an unfold type and the third bypass attachment is overtwisted. (The map
$\Gamma\stackrel{0}\rightarrow \Gamma$ factors into
$\Gamma\stackrel{\mbox{\tiny fold}}\rightarrow
\Gamma''\stackrel{\mbox{\tiny unfold}}{\xrightarrow{\hspace{.8cm}}}  \Gamma$ which
glues into an overtwisted contact structure.)
%$$\begin{diagram} \Gamma & & \rTo^{\mbox{\tiny fold}} & & \Gamma'\\ &\luTo_{0} & & \ldTo_{\mbox{\tiny unfold}} &\\ && \Gamma&& \end{diagram}$$

\begin{figure}[ht]
\begin{overpic}
[scale=0.25]{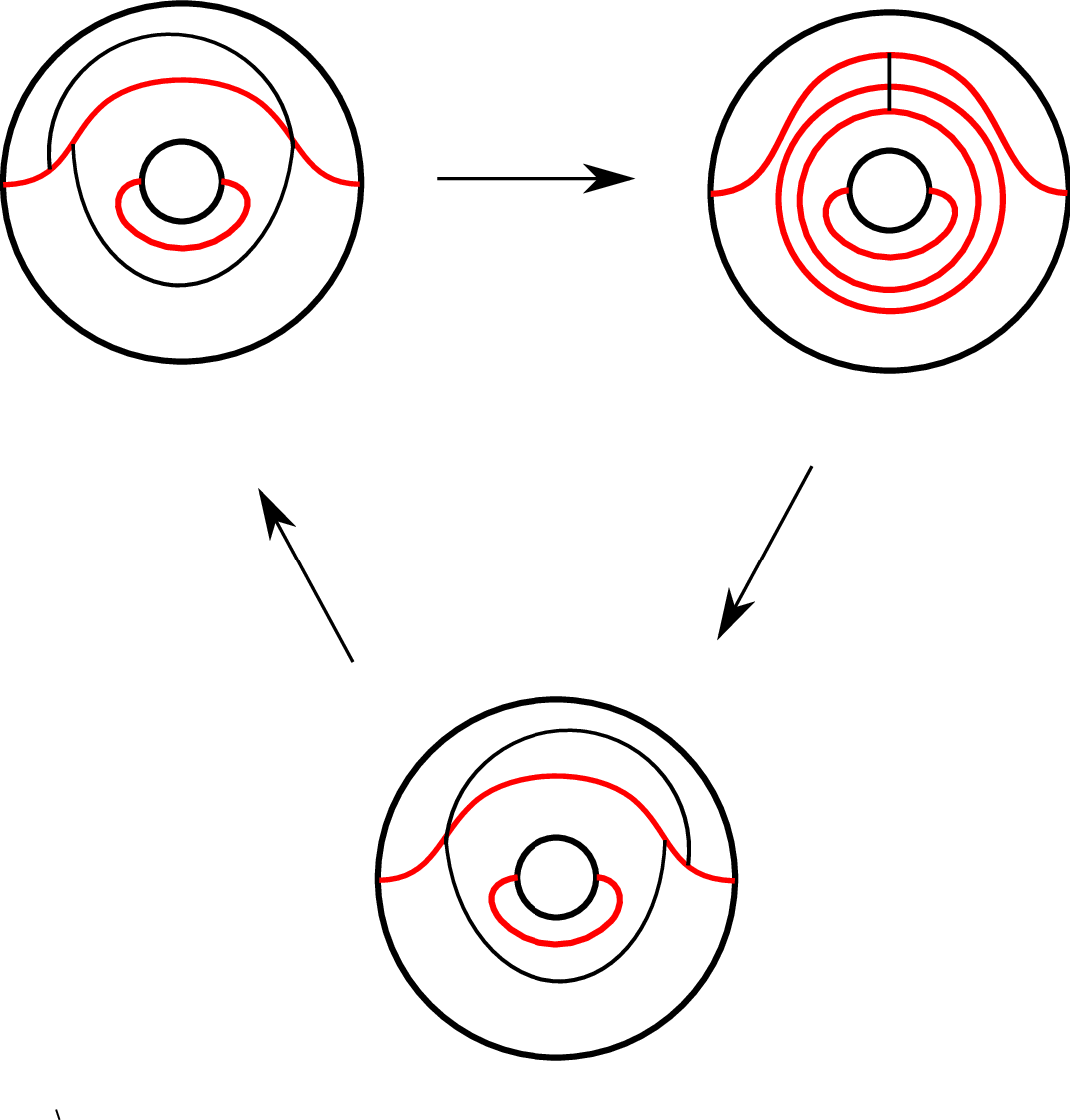}
\put(10,55){$\g$}
\put(80,55){$\g'$}
\put(60,0){$\g$}
\put(40,88){fold}
\put(70,45){unfold}
\put(32,45){$0$}
\end{overpic}
\caption{Fold-unfold triangle.}
\label{folding}
\end{figure}

%\s\n {\bf Extra exact triangles.}  For technical reasons we need
%another type of exact triangle, which we call the {\em extra
%triangle}.  Fix a dividing set $\Gamma$.  Then let
%$\Gamma_1\rightarrow \Gamma_2$ be a composition
%$\Gamma_1\stackrel{\mbox{\tiny unfold}}\rightarrow \Gamma
%\stackrel{\mbox{\tiny fold}}\rightarrow \Gamma_2$. Such a morphism
%can be included into $\Gamma\stackrel{id}\rightarrow \Gamma$, and is
%tight. Similarly define $\Gamma_2\rightarrow \Gamma_3$ and
%$\Gamma_3\rightarrow\Gamma_1$. We require that the extra exact
%triangle factors into three fold-unfold triangles of type:
%$$0\rightarrow \Gamma\stackrel{\mbox{\tiny fold}}\rightarrow
%\Gamma_i\stackrel{\mbox{\tiny unfold}}\rightarrow \Gamma\rightarrow
%0.$$

\subsection{Octahedral axiom}

One of the primary motivations for introducing the contact category was that the bypass triangles often satisfy the octahedral axiom.  In other words, there was evidence that the contact category could be embedded inside some sort of ``triangulated envelope" while still preserving the bypass triangles.  Theorem~\ref{thm embed} realizes this for the contact category of the disk.

We briefly review the octahedral axiom. Refer to Figure~\ref{octahedral}. If
there are three exact triangles $(A,B,C')$, $(B,C,A')$, $(A,C,B')$,
so that the face $ABC$ commutes, then there is a fourth exact
triangle $(C',B',A')$ (i) which makes the other three faces $A'BC'$,
$A'CB'$, $AB'C'$ commute and (ii) such that the compositions
$B'\stackrel{t}\rightarrow A\stackrel{j}\rightarrow B$,
$B'\stackrel{q}\rightarrow A'\stackrel{m}\rightarrow B$ agree and
the compositions $B\stackrel{k}\rightarrow C\stackrel{p}\rightarrow
B'$, $B\stackrel{n}\rightarrow C'\stackrel{s}\rightarrow B'$ agree.

\begin{figure}[ht]
\begin{overpic}[width=7.5cm]{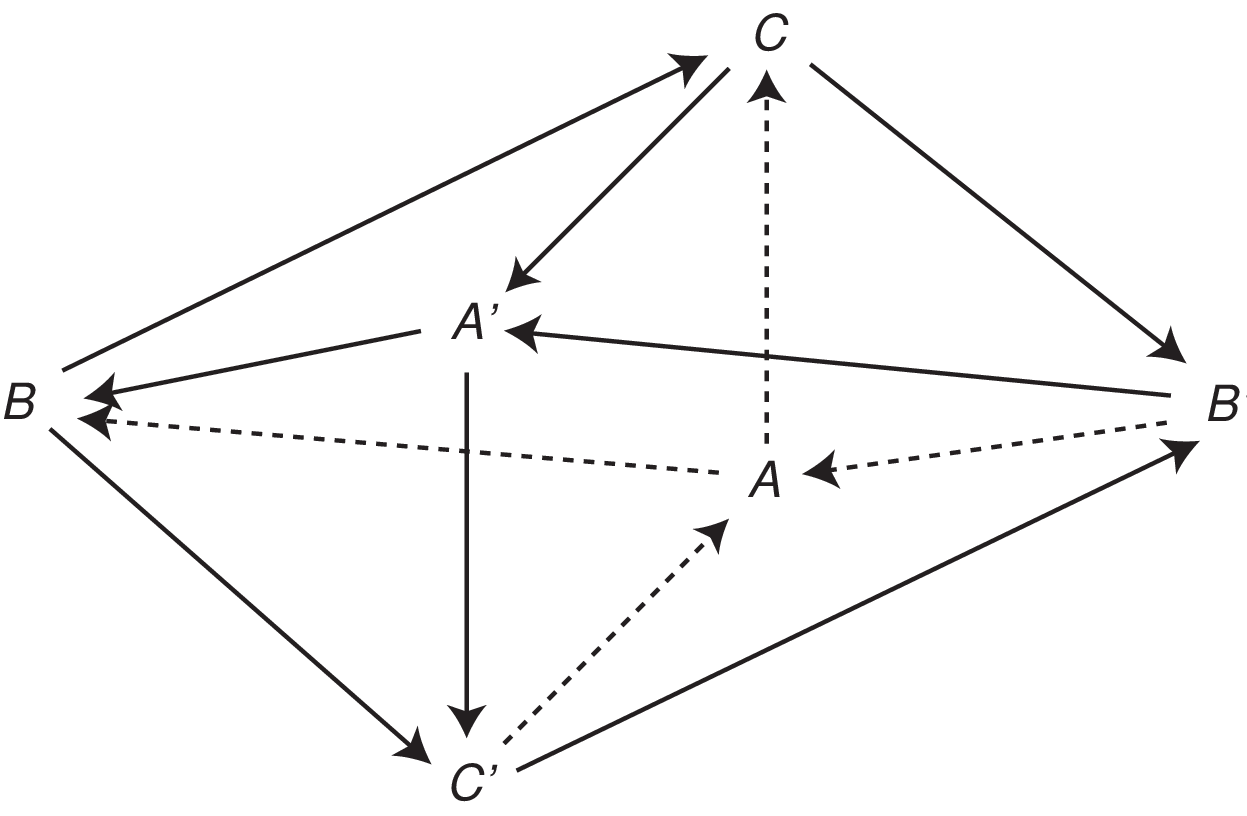}
\put(63,47){\tiny $i$} \put(22,27){\tiny $j$} \put(29,49){\tiny $k$}
\put(35.2,18){\tiny $l$} \put(24.3,38.2){\tiny $m$}
\put(15.3,17.5){\tiny $n$} \put(49,48){\tiny $o$}
\put(80.3,49){\tiny $p$} \put(75,36.7){\tiny $q$}
\put(47.9,16.3){\tiny $r$} \put(68,14.5){\tiny $s$}
\put(75,26.5){\tiny $t$}
\end{overpic}
\caption{} \label{octahedral}
\end{figure}

We present some evidence for the octahedral axiom where $\Sigma=D^2$ and $\#F=8$. The six dividing sets $\Gamma$ with $\phi(\Gamma)=1$, where $\phi$ is the Spin$^c$-map, form the octahedron given in Figure~\ref{octahedral2}, where all the arrows are nontrivial bypass morphisms.

\begin{figure}[ht]
\begin{overpic}[width=7.5cm]{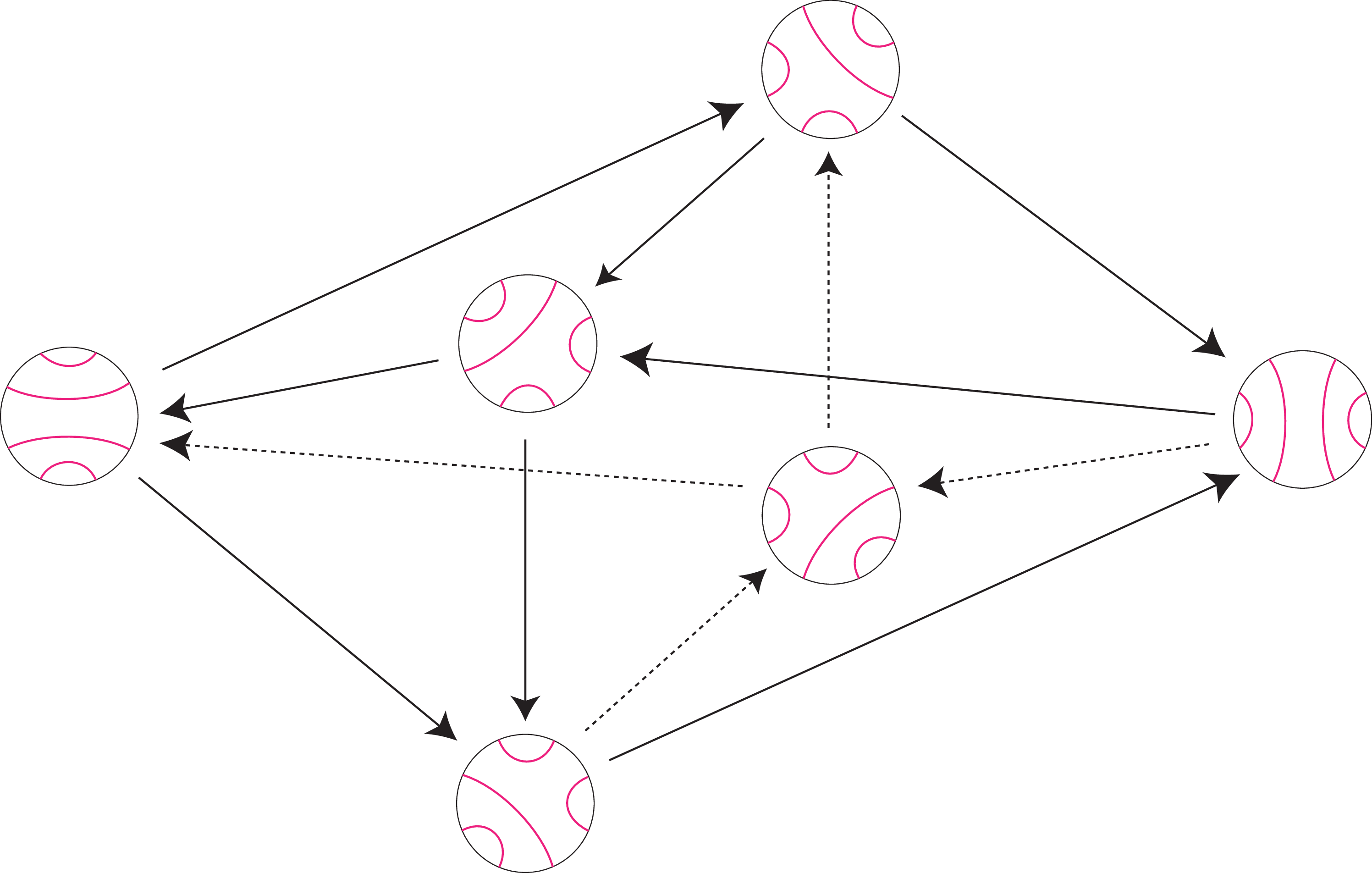}
\end{overpic}
\caption{} \label{octahedral2}
\end{figure}

Using the labeling from Figure~\ref{octahedral}, let us consider the
compositions $B'\stackrel{t}\rightarrow A\stackrel{j}\rightarrow B$
and $B'\stackrel{q}\rightarrow A'\stackrel{m}\rightarrow B$ in
Figure~\ref{octahedral2}. Both correspond to the same two bypass
moves along disjoint arcs of attachment, and differ only in the
order in which the attachment takes place.  Hence both compositions
give the same contact structure up to isotopy, and agree as
morphisms $B'\rightarrow B$. Similarly, the compositions $sn$ and
$pk$ agree.

Let us also discuss the commutativity of the triangle $A'CB'$, for
example. For this we use bypass rotation \cite[Lemma~4.2]{HKM}. The arc of attachment $\delta$ that gives rise to the
morphism $C\stackrel{o}\rightarrow A'$ can be rotated to the left to
give an arc of attachment $\delta'$ for the morphism
$C\stackrel{p}\rightarrow B'$. More precisely, inside a small
neighborhood of the union of $\Sigma=D^2$ and the bypass half-disk
along $\delta$, there exists a bypass half-disk along $\delta'$.
Moreover, the image of the arc $\delta$ on $B'$ (after the bypass
attachment along $\delta'$) is precisely the arc of attachment for
$B'\stackrel{q}\rightarrow A'$. Therefore, $C\stackrel{o}\rightarrow
A'$ can be factored into
$$C\stackrel{p}\rightarrow B'\stackrel{q}\rightarrow
A'\stackrel{x}\rightarrow A'.$$ By Eliashberg's uniqueness theorem
for tight contact structures on the $3$-ball, $x=id$ and it follows
that $o=qp$.

\subsection{The universal cover of the contact category}
\label{subsection: universal cover}

In this subsection we describe the {\em universal covers} of the contact categories $\mathcal{C}(\Sigma)$ and $\mathcal{C}(\Sigma,F)$.

\subsubsection{The universal cover}

Let $\mathcal{C}(\Sigma,i)$ be a connected component of $\mathcal{C}(\Sigma)$ and let $\Gamma_0\in \mathfrak{ob}(\mathcal{C}(\Sigma,i))$. {\em The universal cover of $\mathcal{C}(\Sigma,i)$ with basepoint $\Gamma_0$} is the category $\widetilde{\mathcal{C}}(\Sigma,i,\Gamma_0)$, together with the {\em covering functor}
$$\pi: \widetilde{\mathcal{C}}(\Sigma,i,\Gamma_0)\to \mathcal{C}(\Sigma,i),$$
defined as follows:

The objects are given by $(\Gamma_0\stackrel{[\zeta]}\to\Gamma)$, where $[\zeta]$ is a homotopy class of $2$-plane fields which are contact near $\Sigma\times\{0,1\}$ and have dividing sets $\Gamma_0$ and $\Gamma$.  The objects are also denoted by $(\g,[\zeta])$. 

We define $\Hom_{\widetilde{\mathcal{C}}(\Sigma,i,\Gamma_0)}((\g,[\zeta]),(\g',[\zeta']))$ to be the $\F$-vector space generated by homotopy classes of tight contact structures $\xi\in \Hom(\g,\g')$ such that $[\xi\circ \zeta]=[\zeta']$, where $[\xi\circ \zeta]$ is the homotopy class of $2$-plane fields on $\Sigma\times[0,1]$ obtained by concatenating  $\zeta$ and $\xi$. 

The functor $\pi$ takes $(\Gamma_0\stackrel{[\zeta]}\to\Gamma)$ to $\Gamma$ and takes
$$(\Gamma_0\stackrel{[\zeta]}\to\Gamma) \stackrel{\xi}\to
(\Gamma_0\stackrel{[\zeta']=[\xi\circ\zeta]}{\xrightarrow{\hspace{1.2cm}}} \Gamma')\in \Hom_{\widetilde{\mathcal{C}}(\Sigma,i,\Gamma_0)}((\g,[\zeta]),(\g',[\zeta'])),$$
to $\Gamma\stackrel\xi\to\Gamma'\in \Hom_{ \mathcal{C}(\Sigma,i)}(\g,\g')$. 

The universal cover $\widetilde{\mathcal{C}}(\Sigma,F,i,\Gamma_0)$ of $\mathcal{C}(\Sigma,F,i)$ is defined similarly.  Since $\Gamma_0$ determines the integer $i$, we will sometimes suppress the $i$ and write $\widetilde{\mathcal{C}}(\Sigma,\Gamma_0)$ or $\widetilde{\mathcal{C}}(\Sigma,F,\Gamma_0)$.

\subsubsection{$2$-plane fields}

Suppose that $\Sigma$ is a closed surface.
The preimage $\pi^{-1}(\Gamma)$ of $\Gamma\in\mathfrak{ob}(\mathcal{C}(\Sigma,i))$ is isomorphic to the
$\Z$-module $\Z\oplus H_1(\Sigma;\Z)$, albeit not naturally. Fix a
trivialization of the tangent bundle of $\Sigma\times[0,1]$ and a
reference $2$-plane field $(\Gamma_0\stackrel{[\zeta_0]}\to\Gamma)$.

We explain how to define the map
$$\Theta:\pi^{-1}(\Gamma)\to \Z\oplus H_1(\Sigma;\Z),$$
$$(\Gamma_0\stackrel{[\zeta]}\to\Gamma) \mapsto
(\Theta_1(\zeta),\Theta_2(\zeta));$$ 
see \cite{GH,Hu} for more details:  Using a relative
version of the Pontryagin-Thom construction, we can assign a framed
tangle in $\Sigma\times[0,1]$ to any $\zeta$. Here the framed tangle
is properly embedded and has endpoints on $\Sigma\times\{0,1\}$. 
To the difference $\zeta-\zeta_0$ we can assign a framed link $\mathcal{L}$
in (the interior of) $\Sigma\times[0,1]$. 

Any framed link $\mathcal{L}$ in $\Sigma\times[0,1]$ is the union of the following two types of links, up to framed cobordism: (i) a (not necessarily connected) $1$-manifold $C$ on $\Sigma\times\{1/2\}$, with framing coming from the surface, and (ii) a framed unknot.  The proof is a slight generalization of the usual Pontryagin-Thom proof of $\pi_3(S^2)\simeq \Z$, whose elements are classified by framed unknots: Let $\pi_\Sigma:\Sigma\times[0,1]\to\Sigma$ be the projection to $\Sigma$.  Without loss of generality assume that $\pi_\Sigma(\mathcal{L})$ is an immersion with transverse crossings, and we resolve the crossings to obtain a $1$-manifold $C$ which we can view to be on $\Sigma\times\{1/2\}$.  Adjusting the framing on $C$ and resolving the crossings are both equivalent in the framed cobordism category to adding a framed unknot.  Finally a union of framed unknots is framed cobordant to a single framed unknot.
%The curve in (i) can be taken to be a nonseparating curve on $\Sigma\times\{t\}$. 

The framing of the unknot is $\Theta_1(\zeta)$ and is equal to the Hopf invariant and the homology class of $C\subset \Sigma\times\{1/2\}$ is $\Theta_2(\zeta)$. The class $\Theta_2(\zeta)\in H_1(\Sigma;\Z)$ is dual to one-half of
the first Chern class of the difference $\zeta-\zeta_0$.

Suppose that $\bdry \Sigma\not=\emptyset$, $\bdry\Sigma$ is connected, and
$\#F>0$. Let $\mathbf{b}=\{b_1,\dots,b_{2g}\}$ be a {\em basis} for
$\Sigma$, i.e., it is a collection of disjoint, properly embedded,
oriented arcs which cut $\Sigma$ up into a single polygon; in
particular, $\mathbf{b}$ can be viewed as a basis for $H_1(\Sigma,\bdry\Sigma)$.
Let us also assume that $\mathbf{b}$ is transverse to $\Gamma_0$ and
$\Gamma$. Given the basis $\mathbf{b}$, we can take disks
$D_i=b_i\times [0,1]$, whose orientation agrees with the boundary
orientation given by that of $b_i\times\{1\}$.  We can compute
$\langle c_1(\zeta),D_i\rangle$ to be
$\chi(R_+(\Gamma_{D_i}))-\chi(R_-(\Gamma_{D_i}))$ with respect to
$\zeta$.  (Without loss of generality we may assume that $\zeta$ is an overtwisted contact structure by Eliashberg's classification of overtwisted contact structures~\cite{El1}.) Then ${1\over 2}c_1(\zeta-\zeta_0)$ assigns an integer to
the disks $D_1,\dots,D_g$, and is dual to $\Theta_2(\zeta)$.

\subsubsection{Change of basepoint}

Let $\Gamma_0,\Gamma_0'\in \mathfrak{ob}(\mathcal{C}(\Sigma,i))$ be two basepoints.  If $\zeta$ is a homotopy class of $2$-plane fields which is contact near
$\Sigma\times\{0,1\}$ and has dividing sets $\Gamma_0$ and $\Gamma_0'$ that lie on $\Sigma \times \{0\}$ and $\Sigma \times \{1\}$, respectively, then $\zeta$ induces a {\em change-of-basepoint functor}
$$F_\zeta: \widetilde{\mathcal{C}}(\Sigma,\Gamma_0') \to \widetilde{\mathcal{C}}(\Sigma,\Gamma_0),$$
which is given by:
$$F((\Gamma_0'\stackrel{[\xi]}\to \Gamma_1)\stackrel{\xi'}\to (\Gamma_0'\stackrel{[\xi'\circ \xi]}{\xrightarrow{\hspace{.8cm}}}
\Gamma_2))=(\Gamma_0\stackrel{[\xi\circ\zeta]} \longrightarrow
\Gamma_1) \stackrel{\xi'}\to (\Gamma_0\stackrel{[\xi'\circ \xi\circ \zeta]}{\xrightarrow{\hspace{1.2cm}}} \Gamma_2).$$
The functor $F_\zeta$ gives an equivalence of the two categories.  Also, if $\zeta$ is a homotopy class from $\Gamma_0$ to $\Gamma_0'=\Gamma_0$, then the
functor $F_\zeta$ is a {\em deck transformation} of $\widetilde{\mathcal{C}}(\Sigma,\Gamma_0)$.

\subsubsection{Bypass exact triangles}

A bypass exact triangle
$$...\to \Gamma_1\stackrel{\xi_1}\to \Gamma_2 \stackrel{\xi_2}\to \Gamma_3 \stackrel{\xi_3}\to \Gamma_1 \to ...$$
lifts to a bypass exact triangle
$$... \to (\Gamma_0\stackrel{[\zeta]}\to\Gamma_1) \stackrel{\xi_1}  \to (\Gamma_0\stackrel{[\xi_1\circ \zeta]}{\xrightarrow{\hspace{.8cm}}}  \Gamma_2) \stackrel{\xi_2}\to
(\Gamma_0\stackrel{[\xi_2\circ \xi_1\circ \zeta]}{\xrightarrow{\hspace{1.5cm}}}  \Gamma_3)$$
$$ \stackrel{\xi_3}\to (\Gamma_0\stackrel{[\xi_3\circ\xi_2\circ \xi_1\circ \zeta]}{\xrightarrow{\hspace{1.8cm}}}  \Gamma_1)\to....$$
For convenience, let us write $\zeta'=\xi_3\circ\xi_2\circ\xi_1\circ \zeta$.

Let $T: \widetilde{\mathcal{C}}(\Sigma,\Gamma_0)\to
\widetilde{\mathcal{C}}(\Sigma,\Gamma_0)$ be the deck transformation
which, for each fiber $\pi^{-1}(\Gamma)$ and identification
$\Theta:\pi^{-1}(\Gamma)\stackrel\sim\to \Z\oplus H_1(\Sigma;\Z)$,
sends $(m,x)\in \Z\oplus H_1(\Sigma;\Z)$ to $(m-1,x)$. In other
words, the shift functor is a grading shift which drops the Hopf
invariant (i.e., the linking number of the framed unknot) by one
without changing the relative Spin$^c$-structure.

The following theorem, due to Huang~\cite{Hu}, shows that the
functor $T$ is the shift functor for the bypass exact triangles in
$\widetilde{\mathcal{C}}(\Sigma,\Gamma_0)$.

\begin{thm}[Huang \cite{Hu}] \label{thm: huang}
The homotopy classes of $(\Gamma_0\stackrel\zeta\to\Gamma_1)$ and
$(\Gamma_0\stackrel{\zeta'}\rightarrow \Gamma_1)$ have the same
relative Spin$^c$-structure, and the Hopf invariant of $\zeta$ is
one higher than that of $\zeta'$.
\end{thm}

\section{Contact category of a disk} \label{section: contact category of disk}

In the rest of this paper we restrict attention to contact categories of the disk.  Letting $\Sigma=D^2$, $\# F=2n+2$,  $0\leq e\leq n$, we consider $\mathcal{C}(D^2,F, n-2e)$ and $\widetilde{\mathcal{C}}(D^2,F, n-2e)$. The basepoint $\Gamma^0$ is arbitrary at this point.  If $\Gamma\in \mathfrak{ob} (\mathcal{C}(D^2,F, n-2e))$, then
$$\chi_+(\Gamma):=\chi(R_+(\Gamma))=n-e+1,\quad \chi_-(\Gamma):=\chi(R_-(\Gamma))=e+1.$$
We also write $\chi_\pm$ if $\Gamma$ is understood.

The $n+1$ arcs of $R_+(F)$ (called ``positive arcs") are labeled $\ul{0},\ul{1},\dots,\ul{n}$ in clockwise order around $\bdry D^2$.  The arc $\ul{0}$ is the ``based arc'', analogous to a basepoint.  We will often write $D^2_n$ for $(D^2,F)$ with $\# F=2n+2$ and a fixed labeling of $R_+(F)$.  We also assume that the arcs of $R_+(F)$ are evenly spaced around $\bdry D^2$.

\begin{notation}
The labels of the arcs of $R_+(F)$ will be underlined throughout the paper. We will write $\ul{s} < \ul{t}$ when we mean $s<t$.
\end{notation}

\subsection{Skeletal subcategory $\cne$} 

The category $\mathcal{C}(D^2,F, n-2e)$ has an uncountable number of objects since isotopic dividing sets are treated as different objects (cf.\ Remark~\ref{rmk object isotopy}).  However,
\begin{itemize}
\item any two isotopic dividing sets are isomorphic in $\mathcal{C}(D^2,F, n-2e)$ via a weak identity morphism and 
\item any dividing set with a contractible component is isomorphic to a zero object by Sections~\ref{ssec: zero object} and \ref{subsubsection: contact category}.
\end{itemize}
Hence $\mathcal{C}(D^2,F, n-2e)$ has only finitely many isomorphism classes of objects. In particular, the set of isomorphism classes of nonzero objects in $\mathcal{C}(D^2,F, n-2e)$ is in bijection with the set of isotopy classes of dividing sets without closed components, which in turn is in bijection with the set of crossingless matchings with $\chi_+=n-e+1$ and $\chi_-=e+1$. 

Let $\cne$ be a skeletal subcategory of $\mathcal{C}(D^2,F, n-2e)$, obtained by choosing one representative from each isomorphism class of objects of $\mathcal{C}(D^2,F, n-2e)$ and taking the full subcategory of $\mathcal{C}(D^2,F, n-2e)$ with these objects.    Let $\tcne$ be a skeletal subcategory of $\widetilde{\mathcal{C}}(D^2,F, n-2e)$.

{\em We now shift our perspective slightly and work with $\cne$ and $\tcne$ in the rest of the paper.  At this point it would be convenient to slightly change the definition of a bypass from $\Gamma$ to $\Gamma'$ so that:}
$$\mbox{new bypass} = \mbox{old bypass, followed by a weak identity morphism}.$$
Given nonzero objects $\g,\g'$ of $\cne$ with a (new) bypass from $\g$ to $\g'$, the (new) bypass does not depend on the choices of $U$ and $\phi$ that appear in the definition of the old bypass as well as the weak identity morphism. In particular, the relation (R$_2$) in Theorem~\ref{thm: relations in contact category} can be rephrased as ``adding a trivial bypass'' in the case of a disk.

\subsection{Compositions in $\cne$}

Given $\g, \g'\in \mathfrak{ob}(\cne)$, $\gamma_{\g, \g'}$ denotes the dividing set on $\bdry (D^2 \times [0,1])$ obtained by {\em edge rounding} $\g$ on $D^2 \times \{0\}$, $\g'$ on $D^2 \times \{1\}$, and the vertical dividing set on $\bdry D^2 \times [0,1]$.
See \cite[Lemma 3.11]{H1} for the definition of {\em edge rounding} of two dividing sets along a common boundary Legendrian curve.
We write $\#\gamma_{\g,\g'}$ for the number of components of $\gamma_{\g,\g'}$.  If $\#\gamma_{\g, \g'}>1$, then $\Hom(\g,\g')=0$; if $\#\gamma_{\g, \g'}=1$, then $\Hom(\g, \g')\simeq \F$ and we denote its generator by $\xi_{\g, \g'}$.

\begin{convention}
For the rest of the paper, if $\g, \g'\in \mathfrak{ob}(\cne)$, then $\Hom(\g,\g')$ is always understood to be $\Hom_{\cne}(\g,\g')$.
\end{convention}

We study the composition $\Hom(\g', \g'') \times \Hom(\g, \g') \ra \Hom(\g, \g'')$ when all three spaces are nonzero.  The following lemma is similar to \cite[Lemma 3.12]{M1}:

\begin{lemma}\label{lem stack}
Suppose that $\Hom(\g, \g'), \Hom(\g', \g''), \Hom(\g, \g'')$ are nonzero and $\g' \neq \g, \g''$.
Then the following are equivalent:
\be
\item The composition $\Hom(\g', \g'') \times \Hom(\g, \g') \ra \Hom(\g, \g'')$ is nontrivial.
\item There exists a sequence of dividing sets $\g^i$ for $0 \leq i \leq k, k\geq 1$ satisfying:
\be
\item $\g^0=\g, \g^k=\g'$;
\item each $\Hom(\g^i, \g^{i+1})$, $0 \leq i \leq k-1$, is nonzero and is generated by a bypass;
\item $\Hom(\g^i, \g'')\not=0$ for $0 \leq i \leq k$.
\ee
\item There exists a sequence of dividing sets $\g^i$ for $0 \leq i \leq k, k\geq 1$ satisfying:
\be
\item $\g^0=\g', \g^k=\g''$;
\item each $\Hom(\g^i, \g^{i+1})$, $0 \leq i \leq k-1$, is nonzero and is generated by a bypass;
\item $\Hom(\g, \g^i)\not=0$ for $0 \leq i \leq k$.
\ee
\ee
\end{lemma}

\begin{proof}
We prove the equivalence of (1) and (2). The proof of the equivalence of (1) and (3) is similar.

\s\n
(1) $\Rightarrow$ (2): The tight contact structure $\xi_{\g,\g'}$ can be written as a composition of bypasses $\xi_{\g^{k-1}, \g^{k}} \circ \cdots \circ \xi_{\g^{0}, \g^{1}}$, where $\g^0=\g$ and $\g^k=\g'$. For $0 \leq i \leq k$, the contact structure $\xi_{\g^{k}, \g''} \circ \xi_{\g^{k-1}, \g^{k}} \circ \cdots  \circ \xi_{\g^{i}, \g^{i+1}}$ is tight because it can be embedded in $\xi_{\g',\g''} \circ \xi_{\g, \g'} =\xi_{\g, \g''}$ which is tight.  Hence $\Hom(\g^i, \g'')\not=0$ for $0 \leq i \leq k$.

\s
\n (2) $\Rightarrow$ (1): Suppose $k=1$, i.e., $\Hom(\g, \g')$ is generated by a nontrivial bypass $\xi_{\g, \g'}$ and $\Hom(\g, \g'')$ and $\Hom(\g', \g'')$ are nonzero. The dividing sets $\g$ and $\g'$ only differ on a neighborhood of the bypass arc of attachment. Since $\Hom(\g, \g'')$ and $\Hom(\g', \g'')$ are nonzero, $\#\gamma_{\g', \g''}= \#\gamma_{\g, \g''}=1$ and the portion of $\gamma_{\g, \g''}$ which is outside a neighborhood of the arc of attachment is given by the three black arcs in Figure \ref{S-1}.  (There are a priori three possibilities for the black arcs by Euler class considerations and only one of them satisfies $\#\gamma_{\g', \g''}= \#\gamma_{\g, \g''}=1$.)  The tight contact structure $\xi_{\g,\g''}$ is obtained from $\xi_{\g',\g''}$ by attaching a bypass which is trivial when viewed as a bypass on $\bdry (D^2\times[0,1])$.  Hence (1) follows when $k=1$.

%Inside the tight contact structure $\xi_{\g, \g''}$ on $D^2 \times [0,1]$, $\xi_{\g, \g'}$ is a trivial bypass in a neighborhood of $\bdry (D^2 \times [0,1])$ which is attached on $\gamma_{\g, \g''}$. The existence of the trivial bypass follows from \cite[Lemma 2.9]{H2}.  Hence $\xi_{\g, \g'}$ can be embedded in $\xi_{\g, \g''}$.
\begin{figure}[ht]
\begin{overpic}
[scale=0.3]{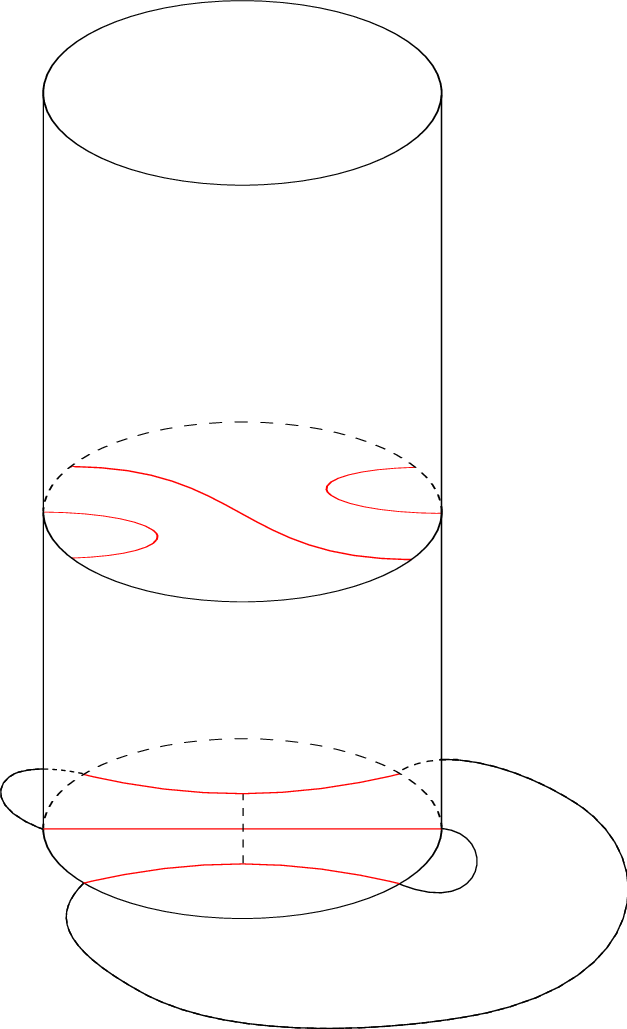}
\put(46,18){$\g$}
\put(46,48){$\g'$}
\put(46,88){$\g''$}
\end{overpic}
\caption{Peeling off the bypass $\xi_{\g, \g'}$ inside $\xi_{\g, \g''}$.}
\label{S-1}
\end{figure}

When $k>1$, one can show by induction on $k-i$ that the compositions $\Hom(\g', \g'') \times \Hom(\g^i, \g') \ra \Hom(\g^i, \g'')$ are nontrivial for all $0 \leq i \leq k-1$.  This implies (1) in general.
\end{proof}

\begin{lemma}\label{lemma: reduction 1}
Let $\g,\g',\g''$ be nonzero dividing sets.  Suppose $\xi\in \Hom(\g,\g')$ is nonzero, $\beta'\in \Hom(\g',\g'')$ is a nontrivial bypass, and $\beta'\circ \xi\in \Hom(\g,\g'')$ is zero.  Then $\xi$ can be factored into $\beta\circ \zeta$, where $\beta,\beta'$ are two consecutive bypasses of a bypass triangle.
\end{lemma}

\begin{proof}
Since $\xi\in \Hom(\g,\g')\not=0$, we have $\#\gamma_{\g,\g'}=1$.   Since attaching $\beta'$ to $\bdry (D^2\times I)$ yields an overtwisted contact structure, there exists an anti-bypass along the same arc of attachment inside $(D^2\times I, \xi)$. This implies that $\xi$ can be factored into $\beta\circ \zeta$, where $\beta,\beta'$ are two consecutive bypasses of a bypass triangle.
\end{proof}

Using the same line of argument (details left to the reader), one can also prove the following:

\begin{lemma} \label{lemma: exact functor}
Let $\tilde \g$ be a nonzero dividing set.  Then $\Hom(\tilde\g,-) $ is an exact functor from $\cne$ to the the category of $\F$-vector spaces, i.e., it takes bypass exact triangles to short exact sequences.  Similarly, $\Hom(-,\tilde\g)$ is an exact functor from $\cne$ to $\F$-vector spaces.
\end{lemma}

We also sketch the proof of Theorem~\ref{thm: relations in contact category} for $D^2_n$ and a nonzero $\xi_{\g,\g'}\in \Hom(\g,\g')$.  This will be used later in Section~\ref{Sec comp}.

\begin{proof}[Sketch of proof of Theorem~\ref{thm: relations in contact category} for $D^2_n$ and nonzero $\xi_{\g,\g'}\in \Hom(\g,\g')$.]
Let $\xi_{\g,\g'}$ be a tight contact structure on $D^2\times[0,k]$.  Suppose we are given a sequence of bypasses $ \xi_{\g^{0}, \g^{1}},\dots,\xi_{\g^{k-1}, \g^{k}}$ which compose to give $\xi_{\g,\g'}$.  Here $\g^0=\g$, $\g^k=\g'$, and $\g^i$ is a dividing set on $D^2\times\{i\}$.  Let $\delta_i\subset D^2\times \{i\}$ be the arc of attachment for $\xi_{\g^i,\g^{i+1}}$.

Let $\kappa_1$ be a boundary parallel component of $\g'$ and let ${\frak c}_1$ be a component of $(D^2\times\{k\})-\g'$ which is bounded by $\kappa_1$ and an arc $d_1$ of $\bdry (D^2\times\{k\})$. (This is unique if $n\geq 1$, which we assume.) Extend $d_1$ in the clockwise direction along $\bdry (D^2\times\{k\})$ until it reaches the next endpoint of $\g'$, and call it $d'_1$.   Let $\delta'_1$ be an arc of attachment obtained by slightly pushing $d'_1$ into $D^2\times\{k\}$ and let $\beta_1$ be the corresponding trivial bypass.  Now
$$\xi_{\g^{k-1}, \g^{k}}\circ\dots\circ \xi_{\g^{0}, \g^{1}}=\beta_1\circ \xi_{\g^{k-1}, \g^{k}}\circ\dots\circ \xi_{\g^{0}, \g^{1}},$$
where $=$ means equality as morphisms.  Since $\delta'_1$ is close to $\bdry D^2$, $\beta_1$ commutes with all the $\xi_{\g^i,\g^{i+1}}$, and
$$\beta_1\circ \xi_{\g^{k-1}, \g^{k}}\circ\dots\circ \xi_{\g^{0}, \g^{1}}=\xi_{\g^{k-1}, \g^{k}}\circ\dots\circ \xi_{\g^{0}, \g^{1}}\circ \beta_1.$$
Here we are abusing notation: there are analogous arcs $d_1'$ and $\delta_1'$ on each $D^2\times\{i\}$ and we also refer to a bypass attached along $\delta_1'\subset D^2\times\{i\}$ by $\beta_1$.

When we attach $\beta_1$ first along $D^2\times\{0\}$, we obtain a boundary parallel component of the dividing set which is unchanged through the attachments of all other
bypasses.  Hence this boundary parallel component can be removed from consideration, and the same construction can be applied to a dividing set with fewer components.
We can iteratively write down a sequence of bypasses $\beta_1,\dots,\beta_l$ so that
\begin{align}
\label{first line} \xi_{\g,\g'}=\xi_{\g^{k-1}, \g^{k}}\circ\dots\circ \xi_{\g^{0}, \g^{1}}&=\beta_l\circ\dots\circ\beta_1\circ \xi_{\g^{k-1}, \g^{k}}\circ\dots\circ \xi_{\g^{0}, \g^{1}}\\
\label{second line} &= \xi_{\g^{k-1}, \g^{k}}\circ\dots\circ \xi_{\g^{0}, \g^{1}}\circ \beta_l\circ\dots\circ\beta_1,
\end{align}
and $\xi_{\g,\g'}=\beta_l\circ\dots\circ \beta_1.$  This means that the bypasses corresponding to $\xi_{\g^i,\g^{i+1}}$ in Equation~\eqref{second line} are all trivial.  The theorem then follows.
\end{proof}

\subsection{Serre functors}

In this subsection we define endofunctors $\cs$ of $\cne$ and $\cstc$ of $\tcne$ which we call the {\em Serre functors} of $\cne$ and $\tcne$ by analogy with the Serre functors of triangulated categories introduced by Bondal and Kapranov \cite{BK}. The reader is referred to \cite{Ke} for more details on Serre functors.

\subsubsection{Serre functor of $\cne$}

\begin{defn}[Serre functor $\cs$] \label{def serre c}
The {\em Serre functor} $\cs$ is an endofunctor of $\cne$ which rotates dividing sets and contact structures by a counterclockwise angle of $\frac{2\pi}{n+1}$.  (Recall that we are assuming that $R_+(F)$ is evenly spaced around $\bdry D^2$.)
\end{defn}

See Figure \ref{figure:S-2} for an example of $\cs$ acting on a dividing set $\g$.
\begin{figure}[ht]
\begin{overpic}
[scale=0.3]{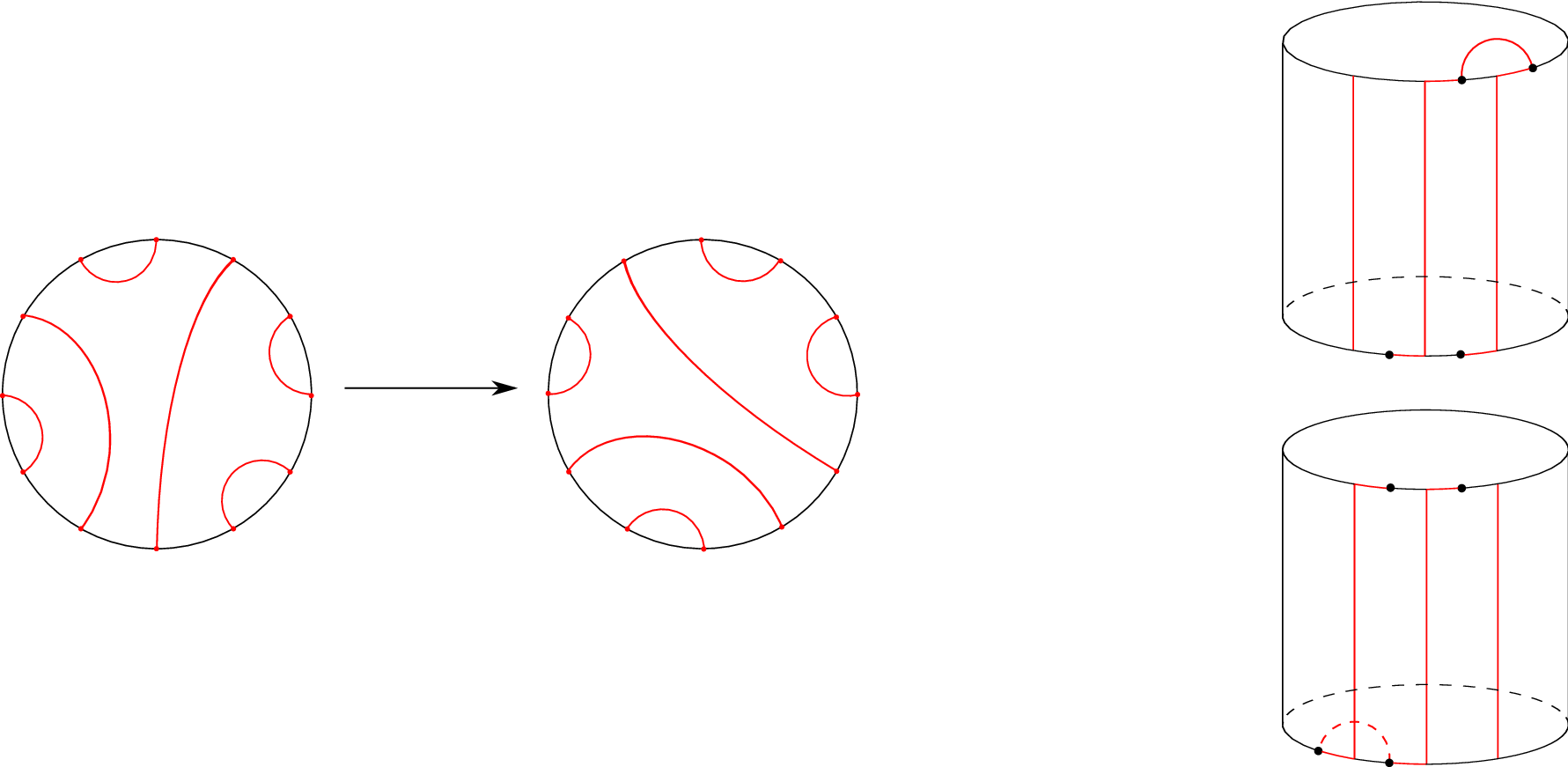}
\put(9,8){$\g$}
\put(40,8){$\cs(\g)$}
\put(25.7,26){$\cs$}
\put(76,0){$\g$}
\put(75.5,22){$\g'$}
\put(71,45){$\cs(\g)$}
\end{overpic}
\caption{An example of $\cs(\g)$ on the left; dividing sets $\gamma_{\g, \g'}$ and $\gamma_{\g', \cs(\g)}$ on the right, used in the proof of Lemma~\ref{lem serre c}.}
\label{figure:S-2}
\end{figure}

\begin{rmk} \label{rmk Mathews2}
This rotation operation was first studied in \cite{M1, M2}.
\end{rmk}

\begin{lemma} \label{lem serre c}
$\Hom(\g ,\g')\not=0$ if and only if $\Hom(\g' ,\cs(\g))\not=0$.  Hence, $\Hom(\g, \cs(\g))\not=0$.
\end{lemma}

We denote the generator of $\Hom(\g, \cs(\g))$ by $\zeta(\g)$.

\begin{proof}
Consider the dividing sets $\gamma_{\g, \g'}$ and $\gamma_{\g', \cs(\g)}$. For any boundary parallel component $\g_0$ of $\g$, there is a corresponding boundary parallel component $\cs(\g_0)$ of $\cs(\g)$.  The results of edge rounding $\g_0$ in $\gamma_{\g, \g'}$ and $\cs(\g_0)$ in $\gamma_{\g', \cs(\g)}$ are the same; see Figure \ref{figure:S-2}.  By iterating the above procedure, we obtain $\#\gamma_{\g, \g'}=\#\gamma_{\g', \cs(\g)}$, which implies the lemma.
\end{proof}

\begin{lemma} \label{lemma: serre c 2}
If $\Hom(\g, \g')\not=0$, then the composition
$$\Hom(\g', \cs(\g)) \times \Hom(\g, \g') \ra \Hom(\g, \cs(\g))$$
is nontrivial.
\end{lemma}

\begin{proof}
We decompose $\xi_{\g,\g'}$ into a composition of bypasses $\xi_{\g^{k-1}, \g^{k}} \circ \cdots \circ \xi_{\g^{1}, \g^{2}}$, where $\g^1=\g$ and $\g^k=\g'$.
Then $\Hom(\g, \g^i)\not=0$ for $1 \leq i \leq k$. By Lemma~\ref{lem serre c}, $\Hom(\g^i, \cs(\g))\not=0$.
The lemma then follows from Lemma \ref{lem stack}.
\end{proof}

\subsubsection{Serre functor of $\tcne$}

\begin{defn}[Serre functor $\cstc$] \label{def serre tc}
The {\em Serre functor} $\cstc$ is an endofunctor of $\tcne$ which is defined on objects by $\cstc(\g, [\xi])=(\cs(\g), [\zeta(\g) \circ \xi])$ and on morphisms by rotating contact structures by a counterclockwise angle of $\frac{2\pi}{n+1}$.
\end{defn}

\begin{claim} \label{claim: well-definition of serre}
The Serre functor $\cstc$ is well-defined.
\end{claim}

\begin{proof}
It suffices to show the following diagram commutes:
$$\xymatrix@1{
(\g,[\xi_0]) \ar[r]^(.4){\zeta(\g)} \ar[d]_{\xi} & (\cs(\g),[\zeta(\g)\circ\xi_0]) \ar[d]^{\cstc(\xi)}\\
(\g',[\xi \circ \xi_0]) \ar@{-->}[ur] \ar[r]_(.4){\zeta(\g')}  & (\cs(\g'),[\zeta(\g')\circ\xi\circ\xi_0])
}$$
i.e., $[\zeta(\g')\circ\xi\circ\xi_0]=[\cstc(\xi)\circ\zeta(\g)\circ\xi_0]$. Here we are assuming that $\Hom(\g,\g')\not=0$.
By Lemma \ref{lem serre c}, $\Hom(\g', \csg)\not=0$ and is generated by $\xi_{\g',\csg}$.
By applying Lemma~\ref{lemma: serre c 2} to the lower and upper triangles in the diagram, we obtain
$$[\zeta(\g')\circ\xi\circ\xi_0]=[\cstc(\xi)\circ\xi_{\g',\csg}\circ\xi\circ\xi_0]=[\cstc(\xi)\circ\zeta(\g)\circ\xi_0].$$
This proves the claim.
\end{proof}

\subsubsection{Calabi-Yau property}
%Lemma~\ref{lemma: serre c 2} shows that $\cstc$ satisfies a property similar to that of a Serre functor of a triangulated category.
According to \cite{Ke}, a triangulated category $\cal{T}$ is {\em weakly d-Calabi-Yau} if it admits a Serre functor $\cs'$ and there is an isomorphism of functors $T^{d} \xra{\sim} \cs'$, where $d$ is an integer and $T$ is the shift functor on $\cal{T}$. The analogous result for $\cstc$ on $\tcne$ is the following:

\begin{lemma} \label{lem serre c deg}
The endofunctor $\cstc^{n+1}$ is isomorphic to $T^{e(n-e)}$ on $\tcne$, i.e., $\tcne$ is ``$d$-Calabi-Yau" for a fraction $d={e(n-e)\over n+1}$.
\end{lemma}

\begin{proof}
For any dividing set $\g$, $\cs^{n+1}(\g)=\g$ since $\cs(\g)$ rotates $\g$ by $\frac{2\pi}{n+1}$.  Also
$$\cstc^{n+1}(\g, [\xi])=(\g, [\xi_n(\g)\circ \dots \circ \xi_0(\g)\circ\xi]),$$
where $\xi_i(\g)=\xi_{\cs^{i}(\g), \cs^{i+1}(\g)}$ for $0 \leq i \leq n$.  Let $k(\g)$ be minus the Hopf invariant of $\xi_n(\g)\circ\dots \circ\xi_0(\g)$.  The proof of Claim~\ref{claim: well-definition of serre} implies that $k(\g)$ is independent of $\g$.    Hence $\cstc^{n+1}$ is isomorphic to some $T^{k}$, where $k=k(\g)$ for any $\g\in {\frak ob}(\cne)$.

We compute $k$ by choosing a special $\g\in {\frak ob}(\cne)$ which has $n$ boundary parallel components; such a $\g$ is unique in ${\frak ob}(\cne)$ up to rotation.
The case of $n=5, e=3$ is depicted in Figure \ref{S-3}. % and leave it to the reader to generalize the procedure.
We can write $\xi_0(\g)$ as a composition of $n-e$ bypasses, illustrated as in the upper left diagram of Figure~\ref{S-3}. The bypasses of $\xi_i(\g)$ are obtained from those of $\xi_{i-1}(\g)$ by a $\frac{2\pi}{n+1}$ rotation, followed by an isotopy in the radial direction so that they are closer to $\bdry D^2_{n+1}$; the $(n-e)(n+1)$ bypasses are then mutually disjoint.  The $(n-e)(n+1)$ bypasses can be grouped into $n-e$ copies of $n+1$ bypasses which are arranged in a circle; see the upper right diagram of Figure~\ref{S-3}.

\begin{figure}[ht]
\begin{overpic}
[scale=0.28]{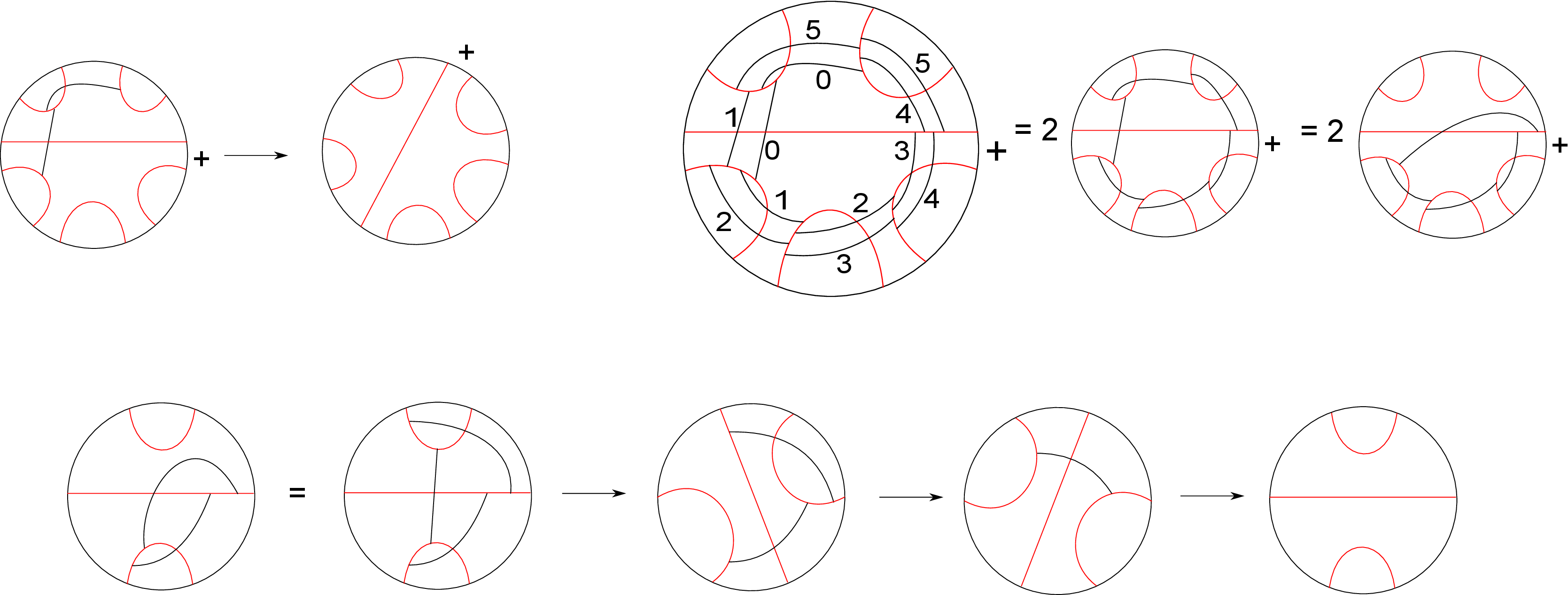}
\put(5,19){$\g$}
\put(24,19){$\cs(\g)$}
\put(15,30){$\xi_0$}
\end{overpic}
\caption{$\xi_0$ is written as a composition of $n-e=2$ bypasses on the upper left.  The $(n-e)(n+1)=2(5+1)$ bypasses for $\xi_n(\g)\circ \dots\circ \xi_0(\g)$ are drawn on the upper right, where the black arcs with label $i$ denote the bypasses of $\xi_i(\g)$ for $0 \leq i \leq n=5$.  The computation of the Hopf invariant is given on the bottom row for $e=1$.}
\label{S-3}
\end{figure}

It suffices to show that the Hopf invariant of the composition of the $n+1$ bypasses is $-e$.  Among the $n+1$ bypasses, $n-e$ of them are trivial.  After attaching the trivial bypasses we are left with $e+1$ overtwisted bypasses arranged in a pinwheel~\cite[Section 1]{HKM}; see the upper right diagram of Figure~\ref{S-3}.  When $e=1$ the bottom row of Figure~\ref{S-3} shows that attaching the $e+1=2$ overtwisted bypasses in pinwheel position is equivalent to the composition of the three bypasses in a bypass triangle.   Hence the Hopf invariant of the composition is $-1$ by Theorem~\ref{thm: huang}.  We can inductively write the pinwheel with $e+1$ bypasses into two pinwheels, one with two bypasses and another with $e$ bypasses (left to the reader).  Hence the Hopf invariant of $\xi_n(\g)\circ\dots\circ\xi_0(\g)$ is $-e$.
\end{proof}

\begin{rmk}
Fractional Calabi-Yau categories have recently been studied by Kuznetsov~\cite{Ku}.
\end{rmk}

In Section \ref{Sec serre d} we will show that $\tcne$ can be embedded into a triangulated category $\tdne$ which admits a Serre functor.
Moreover, the Serre functors commute with the embedding as proved in Proposition \ref{prop serre d}.

\section{Algebraic description of $\cne$} \label{section: algebraic description}

The contact category $\cne$ is defined over a disk $D^2_n$ with $2(n+1)$ marked points on the boundary, and the Euler number equal to $n-2e$.
In particular, $\chi_+=n-e+1$, and $\chi_-=e+1$.

\subsection{Notation for dividing sets}

In this subsection let $\g\in {\frak ob}(\cne)$ be a {\em nonzero} dividing set.
Such $\g$ is a crossingless matching of $2(n+1)$ marked points on $\bdry D^2_n$.
We introduce notation to algebraically encode $\g$.

A dividing set $\g$ is determined by its positive region $R_+(\g)$. Let $\pi_0(\rg)$ be the set of components of $\rg$. Each component ${\frak c}$ of $R_+(\g)$ is a (partially open) disk which intersects $\bdry D^2_n$ at one or more positive arcs and the labels of ${\frak c}\cap \bdry D^2_n$ form the set of labels of ${\frak c}$.  Observe that a component ${\frak c}$ is determined by its set of labels. A component which has only one label is said to be {\em boundary parallel}.

The relative positions of the components of $\rg$ are described by the following {\em nesting}  and {\em adjacency} relations:

\begin{defn}[Nesting and adjacency] \label{def nest g} Let ${\frak c}$ and ${\frak c}'$ be components of $\rg$. Then:
\be
\item ${\frak c}$ {\em nests inside} ${\frak c}'$ if any path $[0,1]\to D^2_n$ from ${\frak c}$  to  the component of $R_+(F)$ corresponding to the label $\ul{0}$ nontrivially intersects ${\frak c}'$. 
\item ${\frak c}$ and ${\frak c}'$ are {\em adjacent} if there is a path in $D^2_n$ between ${\frak c}$ and ${\frak c}'$ which does not intersect any other component of $\rg$. 
\item ${\frak c}$ {\em directly nests inside} ${\frak c}'$ if ${\frak c}$ nests inside ${\frak c}'$ and ${\frak c}$ and ${\frak c}'$ are adjacent.
\ee
\end{defn}

In particular, the component containing $\ul{0}$ does not nest inside any other component and there is no component which nests inside a boundary parallel component.

Let $\Z_+$ be the set of positive integers. Define
$$\vz=\bigsqcup\limits_{k\geq 0}\Z_+^k,$$
where $\Z_+^0=\{\ast\}$ is a set of one special element. The {\em dimension $\dim(\mf{v})$} of $\mf{v}=(v_1, \dots, v_k ) \in \Z_+^k$ is $k$ and the dimension of $\ast$ is $0$.

\begin{defn}[Direct nesting of vectors] \label{def nest v}
For $\mf{v} \in \vz, t \in \Z_+$, define $\mf{v} \sqcup t \in \vz$ by $\dim (\mf{v} \sqcup t)=\dim(\mf{v})+1$ and
$$ (\mf{v} \sqcup t)_j= \left\{
\begin{array}{cl}
v_j & \mbox{if} \hspace{0.1cm} 1 \leq j \leq \dim(\mf{v}), \\
t & \mbox{if} \hspace{0.1cm} j=\dim(\mf{v})+1.
\end{array}\right.
$$
The vector $\mf{v} \sqcup t$ is said to {\em directly nest inside} $\mv$.
\end{defn}

\s
\n {\bf The assignment $\Phi_\g$.}  We label regions of $\rg$ by some vectors in $\vz$. More precisely, we inductively define an injective map
$$\Phi_\g: \pi_0(\rg) \ra \vz.$$
%based on the nesting relations in Definitions \ref{def nest g}, and \ref{def nest v}.
We use the notation $\gv=\Phi_\g^{-1}(\mf{v})$ for $\mf{v} \in \op{Im}(\Phi_\g)$. The component which contains $\ul{0}$ is defined to be $\g_\ast$ and is called the {\em based component}. Next, given a component $\gv$, suppose there are $k$ components which directly nest inside $\gv$, arranged in clockwise order with respect to the label $\ul{0}$. The $t$-th component is then defined to be $\g_{\mf{v} \sqcup t}$ for $1 \leq t \leq k$. 
Note that the two notions of direct nesting --- for vectors in $\vz$ and for regions of $\rg$ --- agree under the map $\Phi_{\g}$.
We also sometimes mix up the notation and say that a region directly nests inside a vector.

\s
We now define
$$V(\g)=\op{Im}(\Phi_\g), \quad \tpv(\g)=V(\g) \backslash \{\ast\}, \quad \vnb(\g)=\{\mf{v} \in \tpv(\g)~|~ |\gv|>1\}.$$
By abuse of notation, we are using $\gv$ to denote its set of labels (a subset of $\{\ul{0},\dots,\ul{n}\}$) and $|\gv|$ to denote its cardinality.  Observe that the component $\gv$ is not boundary parallel for any $\mf{v} \in \vnb(\g)$. The cardinality $|V(\g)|$ is the number of components of $R_+(\g)$, which is equal to $\chi_+(\g)=n-e+1$ for $\g$ in $\cne$.

\begin{defn} \label{def bne}
A (nonzero) dividing set $\g$ is {\em basic} if $\vnb(\g)=\es$, i.e., every component $\gv \neq \g_{\ast}$ is boundary parallel.
\end{defn}

The set of all basic dividing sets in $\cne$ is denoted by $\bne$.
For $\ul{0} < \ul{s_1} < \cdots < \ul{s_e} \leq \ul{n}$, let $\g(\ul{s_1}, \dots, \ul{s_e})$ denote the basic dividing set $\g \in \bne$ such that $\g_{\ast}=\{\ul{0}, \ul{s_1}, \dots, \ul{s_e}\}$.  Any $\g \in \bne$ is determined by its based component $\g_{\ast}$ which is a subset of $\{\ul{0},\dots,\ul{n}\}$ containing $\ul{0}$.
Moreover, $(n+1-|\g_{\ast}|)+1=\chi_+(\g)=n-e+1$.
Hence $|\g_{\ast}|=e+1$ and $|B_{n,e}|=\binom{n}{e}$.

%\begin{rmk} \label{rmk e>0} If $\tpv(\g)=\es$, then $\g$ is the unique dividing set in $\cal{C}_{n,n}$ where $\g_{\ast}=\{0,\dots,n\}$. The contact category $\cal{C}_{n,n}$ is equivalent to the additive category of finite dimensional vector spaces over $\F$. For later use (See Definition \ref{def oi}) we will assume $\tpv(\g) \neq \es$, i.e., $e<n$ from now on. \end{rmk}

Let $l_{\gv}=|\gv|-1$. We order the elements of $\gv$ so that
$$\gv=\{ \gv(0), \dots, \gv(l_{\gv})\} \quad \mbox{and} \quad \gv(0) < \cdots < \gv(l_{\gv}).$$

\begin{example}
Let $\g$ be the dividing set in $\cne$ for $n=7, e=4$ as shown in Figure \ref{2-2-1}.
There are $4$ components of $\rg$:
$$\g_{\ast}=\{\ul{0},\ul{4}\}; \quad \g_{(1)}=\{\ul{1},\ul{3}\}, \quad \g_{(2)}=\{\ul{5},\ul{6},\ul{7}\}; \quad \g_{(1,1)}=\{\ul{2}\}.$$
The elements of $\g_{(2)}$ satisfy $\g_{(2)}(0)=\ul{5}$, $\g_{(2)}(1)=\ul{6}$, and $\g_{(2)}(2)=\ul{7}.$
\begin{figure}[ht]
\begin{overpic}
[scale=0.3]{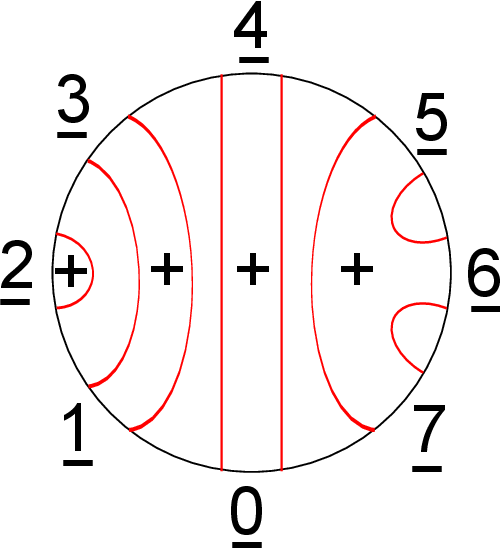}
\end{overpic}
\caption{The four components of $\rg$.}
\label{2-2-1}
\end{figure}
\end{example}

To summarize, we describe a dividing set $\g$ by a partition $\{\gv ~|~ \mf{v} \in V(\g)\}$ of $\{\ul{0},\dots,\ul{n}\}$, where each $\gv$ is a component of $R_+(\g)$.
We will write $\g=\{\gv\}$ for simplicity.
The collection $\{\gv\}$ satisfies the following:
\be
\item $V(\g)$ is a finite subset of $\vz$ such that $|V(\g)|=n-e+1$.
\item $\ast \in V(\g)$ for any $\g$ and $\ul{0} \in \g_{\ast}$.
\item If $\mf{v}, \mf{v} \sqcup t \in V(\g)$, then there exists unique $i \in \{0,1,\dots,l_{\gv}\}$ such that $\g_{\mf{v} \sqcup~ t}$ is a subset of an open interval $(\gv(i), \gv(i+1))$. Here $\gv(l_{\gv}+1)$ is understood to be $\ul{n+1}$.
\item If $\mf{v}, \mf{v} \sqcup t_0 \in V(\g)$, then $\mf{v} \sqcup t \in V(\g)$ for $1 \leq t \leq t_0$, and $\g_{\mf{v} \sqcup ~t}(l_{\g_{\mf{v} \sqcup~ t}}) < \g_{\mf{v} \sqcup ~t'}(0)$ for $1 \leq t < t' \leq t_0$.
\item $\bigsqcup\limits_{\mf{v} \in V(\g)} \gv=\{\ul{0},\dots,\ul{n}\}$.
\ee

\begin{rmk}
Properties (3) and (4) follow from the fact that a dividing set $\Gamma$ is properly embedded in $D^2$.
\end{rmk}

\subsection{Notation for bypasses} \label{subsection: notation for bypass}

In this subsection we introduce notation to describe bypasses.

Let $\g, \g'$ be nonzero objects of $\cne$, let $\beta \in \Hom(\g,\g')$ be a nontrivial bypass, and let $\delta=\delta_+ \cup \delta_-$ be the arc of attachment for $\beta$.  Since $\beta$ is nontrivial, $\delta$ intersects three distinct components of $\g$.  We position $\delta$ and the three components of $\g$ as in Figure~\ref{2-2-2} so that $\delta$ is vertical, $int(\delta_+)\subset R_+(\Gamma)$ is the lower subarc, and $int(\delta_-)\subset R_-(\Gamma)$ is the upper subarc.
%Note that definitions below can be intrinsically defined without fixing a direction of $\delta$.

\begin{notation} \label{notation: beta} $\mbox{}$
\begin{enumerate}
\item $\uv(\beta)$ is the vector in $V(\g)$ such that $int(\delta_+)$ is contained in the component $\g_{\uv(\beta)}$.
\item $\ov(\beta)$ is the vector in $V(\g)$ such that $int(\delta_-)$ connects the components $\g_{\uv(\beta)}$ and $\g_{\ov(\beta)}$.
\item $x(\beta), y(\beta)$ are elements of $\{0,\dots,l_{\g_{\uv(\beta)}}\}$ such that labels $\g_{\uv(\beta)}(x(\beta))$ and $\g_{\uv(\beta)}(y(\beta))$ appear at the bottom left and top left corners of $\g_{\uv(\beta)}$, respectively.
\item $z(\beta)$ is the element of $\{0,\dots,l_{\g_{\ov(\beta)}}\}$ such that the label $\g_{\ov(\beta)}(z(\beta))$ appears at the bottom left corner of $\g_{\ov(\beta)}$.
\end{enumerate}
Refer to the left-hand side of Figure~\ref{2-2-2} for an illustration.
\end{notation}

%Since $\beta$ is nontrivial, $\delta$ intersects $\g$ at three different points.
Observe that $\uv(\beta)\not =\ov(\beta)$ since $\beta$ is nontrivial. We will omit the variable $\beta$ and write $\uv,\ov,x,y,z$ for simplicity when $\beta$ is understood.

Depending on the position of the label $\ul{0}$, both $x\leq y$ and $x > y$ are possible. We use the notation $[[x,y]]$ for the generalized interval between $x,y \in \Z$ given by:
$$ [[x,y]]:= \left\{
\begin{array}{cl}
[x,y] & \mbox{if} \hspace{0.1cm} x \leq y; \\
(-\infty,y] \cup [x,+\infty) & \mbox{otherwise.}
\end{array}\right.
$$

\begin{notation} \label{notation: beta 2}
The component $\guv$ is cut into two parts $\guv^l$ and $\guv^r$ by $\delta_+$, where:
\be
\item $\guv^l:=\{\guv(i) ~|~ i \in [[x,y]]\}$ is the subset of $\guv$ which consists of labels to the left of $\delta_+$.
\item $\guv^r:=\guv \backslash \guv^l$.
\ee
\end{notation}

\begin{figure}[ht]
\begin{overpic}
[scale=0.5]{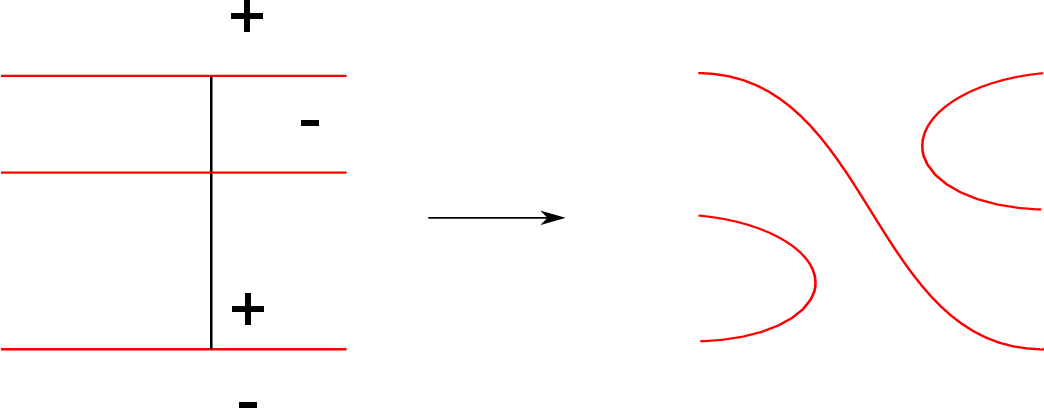}
\put(12,13){$\guv^l$}
\put(27,13){$\guv^r$}
\put(-3,8){$\guv(x)$}
\put(-3,18){$\guv(y)$}
\put(-3,35){$\gov(z)$}
\put(63,11.3){$\g'_{\beta(\uv)}$}
\put(79,32){$\g'_{\beta(\ov)}$}
\put(21,16){$\delta_+$}
\put(21,25){$\delta_-$}
\put(46,20){$\beta$}
\end{overpic}
\caption{Definitions of $\uv, \ov, x, y, z$, and the change of components of $R_+$ under a bypass $\beta$.}
\label{2-2-2}
\end{figure}

%The bypass $\beta \in \Hom(\g,\g')$ changes components of $\rg$ to those of $R_+(\g')$.
\n
{\bf The map $\beta$.} Given a nontrivial $\beta\in \Hom(\g,\g')$, by abuse of notation we write
\begin{equation} \label{eqn: beta}
\beta: V(\g) \ra V(\g')
\end{equation}
for the map which satisfies
$$ \g'_{\beta(\mf{v})}= \left\{
\begin{array}{cl}
\guv^l & \mbox{if} \hspace{0.1cm} \mf{v}=\uv, \\
\gov \sqcup \guv^r & \mbox{if} \hspace{0.1cm} \mf{v}=\ov, \\
\gv & \mbox{otherwise,}
\end{array}\right.
$$
as subsets of $\{\ul{0}, \dots, \ul{n}\}$ for $\mf{v} \in V(\g)$.  The map $\beta$ is a bijection. % By definition, $\gv$ is left unchanged under $\beta$ unless $\mf{v}=\uv, \ov$.

\begin{rmk} \label{rmk beta ast}
We have $\beta(\ast)=\ast \in V(\g')$ unless $\ul{0} \in \guv^r$; in that case $\beta(\ov)=\ast \in V(\g')$.
\end{rmk}

\section{Definition of $\tdne$} \label{section: defn of algebra}

Recall from Definition \ref{def bne} that $\bne$, $0\leq e\leq n$, is the set of basic dividing sets $\g$ in $\cne$ and that each basic dividing set is determined by $\g_{\ast}$.
Let $(\g)$ and $(\g|\g')$ denote the generators of $\op{End}(\g)$ and $\Hom(\g,\g')$ which are $1$-dimensional when $\Hom(\g,\g')\neq 0$.

For $0\leq e\leq n$, define the $\F$-algebra
$$\rne=\bigoplus\limits_{\g,\g' \in \bne} \Hom_{\cne}(\g, \g'),$$
where the multiplication $a \cdot b$ is given by the composition $b \circ a$ in $\cne$ for $a,b\in\{(\g), (\g'|\g'')\}$ if they are composable and zero otherwise.

\begin{prop} [Tightness criterion for basic dividing sets] \label{prop: tightness criterion}
For $\g,\g' \in \bne, \g \neq \g'$, the following are equivalent:
\be
\item $\Hom(\g, \g')\not=0$.
\item There exists a sequence of labels $\ul{0} < \ul{s_1} < \ul{s'_1} < \cdots <\ul{s_k} < \ul{s'_k}$ such that $\g_{\ast} \cap [\ul{s_i}, \ul{s'_i}]=\{\ul{s_i}\}$ for $1\leq i\leq k$ and $\g'_{\ast} = (\g_{\ast} \backslash \{\ul{s_1},\dots,\ul{s_k}\}) \cup \{\ul{s'_1},\dots,\ul{s'_k}\}$.
\ee
\end{prop}

We will refer to (2) as the {\em tightness condition.}

\begin{proof}
(2) $\Rightarrow$ (1):
The dividing set $\g'$ can be obtained from $\g$ by attaching $k$ disjoint bypasses corresponding to $k$ disjoint closed intervals $[\ul{s_i}, \ul{s'_i}]$ for $1 \leq i \leq k$; see the left-hand side of Figure \ref{3-1}. The resulting contact structure is tight by \cite[Theorem 1.2]{HKM}.

\begin{figure}[ht]
\begin{overpic}
[scale=0.3]{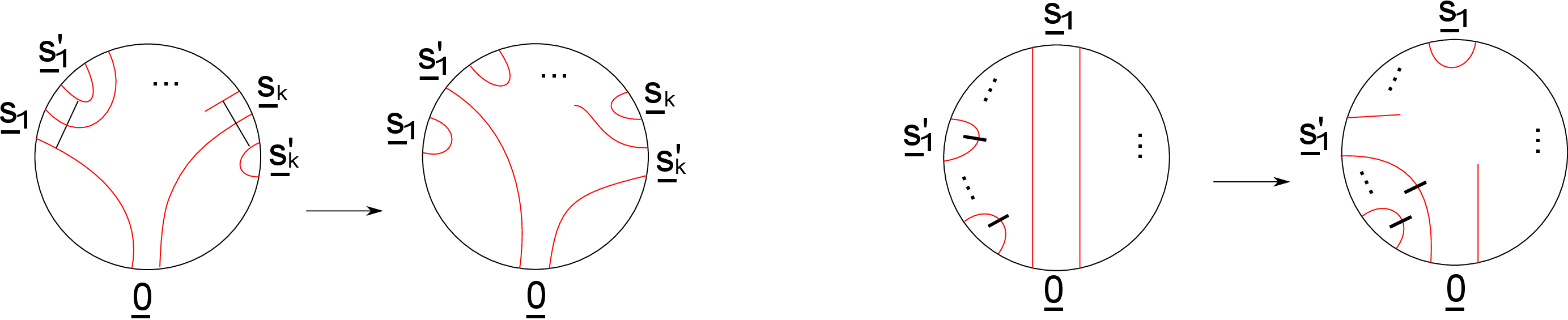}
\put(12,0){$\g$}
\put(38,0){$\g'$}
\put(71,0){$\g$}
\put(96,0){$\g'$}
\end{overpic}
\caption{The $k$ disjoint bypasses in the proof of $(2) \Rightarrow (1)$ on the left. The case of $\ul{s_1}>\ul{s'_1}$ in the proof of $(1) \Rightarrow (2)$ on the right. The components of $\g$ and $\g'$ with markings form a closed loop after edge rounding.}
\label{3-1}
\end{figure}

\s \n(1) $\Rightarrow$ (2):
Let $\ul{s_1}=\op{min}(\g_{\ast} \backslash \g'_{\ast}), \ul{s'_1}=\op{min}(\g'_{\ast} \backslash \g_{\ast})$.

We first prove that $\ul{s_1}<\ul{s'_1}$. Arguing by contradiction, suppose that $\ul{s_1}>\ul{s'_1}$.  (Note that $\ul{s_1}=\ul{s'_1}$ is not possible by definition.) Let $\ul{s_1}=\g_*(r_0), \ul{s_1'}=\g_*'(r_0)$, and $\g_*(r)=\g_*'(r)$ for $r \le r_0$.
The case of $r_0=1$ is depicted on the right-hand side of Figure~\ref{3-1}.  The dividing curve on $\bdry(D^2 \times [0,1])$ obtained by edge rounding (i) the boundary parallel components $\g_{(t)}=\{\ul{t}\}$ for $1 \leq t \leq s'_1$, (ii) the boundary parallel components $\g'_{(t')}=\{\ul{t'}\}$ for $1 \leq t' \leq s'_1-1$, and (iii) the arc of $\g'$ joining the labels $\ul{0}$ and $\ul{s'_1}$, forms a closed loop.  Hence $\gamma_{\g,\g'}$ has more than one component and $\Hom(\g,\g')=0$, which is a contradiction.
The case of $r_0$ in general is no more difficult and is left to the reader.

Next we prove that $\g_{\ast} \cap [\ul{s_1}, \ul{s'_1}]=\{\ul{s_1}\}$ and $\g'_{\ast} \cap [\ul{s_1}, \ul{s'_1}]=\{\ul{s'_1}\}$. Suppose that $\g_{\ast} \cap [\ul{s_1}, \ul{s'_1}]\neq\{\ul{s_1}\}$. Let $\ul{t_1}=\op{min}(\g_{\ast} \cap (\ul{s_1}, \ul{s'_1}))$. The dividing curve on $\bdry(D^2 \times [0,1])$ obtained by edge rounding (i) the components of $\g$ with both endpoints on arcs of $R_+(F)$ labeled $\ul{s_1}$ to $\ul{t_1}$ and (ii) the boundary parallel components of $\g'$ with both endpoints on arcs labeled $\ul{s_1}$ to $\ul{t_1-1}$ forms a closed loop.  Hence $\Hom(\g,\g')=0$ and we have a contradiction.  This implies that $\g_{\ast} \cap [\ul{s_1}, \ul{s'_1}]=\{\ul{s_1}\}$.  The proof of $\g'_{\ast} \cap [\ul{s_1}, \ul{s'_1}]=\{\ul{s'_1}\}$ is similar.

Let $\g^1 \in \bne$ such that $\g^1_{\ast}=\g_{\ast} \backslash \{\ul{s_1}\} \cup \{\ul{s'_1}\}$. We claim that $\Hom(\g^1,\g')\not=0$.
From the proof of (2) $\Rightarrow$ (1) above, $\Hom(\g,\g^1)\not=0$ and is generated by a bypass that we denote by $(\g|\g^1)$. By an argument similar to the proof of (2) $\Rightarrow$ (1) in Lemma \ref{lem stack}, the bypass $(\g|\g^1)$ is a trivial bypass when viewed as a bypass on $\bdry(D^2 \times [0,1])$ with dividing set $\gamma_{\g,\g'}$.  By peeling off the bypass $(\g|\g^1)$ from the contact structure which generates $\Hom(\g,\g')$, we obtain a tight contact structure which generates $\Hom(\g^1,\g')$. In particular, this implies the claim.

If $\g^1=\g'$, then we are done.  Otherwise, % we apply the same method to the pair $(\g^1,\g')$ as we did to $(\g,\g')$.
we inductively define $\ul{s_i}=\op{min}(\g^{i-1}_{\ast} \backslash \g'_{\ast}), \ul{s'_i}=\op{min}(\g'_{\ast} \backslash \g^{i-1}_{\ast})$, and $\g^i \in \bne$ such that $\g^i_{\ast}=\g^{i-1}_{\ast} \backslash \{\ul{s_i}\} \cup \{\ul{s'_i}\}$. After finitely many steps, we have $\g^k=\g'$ for some $k$, which implies (2).
\end{proof}

\begin{cor} \label{cor stack}
For $\g, \g', \g'' \in \bne$, if $\Hom(\g,\g'), \Hom(\g',\g'')$, and $\Hom(\g,\g'')$ are all nonzero, then $(\g|\g')(\g'|\g'')=(\g|\g'')$.
\end{cor}

\begin{cor} \label{cor stack2}
Suppose $\Hom(\g,\g')$ and $\Hom(\g',\g'')$ are both nonzero for $\g,\g',\g'' \in \bne$.
If there exists $\ul{s}$ such that $\ul{s} \in \g_\ast \cap \g''_\ast$ but $\ul{s} \notin \g'_\ast$, then $\Hom(\g,\g'')=0$.
\end{cor}

Given $\tg\in \bne$, let us define $\#(\tg,\ul{s})=|\{\ul{t} \in \tg_\ast ~|~ \ul{t}>\ul{s}\}|$.

\begin{proof}
Suppose that $\ul{s} \in \g_\ast \cap \g''_\ast$, $\ul{s} \notin \g'_\ast$, and $\Hom(\g,\g'')\not=0$. Let $\ul{0} < \ul{s_1} < \ul{s''_1} < \cdots <\ul{s_k} < \ul{s''_k}$ be the sequence of labels in Proposition \ref{prop: tightness criterion} which correspond to $\Hom(\g,\g'')$. Since $\ul{s} \in \g_\ast \cap \g''_\ast$, $\ul{s} \notin [\ul{s_i}, \ul{s''_i}]$ for all $i$.  Then we immediately have $\#(\g,\ul{s})=\#(\g'',\ul{s})$.

On the other hand, $\#(\g',\ul{s})=\#(\g,\ul{s})+1$ since $\Hom(\g,\g')\not=0$ and $\ul{s} \in \g_\ast, \ul{s} \notin \g'_\ast$; and $\#(\g'',\ul{s})\geq \#(\g',\ul{s})$ since $\Hom(\g',\g'')\not=0$.  Hence $\#(\g'',\ul{s})>\#(\g,\ul{s})$, a contradiction.
\end{proof}

\begin{notation} \label{notation: gen rne}
We write $\g \xra{s} \g'$ for $\g, \g' \in \bne, s \in \{1,\dots,n-1\}$ if $\ul{s} \in \g_{\ast}, \ul{s+1} \notin \g_{\ast}$ and $\g'_{\ast}=\g \backslash \{\ul{s}\} \cup \{\ul{s+1}\}$. In this case $\Hom(\g,\g')\not=0$.
\end{notation}

The bypasses $\g \xra{s} \g'$ are the elementary blocks of tight contact structures between basic dividing sets.

\begin{lemma} \label{lem rne}
The algebra $\rne$ has idempotents $(\g)$, generators $(\g|\g')$ where $\g, \g' \in \bne, \g \xra{s} \g'$ for some $s$, and relations:
\begin{gather} 
(\g)(\g') = \delta_{\g,\g'}; \\
(\g)(\g|\g')=(\g|\g')(\g')=(\g|\g');\\
(\g|\g')(\g'|\g'')=0 ~~\mbox{if}~ \g \xra{s} \g', \g' \xra{s-1} \g''; \label{eq rel}\\
(\g|\g')(\g'|\g'')=(\g|\g''')(\g'''|\g'') ~~\mbox{if}~ \g \xra{s} \g', \g' \xra{t} \g'', \g \xra{t} \g''', \g''' \xra{s} \g'' ~\mbox{for}~|s-t|>1.
\end{gather}
\end{lemma}

\begin{proof}
Suppose that $\Hom(\g, \g')\not=0$ and $\g \neq \g'$.
Let $\ul{s_1}=\op{min}(\g_{\ast} \backslash \g'_{\ast})$.
Define $\tg \in \bne$ such that $\g \xra{s_1} \tg$.
By Proposition \ref{prop: tightness criterion}, $\Hom(\g, \tg)$ and $\Hom(\tg, \g')$ are nonzero.
By Lemma \ref{lem stack}, the composition $\Hom(\tg, \g') \times \Hom(\g, \tg) \ra \Hom(\g, \g')$ is nontrivial since the generator $(\g|\tg)$ is a bypass.
By an iterated peeling off of bypasses, one can prove that $\{(\g), (\g|\g') ~|~ \g \xra{s} \g' ~\mbox{for some}~ s\}$ generate $\rne$ as an algebra.

The first two relations of $\rne$ are immediate from the definition of $\rne$.

%The relations of the composition in $\cne$ are generated by (1) the commutativity of a pair of disjoint bypasses, and (2) the composition of two bypasses in a bypass triangle being zero.

For a composition of two bypasses $\g \xra{s} \g', \g' \xra{t} \g''$, there are $3$ possibilities:
\be
\item If $t=s-1$, then $\g_{\ast} \cap [\ul{s-1}, \ul{s+1}]=\{\ul{s-1},\ul{s}\}$ and $\g''_{\ast} \cap [\ul{s-1}, \ul{s+1}]=\{\ul{s},\ul{s+1}\}$. Hence $\Hom(\g,\g'')=0$ by Corollary~\ref{cor stack2}, implying the third relation of $\rne$.
\item If $t=s+1$, then $\Hom(\g,\g'')\not=0$ and the product $(\g|\g')(\g'|\g'')$ is the generator of $\Hom(\g,\g'')$. $\g_{\ast}\cap [\ul{s},\ul{s+2}]=\{\ul{s}\}$ and $\g''_{\ast}\cap [\ul{s},\ul{s+2}]=\{\ul{s+2}\}$ and there is no relation in this case.
\item If $|s-t|>1$, then $\g' \xra{t} \g''$ induces a bypass $\g \xra{t} \g'''$ on $\g$ which is disjoint from the bypass $\g \xra{s} \g'$.  We have $\g_{\ast}\cap [\ul{s},\ul{s+1}]=\{\ul{s}\}$, $\g_{\ast}\cap [\ul{t},\ul{t+1}]=\{\ul{t}\}$, $\g''_{\ast}\cap [\ul{s},\ul{s+1}]=\{\ul{s+1}\}$, and $\g''_{\ast}\cap [\ul{t},\ul{t+1}]=\{\ul{t+1}\}$. The last relation of $\rne$ follows from the commutativity of a pair of disjoint bypasses.
\ee

Now let $\rnea$ denote the algebra with the generators and defining relations as in the lemma. The discussion above gives a homomorphism of algebras $\phi: \rnea \ra \rne$, which is obviously surjective. To prove the injectivity, it suffices to show that 
$$(\g)\cdot \rnea \cdot (\g')\cong (\g)\cdot \rne \cdot (\g') \quad \forall \g, \g'.$$ 
By Proposition~\ref{prop: tightness criterion}, $(\g)\cdot \rne \cdot (\g')$ is one-dimensional if the tightness condition (Proposition~\ref{prop: tightness criterion}(2)) holds; otherwise, it is zero.  If the tightness condition does not hold, then either
\begin{itemize}
\item[(i)] there is no path from $\g$ to $\g'$, i.e., a sequence $\g=\g_1,\dots, \g_n=\g'$ such that $(\g_i|\g_{i+1})$ is a generator (this is the case if and only if there exists $r_0$ such that $\g_*(r)\leq \g_*'(r)$ for $r=0,\dots,r_0-1$ and $\g_*(r_0)> \g_*'(r_0)$) or
\item[(ii)]  there exist $\g_i \xra{s} \g_{i+1}, \g_{i+1} \xra{s-1} \g_{i+2}$ such that 
\begin{equation} \label{eqn: factorization}
(\g)\cdot \rnea \cdot (\g') \cong \left((\g)\cdot \rnea \cdot (\g_i)\right) \cdot \left((\g_i)\cdot \rnea \cdot (\g_{i+2})\right) \cdot \left((\g_{i+2})\cdot \rnea \cdot (\g')\right).
\end{equation}
\end{itemize}
But then \eqref{eqn: factorization} is zero since $(\g_i)\cdot \rnea \cdot (\g_{i+2})=0$ by \eqref{eq rel}.  If the tightness condition holds, then no factorization of $(\g)\cdot \rnea \cdot (\g')$ contains $(\g_{i})\cdot \rnea \cdot (\g_{i+2})$ satisfying $\g_i \xra{s} \g_{i+1}, \g_{i+1} \xra{s-1} \g_{i+2}$. Hence $(\g)\cdot \rnea \cdot (\g')$ is nonzero and one-dimensional. 
\end{proof}

\begin{rmk} \label{rmk strands}
The algebra $\rne$ is isomorphic to the homology of a {\em strands algebra} of a disk which is a differential graded algebra.
Relationships between the contact category and bordered/sutured Heegaard Floer homology have been studied in \cite{Za,M1,M2,M3,M4,Co}.
\end{rmk}

\n
{\bf The quiver $Q_{n,e}$.} Let $\qne$ be the oriented quiver whose set of vertices is $V(\qne)=\bne$ and whose set of arrows is $I(\qne)=\{\g \xra{s} \g' ~\mbox{for some}~s\}.$  A path in $\qne$ from $\g$ to $\g'$ is said to be {\em nonzero} if $\Hom(\g,\g')\not=0$ and a nonzero path is denoted by $\g \ra \g'$.  We define a partial order ``$\leq$" on the set of all nonzero paths: $(\g_1 \ra \g'_1)\leq(\g_2 \ra \g'_2)$ if $\g_1 \ra \g'_1$ can be extended to $\g_2 \ra \g'_2$ in $\qne$.
The partial order motivates the constructions in Section \ref{section: defn of functor}; see Remarks \ref{rmk diff path} and \ref{rmk beta path}.

\s
The finite dimensional algebra $\rne$ is isomorphic to a quotient of the path algebra $\F\qne$ of $\qne$. We refer to \cite{ASS} for an introduction to the representation theory of finite dimensional algebras and quivers. In particular, by \cite[Section I.4]{ASS},  $\{(\g)~|~\g\in \bne \}$ is a complete set of primitive orthogonal idempotents in $\rne$ and $\{P(\g)=\rne (\g) ~|~ \g \in \bne \}$ forms a complete set of non-isomorphic indecomposable projective left $\rne$-modules.
A nice property of the finite quiver $\qne$ is that it has no oriented cycles. 
It implies that any simple module has a finite projective resolution. Hence the algebra $\rne$ has finite global dimension.

Define $\tdne$ as the homotopy category of bounded cochain complexes of finitely generated projective left $\rne$-modules.
By a standard result in homological algebra $\tdne$ is equivalent to the bounded derived category $\mf{D}^b(\rne)$ of finitely generated left $\rne$-modules as triangulated categories.  
The Grothendieck group $K_0(\tdne)$ is isomorphic to $\Z^{\oplus \binom{n}{e}}$.

Let $\dne$ be the ungraded version of $\tdne$, whose objects are the same as $\tdne$ and whose morphisms are given by
$$\Hom_{\dne}(M,N):=\bigoplus\limits_{n\in \Z}\Hom_{\tdne}(M,N[n]).$$

\section{The functors $\fne$} \label{section: defn of functor}

In this section, we define a family of functors $\fne: \cne \ra \dne$ for $0 \leq e \leq n$. We write $\cal{F}$ for $\fne$ when $n,e$ are understood.  Since the definition of $\cal{F}$ is highly technical, we first give some motivating examples in Section \ref{Sec motivation}. In Section \ref{Sec dividing}, we define a complex $\cal{F}(\g)$ in $\dne$ for each dividing set $\g$ in $\cne$. In Section \ref{Sec bypass chain}, we define a chain map $\cal{F}(\beta) \in \Hom(\cal{F}(\g), \cal{F}(\g'))$ for any nontrivial bypass morphism $\beta \in \Hom(\g,\g')$ and then define $\cal{F}(\xi)$ in general as a composition of chain maps corresponding to bypasses.  In Section \ref{Sec comp}, we show that the functor $\cal{F}$ is well-defined.

\subsection{Motivation from $\cne$} \label{Sec motivation}

%We give Examples \ref{ex g1}, \ref{ex g2}, \ref{ex g3} and \ref{ex g4} of $\cal{F}(\g)$ in this section.
The goal of this subsection is to give some motivating examples.

We say that $\g$ is {\em represented} by $\g', \g''$ if there exists a bypass triangle $\g \ra \g' \ra \g''$ in $\cne$. The idea for constructing $\cal{F}(\g)$ is to iteratively represent $\g$ by basic dividing sets using iterated bypass triangles, and then form a complex of (left) projective $\rne$-modules corresponding to the basic dividing sets.

Recall that the indecomposable projective $\rne$-modules are of the form $\rne(\g)$ for $\g \in \bne$.  Using the notation $\g(\ul{s_1}, \dots, \ul{s_e})$ for the basic dividing set in $\bne$ satisfying $\g(\ul{s_1}, \dots, \ul{s_e})_{\ast}=\{\ul{0}, \ul{s_1}, \dots, \ul{s_e}\}$, we write $P(\ul{s_1}, \dots, \ul{s_e})$ for the projective module corresponding to $\g(\ul{s_1}, \dots, \ul{s_e})$.

\begin{example} \label{ex g1}
There are three dividing sets $\g, \g (\ul{1}), \g (\ul{2})$ in $\mathfrak{ob}(\cal{C}_{2,1})$ as shown in Figure \ref{4-1-1}. Among them $\g$ is not basic since $\g_{\ast}=\{\ul{0}\}$ and $ \g_{(1)}=\{\ul{1},\ul{2}\}$.
There is a bypass triangle $\g \ra \g (\ul{1}) \ra \g (\ul{2})$ in $\cal{C}_{2,1}$ which is induced by a bypass $\beta(\g) \in \Hom(\g , \g(\ul{1}))$.
Here $\beta(\g)$ is the unique nontrivial bypass on $\g$ whose arc of attachment $\delta$ intersects $\g_{\ast}$ at one point and $\delta_+ \subset \g_{(1)}$.
Hence $\g$ is represented by the basic dividing sets $\g (\ul{1}), \g (\ul{2})$.
The bypass in $\Hom(\g (\ul{1}), \g (\ul{2}))$ gives a generator of $R_{2,1}$.
We define $\cal{F}(\g) \in \cal{D}_{2,1}$ as the cochain complex $P (\ul{1}) \ra P(\ul{2})$, where the differential is the multiplication by the generator of $R_{2,1}$ and $P(\ul{2})$ is at degree $0$.
\begin{figure}[ht]
\begin{overpic}
[scale=0.25]{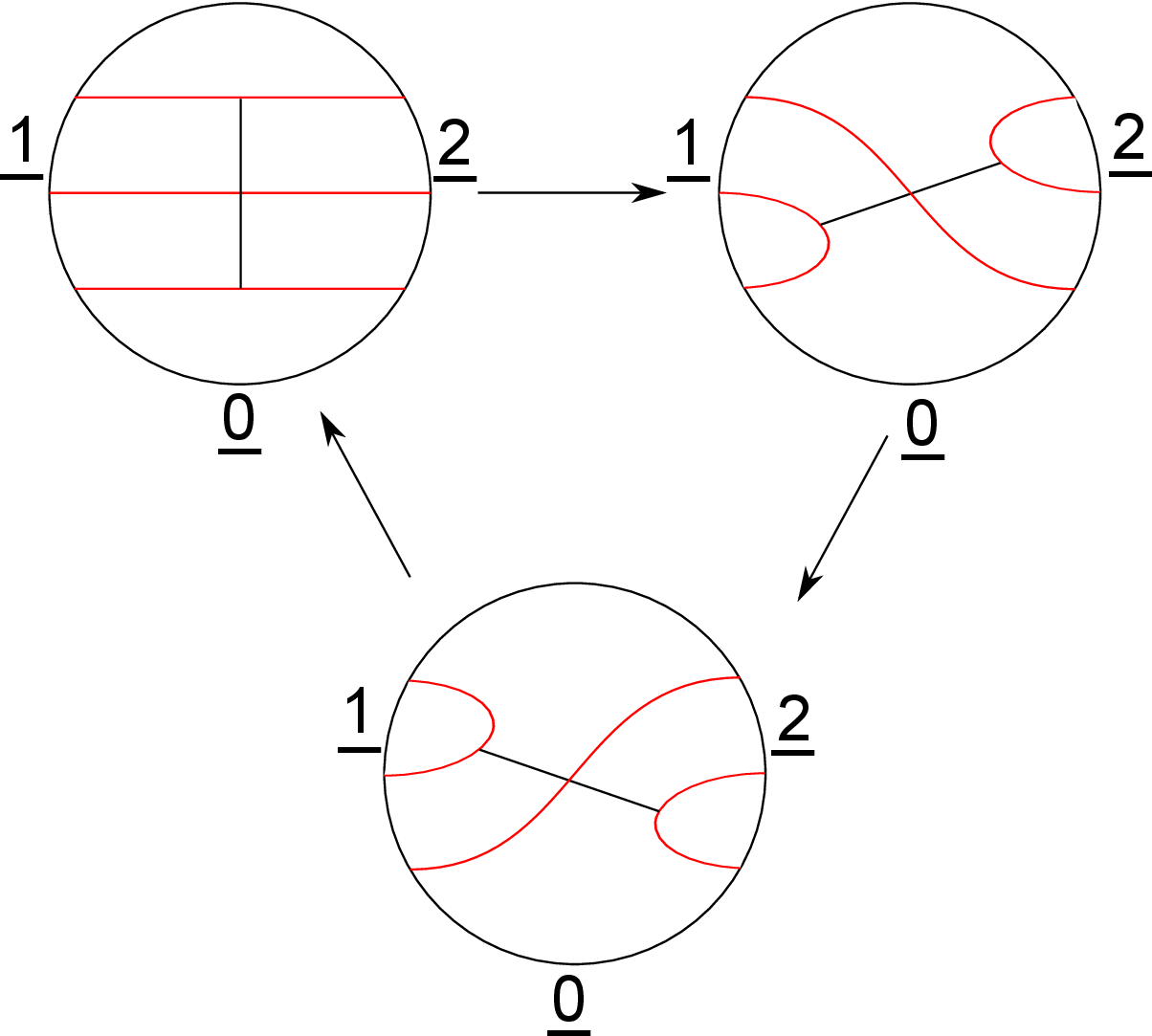}
\put(90,50){$\g (\ul{1})$}
\put(0,50){$\g$}
\put(60,0){$\g (\ul{2})$}
\put(44,78){${\scriptstyle \beta(\g)}$}
\put(18,59){${\scriptstyle \g_{\ast}}$}
\put(8,76){${\scriptstyle \g_{(1)}}$}
\end{overpic}
\caption{$\cal{F}(\g)$ for a non-basic $\g$ in $\cal{C}_{2,1}$.}
\label{4-1-1}
\end{figure}

\end{example}

\s\n {\bf Representing non-basic dividing sets.}
Before proceeding to the next examples, we describe the bypass triangles we choose to iteratively represent a non-basic dividing set by basic ones.  Given a non-basic $\g$, the set $\{i\in \Z_+ ~|~ (i) \in \vnb(\g)\}$ is nonempty.  Let $i_0$ be the smallest element of this set. Let $By(\g)$ denote the set of nontrivial bypasses on $\g$ whose arcs of attachment $\delta$ intersect the closure of $\g_{\ast}$ at one point and satisfy $\delta_+ \subset int( \g_{(i_0)})$.  The set $By(\g)$ is nonempty since the component $\g_{(i_0)}$ is not boundary parallel.  We make the following choice:

\begin{defn}[Choice of $\beta(\g)$] \label{def beta g}
Given a non-basic $\g$, let $\beta(\g)$ be the first bypass in the clockwise direction starting from $\ul{0}$ in $By(\g)$.
As the based arc $\ul{0}$ is usually put at the bottom, $\beta(\g)$ is called the {\em leftmost} bypass in $By(\g)$. 
\end{defn}

Let $\g \xra{\beta(\g)} \g' \ra \g''$ be the triangle induced by $\beta(\g)$. One immediately sees that $\g_{\ast}$ can be viewed as a proper subset of $\g'_{\ast}$ and $\g''_{\ast}$  (in fact this holds for any bypass in $By(\g)$). In other words, $\g'$ and $\g''$ are ``closer" to being basic. If $\g'$ or $\g''$ is not basic, we can further represent it using a triangle induced by $\beta(\g')$ or $\beta(\g'')$. After finitely many steps we can iteratively represent $\g$ by basic dividing sets.

\s
In Example \ref{ex g1}, $\g$ is not basic, $\tpv(\g)=\{(1)\}$, and $l_{\g_{(1)}}=1$.  In the next example we have $\tpv(\g)=\{(1)\}$ and $l_{\g_{(1)}}=2$.

\begin{example} \label{ex g2}
Let $\g\in\mathfrak{ob}(\cal{C}_{4,3})$ such that $\g_{\ast}=\{\ul{0},\ul{4}\}, \g_{(1)}=\{\ul{1},\ul{2},\ul{3}\}$; see Figure \ref{4-1-2}.
The bypasses $\beta(\g)$ and $\beta(\g')$ induce two triangles:
$$\g \xra{\beta(\g)} \g' \ra \g (\ul{2},\ul{3},\ul{4}), \quad \g' \xra{\beta(\g')} \g (\ul{1},\ul{2},\ul{4}) \ra \g (\ul{1},\ul{3},\ul{4}),$$
where $\g'$ is not basic: $\g'_{\ast}=\{\ul{0},\ul{1},\ul{4}\}, \g'_{(1)}=\{\ul{2},\ul{3}\}$.
The non-basic dividing set $\g$ is iteratively represented by basic dividing sets $\g (\ul{1},\ul{2},\ul{4}), \g (\ul{1},\ul{3},\ul{4})$, and $\g (\ul{2},\ul{3},\ul{4})$.

Each of $\Hom(\g(\ul{1},\ul{2},\ul{4}),\g(\ul{1},\ul{3},\ul{4}))$ and $\Hom(\g(\ul{1},\ul{3},\ul{4}),\g(\ul{2},\ul{3},\ul{4}))$ is generated by a nontrivial bypass and their composition is zero by the tightness criterion (Proposition~\ref{prop: tightness criterion}). We define $\cal{F}(\g) \in \cal{D}_{4,3}$ as the cochain complex
$$P (\ul{1},\ul{2},\ul{4}) \ra P(\ul{1},\ul{3},\ul{4}) \ra P (\ul{2},\ul{3},\ul{4}),$$
where the differentials are given by the bypasses and $ P (\ul{2},\ul{3},\ul{4})$ is at degree $0$.
\begin{figure}[ht]
\begin{overpic}
[scale=0.22]{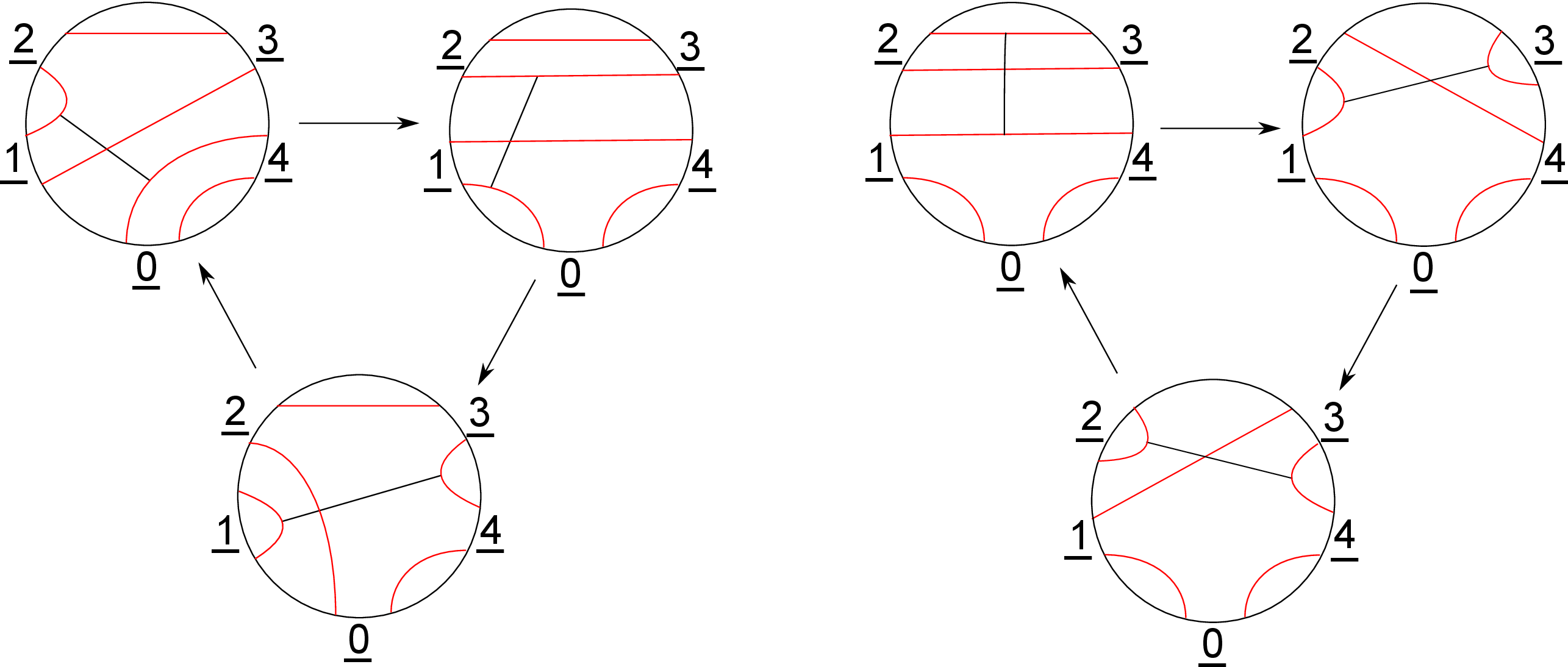}
\put(3,22){$\g$}
\put(40,22){$\g'$}
\put(58,22){$\g'$}
\put(94,22){$\g (\ul{1},\ul{2},\ul{4})$}
\put(84,0){$\g(\ul{1},\ul{3},\ul{4})$}
\put(30,0){$\g(\ul{2},\ul{3},\ul{4})$}
\put(20,37){${\scriptstyle \beta(\g)}$}
\put(74,37){${\scriptstyle \beta(\g')}$}
\end{overpic}
\caption{Two triangles to represent a non-basic $\g$ in $\cal{C}_{4,3}$.}
\label{4-1-2}
\end{figure}
\end{example}

In Examples \ref{ex g1} and \ref{ex g2}, $|\tpv(\g)|=1$.
In the following two examples, we consider the case $|\tpv(\g)|=2$.

\begin{example} \label{ex g3}
Consider $\g\in \mathfrak{ob}(\cal{C}_{4,2})$ such that $\g_{\ast}=\{\ul{0}\}, \g_{(1)}=\{\ul{1},\ul{2}\}, \g_{(2)}=\{\ul{3},\ul{4}\}$; see Figure \ref{4-1-3}.
The bypass $\beta(\g)$ induces a triangle $\g \xra{\beta(\g)} \g' \ra \g''$, where
$$\g'_{\ast}=\{\ul{0},\ul{1}\}, \g'_{(1)}=\{\ul{2}\}, \g'_{(2)}=\{\ul{3},\ul{4}\}; \quad \g''_{\ast}=\{\ul{0},\ul{2}\}, \g'_{(1)}=\{\ul{1}\}, \g'_{(2)}=\{\ul{3},\ul{4}\}.$$
Since $\g'$ and $\g''$ are not basic, there are two more triangles induced by $\beta(\g')$ and $\beta(\g'')$:
$$ \g' \xra{\beta(\g')} \g (\ul{1},\ul{3}) \ra \g (\ul{1},\ul{4}), \quad \g'' \xra{\beta(\g'')} \g (\ul{2},\ul{3}) \ra \g (\ul{2},\ul{4}).$$

Each of the nontrivial morphisms $\g (\ul{1},\ul{3}) \ra \g (\ul{1},\ul{4})$, $\g (\ul{1},\ul{3}) \ra \g (\ul{2},\ul{3})$, $\g (\ul{1},\ul{4}) \ra \g (\ul{2},\ul{4})$, $\g (\ul{2},\ul{3}) \ra \g (\ul{2},\ul{4})$ is given by a nontrivial bypass and the
two ways of composing the bypasses in $\Hom(\g (\ul{1},\ul{3}),\g (\ul{2},\ul{4}))$ commute.  We define $\cal{F}(\g) \in \cal{D}_{4,2}$ as the cochain complex
$$P(\ul{1},\ul{3}) \ra (P(\ul{1},\ul{4})\oplus P(\ul{2},\ul{3})) \ra P(\ul{2},\ul{4}),$$
where the differentials are induced by the bypasses and $P(\ul{2},\ul{4})$ is at degree $0$.
\begin{figure}[ht]
\begin{overpic}
[scale=0.22]{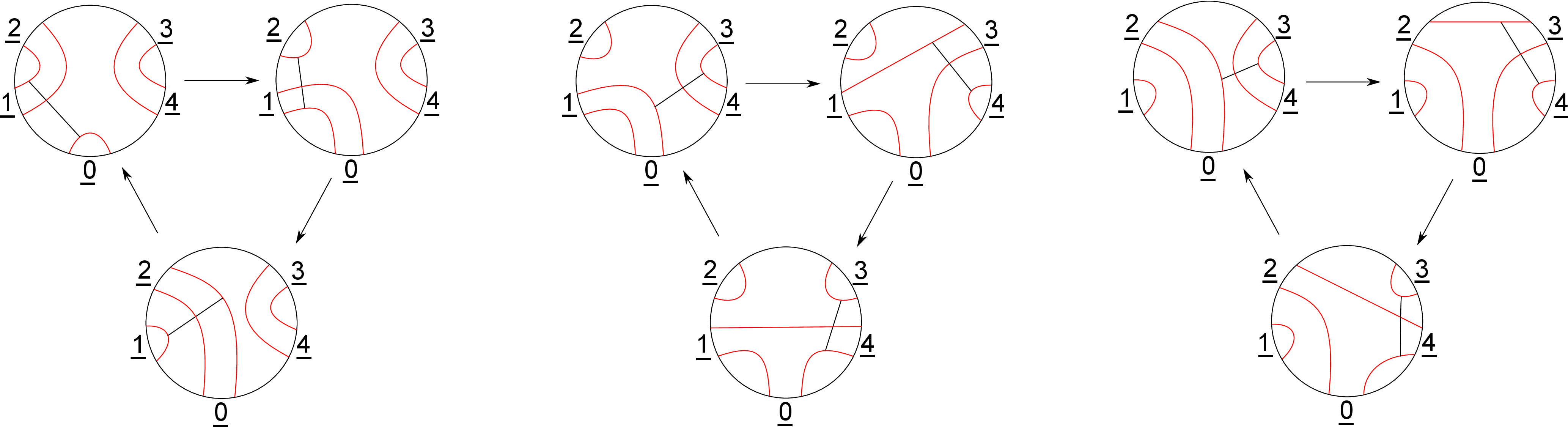}
\put(2,14){$\g$}
\put(25,14){$\g'$}
\put(36,14){$\g'$}
\put(60,14){$\g (\ul{1},\ul{3})$}
\put(72,14){$\g''$}
\put(97,14){$\g (\ul{2},\ul{3})$}
\put(18,0){$\g''$}
\put(55,0){$\g (\ul{1},\ul{4})$}
\put(90,0){$\g (\ul{2},\ul{4})$}
\put(12,23){${\scriptstyle \beta(\g)}$}
\put(48,23){${\scriptstyle \beta(\g')}$}
\put(84,23){${\scriptstyle \beta(\g'')}$}
\end{overpic}
\caption{Three triangles to represent a non-basic $\g$ in $\cal{C}_{4,2}$.}
\label{4-1-3}
\end{figure}

\end{example}

\begin{example} \label{ex g4}
Consider $\g\in\mathfrak{ob}(\cal{C}_{4,2})$ such that $\g_{\ast}=\{\ul{0}\}, \g_{(1)}=\{\ul{1},\ul{4}\}, \g_{(1,1)}=\{\ul{2},\ul{3}\}$; see Figure \ref{4-1-4}.
The bypass $\beta(\g)$ induces a triangle $\g \xra{\beta(\g)} \g' \ra \g''$, where
$$\g'_{\ast}=\{\ul{0},\ul{1}\}, \g'_{(1)}=\{\ul{2},\ul{3}\}, \g'_{(2)}=\{\ul{4}\}; \quad \g''_{\ast}=\{\ul{0},\ul{4}\}, \g'_{(1)}=\{\ul{1}\}, \g'_{(2)}=\{\ul{2},\ul{3}\}.$$
Since $\g'$ and $\g''$ are not basic, there are two more triangles induced by $\beta(\g')$ and $\beta(\g'')$:
$$ \g' \xra{\beta(\g')} \g(\ul{1},\ul{2}) \ra \g(\ul{1},\ul{3}), \quad \g'' \xra{\beta(\g'')} \g(\ul{2},\ul{4}) \ra \g(\ul{3},\ul{4}).$$

There is a tight contact structure in $\Hom(\g(\ul{1},\ul{3}),\g(\ul{2},\ul{4}))$ which is a composition of two bypasses.
The spaces $\Hom(\g(\ul{1},\ul{2}), \g(\ul{2},\ul{4}))$ and $\Hom(\g(\ul{1},\ul{3}), \g(\ul{3},\ul{4}))$ are zero.
We then define $\cal{F}(\g) \in \cal{D}_{4,2}$ as the cochain complex
\begin{gather} \label{ex g4 cx}
\mathcal{N}:= (P(\ul{1},\ul{2}) \ra P(\ul{1},\ul{3}) \ra P(\ul{2},\ul{4}) \ra P(\ul{3},\ul{4})),
\end{gather}
where the differentials are induced by the tight contact structures and $P(\ul{3},\ul{4})$ is at degree $0$.
\begin{figure}[ht]
\begin{overpic}
[scale=0.22]{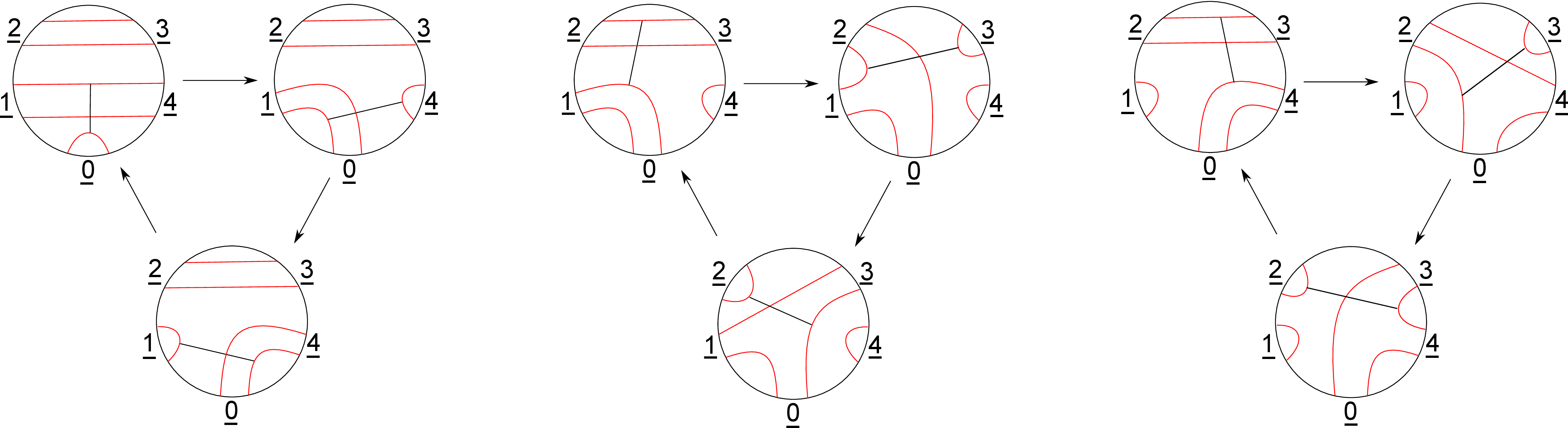}
\put(2,14){$\g$}
\put(25,14){$\g'$}
\put(36,14){$\g'$}
\put(60,14){$\g(\ul{1},\ul{2})$}
\put(72,14){$\g''$}
\put(95,14){$\g(\ul{2},\ul{4})$}
\put(18,0){$\g''$}
\put(55,0){$\g(\ul{1},\ul{3})$}
\put(90,0){$\g(\ul{3},\ul{4})$}
\put(12,23){${\scriptstyle \beta(\g)}$}
\put(48,23){${\scriptstyle \beta(\g')}$}
\put(84,23){${\scriptstyle \beta(\g'')}$}
\end{overpic}
\caption{Three triangles to represent a non-basic $\g$ in $\cal{C}_{4,2}$.}
\label{4-1-4}
\end{figure}
\end{example}

In Example \ref{ex g3}, $V(\g)=\{\ast, (1), (2)\}$, where $\g_{(1)}$ and $\g_{(2)}$ directly nest inside $\g_\ast$.
In Example \ref{ex g4}, $V(\g)=\{\ast, (1), (1,1)\}$ and $\g_{(1,1)}$ directly nests inside $\g_{(1)}$ which in turn directly nests inside $\g_\ast$.
As we will see in Equation~\eqref{eqn: def of cal F Gamma in general}, the differentials in $\cal{F}(\g)$ are defined differently for the two examples.

\subsection{Definition of $\cal{F}(\g)$} \label{Sec dividing}

%\begin{defn}[$\cal{F}(\g)$ for basic $\g$] \label{def fgI}

In this subsection we define $\cal{F}(\g)$ for $\g\in \mathfrak{ob}(\cne)$.  If $\g \in \bne$, then we set
\begin{equation} \label{eqn: def of cal F for basic Gamma}
\cal{F}(\g):=P(\g) \in {\frak{ob}}(\dne),
\end{equation}
viewed as a complex centered at degree $0$, and if $\g$ is a zero object, then we set $\cal{F}(\g)=0$.
%\end{defn}

{\em In the rest of this subsection suppose $\g$ is nonzero.} We use bypass triangles in $\cne$ and construct $\cal{F}(\g) \in {\frak{ob}}(\dne)$ in $3$ steps:
\be
\item[{\em Step 1.}] Make a list of projective $\rne$-modules that appear in $\cal{F}(\g)$.
\item[{\em Step 2.}] Define the cohomological degree for each term in Step 1.
\item[{\em Step 3.}] Define the differential using Steps 1 and 2.
\ee

\s
The following definition generalizes the ad hoc definitions in Examples \ref{ex g1}, \ref{ex g2}, \ref{ex g3}, \ref{ex g4}.

\s
\n{\em Step 1.} The list of projective $\rne$-modules that appear in $\cal{F}(\g)$ is given by $\{\g(\mf{i}) \in \bne ~|~ \mf{i} \in \oi(\g)\}$, which we describe now.

In the examples from Section \ref{Sec motivation}, every $\g$ was represented by an iterated cone of certain basic dividing sets.
The key observation is that the based component of each of these basic dividing sets was obtained from the total set $\{\ul{0},\dots,\ul{n}\}$ by omitting one label from each component in $\tpv(\g)$.

\begin{defn}[Omitting index] \label{def oi}
Let
$$\oi(\g)=\prod\limits_{\mf{v} \in \tpv(\g)}\{0,\dots,l_{\gv}\}.$$
An element $\mi=\lan \iv \ran \in \oi(\g)$ is called an {\em omitting index} of $\g$ and $\iv$ is called the {\em entry of $\mi$ corresponding to $\mv$} (or the ``$\mv$-entry of $\mi$", for short).  Also, the set $\{\g_\mv(\iv)\}_{\mv\in \tpv(\g)}$ is called the {\em set of omitting labels for $\mi$}.
\end{defn}

We define ${\bf 0} \in \oi(\g)$ to be the omitting index such that ${\bf 0}_{\mf{v}}=0$ for any $\mf{v} \in \tpv(\g)$.

\begin{rmk} \label{rmk OI RPV}
Any $\mf{i} \in \oi(\g)$ is determined by $\{ \iv ~|~ \mf{v} \in \vnb(\g) \}$ since $\iv=0$ for all $\mf{v} \in \tpv(\g) \backslash \vnb(\g)$.
\end{rmk}

%\begin{defn} \label{def gi}
Given $\mf{i} \in \oi(\g)$, define $\g(\mf{i}) \in \bne$ such that
\begin{equation}
\g(\mf{i})_{\ast}=\g_{\ast}\cup\bigsqcup_{\mf{v} \in \tpv(\g)}(\gv \backslash \{\gv(i_{\mf{v}})\})=\{\ul{0},\dots,\ul{n}\}\backslash \{\gv(i_{\mf{v}})\}_{\mf{v} \in \tpv(\g)} .
\end{equation}
%\end{defn}
Since $|V(\g)|=n-e+1, |\tpv(\g)|=n-e$, and $\bigsqcup_{\mf{v} \in V(\g)} \gv=\{\ul{0},\dots,\ul{n}\}$, it follows that $|\g(\mf{i})_{\ast}|=e+1$.
%$$\left|\bigcup\limits_{\mf{v} \in \tpv(\g)}(\gv \backslash \{\gv(i_{\mf{v}})\}) \cup \g_{\ast}\right|=\left|\{\ul{0},\dots,\ul{n}\}\backslash \left(\bigcup\limits_{\mf{v} \in \tpv(\g)} \{\gv(i_{\mf{v}})\} \right)\right|=e+1.$$
Observe that (i) if $\g \in \bne$, then we have $\oi(\g)=\{{\bf 0}\}$ and $\g({\bf 0})=\g$, since $l_{\gv}=0$ for all $\mf{v} \in \tpv(\g)$, and (ii) if $\g \notin \bne$ and $\mi \in \oi(\g)$, then $\g(\mi)_{\ast}$ always contains $\g_{\ast}$ as a proper subset.

\begin{rmk}
We have two ways of describing basic dividing sets; they are complementary in some sense. The first one $\g(\ul{s_1},\dots,\ul{s_e})$ describes the labels that are contained in the based component and is mainly used in examples. The second one $\g(\mi)$ emphasizes the labels that are omitted (i.e., the set $\mi \in \oi(\g)$) and is our choice for most of the paper.\footnote{The notations are similar, but we note that the former has entries that are underlined.}
\end{rmk}

\s \n{\em Step 2.} We define the cohomological degree $h(\mf{i})$ for each $\mf{i} \in \oi(\g)$.

\begin{defn} \label{def contain}$\mbox{}$
\be
\item[(i)] Given $\mv \in V(\g)$ and $0 \leq i \leq l_{\gv}$, a vector $\mf{w} \in \vnb(\g)$ {\em nests inside $\mf{v}$ up to $i$} if $\gw \subset (\gv(0), \gv(i))$.
\item[(ii)] Given $\mf{v} \in V(\g)$ and $0 < i \leq l_{\gv}$, a vector $\mf{w} \in \vnb(\g)$ is a {\em direct nesting vector} of $(\mf{v},i)$ if $\gw$ directly nests inside $\gv$ and $\gw \subset (\gv(i-1),\gv(i))$.
\ee
\end{defn}

%\begin{defn} \label{def nest vi}

Let $\nv(\mf{v},i)$ denote the set of vectors in $\vnb(\g)$ that nest inside $\mf{v}$ up to $i$ and let $\dnv(\mf{v},i)$ denote the set of direct nesting vectors of $(\mf{v},i)$.

\begin{defn}[Nesting degree] \label{def nest degree}
Given $\mv \in V(\g)$ and $0 \leq i \leq l_{\gv}$, the {\em nesting degree} $c_{\mf{v}}(i)$ is given by
$$c_{\mf{v}}(i)=\sum\limits_{\mf{w} \in \nv(\mf{v}, i)}l_{\gw}=\sum\limits_{\mf{w} \in \nv(\mf{v}, i)}(|\gw|-1),$$
if $\nv(\mf{v},i)\neq\es$ and is zero otherwise. %$c_{\mf{v}}(i)=0$.
\end{defn}

%The expression of $c_{\mf{v}}(i)$ is simpler using $l_{\gw}$ other than $|\gw|$.

\begin{defn}[Cohomological degree] \label{def coh degree}
The {\em cohomological degree $h(\mf{i})$} of $\mf{i} \in \oi(\g)$ is given by
$$h(\mf{i})=\sum\limits_{\mf{v} \in \tpv(\g)}h(\mf{i},\mf{v}),\quad h(\mf{i},\mf{v})=\iv + c_{\mf{v}}(\iv).$$
\end{defn}

\begin{rmk}
Since $\nv(\mf{v},0)=\es$, we have $c_{\mf{v}}(0)=0$ and $h({\bf 0})=0$.
\end{rmk}

The nesting degree is trivial in Examples \ref{ex g1}, \ref{ex g2}, and \ref{ex g3}.  Hence $h(\mf{i})$ is simply the sum of all the entries of $\mf{i}$.

\s
\n{\em Example \ref{ex g4} revisited.}
We have $\tpv(\g)=\{(1),(1,1)\}$ and $l_{\gv}=1$ for $\mf{v}\in \tpv(\g)$.
The only nonzero nesting degree is $c_{(1)}(1)=1$ since $\g_{(1,1)}$ nests inside $\g_{(1)}$ up to $1$: $$\g_{(1,1)}=\{\ul{2},\ul{3}\} \subset (\ul{1},\ul{4}) =(\g_{(1)}(0),\g_{(1)}(1)).$$
Writing $\mi=\lan \mi_{(1)}, \mi_{(1,1)} \ran$ for $\mf{i} \in \oi(\g)$, we have:
$$h\lan 1,1 \ran=1+1+c_{(1)}(1)=3, \quad h \lan 1,0 \ran=1+0+c_{(1)}(1)=2, \quad h\lan 0,1 \ran=0+1=1, \quad h\lan 0,0 \ran=0,$$
which agree with the negatives of the degrees in the complex $\mathcal{N}$ from \eqref{ex g4 cx}.

\s
\n{\em Step 3.} We define the differential, which is induced by the morphisms between basic dividing sets in $\cne$.

\s\n {\em Example \ref{ex g4} revisited.}  We have
$$\g \lan 1,1 \ran=\g(\ul{1},\ul{2}), \quad \g \lan 1,0 \ran=\g (\ul{1},\ul{3}),\quad  \g \lan 0,1 \ran=\g(\ul{2},\ul{4}),\quad \g \lan 0,0 \ran=\g(\ul{3},\ul{4}).$$
There are two types of morphisms between the $\g(\mf{i})$'s for $\mf{i}=\lan \mi_{(1)}, \mi_{(1,1)} \ran \in OI(\g)$: \be \item[(SL)] $\Hom(\g \lan i,1 \ran,\g \lan i,0 \ran)$ for $i=0,1$;
\item[(SH)] $\Hom(\g\lan 1,0 \ran ,\g \lan 0,1 \ran)$.
\ee

For Type (SL), only the $(1,1)$-entry of $\mf{i}=\lan i,1 \ran$ decreases by $1$ and the other entry is left unchanged.
In Definition \ref{def diff vector} the vector $(1,1)$ is called an {\em sliding vector} of $\mf{i}$.

For Type (SH), the $(1)$-entry of $\mf{i}=\lan 1,0 \ran$ decreases from $1$ to $0$ and the $(1,1)$-entry increases from $0$ to $1=l_{\g_{(1,1)}}$.
In this case $(1,1)$ directly nests inside $(1)$.
In Definition \ref{def diff vector} the vector $(1)$ is called a {\em shuffling vector} of $\mf{i}=\lan 1,0 \ran$.

\begin{defn}[Sliding and shuffling vectors] \label{def diff vector}  $\mbox{}$
\be
\item A vector $\mf{v} \in \tpv(\g)$ is a {\em sliding vector} of $\mf{i} \in \oi(\g)$ if $\iv>0$ and $\dnv(\mf{v},\mi_{\mf{v}})=\es$.
\item A vector $\mf{v} \in \tpv(\g)$ is a {\em shuffling vector} of $\mf{i} \in \oi(\g)$ if $\iv>0, \dnv(\mf{v},\mi_{\mf{v}})\neq\es$ and $\mi_{\mf{w}}=0$ for all $\mf{w} \in \dnv(\mf{v},\iv)$.
\ee
%Let $DV(\mf{i})$ denote the union $\sv(\mf{i})$.
\end{defn}

Let $\slv(\mf{i})$ denote the set of sliding vectors of $\mf{i}$, let $\shv(\mf{i})$ denote the set of shuffling vectors of $\mf{i}$, and let $\sv(\mf{i})= \slv(\mf{i})\cup \shv(\mf{i})$.
Note that not every vector in $\tpv(\g)$ is a sliding vector or a shuffling vector.

\begin{defn}[Modified omitting index] \label{def diff ind}
Given $\mf{i} \in \oi(\g)$ and $\mf{v} \in \sv(\mf{i})$, the {\em $\mf{v}$-modified omitting index $\mf{v} | \mf{i} \in \oi(\g)$} satisfies
$$ (\mf{v} | \mf{i})_{\mf{w}}= \left\{
\begin{array}{cl}
\mi_{\mf{v}}-1 & \mbox{if}~ \mf{w} = \mf{v}, \\
\mi_{\mf{w}} & \mbox{otherwise,}
\end{array}\right.
$$
if $\mf{v} \in \slv(\mf{i})$;
$$ (\mf{v} | \mf{i})_{\mf{w}}= \left\{
\begin{array}{cl}
\mi_{\mf{v}}-1 & \mbox{if}~ \mf{w} = \mf{v}, \\
l_{\gw} & \mbox{if}~ \mf{w} \in \dnv(\mf{v},\mi_{\mf{v}}), \\
\mi_{\mf{w}} & \mbox{otherwise,}
\end{array}\right.
$$
if $\mf{v} \in \shv(\mf{i})$.
\end{defn}

Vectors in $\dnv(\mf{v},\mi_{\mf{v}}) \subset \vnb(\g)$ must be indices of non-boundary-parallel regions, so that the $l_{\gw}$ are positive for all $\mf{w} \in \dnv(\mf{v},\mi_{\mf{v}})$.

\begin{rmk} \label{rmk diff}
As we change from $\mi$ to $\vi$,
\begin{itemize}
\item[(i)] the $\mv$-entry of $\mi$ is the only entry which decreases;
\item[(ii)] a $\mf{w}$-entry of $\mi$ increases if and only if $\mf{v} \in \shv(\mf{i})$ and $\mf{w} \in \dnv(\mf{v},\mi_{\mf{v}})$; and
\item[(iii)] all other entries of $\mi$ are left unchanged.
\end{itemize}
\end{rmk}

\begin{lemma} \label{lem diff path}
If $\mf{v} \in \sv(\mf{i})$, then $\Hom(\g(\mf{i}), \g(\mf{v} | \mf{i}))\not=0$.
\end{lemma}

Given $\mf{v} \in \sv(\mf{i})$, let $r(\mf{i},\mf{v}) \in \rne$ be the generator of $\Hom(\g(\mf{i}), \g(\mf{v} | \mf{i}))$.

\begin{proof}
If $\mf{v} \in \slv(\mf{i})$, then $\g(\mf{i})_{\ast} \cap [\gv(\mi_{\mf{v}}-1),\gv(\mi_{\mf{v}})]=\{\gv(\mi_{\mf{v}}-1)\}$ and
$$\g(\mf{v} | \mf{i})_{\ast}=\g(\mf{i})_{\ast} \backslash \{\gv(\mi_{\mf{v}}-1)\} \cup \{\gv(\mi_{\mf{v}})\}.$$
Then $\Hom(\g(\mf{i}), \g(\mf{v} | \mf{i}))\not=0$ by Proposition \ref{prop: tightness criterion}.

If $\mf{v} \in \shv(\mf{i})$, then we can write $\dnv(\mf{v},\mi_{\mf{v}})=\{\mf{u}^1, \dots, \mf{u}^k\}$ so that
$$\gv(\mi_{\mf{v}}-1) < \g_{\mf{u}^{1}}(0); \quad \g_{\mf{u}^{j}}(l_{\g_{\mf{u}^{j}}})<\g_{\mf{u}^{j+1}}(0),~ 1 \leq j \leq k-1; \quad \g_{\mf{u}^{k}}(l_{\g_{\mf{u}^{k}}}) < \gv(\mi_{\mf{v}}).$$
If $[a,b]$ is any of the corresponding $k+1$ closed disjoint intervals:
\begin{equation} \label{eqn: shuffling intervals}
[\gv(\mi_{\mf{v}}-1) , \g_{\mf{u}^{1}}(0)]; \quad [\g_{\mf{u}^{j}}(l_{\g_{\mf{u}^{j}}}),\g_{\mf{u}^{j+1}}(0)], ~1 \leq j \leq k-1; \quad [\g_{\mf{u}^{k}}(l_{\g_{\mf{u}^{k}}}) , \gv(\mi_{\mf{v}})],
\end{equation}
the intersections $\g(\mf{i})_{\ast} \cap [a,b]=\{a\}$. We have
\begin{gather*}
\g(\mf{v} | \mf{i})_{\ast}=\g(\mf{i})_{\ast} \backslash \left(  \{ \gv(\mi_{\mf{v}}-1)\}\cup \{\g_{\mf{u}^{j}}(l_{\g_{\mf{u}^{j}}})\}_{j=1}^k \right) \cup \left(\{\g_{\mf{u}^{j}}(0)\}_{j=1}^k \cup \{\gv(\mi_{\mf{v}})\}\right).
\end{gather*}
Hence $\Hom(\g(\mf{i}), \g(\mf{v} | \mf{i}))\not=0$ by Proposition \ref{prop: tightness criterion}. \end{proof}

\begin{rmk} \label{rmk diff path}
With a little more work one can show that
$\{\g(\mi) \ra \g(\vi) ~|~ \mf{v} \in \sv(\mf{i})\}$
coincides with the set of minimal elements of $\{\mbox{nonzero path}~~ \g(\mi) \ra \g(\mj) ~|~ \mi,\mj \in \oi(\g)\}$
with respect to the partial order $\leq$ from Section~\ref{section: defn of algebra}.
This is actually the motivation behind Definitions \ref{def diff vector} and \ref{def diff ind}.
\end{rmk}

%\begin{defn}[Differential intervals] When $\mf{v} \in \slv(\mf{i})$, the closed interval $[\gv(\mi_{\mf{v}}-1),\gv(\mi_{\mf{v}})]$ is called the {\em sliding differential interval} of $(\mf{i},\mf{v})$.  When $\mf{v} \in \shv(\mf{i})$, the $k+1$ closed intervals in Equation~\eqref{eqn: shuffling intervals} are called the {\em shuffling differential intervals} of $(\mf{i},\mf{v})$. Both are called the {\em differential intervals} of $(\mf{i},\mf{v})$. \end{defn}

We now define $\cal{F}(\g)$ for $\g$ in general:
\begin{equation}\label{eqn: def of cal F Gamma in general}
\cal{F}(\g)=\left(\bigoplus\limits_{\mf{i} \in \oi(\g)}P(\g(\mf{i}))[h(\mf{i})], ~~~ d_{\g}=\sum\limits_{\mf{i} \in \oi(\g)}\sum\limits_{\mf{v}}d(\mf{i},\mf{v})\right),
\end{equation}
where the second summation is taken over $\mv \in \sv(\mf{i})$ and  $d(\mf{i},\mf{v}): P(\g(\mf{i})) \ra P(\g(\mf{v}|\mf{i}))$
is given by right multiplication by $r(\mf{i},\mf{v})$.

\begin{rmk} $\mbox{}$
\be
\item If $\g \in \bne$, then $\oi(\g)=\{{\bf 0}\}, \g({\bf 0})=\g$, and $h({\bf 0})=0$.  Hence $\cal{F}(\g)=(P(\g), d_{\g}=0)$, which agrees with $\cal{F}(\g)$ from Equation~\eqref{eqn: def of cal F for basic Gamma}.
\item By the usual grading shift convention, $P(\g(\mf{i}))[h(\mf{i})]$ is at degree $-h(\mi)$.  Since $h(\mi)$ is nonnegative, the highest degree term of $\cf(\g)$ has degree $0$.
\ee
\end{rmk}

We will write $d$ for $d_{\g}$ when $\g$ is understood. By definition $d: \cal{F}(\g) \ra \cal{F}(\g)$ is a map of $\rne$-modules. It remains to verify that $d^2=0$ and $d$ is homogeneous of degree $1$; they are proved in Lemmas~\ref{lem dv} and \ref{lem comp triangle}.

\s\n
{\bf Interpretation in terms of negative regions.}  We now give a slightly more unified way of describing $d=d_\g$ in terms of the negative region $R_-(\g)$.  Let $\mf{i}\in \oi(\g)$ and let ${\frak c}$ be a component of $R_-(\g)$ such that:
\be
\item[(*)] if the component $\g_{\mv}$ of $R_+(\g)$ has (nonempty) boundary $\gamma_{\mv}$ in common with ${\frak c}$, then $\mv\not=\ast$.
\ee 
Then we say $\mf{i}$ is {\em ${\frak c}$-admissible} if
\be
\item[(**)] the omitting label of $\mv$ satisfying (*) is the label of the interval of $\bdry D^2$ which is adjacent to the initial point of $\gamma_{\mv}$.   (Recall that $\Gamma$ is oriented as the boundary of $R_+(\g)$.) 
\ee
 If $\mf{i}$ is ${\frak c}$-admissible, then ${\frak c}|\mf{i}$ is obtained from $\mf{i}$ by replacing the label of each $\mv$ satisfying (*) by the label which is adjacent to the terminal point of $\gamma_{\mv}$; see Figure \ref{negative}.
Observe that if $\mf{i}\in \oi(\g)$ and $\mv\in \sv(\mf{i})$, then there is a unique ${\frak c}$ such that $\mf{i}$ is ${\frak c}$-admissible and $\mv|\mf{i}={\frak c}|\mf{i}$.  In such a case we write $r(\mf{i},{\frak c})=r(\mf{i},\mf{v})$ and $d(\mf{i},{\frak c})=d(\mf{i},\mf{v})$.  Hence $d(\mf{i},{\frak c})$ can be viewed as a refinement of $d(\mf{i},\mf{v})$.

\begin{figure}[ht]
\begin{overpic}
[scale=0.25]{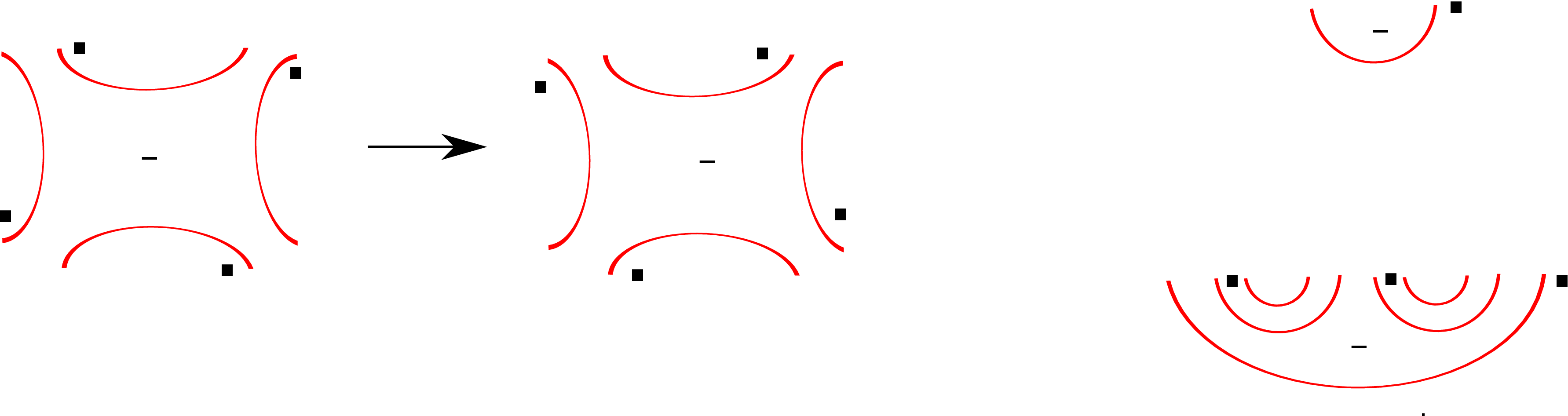}
\put(10,5){$\mf{i}$}
\put(42,5){${\frak c}|\mf{i}$}
\put(12,15){${\frak c}$}
\put(85,24){${\frak c}$}
\put(82,22){$\mf{v}$}
\put(82,3){${\frak c}$}
\put(80,0){$\mf{v}$}
\end{overpic}
\caption{The picture on the left describes $d(\mf{i},{\frak c})$, where dividing sets are $\gamma_{\mv}$, and black boxes are locations of omitting labels for $\mf{i}$ and ${\frak c}|\mf{i}$. The pictures on the right are ${\frak c}$-admissible omitting indices $\mi$ such that $\vi={\frak c}|\mf{i}$, where $\mv \in \slv(\mi)$ and $\shv(\mi)$, respectively.}
\label{negative}
\end{figure}

\s
Given ${\frak c} \in \pi_0(R_-(\g))$, let us define
\begin{equation} \label{eqn: def of dc}
d_{\frak c}=\sum_{\mi}d(\mf{i},{\frak c}),
\end{equation}
where the summation in the first equation is taken over $\mi \in \oi(\g)$ such that $\mi$ is ${\frak c}$-admissible.  We immediately see that $d=\sum_{{\frak c}\in \pi_0(R_-(\g))} d_{\frak c}$.  Let us also write
\begin{equation} \label{eqn: def of dv}
d_{\mv}=\sum_{\frak c} d_{\frak c},
\end{equation}
where ${\frak c}$ has a boundary component in common with $\mv$ and $\mv$ is closer to the based component.

\begin{lemma} \label{lem dv} $\mbox{}$
\be
\item If ${\frak c} \in \pi_0(R_-(\g))$, then $d_{\frak c}^2=0$.
\item If ${\frak c}, {\frak c}' \in \pi_0(R_-(\g))$ and ${\frak c}\not={\frak c}'$, then $d_{{\frak c}}d_{{\frak c}'}=d_{\frak c}d_{{\frak c}'}$.
\item $d^2=0$.
\ee
\end{lemma}

\begin{proof}
(1) This is immediate from observing that ${\frak c}|\mf{i}$ is not ${\frak c}$-admissible.

\s\n
(2) Suppose ${\frak c}$ and ${\frak c}'$ are adjacent, i.e., there is a component of $R_+(\g)$ labeled by $\mv$ which has boundary in common with both ${\frak c}$ and ${\frak c}'$.  If $\mf{i}$ is not ${\frak c}$- or ${\frak c}'$-admissible, then $d(\mf{i},{\frak c})=d(\mf{i},{\frak c}')=0$.  If $\mf{i}$ is ${\frak c}$-admissible, then $\mf{i}$ cannot be ${\frak c}'$-admissible and $d(\mf{i},{\frak c}')=0$.  If ${\frak c}|\mf{i}$ is ${\frak c}'$-admissible, then there are components of $\bdry D^2\cap {\frak c}$ and $\bdry D^2\cap {\frak c}'$ that are adjacent to a label of $\mv$; however, $d({\frak c}|\mf{i},{\frak c}')\circ d(\mf{i},{\frak c})=0$ by Corollary~\ref{cor stack2}. In any case $d_{\frak c}d_{{\frak c}'}=d_{\frak c'}d_{\frak c}=0$.

Suppose ${\frak c}$ and ${\frak c}'$ are not adjacent.  Then clearly ${\frak c}|({\frak c}'|\mi)={\frak c}'|({\frak c}|\mi)$ (if either side exists) and  $d_{\frak c}d_{{\frak c}'}=d_{\frak c'}d_{\frak c}$.

\s\n
(3) It follows from (1), (2), and $d=\sum _{\frak c} d_{\frak c}$.
\end{proof}

\begin{lemma} \label{lem rel nc}
If $\mf{v} \in V(\g)$ and $0 < i \leq l_{\gv}$, then the following holds as subsets of $\vnb(\g)$:
$$\nv(\mf{v},i) = \nv(\mf{v},i-1) \sqcup \dnv(\mf{v},i) \sqcup \bigsqcup\limits_{\mf{w}\in \dnv(\mf{v},i)} \nv(\mf{w}, l_{\gw}).$$
\end{lemma}

\begin{proof}
By Definition \ref{def contain},
$$\nv(\mf{v},i) = \nv(\mf{v},i-1) \sqcup \{\mf{u} \in \vnb(\g) ~|~ \g_{\mf{u}} \subset (\gv(i-1),\gv(i))\}.$$
For any $\mf{u}$ satisfying $\g_{\mf{u}} \subset (\gv(i-1),\gv(i))$, either $\gu$ directly nests inside $\gv$, i.e., $\mf{u} \in \dnv(\mf{v},i)$; or $\gu$ nests inside $\gw$ for a unique $\mf{w} \in \dnv(\mf{v},i)$, i.e., $\mf{u} \in \nv(\mf{w}, l_{\gw})$.
\end{proof}

\begin{lemma} \label{lem d degree}
The degree of $d$ is $1$.
\end{lemma}

\begin{proof}
Since the term $P(\g(\mf{i}))[h(\mf{i})]$ is at cohomological degree $-h(\mf{i})$, it suffices to show that $h(\mf{v} | \mf{i})=h(\mf{i})-1$ for $\mf{v} \in \sv(\mf{i})$.

If $\mf{v} \in \slv(\mf{i})$, then $\mf{i}\mapsto \mf{\vi}$ leaves all the entries of $\mf{i}$ unchanged except for the $\mf{v}$-entry.
In particular, $h(\vi, \mf{w})=h(\mf{i}, \mf{w})$ for $\mf{w} \neq \mf{v}$.
Since $\dnv(\mf{v},\mi_{\mf{v}})=\es$, $\nv(\mf{v},\mi_{\mf{v}}-1)=\nv(\mf{v},\mi_{\mf{v}})$ by Lemma \ref{lem rel nc}.
Hence $c_{\mf{v}}(\mi_{\mf{v}}-1)=c_{\mf{v}}(\mi_{\mf{v}})$ and $h(\mf{v} | \mf{i}, \mf{v})=h(\mf{i}, \mf{v})-1$.  This implies that $h(\mf{v} | \mf{i})=h(\mf{i})-1$.

If $\mf{v} \in \shv(\mf{i})$, then the entries of $\mf{i}$ that are unchanged by $\mf{i}\mapsto \mf{\vi}$ are those of $\mf{w} \notin \{\mf{v}\}\cup \dnv(\mf{v},\mi_{\mf{v}}) $.  Hence $h(\vi, \mf{w})=h(\mf{i}, \mf{w})$ for $\mf{w} \notin \{\mf{v}\}\cup \dnv(\mf{v},\mi_{\mf{v}}) $.
It remains to show that
$$\left(\sum\limits_{\mf{w} \in \dnv(\mf{v},\mi_{\mf{v}})}h(\vi,\mf{w})\right)+h(\vi,\mf{v}) =\left(\sum\limits_{\mf{w} \in \dnv(\mf{v},\mi_{\mf{v}})}h(\mf{i},\mf{w})\right)+h(\mf{i},\mf{v})-1.$$
By Definitions \ref{def nest degree} and \ref{def coh degree} this can be rewritten as
\begin{align*}
 \sum\limits_{\mf{w} \in \dnv(\mf{v},\mi_{\mf{v}})}(c_{\mf{w}}(l_{\gw})+l_{\gw})+c_{\mf{v}}(\mi_{\mf{v}}-1)+\mi_{\mf{v}}-1 & =\sum\limits_{\mf{w} \in \dnv(\mf{v},\mi_{\mf{v}})}(c_{\mf{w}}(0)+0)+c_{\mf{v}}(\mi_{\mf{v}})+\mi_{\mf{v}}-1,\\
 \sum\limits_{\mf{w} \in \dnv(\mf{v},\mi_{\mf{v}})}(c_{\mf{w}}(l_{\gw})+l_{\gw})+c_{\mf{v}}(\mi_{\mf{v}}-1) & =c_{\mf{v}}(\mi_{\mf{v}}).
\end{align*}
The last equation follows from Lemma \ref{lem rel nc}.
\end{proof}

\subsection{Definition of $\cal{F}(\beta)$} \label{Sec bypass chain}

In this subsection we define $\cal{F}(\beta) \in \Hom(\cal{F}(\g), \cal{F}(\g'))$ for any nontrivial bypass $\beta \in \Hom(\g,\g')$. Recall from Notation~\ref{notation: beta} and Notation~\ref{notation: beta 2} (see also Figure~\ref{2-2-2}) that a bypass $\beta$ is described by two vectors $\uv, \ov \in V(\g)$, three integers $x,y,z$, and a partition $\guv=\guv^l \sqcup \guv^r$.

\subsubsection{Identity and shuffling indices}

Given $\mf{i} \in \oi(\g)$ and $\mf{j}\in \oi(\g')$, the restriction of $\cal{F}(\beta)$ to $P(\g(\mf{i}))\to P(\g'(\mf{j}))$ will be the zero map or one of two types:
\be
\item[(Id)] the identity map $P(\g(\mf{i})) \ra P(\g'(\mf{j}))$ for a unique $\mf{j} \in \oi(\g')$;
\item[(Sh)] a nonzero map $P(\g(\mf{i})) \ra P(\g'(\mf{j}))$ for a unique $\mf{j} \in \oi(\g')$.
\ee
For each $\mf{i}\in \oi(\g)$, the unique $\mf{j}\in \oi(\g')$, if it exists, satisfies the condition that the path from $\g(\mf{i})$ to $\g'(\mf{j})$ (possibly the identity path) is the shortest nonzero path in $\qne$ starting from $\g(\mf{i})$.
%In practice, $\cal{F}(\beta)|_{\g(\mf{i})}$ moves a collection of omitting labels associated to $\mi \in \oi(\g)$ as small as possible to obtain another collection of omitting labels associated to some $\mj \in \oi(\g')$.
The main distinction between Types (Id) and (Sh) is whether there exists $\mf{j} \in \oi(\g')$ such that $\g'(\mf{j}) = \g(\mf{i})$.
In each type there are two subcases $\uv=\ast$ or $\uv \neq \ast$; see Figures \ref{4-2-1} and \ref{4-2-2}.

\s\n {\bf Type (Id).}

\begin{defn}[Identity index] \label{def ii}
An omitting index $\mf{i} \in \oi(\g)$ is called an {\em identity index} of $\beta$ if either
\be
\item[(a)] $\ul{0} \in \guv^l$; or
\item[(b)] $\ul{0} \notin \guv$ and $\mi_{\uv} \in [[x,y]]$ (i.e., $\guv(\mi_{\uv}) \in \guv^l$).
\ee
Let $II(\beta) \subset \oi(\g)$ denote the set of identity indices of $\beta$.
\end{defn}

If $\mf{i}\in II(\beta)$, then there exists a unique $\mf{j} \in \oi(\g')$ such that $\g'(\mf{j}) = \g(\mf{i})$.

\s\s
\begin{figure}[ht]
\begin{overpic}
[scale=0.3]{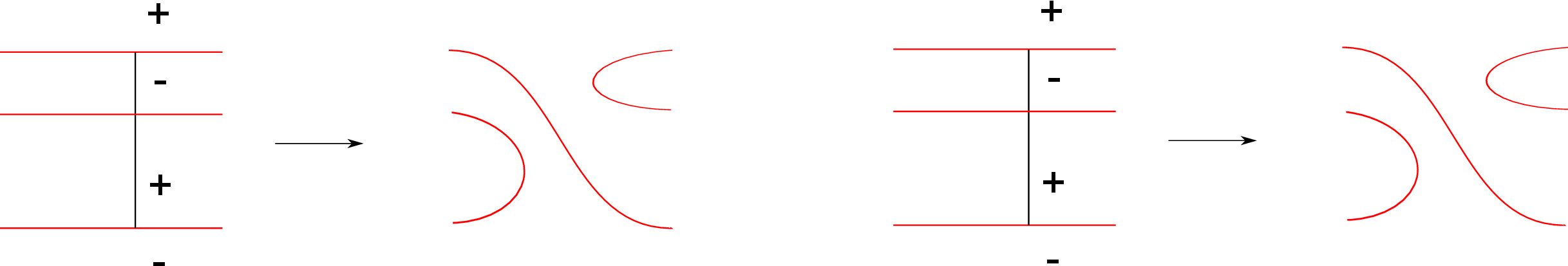}
\put(0,5){$\ul{0}$}
\put(28,5){$\ul{0}$}
\put(5,16){$\mi_{\ov}$}
\put(34,16){$\mj_{\beta(\ov)}$}
\put(57,5){$\mi_{\uv}$}
\put(83,5){$\mj_{\beta(\uv)}$}
\put(62,16){$\mi_{\ov}$}
\put(91,16){$\mj_{\beta(\ov)}$}
\end{overpic}
\caption{Type (Id).  The case $\ul{0} \in \guv^l$ is on the left and the case $\ul{0} \notin \guv$ and $\mi_{\uv} \in [[x,y]]$ is on the right. Here $\mi_{\mf{v}}, \mj_{\mf{w}}$ denote entries of $\mf{i} \in \oi(\g), \mf{j} \in \oi(\g')$ and $\ul{0}$ is the label on $\bdry D^2$.}
\label{4-2-1}
\end{figure}

\s
\n {\bf Type (Sh).}  We first require that $\mf{i} \in \oi(\g)$ satisfies: (1) either $\ul{0} \in \guv^r$; or $\ul{0} \notin \guv$ and $\mi_{\uv} \notin [[x,y]]$; and (2A) $\ov \in \tpv(\g)$. Additional conditions are more involved and will be motivated in the following several paragraphs; the full description will then be given in Definition \ref{def si}.

We first observe that if  $\mf{i} \in \oi(\g)$ satisfies (1), then there is no $\mf{j}\in \oi(\g')$ such that $\g'(\mf{j}) =\g(\mf{i})$: this is because there is no $\g'_{\beta(\uv)}(\mf{j}_{\beta(\uv)})$ if $\g'(\mf{j}) =\g(\mf{i})$. 

\s\s

\begin{figure}[ht]
\begin{overpic}
[scale=0.3]{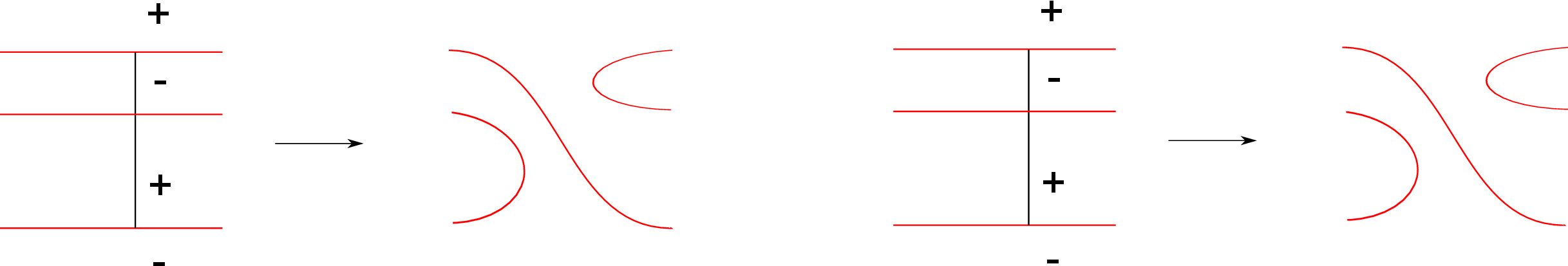}
\put(14,5){$\ul{0}$}
\put(42,5){$\ul{0}$}
\put(0,16){${\scriptstyle z=\mi_{\ov}}$}
\put(0,3){${\scriptstyle x}$}
\put(0,7){${\scriptstyle y}$}
\put(-2,11){${\scriptstyle LSV(\beta)}$}
\put(27,7){${\scriptstyle \mj_{\beta(\uv)}}$}
\put(70,5){${\scriptstyle \mi_{\uv}}$}
\put(98,5){${\scriptstyle \mj_{\beta(\ov)}}$}
\put(56,16){${\scriptstyle z=\mi_{\ov}}$}
\put(84,7){${\scriptstyle \mj_{\beta(\uv)}}$}
\put(56,3){${\scriptstyle x}$}
\put(56,7){${\scriptstyle y}$}
\put(54,11){${\scriptstyle LSV(\beta)}$}
\end{overpic}
\caption{Type (Sh). The case $\ul{0} \in \guv^r$ is on the left and the case $\ul{0} \notin \guv$ and $\mi_{\uv} \notin [[x,y]]$ is on the right. In both cases $\ov \neq \ast$. See Definition \ref{def lsv} for $LSV(\beta)$.}
\label{4-2-2}
\end{figure}

If $\mi\in  \oi(\g)$ satisfies (1) and (2A), the most efficient way to move omitting labels of $\mi$ to omitting labels of $\mj$ is to send $\g_{\ov}(z)$ to $\g_{\uv}(y)$ and leave the other labels intact. In particular, this means that $\mi_{\ov}=z$ and $\g'_{\beta(\uv)}(\mj_{\beta(\uv)})=\guv(y)$; at the same time the omitting labels in $\gw$ may be moved for $\gw$ lying between $\guv(y)$ and $\gov(z)$.

\begin{defn} \label{def lsv}
A vector $\mf{w} \in \tpv(\g)$ is called a {\em left shuffling vector} of $\beta$ if $l_{\gw}>0$, the component $\gw$ is adjacent to $\guv$ and $\gov$, and $\gw \subset [[\guv(y), \gov(z)]]$. Let $LSV(\beta)$ denote the set of left shuffling vectors of $\beta$.
\end{defn}

The condition $\gw \subset [[\guv(y), \gov(z)]]$ implies that the component $\gw$ is to the left to the arc of attachment $\delta$, assuming that $\delta$ is positioned as in Figure~\ref{2-2-2}.

%There are two possible cases of the nesting relation between $\guv, \gov$ and $\gw \in LSV(\beta)$; see Figure \ref{4-2-7}.

\begin{defn}[Shuffling types] \label{def beta shuff} $\mbox{}$
\be
\item[(Y)] A bypass $\beta \in \Hom(\g,\g')$ is of {\em shuffling type} (Y) if $\guv(y) < \gov(z)$.
\item[(Z)] A bypass $\beta \in \Hom(\g,\g')$ is of {\em shuffling type} (Z) if $\guv(y)> \gov(z)$ and there exist $\mf{w}(\beta) \in LSV(\beta)$ and $0<k(\beta)\leq l_{\g_{\mf{w}(\beta)}}$ such that $\guv \cup \gov \subset (\g_{\mf{w}(\beta)}(k(\beta)-1), \g_{\mf{w}(\beta)}(k(\beta)))$.
\ee
\end{defn}

For type (Y), neither $\guv$ nor $\gov$ directly nests inside any $\mf{w} \in LSV(\beta)$.  For type (Z), $\guv$ and $\gov$ directly nest inside a unique $\mf{w}(\beta) \in LSV(\beta)$.

\begin{rmk} \label{rmk: not always in Y and Z} $\mbox{}$
\be
\item The two shuffling types (Y) and (Z) are mutually exclusive but {\em some bypasses do not belong to either type} when the conditions $\mf{w}\in \tpv(\g)$ and $l_{\gw}>0$ in the definition of a left shuffling vector are not met.   This happens when $\Gamma_\ast$ is adjacent to $\guv$ and $\gov$ and $\Gamma_\ast \subset [[\guv(y), \gov(z)]]$.

\item If $\beta$ is of shuffling type (Z), then the pair $(\mf{w}(\beta), k(\beta))$ is unique. We use $(\mf{w}(\beta), k(\beta))$ to denote this pair.
\ee
\end{rmk}

\begin{defn}[Shuffling index] \label{def si}
An omitting index $\mf{i} \in OI(\g)$ is a {\em shuffling index} of $\beta$ if the following conditions hold:
\be
\item either
\be
\item $\ul{0} \in \guv^r$; or
\item $\ul{0} \notin \guv$ and $\mi_{\uv} \notin [[x,y]]$ (i.e., $\guv(\mi_{\uv}) \in \guv^r$);
\ee
\item $\ov \in \tpv(\g)$ and $\mi_{\ov}=z$;
\item $\mi_{\mf{w}}=0$ for all $\mf{w} \in LSV(\beta)$ if $\beta$ is of shuffling type (Y);
\item $\mi_{\mf{w}(\beta)}=k(\beta)$ and $\mi_{\mf{w}}=0$ for all $\mf{w} \in LSV(\beta) \backslash \{\mf{w}(\beta)\}$ if $\beta$ is of shuffling type (Z).
\ee
Let $SI(\beta) \subset \oi(\g)$ denote the set of shuffling indices of $\beta$.
\end{defn}

\begin{figure}[ht]
\begin{overpic}
[scale=0.2]{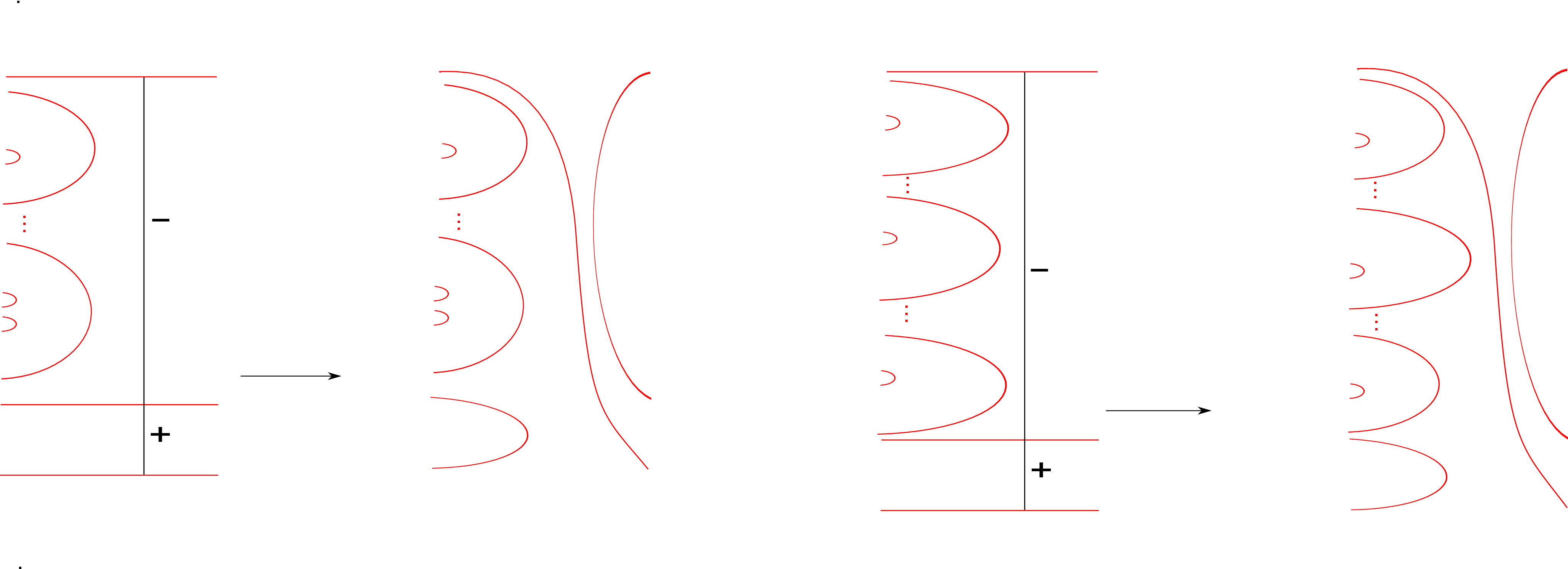}
\put(5,2){$\g(\mi)$}
\put(30.5,2){$\g'(\bi)$}
\put(61,0){$\g(\mi)$}
\put(90,0){$\g'(\bi)$}
\put(0,9){$\scriptstyle{y}$}
\put(-4,32){$\scriptstyle{\mi_{\ov}=z}$}
\put(13,8){$\scriptstyle{\mi_{\uv}}$}
\put(41,8){$\scriptstyle{\bi_{\beta(\ov)}}$}
\put(69,5){$\scriptstyle{\mi_{\uv}}$}
\put(100,5){$\scriptstyle{\bi_{\beta(\ov)}}$}
\put(2,15){$\scriptstyle{\mw^1}$}
\put(2,27){$\scriptstyle{\mw^k}$}
\put(0,13){$\scriptscriptstyle{0}$}
\put(0,24){$\scriptscriptstyle{0}$}
\put(25,9){$\scriptscriptstyle{\bi_{\beta(\uv)}}$}
\put(28,19){$\scriptscriptstyle{l^1}$}
\put(28,28){$\scriptscriptstyle{l^k}$}
\put(58,20){$\scriptstyle{\mw(\beta)}$}
\put(58,11){$\scriptstyle{\mw^{s+1}}$}
\put(59,28){$\scriptstyle{\mw^{s}}$}
\put(54,18){$\scriptscriptstyle{k(\beta)}$}
\put(56,7){$\scriptstyle{y}$}
\put(56,9){$\scriptscriptstyle{0}$}
\put(56,25.5){$\scriptscriptstyle{0}$}
\put(53,33){$\scriptstyle{\mi_{\ov}=z}$}
\put(82,6){$\scriptscriptstyle{\bi_{\beta(\uv)}}$}
\put(85,12.5){$\scriptscriptstyle{l^{s+1}}$}
\put(84,21){$\scriptscriptstyle{k(\beta)-1}$}
\put(87,29){$\scriptscriptstyle{l^{s}}$}
\end{overpic}
\caption{Examples of $SI(\beta)$, where $\beta(\mi)$ is defined in Definition \ref{def beta ind}.
A bypass of shuffling type (Y), $\guv(y) < \gov(z)$, is on the left, where $LSV(\beta)=\{\mw^1, \dots, \mw^k\}$ and $l^t=l_{\g_{\mw^t}}$.
A bypass of shuffling type (Z), $\guv(y) > \gov(z)$, with the pair $(\mw(\beta),k(\beta))$ is on the right, where $LSV(\beta)=\{\wb, \mf{w}^1,\dots, \mf{w}^s,\dots, \mf{w}^{k-1}\}$ and $l^t=l_{\g_{\mw^t}}$.  The entries $\mi_{\mv}$ and $\bi_{\beta(\mv)}$ are drawn for $\mv \in LSV(\beta) \sqcup \{\uv, \ov\}$ for both types of shuffling.}
\label{4-2-7}
\end{figure}

In the special case where $\ul{0} \in \guv$ or $\gov$, the following descriptions of $II(\beta)$ and $SI(\beta)$ are straightforward.

\begin{lemma} \label{lem IS ind}$\mbox{}$
\be
\item If $\ul{0} \in \guv^l$ then $II(\beta)=OI(\g)$ and $SI(\beta)=\es$.
\item If $\ul{0} \in \guv^r$, then $II(\beta)=\es$.
\item If $\ul{0} \in \gov$, then $SI(\beta)=\es$.
\ee
\end{lemma}

By Definitions \ref{def ii} and \ref{def si}(1), $II(\beta) \cap SI(\beta)=\es$.
The disjoint union $II(\beta) \sqcup SI(\beta)$ is always nonempty, but is not equal to $\oi(\g)$ in general.

%\s \n{\bf Type (N).} any $\mi \in \oi(\g) \backslash (II(\beta) \sqcup SI(\beta))$ is called a Type (N) index of $\beta$.

\subsubsection{Definition of the chain map} \label{subsubsection: defn of chain map}

The bypass $\beta$ changes omitting indices in $II(\beta) \sqcup SI(\beta) \subset OI(\g)$ to those in $OI(\g')$.

\begin{notation}
We abuse notation and use $\beta$ to denote three related things:
\be
\item a bypass;
\item the map $V(\g)\to V(\g')$ from Equation~\eqref{eqn: beta}; and
\item the map $II(\beta) \sqcup SI(\beta) \ra OI(\g')$, defined below.
\ee
\end{notation}

\begin{defn} \label{def beta ind}
For $\beta \in \Hom(\g,\g')$, define $\beta: II(\beta) \sqcup SI(\beta) \ra OI(\g')$ by:

\s
\n {\bf (C1).} If $\mf{i} \in II(\beta)$, then define $\beta(\mf{i}) \in OI(\g')$ by $\beta(\mf{i})_{\beta(\mf{v})} \in \Z_+$ for $\beta(\mf{v})\neq \ast$ such that
$$\g'_{\beta(\mf{v})}(\beta(\mf{i})_{\beta(\mf{v})})=\gv(\mi_{\mf{v}}).$$

\n {\bf (C2).} If $\mf{i} \in SI(\beta)$ and $\beta$ is of shuffling type (Y), then define $\beta(\mf{i}) \in OI(\g')$ by $\beta(\mf{i})_{\beta(\mf{v})} \in \Z_+$ for $\beta(\mf{v}) \neq \ast$ such that
$$ \g'_{\beta(\mf{v})}(\beta(\mf{i})_{\beta(\mf{v})})= \left\{
\begin{array}{cl}
\guv(y) & \mbox{if} \hspace{0.1cm} \mf{v}=\uv, \\
\guv(\mi_{\uv}) & \mbox{if} \hspace{0.1cm} \mf{v}=\ov \mbox{ and } \beta(\ov) \neq \ast,\\
\gv(l_{\gv}) & \mbox{if} \hspace{0.1cm} \mf{v} \in LSV(\beta), \\
\gv(\mi_{\mf{v}}) & \mbox{otherwise.}
\end{array}\right.
$$

\n {\bf (C3).} If $\mf{i} \in SI(\beta)$ and $\beta$ is of shuffling type (Z), then define $\beta(\mf{i}) \in OI(\g')$ by $\beta(\mf{i})_{\beta(\mf{v})} \in \Z_+$ for $\beta(\mf{v}) \neq \ast$ such that
$$ \g'_{\beta(\mf{v})}(\beta(\mf{i})_{\beta(\mf{v})})= \left\{
\begin{array}{cl}
\guv(y) & \mbox{if} \hspace{0.1cm} \mf{v}=\uv, \\
\guv(\mi_{\uv}) & \mbox{if} \hspace{0.1cm} \mf{v}=\ov \mbox{ and } \beta(\ov) \neq \ast,\\
\g_{\mf{w}(\beta)}(k(\beta)-1) & \mbox{if} \hspace{0.1cm} \mf{v} =\mf{w}(\beta), \\
\gv(l_{\gv}) & \mbox{if} \hspace{0.1cm} \mf{v} \in LSV(\beta), \mf{v} \neq \mf{w}(\beta),\\
\gv(\mi_{\mf{v}}) & \mbox{otherwise.}
\end{array}\right.
$$
\end{defn}

See Figure \ref{4-2-7} for examples. We verify that Definition \ref{def beta ind} is well-defined, i.e., $\iv$ exists (equivalently $\mf{v}\not=\ast$) on the right-hand side of the equations. For (C1), $\ul{0} \notin \guv^r$ since $\mf{i} \in II(\beta)$. Hence $\beta(\ast)=\ast$ by Remark \ref{rmk beta ast}. Then $\beta(\mf{v})\neq \ast$ implies that $\mf{v} \neq \ast$.  For the second rows of (C2) and (C3), $0 \notin \guv^r$ since $\beta(\ov)\neq \ast$; and $0 \notin \guv^l$ since $\mf{i} \in SI(\beta)$.
Hence $\uv \neq \ast$.

\begin{rmk} \label{rmk beta ind} $\mbox{}$
\be
\item[(i)] %$\beta(\mi)$ is defined indirectly through the omitting labels associated to $\mi$ and $\beta(\mi)$.
An effective way to understand $\beta$ is to track the movement of the omitting labels from $\g(\mi)$ to $\g'(\bi)$.
\item[(ii)] Since $\beta: V(\g) \ra V(\g')$ only changes the two components $\guv, \gov$ by definition, we have
$$\beta(\mf{i})_{\beta(\mf{v})}=\mi_{\mf{v}}$$
for $\mf{v} \notin \{\uv,\ov\}$ if $\mf{i} \in II(\beta)$, and for $\mf{v} \notin \{\uv,\ov\}\cup LSV(\beta)$ if $\mf{i} \in SI(\beta)$.
\item[(iii)] The map $\beta: II(\beta) \sqcup SI(\beta) \ra OI(\g')$ is injective.
\ee
\end{rmk}

\n
{\bf Interpretation in terms of negative regions.} For $\mi\in SI(\beta)$ we give a description of $\beta(\mi)$ in terms of the component $\overline{\frak c}$ of $R_-(\g)$ that lies between $\uv$ and $\ov$.  Let $\overline{\frak c}^l$ be the left-hand side of $\overline{\frak c}$ cut along the arc of attachment for $\beta$, i.e., the region that lies between $\guv^l$ and $\gov$ as in Figure \ref{2-2-2}. 
Let $\mi\in \oi(\g)$.  Then $\mi\in SI(\beta)$ if and only if the following hold:
\be
\item[(N1)] if $\g_{\mv}\not=\g_{\uv}$ is a component of $R_+(\g)$ which has boundary $\gamma_{\mv}$ in common with $\overline{\frak c}^l$, then $\mv\not=\ast$;
\item[(N2)] the omitting label of any $\mv$ satisfying (N1) is the label of the interval of $\bdry D^2$ which is adjacent to the initial point of $\gamma_{\mv}$;
\item[(N3)] neither $\ul{0}$ nor the omitting label of $\uv$ is on the left-hand side of $\uv$.
\ee
If $\mf{i}\in \oi(\g)$, then $\beta(\mf{i})$ is obtained from $\mf{i}$ by removing the label of each $\mv$ satisfying (N1) and, for each $\mv\not=\ov$ which has boundary  $\gamma_{\mv}$ in common with $\overline{\frak c}^l$, adding the label which is adjacent to the terminal point of $\gamma_\mv$.

\s\n
{\bf A closer look at $SI(\beta)$.}
The conditions $\mi_{\mw}=0$ in Definitions \ref{def diff vector}(2) and \ref{def si}(3),(4) are similar, and there is a good reason for this. % One way to understand $SI(\beta)$ is through the shuffling vectors in $\g''$, where
Let $\g''$ be the third dividing set in the bypass triangle
$$\g \xra{\beta} \g' \xra{\beta'} \g''$$
induced by $\beta$. For each $\mf{i} \in SI(\beta)$ there exists a unique $\mf{k} \in OI(\g'')$ such that $\g''(\mf{k})=\g(\mf{i})$; see Lemma \ref{lem ind triangle}. We will see in the proof of Lemma \ref{lem comp triangle} that there exists $\mf{u} \in \sv_{\g''}(\mf{k})$ such that $\g''(\mf{u} | \mf{k})=\g'(\beta(\mf{i}))$. As a result, the restriction $\cf(\beta)|_{\g(\mi)}: P(\g(\mi)) \ra P(\g'(\bi))$ coincides with part of the differential $d_{\g''}(\mf{k}, \mf{u}): P( \g''(\mf{k})) \ra P(\g''(\mf{u}|\mf{k}))$. This is the key to proving that
$$\cf(\g) \xra{\cf(\beta)} \cf(\g') \xra{\cf(\beta')} \cf(\g'')$$
is a distinguished triangle in Proposition \ref{prop triangle}; see Examples \ref{ex t1} and \ref{ex t2}.

\begin{lemma}  \label{lem beta path} $\mbox{}$
\begin{enumerate}
\item If $\mf{i} \in II(\beta)$, then $\g'(\bi)=\g(\mf{i})$.
\item If $\mf{i} \in SI(\beta)$, then $\Hom(\g(\mf{i}), \g'(\bi))\not=0$.
\ee
\end{lemma}

\begin{proof}
(1) is immediate from Definition \ref{def beta ind}(C1).

\s\n
(2) Suppose that $\beta$ is of shuffling type (Y).
Let $LSV(\beta)=\{\mf{w}^1,\dots, \mf{w}^k\}$ such that
$$\guv(y) < \g_{\mf{w}^{1}}(0); \quad \g_{\mf{w}^{j}}(l_{\g_{\mf{w}^{j}}})<\g_{\mf{w}^{j+1}}(0), ~1 \leq j \leq k-1; \quad
\g_{\mf{w}^{k}}(l_{\g_{\mf{w}^{k}}}) < \gov(z).$$
There are $k+1$ disjoint closed intervals:
\begin{equation} \label{eqn: chain 1}
[\guv(y) , \g_{\mf{w}^{1}}(0)]; \quad [\g_{\mf{w}^{j}}(l_{\g_{\mf{w}^{j}}}),\g_{\mf{w}^{j+1}}(0)], ~1 \leq j \leq k-1; \quad [\g_{\mf{w}^{k}}(l_{\g_{\mf{w}^{k}}}) , \gov(z)].
\end{equation}
The intersections $\g(\mf{i})_{\ast} \cap [a,b]=\{a\}$, where $[a,b]$ is any of the $k+1$ intervals. We have
\begin{gather*}
\g'(\bi)_{\ast}=\g(\mf{i})_{\ast} \backslash  \left(\{\guv(y)\} \cup \{\g_{\mf{w}^{j}}(l_{\g_{\mf{w}^{j}}})\}_{j=1}^k\right) \cup \left(\{\g_{\mf{w}^{j}}(0)\}_{j=1}^k \cup \{\gov(z)\}\right)
\end{gather*}
and $\Hom(\g(\mf{i}), \g'(\bi)\not=0$ by Proposition \ref{prop: tightness criterion}.

Suppose that $\beta$ is of shuffling type (Z).
Let $LSV(\beta)=\{\wb, \mf{w}^1,\dots, \mf{w}^s,\dots, \mf{w}^{k-1}\}$ such that $\g_{\mf{w}^{j}}(l_{\g_{\mf{w}^{j}}})<\g_{\mf{w}^{j+1}}(0)$
for $1 \leq j \leq k-2$, $\g_{\mf{w}^{s}}(l_{\g_{\mf{w}^{s}}}) < \gov(z)$, and $\guv(y) < \g_{\mf{w}^{s+1}}(0)$.
Moreover, $\g_{\wb}(k(\beta)-1) < \g_{\mf{w}^{1}}(0)$, and $\g_{\mf{w}^{k-1}}(l_{\g_{\mf{w}^{k-1}}}) < \g_{\wb}(k(\beta)).$
Then there are $k+1$ disjoint closed intervals:
\begin{gather}\label{eqn: chain 2}
[\g_{\wb}(k(\beta)-1) , \g_{\mf{w}^{1}}(0)]; \quad [\g_{\mf{w}^{j}}(l_{\g_{\mf{w}^{j}}}),\g_{\mf{w}^{j+1}}(0)], ~1 \leq j \leq s-1; \quad [\g_{\mf{w}^{s}}(l_{\g_{\mf{w}^{s}}}) , \gov(z)];\\
\nonumber
[\guv(y) , \g_{\mf{w}^{s+1}}(0)]; \quad [\g_{\mf{w}^{j}}(l_{\g_{\mf{w}^{j}}}),\g_{\mf{w}^{j+1}}(0)], ~s+1 \leq j \leq k-2; \quad [\g_{\mf{w}^{k-1}}(l_{\g_{\mf{w}^{k-1}}}) , \g_{\wb}(k(\beta))].
\end{gather}
Again $\Hom(\g(\mf{i}), \g'(\bi)\not=0$ follows from Proposition~\ref{prop: tightness criterion}.
\end{proof}

\begin{rmk} \label{rmk beta path}
With a little more work one can show that the nonzero path $\g(\mi) \ra \g'(\bi)$ is the unique minimal element of $\{\mbox{nonzero path}~~ \g(\mi) \ra \g'(\mj) ~|~ \mi \in \oi(\g), \mj \in \oi(\g')\}$.
This is actually the motivation behind Definitions \ref{def si} and \ref{def beta ind}.
\end{rmk}

The $k+1$ closed intervals in Equation~\eqref{eqn: chain 1} or \eqref{eqn: chain 2} are called the {\em chain intervals} of $\beta$.

Let $t(\beta,\mf{i}) \in \rne$ denote the idempotent of $\g(\mf{i})$ if $\mf{i} \in II(\beta)$, or the generator of $\Hom(\g(\mf{i}), \g'(\bi))$ if $\mf{i} \in SI(\beta)$.

\begin{defn} \label{def fb}
For a nontrivial bypass morphism $\beta \in \Hom(\g,\g')$, define a map of $\rne$-modules $\cal{F}(\beta): \cal{F}(\g) \ra \cal{F}(\g')$ by
$$\cal{F}(\beta)=\sum\limits_{\mf{i}\in II(\beta) \sqcup SI(\beta)}\cal{F}(\beta,\mf{i}),$$
where $\cal{F}(\beta,\mf{i}): P(\g(\mf{i})) \ra P(\g'(\bi))$ is the right multiplication by $t(\beta,\mf{i})$ for $\mf{i} \in II(\beta) \sqcup SI(\beta)$.
\end{defn}

\subsubsection{Some examples}

Before proving that $\cal{F}(\beta)$ is a chain map, we look at some examples.

\begin{example} \label{ex t1}
Consider the bypass triangle $\g \xra{\beta} \g' \xra{\beta'} \g'' \xra{\beta''} \g $ in Figure \ref{4-2-3}.
By definition,
$$\xymatrix{
\cf(\g): \ar[d]_{\cf(\beta)} & P(\ul{1},\ul{2}) \ar[r] \ar[rd] \ar @{} [drr] |{(A), {\underline{\frak c}}} & P(\ul{1},\ul{3}) \ar @{} [drr] |{(B), {\overline{\frak c}}} \ar[rd] \ar[r] & P(\ul{2},\ul{4}) \ar[rd] \ar[r] & P(\ul{3},\ul{4}) \\
\cf(\g'): \ar[d]_{\cf(\beta')}& & P(\ul{1},\ul{2}) \ar @{} [dr] |{(B), {\overline{\frak c}}} \ar[r] \ar[d] & P(\ul{1},\ul{4}) \ar @{} [dr] |{(A), {\underline{\frak c}}} \ar[r] \ar[d] & P(\ul{2},\ul{4}) \ar[d] \\
\cf(\g''): \ar[d]_{\cf(\beta'')}& & P(\ul{1},\ul{3}) \ar @{} [dr] |{(A), {\underline{\frak c}}}\ar[d] \ar[r] & P(\ul{1},\ul{4}) \ar @{} [dr] |{(B), {\overline{\frak c}}} \ar[d] \ar[r] & P(\ul{3},\ul{4}) \ar[d] \\
\cf(\g):  & P(\ul{1},\ul{2}) \ar[r] & P(\ul{1},\ul{3}) \ar[r] & P(\ul{2},\ul{4}) \ar[r] & P(\ul{3},\ul{4})
}$$

The labels (e.g., (A), $\underline{\frak c}$) above correspond to the cases (e.g., Case (A), ${\frak c}=\underline{\frak c}$) in the proof of Lemma \ref{lem beta chain}. All the squares are commutative by the commutativity relation of $\rne$.  Therefore, $\cf(\beta)$, $\cf(\beta')$, and $\cf(\beta'')$ are all chain maps and the sum of their degrees is $1$. Moreover, $\cf(\g) \ra \cf(\g') \ra\cf(\g'')$ is a distinguished triangle in $\dne$ up to grading shift.

\begin{figure}[ht]
\begin{overpic}
[scale=0.28]{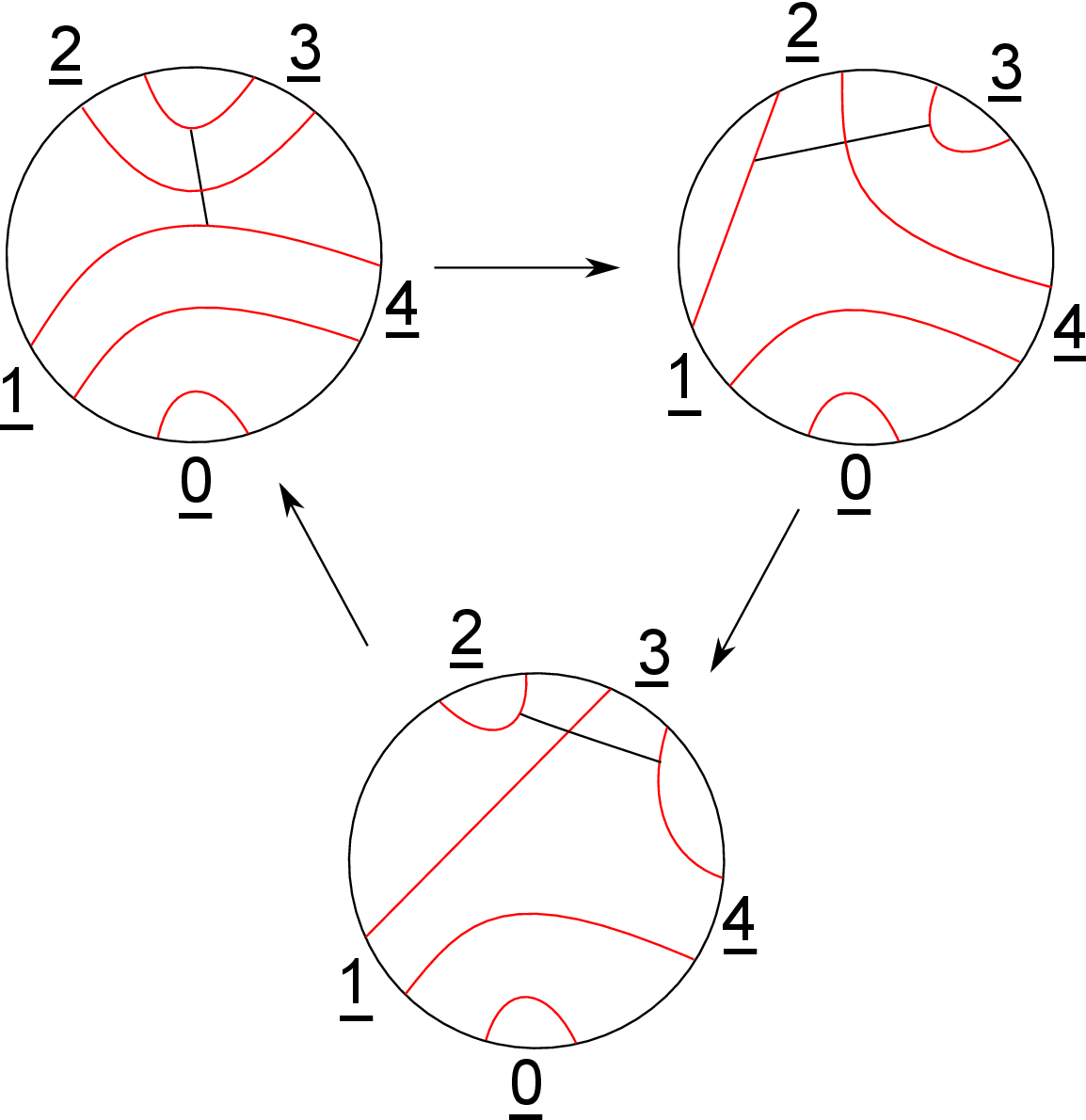}
\put(6,50){$\g$}
\put(90,50){$\g'$}
\put(36,-2){$\g''$}
\put(45,80){$\beta$}
\put(70.5,42){$\beta'$}
\put(21,41.5){$\beta''$}
\end{overpic}
\caption{}
\label{4-2-3}
\end{figure}

We now check the definition of $\cf(\beta)$ in more detail.
We have $\guv^l=\{\ul{3}\}, \guv^r=\{\ul{2}\}, \gov=\{\ul{1},\ul{4}\}$, and $x=y=z=1$.
By definition $\beta$ is of shuffling type (Y) since $\guv(1) < \gov(1)$. %$LSV(\beta)=\es$.
\be
\item
For $P(\ul{1},\ul{2}), P(\ul{2},\ul{4}) \in \cf(\g)$, the corresponding $\mf{i} \in II(\beta)$ since $\guv(\mi_{\uv})=\ul{3} \in \guv^l$.
\item
For $P(\ul{1},\ul{3}) \in \cf(\g)$, the corresponding $\mf{i} \in SI(\beta)$ since $\guv(\mi_{\uv})=\ul{2} \in \guv^r$, $\gov(\mi_{\ov})=\ul{4}=\gov(z)$ and $LSV(\beta)=\es$.
The restriction $\cf(\beta)|_{P(\ul{1},\ul{3})}: P(\ul{1},\ul{3}) \ra P(\ul{1},\ul{4})$ corresponds to $d_{\g''}(\mf{k}, \mf{u})$, where $\g''(\mf{k})=\g(\mf{i})$ and $\mf{u}=\beta''^{-1}(\ov)$.
\item
For $P(\ul{3},\ul{4}) \in \cf(\g)$, the corresponding $\mf{i} \notin II(\beta)$ since $\guv(\mi_{\uv})=\ul{2} \in \guv^r$, and $\mf{i} \notin SI(\beta)$ since $\gov(\mi_{\ov})=\ul{1} \neq \gov(z)$. Hence $\mf{i} \notin II(\beta) \sqcup SI(\beta)$.
\ee
\end{example}

All the three bypasses in Example \ref{ex t1} are of shuffling type (Y).
We consider a bypass of shuffling type (Z) in the next example.

\begin{example} \label{ex t2}
Consider the bypass triangle $\g \xra{\beta} \g' \xra{\beta'} \g'' \xra{\beta''} \g $ in Figure \ref{4-2-4}.
By definition,
$$\xymatrix{
\cf(\g): \ar[d]_{\cf(\beta)}  & P(\ul{1},\ul{3}) \ar[r] \ar[rd] & P(\ul{1},\ul{4}) \ar[r] \ar[rd] & P(\ul{3},\ul{5}) \ar[rd] \ar[r] & P(\ul{4},\ul{5}) \\
\cf(\g'): \ar[d]_{\cf(\beta')} & P(\ul{1},\ul{2}) \ar[d] \ar[r] & P(\ul{1},\ul{3}) \ar[d] \ar[r] & P(\ul{2},\ul{5}) \ar[d] \ar[r] & P(\ul{3},\ul{5}) \ar[d]\\
\cf(\g''): \ar[d]_{\cf(\beta'')} & P(\ul{1},\ul{2}) \ar[r]  \ar[d] & P(\ul{1},\ul{4})  \ar[d] \ar[r] & P(\ul{2},\ul{5})  \ar[d] \ar[r] & P(\ul{4},\ul{5}) \ar[d] \\
\cf(\g): & P(\ul{1},\ul{3}) \ar[r]  & P(\ul{1},\ul{4}) \ar[r] & P(\ul{3},\ul{5}) \ar[r] & P(\ul{4},\ul{5})
}$$

The maps $\cf(\beta), \cf(\beta'), \cf(\beta'')$ are chain maps, the sum of their degrees is $1$, and $\cf(\g) \ra \cf(\g') \ra\cf(\g'')$ is a distinguished triangle in $\dne$ up to grading shift.
\begin{figure}[ht]
\begin{overpic}
[scale=0.25]{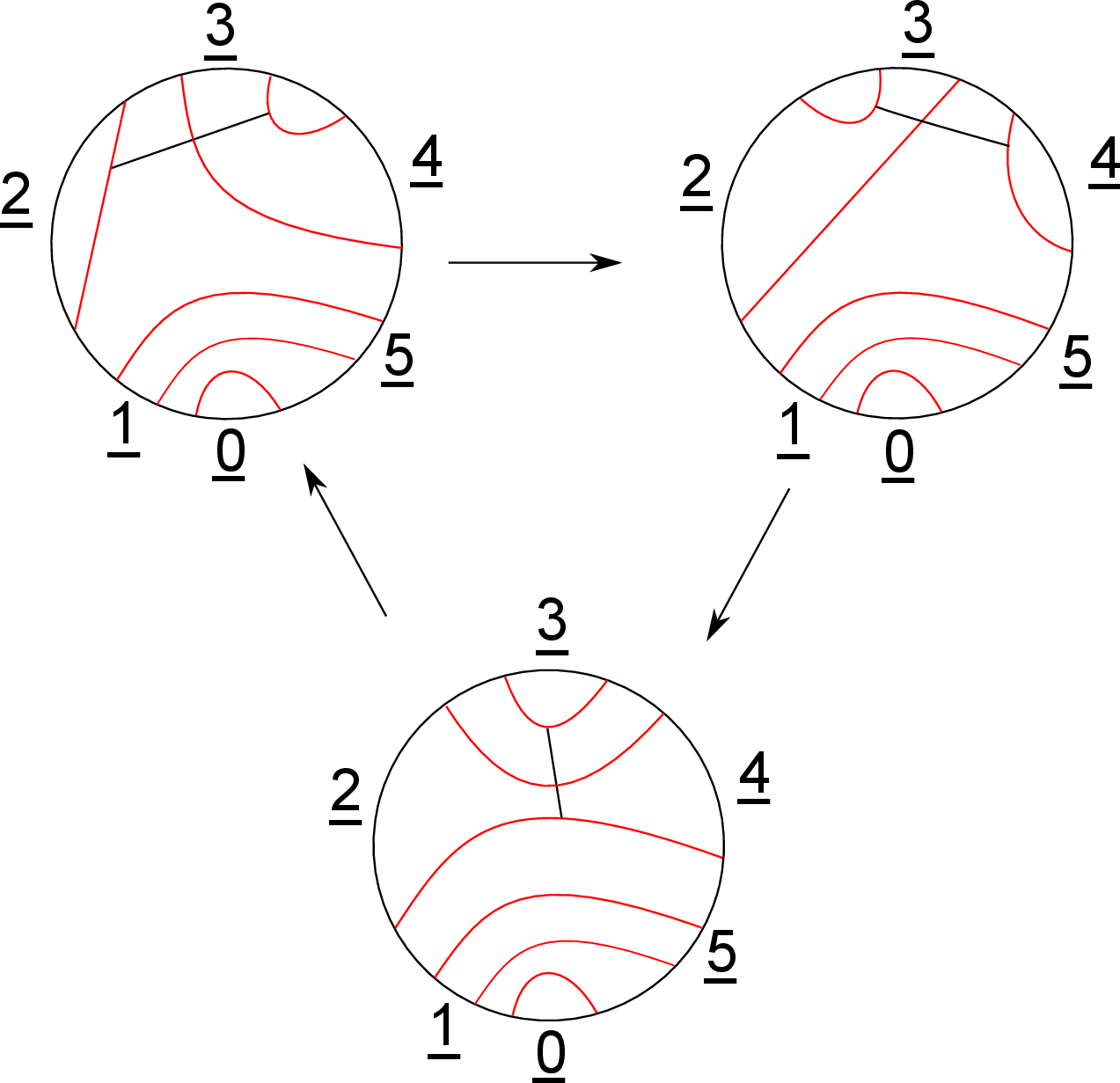}
\put(2,50){$\g$}
\put(93,50){$\g'$}
\put(56,-2){$\g''$}
\put(45.5,77){$\beta$}
\put(70.2,40){$\beta'$}
\put(22.7,40){$\beta''$}
\end{overpic}
\caption{}
\label{4-2-4}
\end{figure}

For the bypass $\beta\in \Hom(\g, \g')$, we have $\guv^l=\{\ul{4}\}, \guv^r=\{\ul{3}\}, \gov=\{\ul{2}\}$, and $x=y=1, z=0$. Also $\guv, \gov \subset (\ul{1},\ul{5})=(\g_{\mf{w}(\beta)}(0), \g_{\mf{w}(\beta)}(1))$.  Hence $\beta$ is of shuffling type (Z), where $\g_{\mf{w}(\beta)}=\{\ul{1},\ul{5}\}$ and $k(\beta)=1$.
\be
\item
For $P(\ul{1},\ul{3}), P(\ul{3},\ul{5}) \in \cf(\g)$, the corresponding $\mf{i} \in II(\beta)$ since $\guv(\mi_{\uv})=\ul{4} \in \guv^l$.
\item
For $P(\ul{1},\ul{4}) \in \cf(\g)$, the corresponding $\mf{i} \in SI(\beta)$ since $\guv(\mi_{\uv})=\ul{3} \in \guv^r$, $\gov(\mi_{\ov})=\ul{2}=\gov(z)$, $\g_{\wb}(\mi_{\wb})=\ul{5}=\g_{\wb}(k(\beta))$, and $LSV(\beta)\backslash \wb =\es$.
The restriction $\cf(\beta)|_{P(\ul{1},\ul{4})}: P(\ul{1},\ul{4}) \ra P(\ul{2},\ul{5})$ corresponds to $d_{\g''}(\mf{k}, \mf{u})$, where $\g''(\mf{k})=\g(\mf{i})$ and $\mf{u}=\beta''^{-1}(\wb)$.
\item
For $P(\ul{4},\ul{5}) \in \cf(\g)$, the corresponding $\mf{i} \notin II(\beta)$ since $\guv(\mi_{\uv})=\ul{3} \in \guv^r$, and $\mf{i} \notin SI(\beta)$ since $\g_{\wb}(\mi_{\wb})=1 \neq \g_{\wb}(k(\beta))$. Hence $\mf{i} \notin II(\beta) \sqcup SI(\beta)$.
\ee
\end{example}

\subsubsection{Proof that $\cal{F}(\beta)$ is a chain map}

\begin{prop} \label{prop beta}
If $\beta \in \Hom(\g,\g')$ is a nontrivial bypass morphism, then $d_{\g'}\circ\cal{F}(\beta)=\cal{F}(\beta)\circ d_\g.$
\end{prop}

\begin{proof}
Observe that $\beta: II(\beta) \sqcup SI(\beta) \ra OI(\g')$ is injective by Remark \ref{rmk beta ind}(iii) and any two paths with the same endpoints in $\qne$ give the same element of $\rne$. Hence the proposition is a consequence of the following lemma.
\end{proof}

\begin{lemma}  \label{lem beta chain} Let $\mf{i} \in OI(\g)$.
\be
\item If $\mi$ is ${\frak c}$-admissible for ${\frak c}\in \pi_0(R_-(\g))$, ${\frak c}|\mi \in II(\beta) \sqcup SI(\beta)$, and the path $\g(\mf{i}) \ra \g({\frak c}|\mi) \ra \g'(\beta({\frak c}|\mi))$ is nonzero, then $\mf{i} \in II(\beta) \sqcup SI(\beta)$ and there exists a unique ${\frak c}'\in\pi_0(R_-(\g'))$ such that $\beta(\mi)$ is ${\frak c}'$-admissible and ${\frak c}'| \beta(\mi)=\beta({\frak c}|\mi)$.
\item If $\mf{i} \in II(\beta) \sqcup SI(\beta)$, there exists ${\frak c}'\in\pi_0(R_-(\g'))$ such that $\beta(\mi)$ is ${\frak c}'$-admissible, and the path $\g(\mf{i}) \ra  \g'(\beta(\mf{i})) \ra \g'({\frak c}'|\bi)$ is nonzero, then there exists ${\frak c}\in \pi_0(R_-(\g))$ such that $\mi$ is ${\frak c}$-admissible, ${\frak c}|\mi \in II(\beta) \sqcup SI(\beta)$, and $\beta({\frak c}|\mi)={\frak c}' | \bi$.
\ee
\end{lemma}

\begin{proof}
(1) We assume that $\uv,\ov \neq \ast$. The proofs of other cases are similar.
The main idea is to track the movements of the omitting labels under $\cal{F}(\beta)$ and the differentials. Note that the uniqueness of ${\frak c}'\in\pi_0(R_-(\g'))$ follows immediately from the existence by Remark~\ref{rmk diff}(i).

Let $\overline{\frak c}(\beta)$ be the component of $R_-(\beta)$ that lies between $\uv(\beta)$ and $\ov(\beta)$ and let $\underline{\frak c}(\beta)$ be the component of $R_-(\g)$ that lies directly below $\uv(\beta)$.  If we omit $\beta$, it is understood that $\overline{\frak c}=\overline{\frak c}(\beta)$, etc. Let $\g \xra{\beta} \g' \xra{\beta'} \g''$ be the bypass triangle starting with $\beta$.

Suppose ${\frak c}$ satisfies the assumptions of (1).

\s\n{\bf Case (A).}  Suppose that $\mi_{\uv} \in [[x,y]]$, i.e., the label is on the left-hand side of $\uv$. Then $\mf{i} \in II(\beta)$ and $\g'(\beta(\mi))=\g(\mi)$.

If ${\frak c}\not =\ul{\frak c},\overline{\frak c}$, then $({\frak c}|\mi)_{\uv}\in [[x,y]]$ and ${\frak c}|\mi\in II(\beta)$.  We can take ${\frak c}'={\frak c}$, viewed as an element of $\pi_0(R_-(\g'))$, and it is immediate that ${\frak c}'| \beta(\mi)=\beta({\frak c}|\mi)$.  See Figure \ref{4-2-5} for two possible locations for ${\frak c}$.

If ${\frak c}= \ul{\frak c}$ and $\mi$ is ${\frak c}$-admissible, then $\mi_{\uv}=x$ and $({\frak c}|\mi)_{\uv}\not \in [[x,y]]$. Since we are assuming that ${\frak c}|\mi \in SI(\beta)$, (N1)--(N3) from Section~\ref{subsection: notation for bypass} must hold. If we take ${\frak c}'=\overline{\frak c}(\beta')$ (this is $\overline{\frak c}$ with respect to $\beta'$), then $\beta(\mi)$ is ${\frak c}'$-admissible and ${\frak c}'| \beta(\mi)=\beta({\frak c}|\mi)$.

If ${\frak c}=\overline{\frak c}$, then $\mi$ is not ${\frak c}$-admissible.

\s
\begin{figure}[ht]
\begin{overpic}
[scale=0.19]{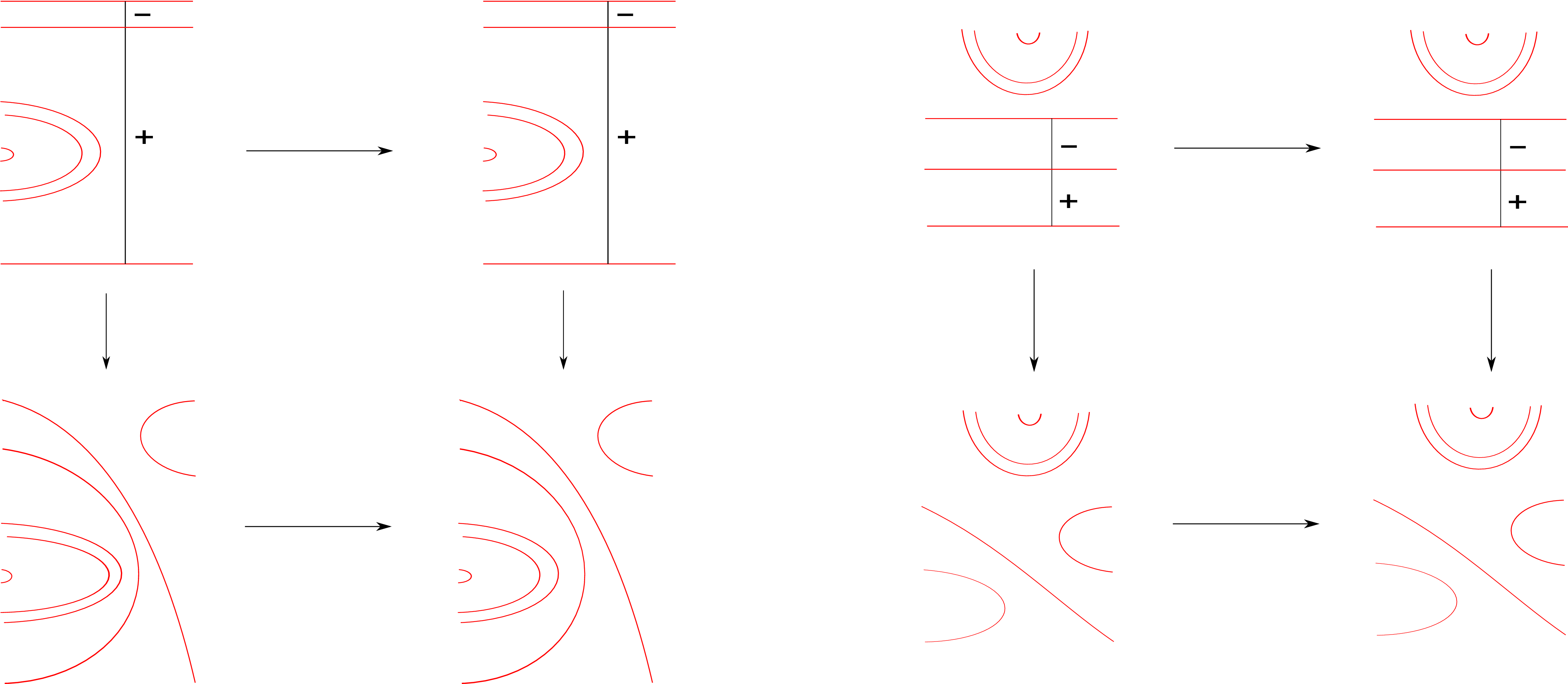}
\put(12,24){$\g(\mi)$}
\put(41,24){$\g({\frak c}|\mi)$}
\put(10,-2.7){$\g'(\bi)$}
\put(38,-2.7){$\g'({\frak c}'|\bi)$}
%\put(0,27){${\scriptstyle{x}}$}
%\put(0,40.5){${\scriptstyle{y}}$}
\put(0,38){${\scriptstyle{\mi_{\uv}}}$}
\put(-2,11.5){$\scriptstyle{{\bi}_{\beta(\uv)}}$}
\put(30,29){${\scriptstyle{({\frak c}|\mi)_{\uv}}}$}
\put(5,29){${\scriptstyle \uv}$}
\put(5.3,33.2){${\scriptstyle {\frak c}}$}
\put(0,2){$\scriptstyle{\beta(\uv)}$}
\put(6.9,6.5){$\scriptstyle{{\frak c}'}$}
\put(8,21){$\beta$}
\put(37,21){$\beta$}
\put(67,21){$\beta$}
\put(96,21){$\beta$}
\put(70,26){$\g(\mi)$}
\put(96,26){$\g({\frak c}|\mi)$}
\put(67,-.7){$\g'(\bi)$}
\put(92,-.7){$\g'({\frak c}'|\bi)$}
%\put(58,37){${\scriptstyle{z}}$}
%\put(70,37){${\scriptstyle{\wt{z}}}$}
\put(70,41){${\scriptstyle{\mi_{\ov}}}$}
\put(84,40){${\scriptstyle{({\frak c}|\mi)_{\ov}}}$}
\put(61,36.5){${\scriptstyle{\ov}}$}
\put(61.5,11){${\scriptstyle{\beta(\ov)}}$}
\put(70,17){$\scriptstyle{{\bi}_{\beta(\ov)}}$}
\put(65.5,37.5){$\scriptstyle{{\frak c}}$}
\put(65.2,13.3){$\scriptstyle{{\frak c}'}$}
\end{overpic}
\s\s
%\caption{Case (A-1-2) on the left $\mi,\vi \in II(\beta), \mv=\uv, \mw=\beta(\uv)$, is on the left, where $\mf{u} \in \dnv(\uv, \mi_{\uv})$; and Case (A-1-4), $\mi,\vi \in II(\beta), \mv=\ov, \mw=\beta(\ov)$, is on the right, where $\mf{u} \in \dnv(\ov, \mi_{\ov})$.}
\caption{${\frak c}$ is to the left of $\uv$ in the left-hand figure and is above $\ov$ in the right-hand figure.}
\label{4-2-5}
\end{figure}

\s
\n{\bf Case (B).} Suppose that $\mi_{\uv} \notin [[x,y]]$.  Then $\mi\in SI(\beta)$.

If ${\frak c}=\ul{\frak c}$, then $\mi$ is not ${\frak c}$-admissible.

If ${\frak c}=\overline{\frak c}$  and $\mi$ is ${\frak c}$-admissible, then $\mi_{\ov}=z$ and $({\frak c}|\mi)_{\uv}=y \in [[x,y]]$. Hence ${\frak c}|\mi\in II(\beta)$.  If we take ${\frak c}'=\ul{\frak c}(\beta')$, then $\beta(\mi)$ is ${\frak c}'$-admissible and ${\frak c}'| \beta(\mi)=\beta({\frak c}|\mi)$; see Figure~\ref{4-2-6}. In this case, we say that ${\frak c}=\overline{\frak c}$ lies below and shares a common boundary with $\ov$. This convention on the relative positions of the different regions (i.e., the positioning as in Figure~\ref{4-2-6}) will be used for the rest of the proof. 

If ${\frak c}\not =\ul{\frak c},\overline{\frak c}$, then $({\frak c}|\mi)_{\uv}\not\in [[x,y]]$ and ${\frak c}|\mi\in SI(\beta)$. Suppose ${\frak c}$ lies above and shares a common boundary with $\ov$, or ${\frak c}$ is to the left of $\overline{\frak c}$ and shares a common boundary with some $\mf{w}\in LSV(\beta)$. 
Since $\mi$ is ${\frak c}$-admissible, we have $\mi\not\in SI(\beta)$ by (N1)--(N3), a contradiction; on the other hand, by Corollary~\ref{cor stack2}, $\g(\mf{i}) \ra \g({\frak c}|\mi) \ra \g'(\beta({\frak c}|\mi))$ is zero, which is consistent. If ${\frak c}$ is to the left of and shares a common boundary with $\uv$, then $\mi$ is not ${\frak c}$-admissible since $\mi_{\uv} \notin [[x,y]]$.   In the remaining cases of $\mi$-admissible ${\frak c}\not =\ul{\frak c},\overline{\frak c}$, we can take ${\frak c}'={\frak c}$, viewed as an element of $\pi_0(R_-(\g'))$, and it is immediate that ${\frak c}'| \beta(\mi)=\beta({\frak c}|\mi)$.

%The proof for the case of ${\frak c}\not =\ul{\frak c},\overline{\frak c}$ is similar, and is left to the reader. 

\begin{figure}[ht]
\hskip-.7in
\begin{overpic}
[scale=0.18]{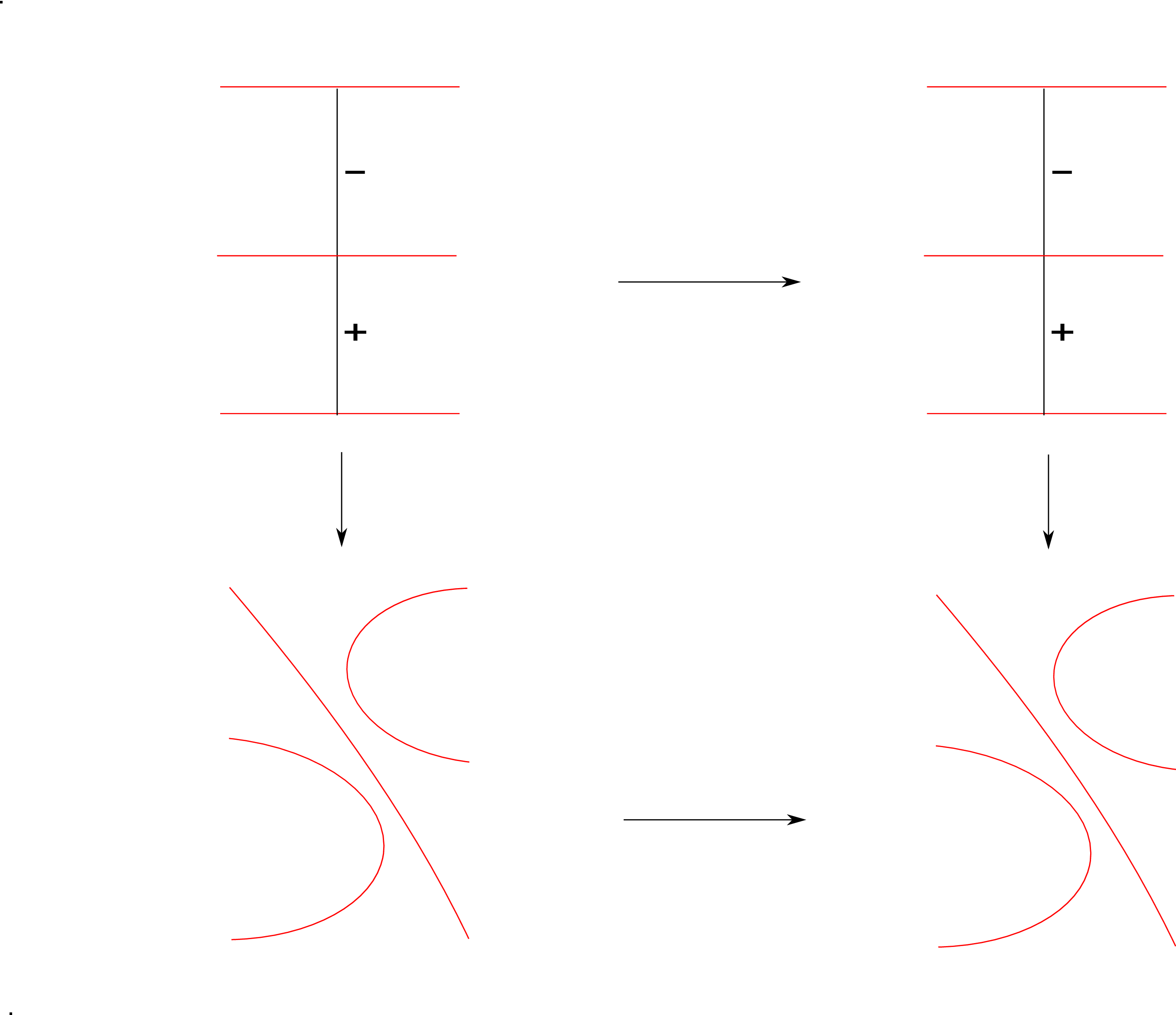}
\put(36,46){$\g(\mi)$}
\put(95,46){$\g({\frak c}|\mi)$}
\put(22,0){$\g'(\beta(\mi))$}
\put(80,0){$\g'({\frak c}'|\bi)$}
\put(16,81){$\scriptstyle{\mi_{\ov}=z}$}
\put(95,82){$\scriptstyle{({\frak c}|\mi)_{\ov}}$}
%\put(12,70){$\scriptstyle{LSV(\beta)}$}
\put(15,70){$\scriptstyle{{\frak c}=\overline{\frak c}}$}
%\put(36,70){$\scriptstyle{RSV(\beta)}$}
\put(36,60){$\scriptstyle{\mi_{\uv}}$}
%\put(20,60){$\scriptstyle{y}$}
\put(72,60){$\scriptstyle{({\frak c}|\mi)_{\uv}=y}$}
\put(12,18){$\scriptstyle{\beta(\mi)_{\beta(\uv)}}$}
\put(36,18){$\scriptstyle{\beta(\mi)_{\beta(\ov)}}$}
%\put(31,28){$\scriptscriptstyle{\dnv(\beta(\ov),\beta(\mi)_{\beta(\ov)})}$}
\put(35,28){$\scriptstyle{{\frak c}'}$}
\put(68,18){$\scriptscriptstyle{({\frak c}'|\bi)_{\beta(\uv)}}$}
\put(93,38){$\scriptscriptstyle{({\frak c}'|\bi)_{\beta(\ov)}}$}
\end{overpic}
\caption{} %$\mi \in SI(\beta)$, $\mv=\uv$, and $\uv|\mi\in II(\beta)$, where $RSV(\beta)=\dnv(\beta(\ov),\beta(\mi)_{\beta(\ov)})$, and $\beta(RSV(\beta))=\dnv(\beta(\ov),\beta(\mi)_{\beta(\ov)})$. }
\label{4-2-6}
\end{figure}

\s
\n(2) The proof is similar to that of (1) and is left to the reader.
\end{proof}

\begin{rmk}\label{rmk beta chain}
Observe that we set ${\frak c}'={\frak c}$ except in Case (A) when ${\frak c}=\underline{\frak c}$ and Case (B) when ${\frak c}=\overline{\frak c}$; see the labels of the commuting squares in Example \ref{ex t1} for example.
\end{rmk}

%By abuse of notation, we also use $\cf(\beta) \in \Hom(\cf(\g),\cf(\g'))$ to denote the homotopy class of the map.

We now complete the definition of $\cf: \cne \ra \dne$. If $\g \stackrel\xi\to \g'$ is a zero morphism, then $\cal{F}(\xi)$ is defined to be the zero morphism; in particular this is the case when $\g$ or $\g'$ is the zero object and $\cal{F}(\g)$ or $\cal{F}(\g')=0$. Any nonzero morphism $\xi$ in $\cne$ can be written as a composition $\beta_k\circ\dots\circ \beta_1$ of nontrivial bypass morphisms and we define $\cal{F}(\xi)$ as the composition $\cal{F}(\beta_k)\circ\dots \circ \cal{F}(\beta_1)$. If $\xi$ is an identity morphism (i.e., induced by a trivial bypass), then we set $\cal{F}(\xi)=\op{id}$. In Section \ref{Sec comp} we will show that $\cal{F}$ is well-defined.

\subsection{Well-definition of the composition} \label{Sec comp}

In order to prove that $\cf: \cne \ra \dne$ is well-defined it suffices to show the following:
\be
\item $\cal{F}(\xi)$ is independent of the choice of decomposition of any nonzero morphism $\xi$ into a sequence of nontrivial bypass morphisms.
\item If the composition of a sequence $\beta_1,\dots,\beta_k$ of nontrivial bypasses is a zero morphism, then the composition $\cf(\beta_k)\circ\dots\circ\cf(\beta_1)=0$ in $\dne$.
\ee

By Theorem~\ref{thm: relations in contact category}, (1) can be reduced to the case where $\xi$ is a composition of two disjoint bypasses.  Here the bypasses may be trivial.

By Lemma~\ref{lemma: reduction 1}, (2) can be reduced to the case where the composition is $\beta' \circ \beta$ for two consecutive (nonzero, non-identity) bypasses $\beta, \beta'$ in any bypass triangle: Let $\beta'=\beta_k$ and $\xi=\beta_{k-1}\circ\dots\circ\beta_1$.  If $\xi$ is a zero morphism, then we can replace $\beta_1,\dots,\beta_k$ by $\beta_1,\dots,\beta_{k-1}$.  If $\xi$ is nonzero, then Lemma~\ref{lemma: reduction 1} implies that $\xi$ can be factored into $\beta\circ \zeta$, where $\beta,\beta'$ are two consecutive bypasses of a bypass triangle.

\subsubsection{Composition in bypass triangles} \label{Sec comp triangle}

We fix the notation $\g \xra{\beta} \g' \xra{\beta'} \g'' \xra{\beta''} \g$ for a bypass triangle throughout this section. We also use the notation $\uv(\beta)$ and $\underline{\frak c}(\beta)$ to mean $\uv$ and $\underline{\frak c}$ for $\beta$.

In Examples \ref{ex t1} and \ref{ex t2}, each projective $\rne$-module in $\cf(\g)$ appears either in $\cf(\g')$ or in $\cf(\g'')$.  This observation can be generalized as follows.

\begin{lemma} \label{lem ind triangle}
For any $\mf{i} \in OI(\g)$, one of the following holds:
\be
\item if $\mf{i} \in II(\beta)$, then $\mf{j}:=\beta(\mf{i})$ is not in $II(\beta')$ and satisfies $\g'(\mf{j})=\g(\mf{i})$;
\item if $\mf{i} \notin II(\beta)$, then there exists a unique $\mf{k} \in II(\beta'')$ such that $\g''(\mf{k})=\g(\mf{i})$.
\ee
\end{lemma}

\begin{proof}
%We assume that $\uv(\beta), \ov(\beta) \neq \ast$. The proofs of the cases $\uv(\beta)=\ast$ or $\ov(\beta)=\ast$ are similar and are left to the reader.
%The relation between $\uv(\wt{\beta}), \ov(\wt{\beta})$ for $\wt{\beta} \in \{\beta, \beta' ,\beta''\}$ is
First observe that $\g_{\ov(\beta)}={\g'}_{\uv(\beta')}^r={\g''}_{\uv(\beta'')}^l$.

If $\mf{i} \in II(\beta)$, then $\mf{j}:=\beta(\mf{i})$ satisfies $\g'(\mf{j})=\g(\mf{i})$ by definition of $II(\beta)$. If $\ov(\beta)\not=\ast$, then $$\g'_{\uv(\beta')}(\mj_{\uv(\beta')})=\g_{\ov(\beta)}(\mi_{\ov(\beta)}) \in \g_{\ov(\beta)}={\g'}_{\uv(\beta')}^r,$$
and, if $\ov(\beta)=*$, then $\ul{0}\in {\g'}_{\uv(\beta')}^r$.  In both cases $\mf{j}\notin II(\beta')$.

If $\mf{i} \notin II(\beta)$, then each label which is omitted in $\g(\mf{i})_{\ast}$ appears exactly once in $\g''_{\mf{w}}$ for some $\mf{w} \neq \ast$.
Hence there exists $\mf{k} \in OI(\g'')$ such that $\g''(\mf{k})=\g(\mf{i})$. If  $\ov(\beta)\not=\ast$, then
$$\g''_{\uv(\beta'')}(\mf{k}_{~\uv(\beta'')})=\g_{\ov(\beta)}(\mi_{\ov(\beta)}) \in \g_{\ov(\beta)}={\g''}_{\uv(\beta'')}^l,$$
and, if $\ov(\beta)=*$, then $\ul{0}\in {\g''}_{\uv(\beta'')}^l$. In both cases $\mf{k} \in II(\beta'')$.
\end{proof}

In view of Lemma~\ref{lem ind triangle}(2), there exists a map
$$\gamma(\beta''): OI(\g) \backslash II(\beta) \ra II(\beta'') \subset OI(\g'')$$
such that $\g''(\gamma(\beta'')(\mf{i}))=\g(\mf{i})$ for $\mf{i} \notin II(\beta)$.
%\begin{defn} \label{def gamma}
Then define the map
\begin{equation}\label{eqn: def gamma}
\cf(\gamma(\beta'')): \cf(\g) \ra \cf(\g'')
\end{equation}
of $\rne$-modules as the direct sum of identity morphisms $P(\g(\mf{i})) \ra P(\g''(\gamma(\beta'')(\mf{i})))$ for $\mf{i} \notin II(\beta)$.

The following lemma implies that $\cf(\beta')\circ \cf(\beta)=0 \in \Hom_{\dne}(\cf(\g), \cf(\g''))$, i.e., in the (ungraded) homotopy category.

\begin{lemma} \label{lem comp triangle}
$d\cf(\gamma(\beta''))=\cf(\beta')\circ \cf(\beta)$ as $\rne$-module maps from $\cf(\g)$ to $\cf(\g'')$.
\end{lemma}

\begin{proof}
We write $\gamma$ for $\gamma(\beta'')$ during the proof.
By definition $d\cf(\gamma)=d_{\g''} \circ \cf(\gamma) + \cf(\gamma) \circ d_{\g}$.
We assume that $\uv(\beta), \ov(\beta) \neq \ast$; the proofs of the other cases are easier.

We will show that  
\begin{equation} \label{eqn: exact}
(\cf(\beta')\circ \cf(\beta) + d_{\g''} \circ \cf(\gamma) + \cf(\gamma) \circ d_{\g})|_{P(\g(\mf{i}))}=0
\end{equation}
for any $\mf{i} \in OI(\g)$.

\s \n{\bf Case A.} Suppose $\mf{i} \in II(\beta)$, i.e., $\mi_{\uv}\in[[x(\beta),y(\beta)]]$. Then $\cf(\gamma)|_{P(\g(\mf{i}))}=0$ and Equation~\eqref{eqn: exact} becomes $(\cf(\beta')\circ \cf(\beta) + \cf(\gamma) \circ d_{\g})|_{P(\g(\mf{i}))}=0$.  Also $\bi \notin II(\beta')$ by Lemma \ref{lem ind triangle}(1). 

We first consider $\cf(\beta')\circ \cf(\beta)|_{P(\g(\mf{i}))}$. If $\cf(\beta')\circ \cf(\beta)|_{P(\g(\mf{i}))}\not=0$, then $\bi \in SI(\beta')$ since $\bi \notin II(\beta')$. This forces $\mi_{\uv(\beta)}=x(\beta)$ and $\mf{i}$ to be $\underline{\frak c}(\beta)$-admissible, and we have $\g''(\beta'(\beta(\mf{i})))=\g(\underline{\frak c}(\beta) | \mf{i})$. Note that, by Remark~\ref{rmk: not always in Y and Z}(1), $\underline{\frak c}(\beta)$ cannot share a common boundary with $\ast$. 

Recall that $d_{\g}=\sum_{\frak c}d_{\frak c}$, where the summation is over $\pi_0(R_-(\g))$. Hence  $\cf(\gamma) \circ d_{\g}|_{P(\g(\mf{i}))}$ is the sum of $\cf(\gamma) \circ d_{\frak c}|_{P(\g(\mf{i}))}$, where $\mi$ is ${\frak c}$-admissible. If ${\frak c}\not=\underline{\frak c}(\beta), \overline{\frak c}(\beta)$, then ${\frak c}|\mi \in II(\beta)$ and $\cf(\gamma)\circ d_{\frak c}|_{P(\g(\mi))}=0$.  If ${\frak c}=\overline{\frak c}(\beta)$, then $\mi$ is not ${\frak c}$-admissible.  If ${\frak c}=\underline{\frak c}(\beta)$, then the ${\frak c}$-admissibility of $\mi$ implies that $\mi_{\uv(\beta)}=x(\beta)$ and $\g''(\beta'(\beta(\mf{i})))=\g({\frak c} | \mf{i})$. See Figure~\ref{4-3-1}.
\begin{figure}[ht]
\begin{overpic}
[scale=0.25]{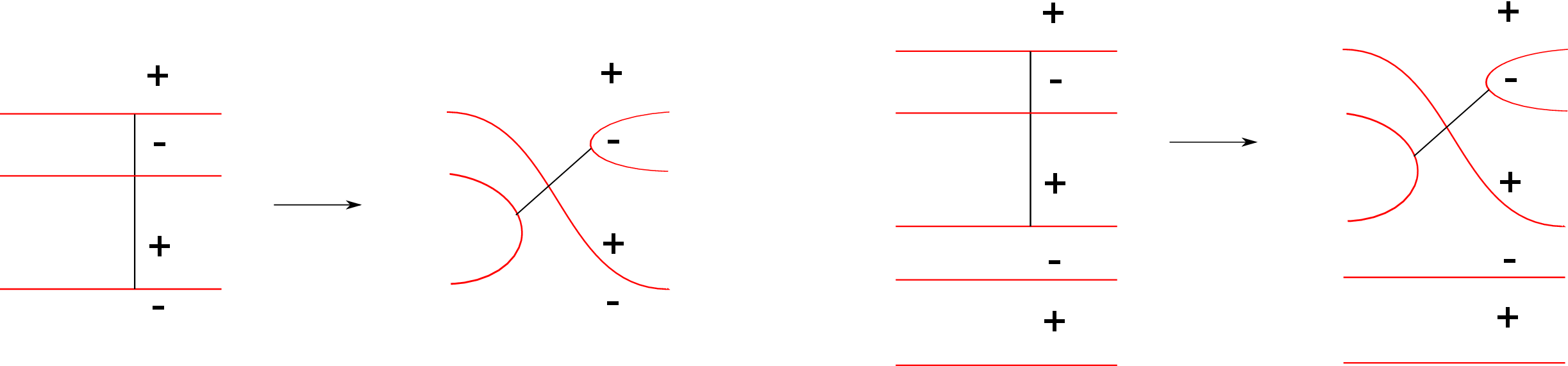}
\put(19,13){$\beta$}
\put(76,17){$\beta$}
\put(-4,0){${\scriptstyle \dnv(\uv(\beta), \mi_{\uv(\beta)})}$}
\put(30,0){${\scriptstyle LSV(\beta')}$}
\put(-1,7){${\scriptstyle x(\beta)}$}
\put(25,7){${\scriptstyle z(\beta')}$}
\put(58,12){${\scriptstyle \uv(\beta)}$}
\put(61,2){${\scriptstyle \mf{v}}$}
\put(70,3){${\scriptstyle i_0}$}
\put(86,2){${\scriptstyle \mf{w}(\beta')}$}
\put(99,3){${\scriptstyle k(\beta')}$}
\end{overpic}
\caption{Two subcases of ${\frak c}=\underline{\frak c}(\beta)$: $\uv(\beta)=\mf{v}$ on the left and $\uv(\beta) \in \dnv(\mf{v},i_0)$ on the right.}
\label{4-3-1}
\end{figure}

\s\n {\bf Case B.} Suppose $\mf{i} \notin II(\beta)$, i.e., $\mi_{\uv}\not\in[[x(\beta),y(\beta)]]$.  

We first consider $\cf(\beta')\circ \cf(\beta)|_{P(\g(\mf{i}))}$.  Note that $\cf(\beta')\circ \cf(\beta)|_{P(\g(\mf{i}))}\not=0$ if and only if $\mi \in SI(\beta)$, since $\bi\in II(\beta')$ is automatic.  

Next let ${\frak c}\in \pi_0(R_-(\g))$. If ${\frak c}\not=\underline{\frak c}(\beta), \overline{\frak c}(\beta)$, then there exists ${\frak c}''={\frak c}$, viewed as an element of $\pi_0(R_-(\g''))$, such that $d_{{\frak c}''} \circ \cf(\gamma)|_{P(\g(\mi))} = \cf(\gamma) \circ d_{\frak c}|_{P(\g(\mi))}$.  (This takes care of all ${\frak c}''\not=\underline{\frak c}(\beta''), \overline{\frak c}(\beta'')$.) If ${\frak c}=\overline{\frak c}(\beta)$ and $\mi$ is not ${\frak c}$-admissible, then $\cf(\gamma) \circ d_{\frak c}|_{P(\g(\mi))}=0$. If ${\frak c}=\overline{\frak c}(\beta)$ and $\mi$ is ${\frak c}$-admissible, then ${\frak c}| \mi\in II(\beta)$ and $\cf(\gamma) \circ d_{\frak c}|_{P(\g(\mi))}=0$. If ${\frak c}=\underline{\frak c}(\beta)$, then $\mi$ is not ${\frak c}$-admissible and $\cf(\gamma) \circ d_{\frak c}|_{P(\g(\mi))}=0$.  If ${\frak c}''=\overline{\frak c}(\beta'')$, then $\gamma(\beta'')(\mi)$ is not ${\frak c}''$-admissible and $d_{{\frak c}''} \circ \cf(\gamma)|_{P(\g(\mi))} = 0$.  Finally, if ${\frak c}''=\underline{\frak c}(\beta'')$, then $d_{{\frak c}''} \circ \cf(\gamma)|_{P(\g(\mi))} \not= 0$ if and only if $\mi\in SI(\beta)$.  Moreover, when this happens, $\beta(\mi)\in II(\beta')$ and $d_{{\frak c}''} \circ \cf(\gamma)|_{P(\g(\mi))}=\cf(\beta')\circ \cf(\beta)|_{P(\g(\mi))}$.
\end{proof}

Lemma \ref{lem comp triangle} will be the key to proving that $\cf$ is exact, i.e.,  maps bypass triangles to distinguished triangles; see Proposition \ref{prop triangle}.

\subsubsection{Disjoint pairs} \label{Sec disjoint}

For $s=0,1$, let $\beta^s \in \Hom(\g, \g^s)$ be a pair of bypass morphisms whose arcs of attachment are disjoint. Let $\wt{\g}$ be the resulting dividing set after attaching both bypasses to $\g$. We assume that the composition of the two bypass morphisms is nonzero; in particular we are assuming that $\wt{\g}$ does not contain contractible components. Let $\tb^s \in \Hom(\g^{1-s}, \wt{\g})$, $s=0,1$, be bypass morphisms such that $\tb^1 \circ \beta^0=\tb^0 \circ \beta^1 \in \Hom(\g, \wt{\g})$.

The goal of this subsection is to prove:

\begin{lemma} \label{lem disjoint}
Given a pair of disjoint bypasses $\beta^s$, $s=0,1$, on $\g$, and $\wt{\g}, \tb^s$, $s=0,1$, as above, we have
$$\cf(\tb^1) \circ \cf(\beta^0) = \cf(\tb^0) \circ \cf(\beta^1) \in \Hom_{\dne}(\cf(\g), \cf(\wt{\g})).$$
\end{lemma}

Since the reduction to disjoint pairs of bypasses at the beginning of Section~\ref{Sec comp} allows for any of $\beta^s$, $\widetilde{\beta}^s$ to be trivial, we first consider the situation where at least one of $\beta^s, \tb^s$ is trivial.  We can enumerate all the possible relative positions of $\beta^0$ and $\beta^1$, assuming $\beta^0$ is a fixed trivial bypass.  The enumeration is left to the reader, but we almost always have $\beta^0$ and $\tb^0$ trivial and $\beta^1=\tb^1$. The only (nontrivial, nonzero) exception is given in Figure~\ref{rotate2}, which is equivalent to a bypass rotation relation (R$_2'$) right below Theorem~\ref{thm: relations in contact category}.

\begin{figure}[ht]
\begin{overpic}[width=3.5cm]{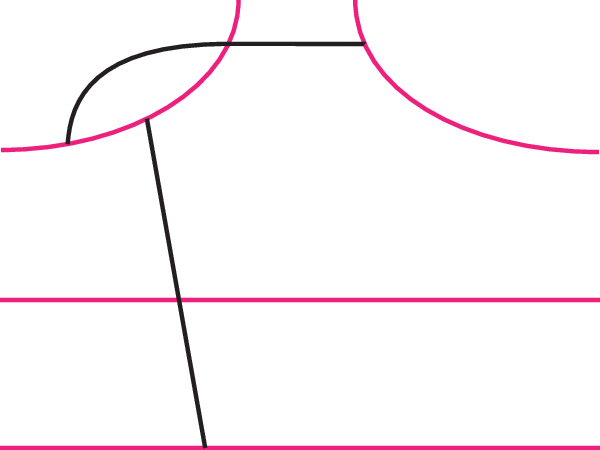}
\put(29,42){\tiny $\beta^1$} \put(48,59){\tiny $\beta^0$}
%\put(2,27){\tiny $\gamma_1$} \put(2,2.5){\tiny $\gamma_2$}
\end{overpic}
\caption{}
\label{rotate2}
\end{figure}

Next we enumerate all the cases where all of $\beta^s$, $\tb^s$ are nontrivial and nonzero.  For any nontrivial morphism $\beta$, $\cf(\beta)$ is determined by the maps $\beta: II(\beta) \sqcup SI(\beta) \ra OI(\g)$. Let
$$BV(\beta)=\{\uv(\beta), \ov(\beta)\} \cup LSV(\beta).$$
Then $\beta(\mf{i})_{\beta(\mf{v})}=\mi_{\mf{v}}$ for $\mf{v} \notin BV(\beta)$ by Remark \ref{rmk beta ind}(ii). Our proof is based on a case-by-case analysis of the relative positions of $BV(\beta^0)$ and $BV(\beta^1)$.  The following cases cover all the possibilities, after possibly switching $\beta^0$ and $\beta^1$:
\be
\item $BV(\beta^0) \cap BV(\beta^1)=\es$;
\item $\ov(\beta^0)=\uv(\beta^1)$;
\item $\ov(\beta^0)=\ov(\beta^1)$;
\item $\uv(\beta^0)=\uv(\beta^1)$;
\item $LSV(\beta^0) \cap BV(\beta^1) \neq \es$.
\ee
See Figure \ref{4-3-2} for a full list of Cases (2)-(5).
%Cases (3-3) and (4-1) are ``the left-right relation" with two possible sign assignments.
Each red arc is assumed to be a distinct arc of the dividing set, except for Case (4) where it is only required that the two positive regions belong to the same component.
The list for Case (5) does not include cases that were already listed (e.g., Case (2-1)).
Note that the various cases may have overlaps.
\begin{figure}[ht]
\begin{overpic}
[scale=0.25]{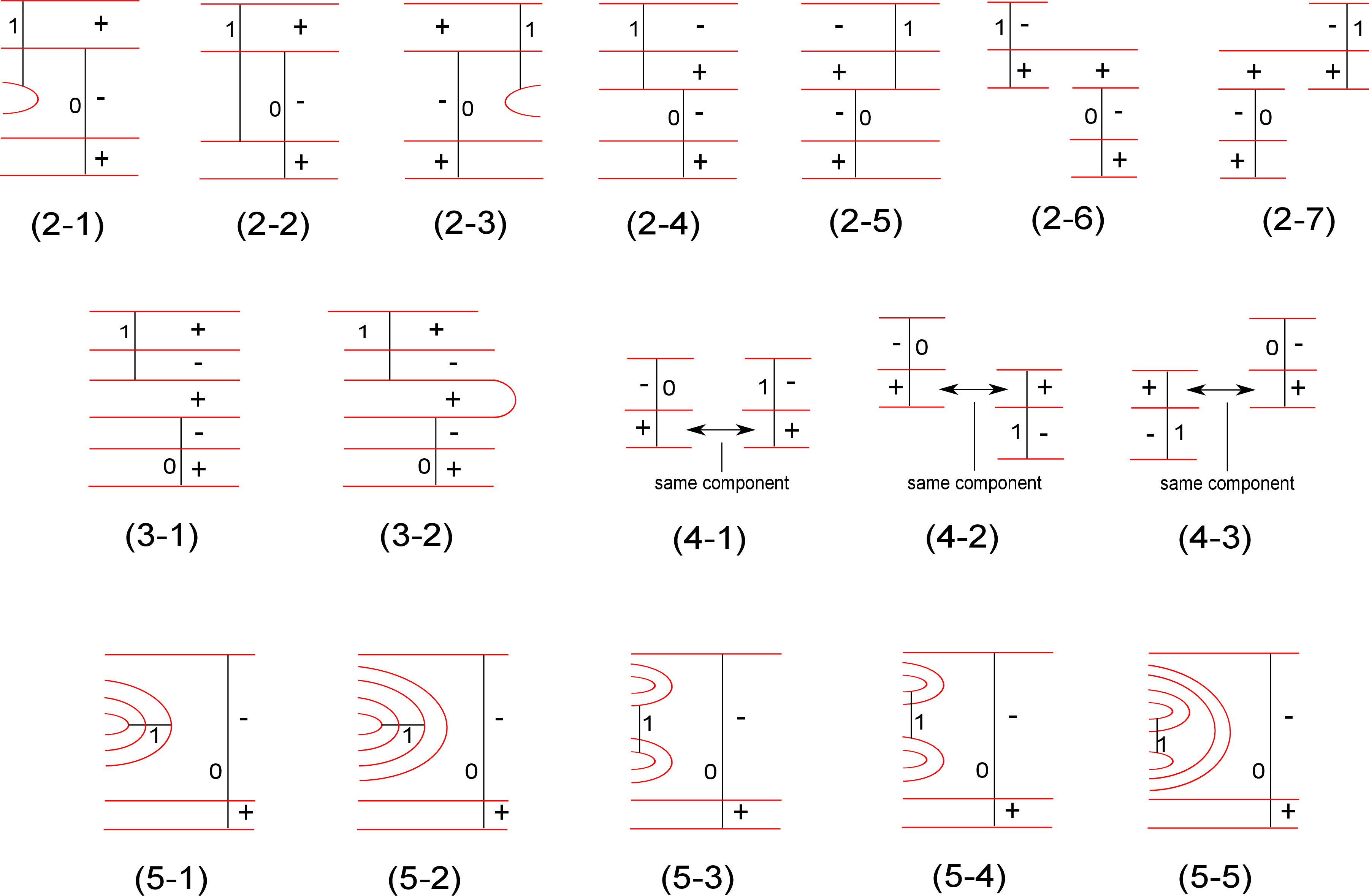}
\end{overpic}
\caption{Black arcs with labels $s=0,1$ are the arcs where bypasses $\beta^s$ are attached.}
\label{4-3-2}
\end{figure}

Before we discuss the general case, we look at an example in Case (2-1) where $\cf(\tb^1) \circ \cf(\beta^0) \neq \cf(\tb^0) \circ \cf(\beta^1)$ as $\rne$-linear maps from $\cf(\g)$ to $\cf(\wt{\g})$.  This illustrates the necessity of working in the homotopy category $\dne$.

\begin{example} \label{ex h}
Consider the bypasses in Figure \ref{4-3-3}, where $\ov(\beta^0)=\uv(\beta^1)$.
\begin{figure}[ht]
\begin{overpic}
[scale=0.2]{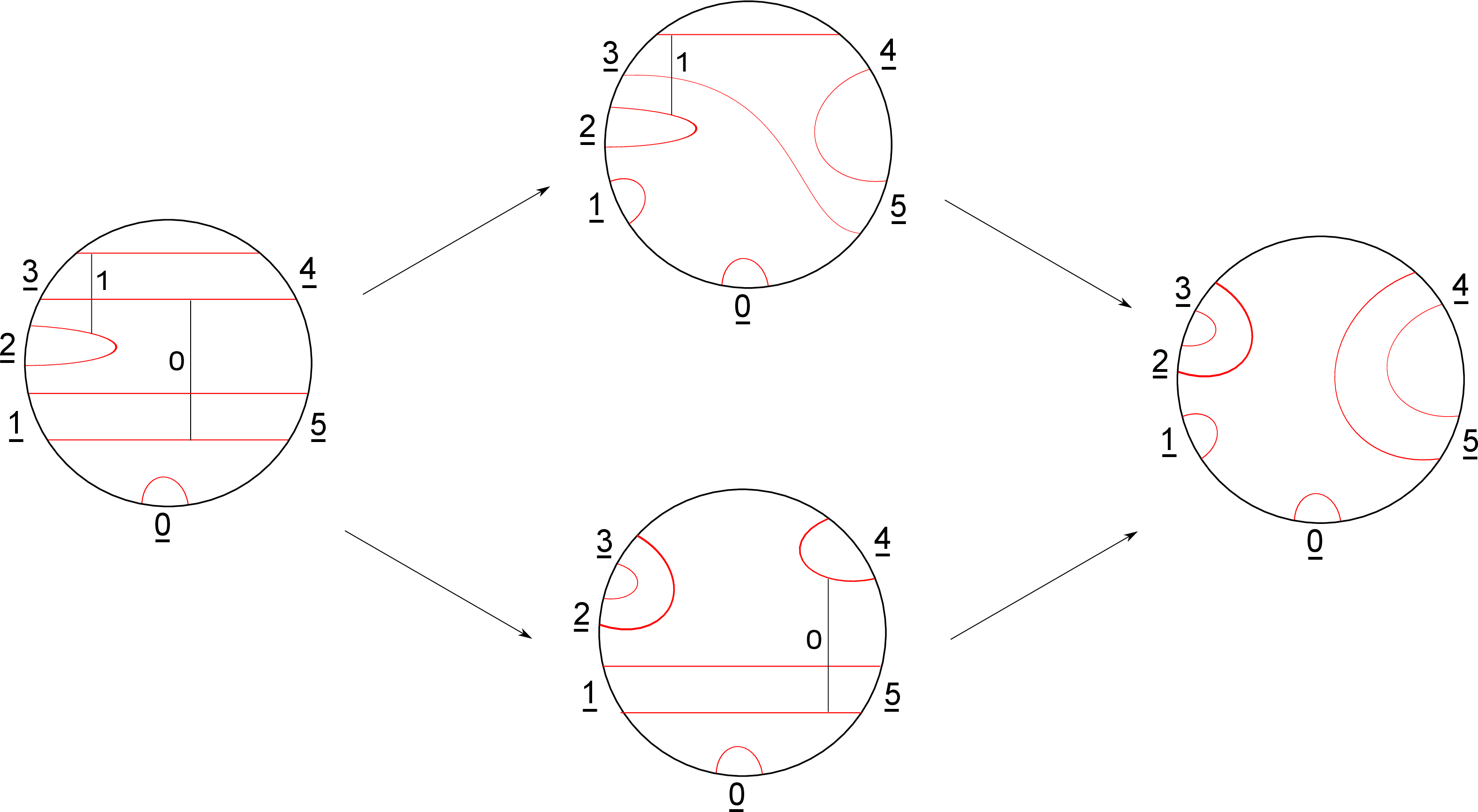}
\put(28,40){$\beta^0$}
\put(70,40){$\tb^1$}
\put(28,18){$\beta^1$}
\put(70,18){$\tb^0$}
\put(12,14){$\g$}
\put(56,30){$\g^0$}
\put(56,0){$\g^1$}
\put(92,14){$\wt{\g}$}
\end{overpic}
\caption{A pair of disjoint bypasses in Case (2-1).}
\label{4-3-3}
\end{figure}
\end{example}
By definition, the two compositions are:
$$\xymatrix{
\cf(\g): \ar[d]_{\cf(\beta^0)} & P(\ul{1},\ul{3})  \ar[r]  & P(\ul{1},\ul{4})  \ar[d] \ar[r]        & P(\ul{3},\ul{5})  \ar[d]  \ar[r]                & P(\ul{4},\ul{5})   \ar[d]\\
\cf(\g^0): \ar[d]_{\cf(\tb^1)} &          & P(\ul{3},\ul{4})    \ar[r] \ar[dr] & P(\ul{3},\ul{5})    \ar[r] \ar[dr]          & P(\ul{4},\ul{5})  \\
\cf(\wt{\g}):                  &          & P(\ul{2},\ul{4})  \ar[r]        & (P(\ul{3},\ul{4})   \oplus P(\ul{2},\ul{5})  ) \ar[r]  & P(\ul{3},\ul{5})   \\
\cf(\g): \ar[d]_{\cf(\beta^1)} & P(\ul{1},\ul{3})  \ar[r]\ar[dr]  & P(\ul{1},\ul{4})  \ar[dr] \ar[r]        & P(\ul{3},\ul{5})   \ar[dr]  \ar[r]                & P(\ul{4},\ul{5})   \\
\cf(\g^1): \ar[d]_{\cf(\tb^0)} & P(\ul{1},\ul{2})  \ar[r]  & P(\ul{1},\ul{3})  \ar[d] \ar[r]        & P(\ul{2},\ul{5})    \ar[d]  \ar[r]                & P(\ul{3},\ul{5})   \ar[d]\\
\cf(\wt{\g}):                  &          & P(\ul{2},\ul{4})  \ar[r]        & (P(\ul{3},\ul{4})   \oplus P(\ul{2},\ul{5})  ) \ar[r]  & P(\ul{3},\ul{5})
}$$

For $P(\ul{3},\ul{5})$, the corresponding index $\mf{i}$ is in $II(\beta^0) \cap II(\beta^1)$. Hence both compositions are identity morphisms when restricted to $P(\ul{3},\ul{5})$.

Let $\cf(h)$ be the following map:
$$\xymatrix{
\cf(\g): \ar[d]_{\cf(h)} & P(\ul{1},\ul{3}) \ar[r]  & P(\ul{1},\ul{4})  \ar[d] \ar[r]        & P(\ul{3},\ul{5})    \ar[r]                & P(\ul{4},\ul{5}) \\
\cf(\wt{\g}):                  &          & P(\ul{2},\ul{4})  \ar[r]        & (P(\ul{3},\ul{4})\oplus P(\ul{2},\ul{5})) \ar[r]  & P(\ul{3},\ul{5})
}$$
We can easily verify that $\cf(\tb^1) \circ \cf(\beta^0) + \cf(\tb^0) \circ \cf(\beta^1)=d\cf(h)$ as maps.
Hence $\cf(\tb^1) \circ \cf(\beta^0) = \cf(\tb^0) \circ \cf(\beta^1) \in \Hom_{\dne}(\cf(\g), \cf(\wt{\g}))$.
%Note that $\cf(h)$ is defined on $P(\ul{1},\ul{4})$ whose corresponding index $\mf{i}$ is not an element of $II(\beta^0) \cup II(\beta^1)$.

\begin{proof}[Proof of Lemma~\ref{lem disjoint}]
The following claims imply Lemma~\ref{lem disjoint}:

\begin{claimA}
With the exception of Cases (2-1), (2-2), and (5-4), in Cases (1)--(5),
$$\cf(\tb^1) \circ \cf(\beta^0) = \cf(\tb^0) \circ \cf(\beta^1)$$
as maps from $\cf(\g)$ to $\cf(\wt{\g})$.
\end{claimA}

\begin{claimB}
In Cases (2-1), (2-2), and (5-4), there exists $\cf(h): \cf(\g) \ra \cf(\wt{\g})$ such that
\begin{gather} \label{eq dh} \tag{H}
(\cf(\tb^1) \circ \cf(\beta^0) + \cf(\tb^0) \circ \cf(\beta^1)+d\cf(h))|_{P(\g(\mf{i}))}=0,
\end{gather}
for any $\mf{i} \in OI(\g)$.
\end{claimB}

\begin{claimC}
Referring to Figure~\ref{rotate}, if $\beta^0$ and $\beta^1$ are bypass morphisms corresponding to $\delta_0$ and $\delta_1$, then
$$\cf(\tb^1)\circ \cf(\beta^0)=\cf(\beta^1).$$
\end{claimC}

We only prove Claim A for Case (1) and Claim B for Case (2-2). The proofs of the other cases are similar and are left to the reader.

\s
\n{\bf Case (1).} Since $BV(\beta^0) \cap BV(\beta^1)=\es$, the maps $\cf(\beta^0)$ and $\cf(\beta^1)$ do not affect each other. More precisely, for $\mf{i} \in XI(\beta^0) \cap YI(\beta^1)$ where $XI, YI \in \{II, SI\}$, we have $\beta^0(\mf{i}) \in YI(\tb^1), \beta^1(\mf{i}) \in XI(\tb^0)$, and $\tb^1(\beta^0(\mf{i}))=\tb^0(\beta^1(\mf{i}))$.
For $\mf{i} \notin (II(\beta^0) \sqcup SI(\beta^0))\cap (II(\beta^1) \sqcup SI(\beta^1))$, $$\cf(\tb^1) \circ \cf(\beta^0)|_{P(\g(\mf{i}))}= \cf(\tb^0) \circ \cf(\beta^1)|_{P(\g(\mf{i}))}=0.$$

\s
\n{\bf Case (2-2).} Suppose that $\uv(\beta^s) = \ov(\beta^{1-s})$. Assume that $\uv, \ov \neq \ast$ for simplicity. Let $\uv^s$, $\ov{}^s$, $\underline{\frak c}^s$, $\overline{\frak c}^s$ denote $\uv$, $\ov$, $\underline{\frak c}$, $\overline{\frak c}$ for $\beta^s$, where $s=0,1$. In this case $\overline{\frak c}^0=\overline{\frak c}^1$.
We say $\mf{i} \in OI(\g)$ is of {\em type $(a^0, a^1)$} for $a^s \in \{l, r\}$ if $\g_{\uv^s}(i_{\uv^s}) \in \g_{\uv^s}^{a^s}$; see Figure \ref{4-3-4}.

For $\mf{i}$ of type $(r,r)$, there exists $\mf{j} \in OI(\tg)$ such that $\tg(\mf{j})=\g(\mf{i})$.
We denote $\mf{j}$ by $h(\mf{i})$ for $\mf{i}$ of type $(r,r)$.
Define $\cf(h): \cf(\g) \ra \cf(\tg)$ as the sum of identity morphisms $P(\g(\mf{i})) \ra  P(\tg(h(\mf{i})))$ for $\mf{i}$ of type $(r,r)$.
%We prove (\ref{eq dh}) by discussing $\mf{i}$ of type $(a^0, a^1)$.

\begin{figure}[ht]
\begin{overpic}
[scale=0.2]{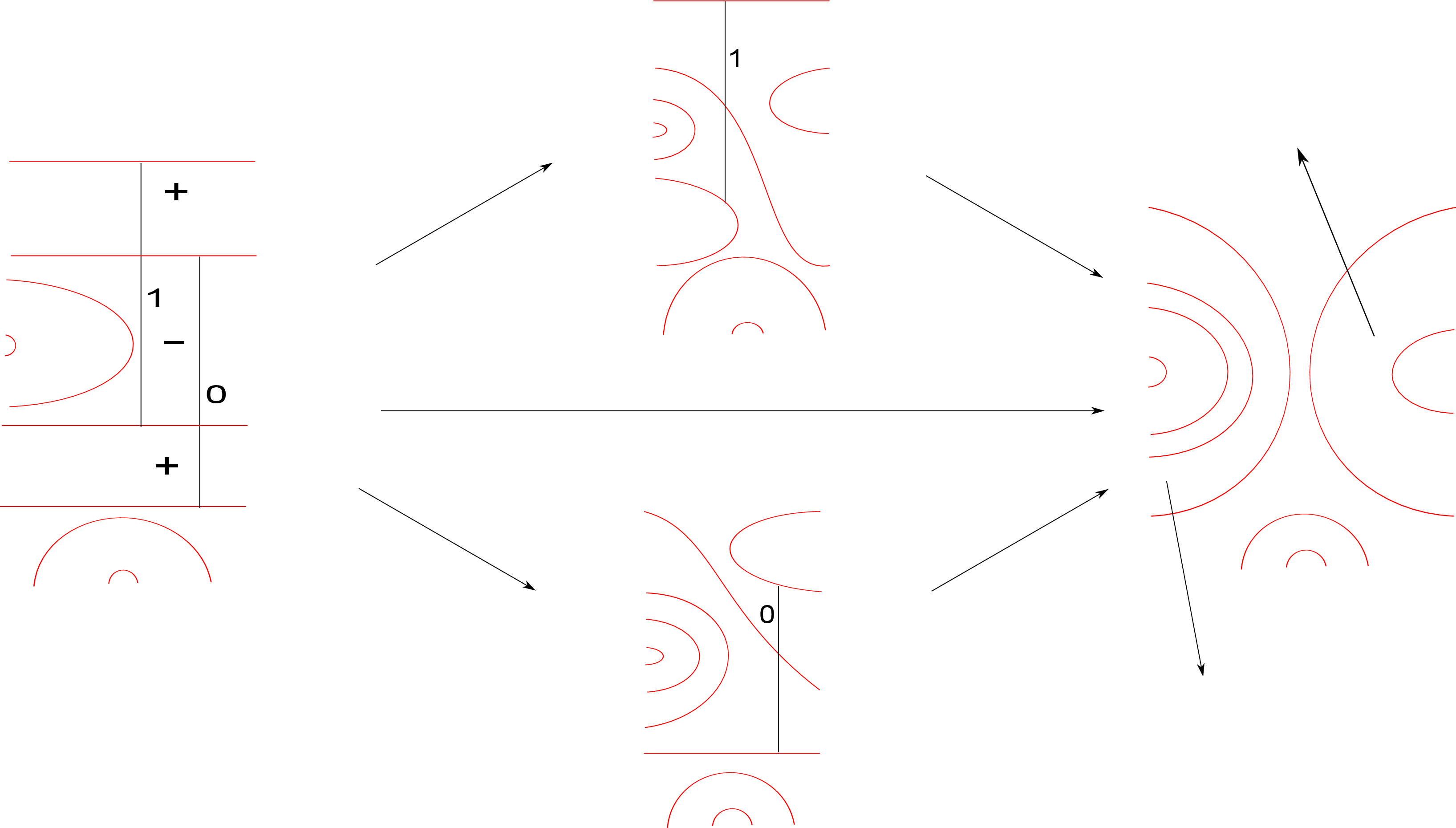}
\put(28,44){$\beta^0$}
\put(68,44){$\tb^1$}
\put(28,13){$\beta^1$}
\put(68,13){$\tb^0$}
\put(16,13){$\g$}
\put(60,33){$\g^0$}
\put(60,0){$\g^1$}
\put(98,10){$\wt{\g}$}
\put(40,26){$h$}
\put(0,23){$l^0$}
\put(0,41){$r^1$}
\put(16,23){$r^0$}
\put(16,41){$l^1$}
\put(43,40){$l^0$}
\put(43,53){$r^1$}
\put(56,6){$r^0$}
\put(56,17){$l^1$}
\put(78,38){$r^1$}
\put(99,24){$r^0$}
\put(0,18){$\small{\underline{\frak c}^0}$}
\put(15,33){$\small{\overline{\frak c}^0=\overline{\frak c}^1}$}
\put(9,47){$\small{\underline{\frak c}^1}$}
\put(83.7,30){$\small{{\frak c}'_1}$}
\put(96,18){$\small{{\frak c}'_2}$}
\put(99,30.5){$\small{{\frak c}'_3}$}
\put(80,6){${\scriptstyle \tb^1(\beta^0(\uv^0))}$}
\put(82,48){${\scriptstyle \tb^1(\beta^0(\uv^1))}$}
%\put(81,30){${\mf{u}}$}
%\put(7,19){${\mf{w}}$}
%\put(2,33){${\scriptscriptstyle\mw(\beta^0)}$}
%\put(48,36){${\scriptscriptstyle\mw(\tb^1)}$}
\end{overpic}
\caption{The letters $l^s, r^s$ on $\g$ indicate types of $\mf{i} \in OI(\g)$; $l^s, r^{1-s}$ on $\g^s$ indicate $\beta^s(\mf{i})$ for $\mf{i} \in II(\beta^s)$; $r^0, r^1$ on $\tg$ indicate $h(\mf{i})$ for $\mf{i}$ of type $(r,r)$.}
%The vectors $\mw$ on $\g$, $\mw(\tb^1)$ on $\g^0$ are used in Case (2-2-B); and $\mf{u}$ on $\tg$, $\mw(\beta^0)$ on $\g$ are used in Case (2-2-D).
\label{4-3-4}
\end{figure}

\s
\n {\bf Case (2-2-A).}
Suppose $\mf{i}$ is of type $(l,l)$. Then $\mf{i} \in II(\beta^0) \cap II(\beta^1)$ and $\beta^s(\mf{i}) \in II(\tb^{1-s})$ and $\tb^1(\beta^0(\mf{i}))=\tb^0(\beta^1(\mf{i}))$. Hence $(\cf(\tb^1) \circ \cf(\beta^0) + \cf(\tb^0) \circ \cf(\beta^1))|_{P(\g(\mf{i}))}=0$ by the commutativity.

We have $\cf(h)|_{P(\g(\mi))}=0$ by the definition of $\cf(h)$; hence $d_{\widetilde\Gamma}\circ \cf(h)|_{P(\g(\mi))}=0$.

It remains to show that $\cf(h)\circ d_\g|_{P(\g(\mi))}=0$. This follows from observing that there is no ${\frak c}\in \pi_0(R_-(\g))$ such that $\mi$ is ${\frak c}$-admissible and ${\frak c}|\mi$ is of type $(r,r)$: This is clear if ${\frak c}\not=\overline{\frak c}^0, \underline{\frak c}^0, \underline{\frak c}^1$ since the labels in $\uv^0$ and $\ov^0$ are not moved. If ${\frak c}=\overline{\frak c}^0=\overline{\frak c}^1$, then $\mi$ is not ${\frak c}$-admissible. If ${\frak c}=\underline{\frak c}^0$, then the label of $\ov^0$ is not moved, and if ${\frak c}=\underline{\frak c}^1$, then the label of $\uv^0$ is not moved.

Summing all the above terms gives (H) for $\mi$ of type $(l,l)$.

\s \n {\bf Case (2-2-B).} Suppose $\mf{i}$ is of type $(l,r)$.
Then $\mf{i} \in II(\beta^0)$, $\mf{i} \notin II(\beta^1)$ and $\beta^0(\mf{i}) \notin II(\tb{}^1)$. Since $\g_{\uv^0}(\mi_{\uv^0}) \in \g_{\uv^0}^l$, we have $\mf{i} \notin SI(\beta^1)$ and $\cf(\beta^1)|_{P(\g(\mf{i}))}=0$.

We have $\cf(h)|_{P(\g(\mi))}=0$ by the definition of $\cf(h)$; hence $d_{\widetilde\Gamma}\circ \cf(h)|_{P(\g(\mi))}=0$.

Let ${\frak c}\in \pi_0(R_-(\g))$. If ${\frak c}\not=\overline{\frak c}^0, \underline{\frak c}^0, \underline{\frak c}^1$, then ${\frak c}|\mi$ cannot be of type $(r,r)$ even if it exists.  If ${\frak c}=\overline{\frak c}^0=\overline{\frak c}^1$ or $\underline{\frak c}^1$, then $\mi$ is not ${\frak c}$-admissible.  Hence $\cf(h)\circ d_{\frak c}|_{P(\g(\mi))}=0$ for ${\frak c}\not=\underline{\frak c}^0$.

Finally, if ${\frak c}=\underline{\frak c}^0$, then $\mi$ is ${\frak c}$-admissible if and only if ${\frak c}|\mi$ is of type $(r,r)$.  This holds precisely when $\beta^0(\mf{i}) \in SI(\tb^1)$. Hence $(\cf(\tb^1) \circ \cf(\beta^0) + \cf(h)\circ d_{\underline{\frak c}^0})|_{P(\g(\mf{i}))}=0$.

Summing all the above terms gives (H) for $\mi$ of type $(l,r)$.

\s
\n {\bf Case (2-2-C).} Suppose $\mf{i}$ is of type $(r,l)$. The proof is the same as that of type $(l,r)$, with $l$ and $r$ reversed.

\s
\n {\bf Case (2-2-D).} Suppose $\mf{i}$ is of type $(r,r)$. Then $\mf{i} \notin II(\beta^0) \cup II(\beta^1)$.

%$\cf(h): P(\g(\mf{i})) \ra  P(\tg(h(\mf{i})))$ is the identity morphism.

Let ${\frak c}\in \pi_0(R_-(\g))$. If ${\frak c}\not=\overline{\frak c}^0, \underline{\frak c}^0, \underline{\frak c}^1$, then the labels in $\uv^0$ and $\ov^0$ are not moved and there exists ${\frak c}'={\frak c}$, viewed as an element of $\pi_0(R_-(\widetilde\g))$, such that $\cf(h)\circ d_{\frak c}+ d_{\frak c'}\circ \cf(h)=0$ on $P(\g(\mi))$. If ${\frak c}=\underline{\frak c}^0$ or $\underline{\frak c}^1$, then $\mi$ is not ${\frak c}$-admissible.  If ${\frak c}=\overline{\frak c}^0$ and $\mi$ is ${\frak c}$-admissible, then ${\frak c}|\mi$ is of type $(l,l)$ and $\cf(h)\circ d_{\frak c}=0$ on $P(\g(\mi))$.

There are three components ${\frak c}'$ of $\pi_0(R_-(\widetilde\g))$ that are not of the form ${\frak c}'={\frak c}$; they will be denoted by ${\frak c}'_1, {\frak c}'_2, {\frak c}'_3$, in order from left to right in the right-hand diagram of Figure~\ref{4-3-4}. $h(\mi)$ is not ${\frak c}'_2$-admissible. One easily checks that
$$(d_{{\frak c}'_1}\circ\cf(h)+\cf(\tb^1) \circ \cf(\beta^0))|_{P(\g(\mi))}=0, \quad (d_{{\frak c}'_3}\circ \cf(h)+\cf(\tb^0) \circ \cf(\beta^1))|_{P(\g(\mi))}=0.$$

Summing all the above terms gives (H) for $\mi$ of type $(r,r)$.
\end{proof}

The well-definition of the functor $\cf: \cne \ra \dne$ follows from Lemmas \ref{lem comp triangle} and \ref{lem disjoint}.

\section{The functors $\tfne$} \label{section: functors on universal cover}

In this section we extend $\cf: \cne \ra \dne$ to $\tf: \tcne \ra \tdne$ by relating the homotopy gradings on both sides.

\subsection{Degree of $\cf(\beta)$}

Let $\beta\in \Hom(\g,\g')$ be a nontrivial, nonzero bypass.  If $h(\mf{i})-h(\beta(\mf{i}))$ is the same for all $\mf{i} \in II(\beta) \sqcup SI(\beta)$, then we say that $\cf(\beta)$ is {\em homogeneous} and define the {\em degree of $\cf(\beta)$} to be $\deg(\cf(\beta))=h(\mf{i})-h(\beta(\mf{i}))$ for any $\mi$.   The goal of this subsection is to show that $\cf(\beta)$ is homogeneous and to compute $\deg(\cf(\beta))$.

The arc of attachment of a nontrivial bypass $0\not=\beta \in \Hom(\g,\g')$, together with the three components of $\Gamma$ that it intersects, cuts the disk $D^2$ into $6$ components.  The components are labeled $P_i(\beta)$, $i=1,\dots,6$, where $P_1(\beta)$ is the bottom component and $i$ increases as we go clockwise around $\bdry D^2$; see the top left diagram of Figure \ref{5-1-1}.  $P_i(\beta')$ and $P_i(\beta'')$ are defined analogously; see the top right and bottom diagrams of Figure~\ref{5-1-1}. $P_i(\beta)$ will be abbreviated as $P_i$ if $\beta$ is understood.  We write $\ul{0} \in P_i$ if the boundary arc with label $\ul{0}$ is contained in $P_i$.
\begin{figure}[ht]
\begin{overpic}
[scale=0.28]{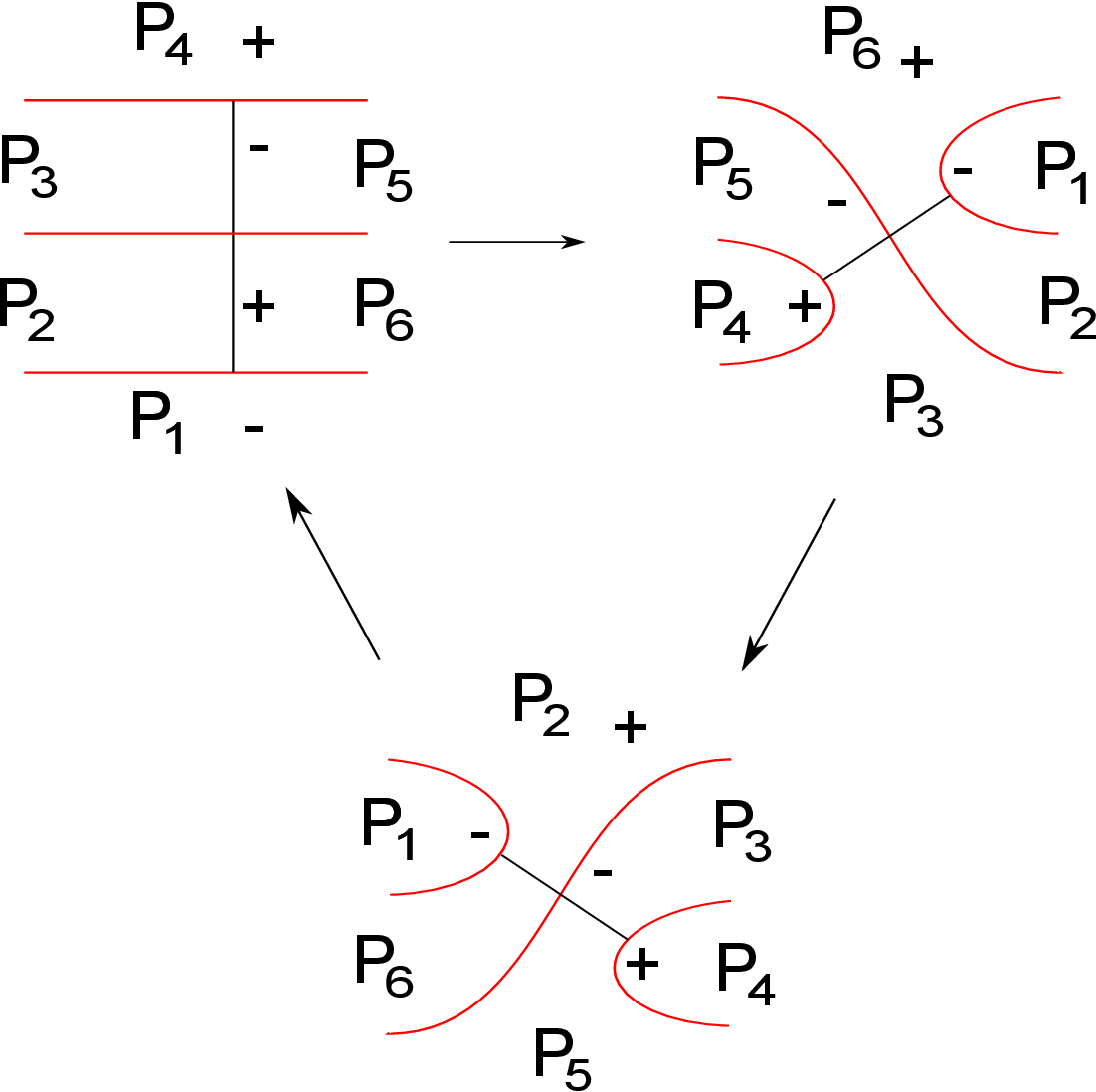}
\put(45,82){$\beta$}
\put(76,38){$\beta'$}
\put(17,40){$\beta''$}
\end{overpic}
\caption{}
\label{5-1-1}
\end{figure}

%\begin{defn} \label{def length A}
For convenience we will write
\begin{equation} \label{eqn: l of A}
l_{\g}(A)=\sum_{\gv \subset A}l_{\gv},
\end{equation}
where $A$ is a subset of $\R$.
%\end{defn}

\begin{lemma} \label{lem deg} $\mbox{}$
\begin{enumerate}
\item  If $\beta \in \Hom(\g,\g')$ is a nontrivial bypass, then $\cf(\beta)$ is homogeneous and 
\begin{equation} \label{eqn: formula for deg}
\deg(\cf(\beta))= \left\{
\begin{array}{cl}
0 & \mbox{if} \hspace{0.1cm} \ul{0} \in P_1, P_2; \\
|\guv^r|+l_{\g}((\guv(0), \guv(x))) & \mbox{if} \hspace{0.1cm} \ul{0} \in P_3, P_4; \\
1-|\guv^l|-l_{\g}((\guv(x), \gov(z))) & \mbox{if} \hspace{0.1cm} \ul{0} \in P_5, P_6.
\end{array}\right.
\end{equation}
\item The sum of degrees of the three bypasses in a bypass triangle is $1$.
\end{enumerate}
\end{lemma}

\begin{proof}
(1) Suppose $\ul{0} \in P_1$. The other cases are similar and are left to the reader.

If $\mf{i} \in II(\beta)$, then $\bi_{\beta(\mf{v})}=\mi_{\mf{v}}$ for any $\mf{v} \in \tpv(\g)$.
Moreover, the nesting degree is unchanged: $c_{\beta(\mf{v})}(\bi_{\beta(\mf{v})})=c_{\mf{v}}(\iv)$ for any $\mf{v}$ since $\ul{0} \in P_1$.
Hence $h(\bi)=h(\mf{i})$ and $\deg(\cf(\beta))=0$.

If $\mf{i} \in SI(\beta)$, then the bypass is of shuffling type (Y);  see Figure \ref{5-1-2}.
Note that
\begin{align}\label{eqn: difference}
h(\bi)-h(\mf{i})&=(h(\bi, \beta(\ov))-h(\mf{i}, \uv)) + (h(\bi, \beta(\uv))-h(\mi,\ov)) \\
\nonumber & \qquad \qquad + \sum_{\mf{w} \in LSV(\beta)}\left(h(\bi, \beta(\mf{w}))-h(\mi,\mf{w})\right),
\end{align}
since the only regions that are modified are $\ov$, $\uv$, and $\mf{w}\in LSV(\beta)$.  We have $x=z=0, y=|\guv^l|-1$, and
\begin{itemize}
\item[(a)] $\bi_{\beta(\ov)}-\mi_{\uv}=l_{\g_{\ov}}-|\guv^l|+1$;
\item[(b)] $\mi_{\ov}=0$ and $\bi_{\beta(\uv)}=|\guv^l|-1$;
\item[(c)] $\mi_{\mf{w}}=0$ and $\bi_{\beta(\mf{w})}=l_{\gw}$ for $\mf{w} \in LSV(\beta).$
\end{itemize}
Using (a)--(c) we compute each of the three terms on the right-hand side of Equation~\eqref{eqn: difference}:
\begin{align*}
h(\bi, \beta(\ov))-h(\mf{i}, \uv)&=(\bi_{\beta(\ov)}-\mi_{\uv})+ l_{\g}((\gov(0), \guv(\mi_{\uv}))) - l_{\g}((\guv(0), \guv(\mi_{\uv}))) \\
&=(l_{\g_{\ov}}-|\guv^l|+1) - (l_{\g}((\guv(0), \gov(0))) + l_{\g_{\ov}}) \\
&=1-|\guv^l|-l_{\g}((\guv(0), \gov(0))),
\end{align*}
$$h(\bi, \beta(\uv))-h(\mi,\ov)=|\guv^l|-1+l_{\g}((\guv(0), \guv(y))),$$
$$\sum_{\mf{w} \in LSV(\beta)}(h(\bi, \beta(\mf{w}))-h(\mi,\mf{w}))=l_{\g}((\guv(y), \gov(0))).$$
Summing the three terms gives:
$$h(\bi)-h(\mf{i})=l_{\g}((\guv(0), \guv(y))) + l_{\g}((\guv(y), \gov(0))) - l_{\g}((\guv(0), \gov(0)) )=0.$$

\begin{figure}[ht]
\begin{overpic}
[scale=0.2]{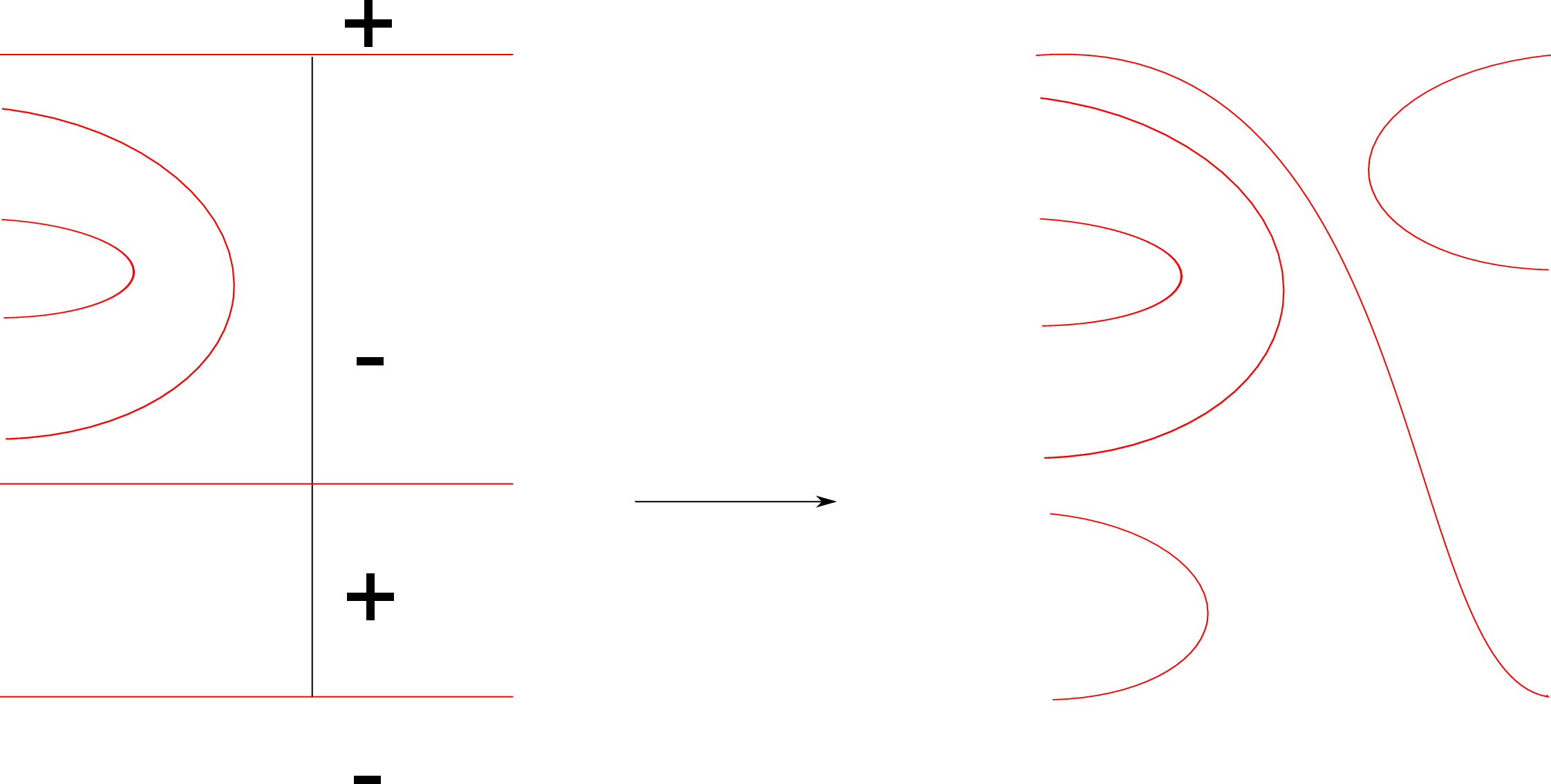}
\put(10,0){$\ul{0}$}
\put(45,20){$\beta$}
\put(30,10){$i_{\uv}$}
\put(0,25){$i_{\mf{w}}$}
\put(0,7){$x$}
\put(0,16){$y$}
\put(0,50){$i_{\ov}=z$}
\put(98,10){${\scriptstyle \bi_{\beta(\ov)}}$}
\put(61,13){${\scriptstyle \bi_{\beta(\uv)}}$}
\put(61,39){${\scriptstyle \bi_{\beta(\mf{w})}}$}
\end{overpic}
\caption{A shuffling index $\mf{i} \in SI(\beta)$, where $\ul{0} \in P_1$.}
\label{5-1-2}
\end{figure}

\s\n
(2) For a bypass triangle $\g \xra{\beta} \g' \xra{\beta'} \g'' \xra{\beta''} \g$, $P_i(\beta)=P_{i+2}(\beta')=P_{i+4}(\beta'')$ where the subscripts are viewed mod $6$. Since any triangle is invariant under rotation, we may assume that $\ul{0} \in P_1(\beta)$ or $P_2(\beta)$. By Equation~\eqref{eqn: formula for deg}, we may further assume that $\ul{0} \in P_1(\beta)$. Then $\ul{0} \in P_3(\beta')=P_5(\beta'')$. Note that the terms $\uv, \ov, x, y, z$ that appear in Equation~\eqref{eqn: formula for deg} are those for $\beta$.  We can verify that:
%The relation between $\beta'$ and $\beta''$ as in Figure \ref{4-3-1} implies that
\begin{gather*}
\g_{\uv(\beta')}^r=\g_{\uv(\beta'')}^l, \quad \g_{\uv(\beta')}(0)=\g_{\uv(\beta'')}(x(\beta'')), \quad \g_{\uv(\beta')}(x(\beta'))=\g_{\ov(\beta'')}(z(\beta'')); \\
l_{\g'}((\g_{\uv(\beta')}(0), \g_{\uv(\beta')}(x(\beta'))))=l_{\g''}((\g_{\uv(\beta'')}(x(\beta'')), \g_{\ov(\beta'')}(z(\beta'')))).
\end{gather*}
Hence $\deg(\cf(\beta))+\deg(\cf(\beta'))+\deg(\cf(\beta''))=1$ by Equation~\eqref{eqn: formula for deg}.
\end{proof}

\subsection{Definition of $\tf$} \label{Sec grading}

Each indecomposable object of $\tcne$ is a pair $(\g, [\xi])$ consisting of a dividing set $\g$ in $\cne$ and a homotopy grading. {\em From now on the source $\g^0 \in \bne$ of the quiver $\qne$, i.e., $\g^0_{\ast}=\{0,1,\dots,e\}$, will be the basepoint of $\tcne$.}

\subsubsection{Definition of $[\xi(\g)]$} \label{subsubsection: xi g}

We first choose a ``canonical" homotopy grading $[\xi(\g)]$ for each $\g\in\cne$.  It is defined by induction on $m(\g)=e+1-|\g_{\ast}|$. Note that $m(\g)=0$ if and only if $\g \in \bne$ is basic.

For any $\g^b \in \bne$, choose a path from $\g^0$ to $\g^b$ in $\qne$ and let $\xi^b$ denote the composition of bypasses corresponding to the path. The homotopy grading $[\xi^b]$ is independent of the choice of path and we define $[\xi(\g^b)]=[\xi^b]$.  (Note that $\xi^b$ may not be a tight contact structure.)

Next suppose that $m(\g)>0$.  Recall $\beta(\g)$ from Definition \ref{def beta g} for any non-basic $\g$. There is a bypass triangle
\begin{equation} \label{eqn: triangle with ob}
\g \xra{\beta(\g)} \g' \ra \g'' \xra{\ob(\g)} \g.
\end{equation}
By the construction of $\beta(\g)$, we have $m(\g'),  m(\g'') < m(\g)$.  Let $\ob(\g)$ denote the bypass from $\g''$ to $\g$.  We then define $[\xi(\g)]=[\ob(\g) \circ \xi(\g'')]$.

\subsubsection{Definition of $\tf$} \label{subsubsection: def of tf}

We define $\tf: \tcne \ra \tdne$ as follows: We set $\tf(\g, [\xi(\g)])=\cf(\g)$ and extend $\tf$ to any object so that it commutes with the shift functors on both sides, i.e., $\tf((\g,[\xi])[i])=\tf(\g,[\xi])[i]$.  Here $(\g,[\xi])[i]=T^i(\g,[\xi])$, where $T$ is the shift functor on $\tcne$. Next, suppose $\beta\in \Hom(\widetilde\g,\g)$ is nonzero, where {\em $\beta$ is not necessarily a bypass.}   Let $c(\beta)$ be the integer encoding minus the Hopf invariant and satisfying
\begin{equation} \label{eqn: c beta}
[\beta\circ \xi(\widetilde\g)]=[\xi(\g)][c(\beta)],
\end{equation}
where $[c(\beta)]$ refers to shifting by $c(\beta)$.
Then we define
$$\tf((\widetilde\g,[\xi(\widetilde\g)])\stackrel\beta\to (\g,[\beta\circ\xi(\widetilde\g)]))= (\cf(\widetilde\g)\stackrel{\cf(\beta)[c(\beta)]}{\xrightarrow{\hspace{1.5cm}}} \cf(\g)[c(\beta)]),$$
where $\cf(\beta)[c(\beta)]$ is $\cf(\beta)$ postcomposed with the shift $[c(\beta)]$.
%We finally extend it to the general case by composing with shift functors.

%\begin{rmk}
%The definition of $\tf$ depends on the choices of $[\xi(\g^0)]$.
%\end{rmk}

%\s
%In order to show that $\tf$ is well-defined, it suffices to prove that $\deg(\tf(\beta))=0$ for any $\beta$.

\subsubsection{Well-definition}

In this subsection we will abuse notation and not distinguish between $\beta\in \Hom(\tg,\g)$ and $\beta\in  \Hom_{\tcne}((\tg,[\xi]),(\g,[\beta\circ \xi]))$.

%In this section only, we write $\g$ for $(\g,[\xi(\g)])\in \mathfrak{ob}(\tcne)$.

\begin{prop} \label{prop grading}
The functor $\tf$ is well-defined, i.e., $\deg(\tf(\beta))=0$ for any nonzero $\beta \in \Hom(\tg,\g)$.
\end{prop}

\begin{proof}
It suffices to prove that
\begin{gather} \label{eq deg}
[\beta \circ \xi(\tg)]= [\xi(\g)][\deg(\cf(\beta))]
\end{gather}
for a nontrivial bypass $\beta$.  Comparing with Equation~\eqref{eqn: c beta},  $c(\beta)=\deg(\cf(\beta))$ and $\deg(\tf(\beta))=\deg(\cf(\beta))-c(\beta)=0$.

We first prove Equation \eqref{eq deg} for $\beta=\ob(\g) \in \Hom(\widetilde\g, \g)$. By Definition \ref{def beta g}, $\ul{0} \in \guv^l$, where $\uv=\uv(\beta)$.
Hence $\ul{0} \in P_2(\beta)$, which implies that $\deg(\cf(\beta))=0$ by Lemma \ref{lem deg}(1).
Equation \eqref{eq deg} is immediate from the definition of $[\xi(\g)]$ as $[\xi(\g)]=[\ob(\g) \circ \xi(\tg)]=[\beta \circ \xi(\tg)]$.

For a general nontrivial bypass $\beta$, we use $\ob(\g)$ and $\ob(\tg)$, defined as in Equation~\eqref{eqn: triangle with ob}, to simplify $\beta$.

\s
\n{\bf Case 1.} $\ul{0} \in P_1(\beta) \cup P_2(\beta)$. 
Equation~\eqref{eq deg} will be proved by induction on 
$$\kappa(\beta):=\dim \uv(\beta) + l_{\tg}((0,\tg_{\ov(\beta)}(0))).$$
If $\kappa(\beta)\geq 1$, then there exists a commutative diagram:
\begin{equation} \label{commut diag}
\xymatrix{
\tg \ar[r]^{\beta} & \g \\
\tg'' \ar[r]^{\beta''} \ar[u]^{\ob(\tg)}  & \g'' \ar[u]_{\ob(\g)} }
\end{equation}
where $\kappa(\beta'') < \kappa(\beta) $ and $\ul{0} \in P_1(\beta'') \cup P_2(\beta'')$.  (In some cases, e.g., Figure~\ref{5-1-3}, there may be degeneracies.)
Since we have already shown that Equation \eqref{eq deg} holds for $\ob(\g)$ and $\ob(\tg)$, it suffices to prove Equation \eqref{eq deg} for $\beta''$.
We then reduce to the case where $\kappa(\beta)=0$, i.e., $l_{\tg}((0,\tg_{\ov(\beta)}(0)))=0$, $\dim \uv(\beta) =0$, and $\ul{0} \in \tg_{\uv(\beta)}^l$.
\begin{figure}[ht]
\begin{overpic}
[scale=0.25]{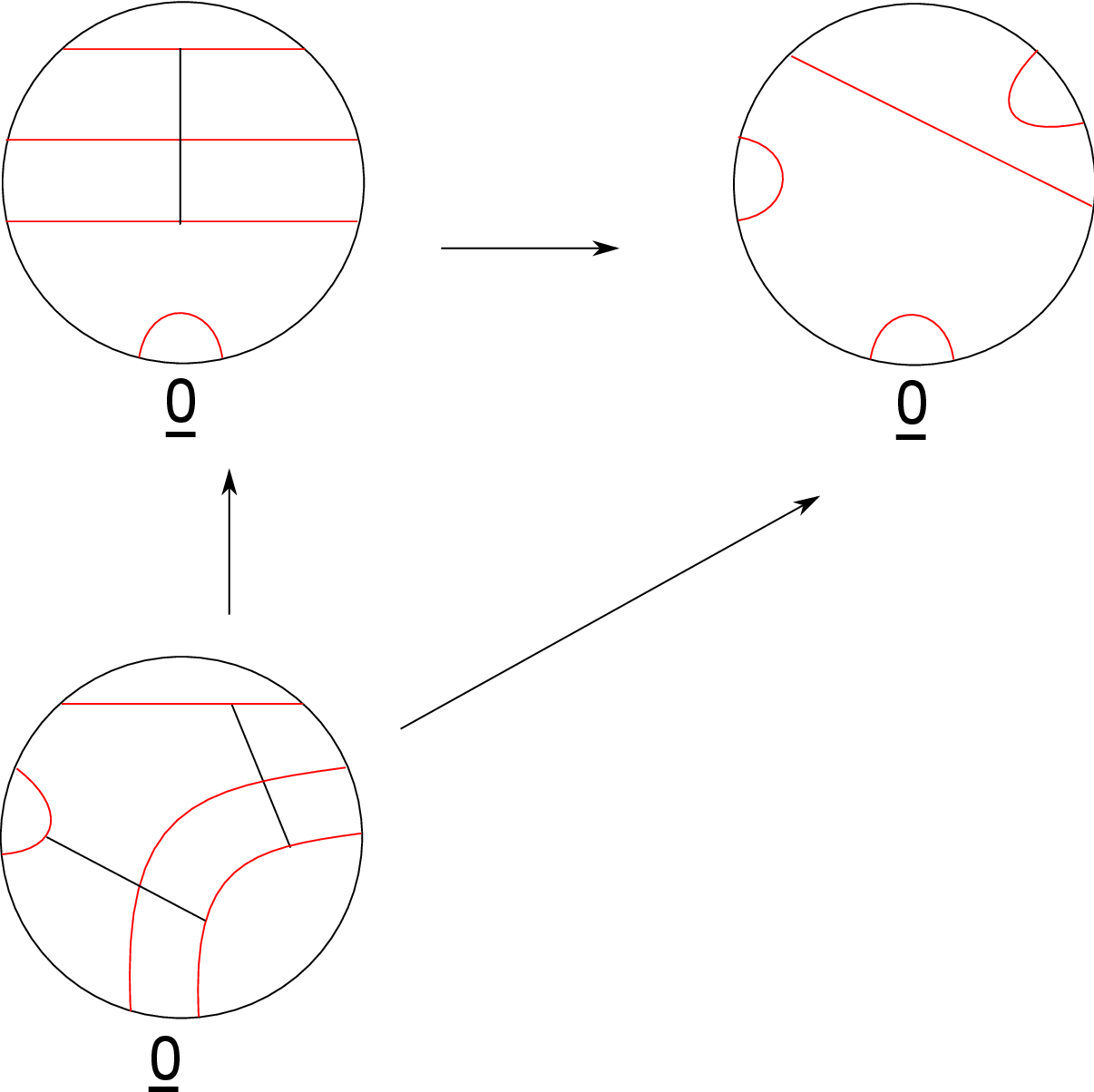}
\put(45,80){$\beta$}
\put(23,45){$\ob(\tg)$}
\put(60,35){$\beta''$}
\put(30,0){$\tg''$}
\put(30,60){$\tg$}
\put(95,60){$\g$}
\end{overpic}
\caption{The case $\ul{0} \in P_1(\beta)$ and $\dim \uv(\beta)=1$.}
\label{5-1-3}
\end{figure}

Let $\tg\stackrel\beta\to \g$ be a bypass satisfying $\kappa(\beta)=0$. If $l_{\tg_{\ov(\beta)}}=0$, then $\beta=\ob(\g)$ and Equation \eqref{eq deg} holds. If $l_{\tg_{\ov(\beta)}}>0$, then we apply the same commutative diagram~\eqref{commut diag} with $\beta''$ trivial to iteratively reduce to the case where $\beta=\ob(\g)$.
\s
\n{\bf Case 2.} $\ul{0} \in P_5(\beta) \cup P_6(\beta)$. Then $\deg(\cf(\beta))\leq 0$ by Lemma \ref{lem deg}(1).
We use bypass triangles to reduce to the case where $\deg(\cf(\beta))=0$.

Suppose that $\deg(\cf(\beta))<0$ for $\beta \in \Hom(\tg, \g)$. Then by Lemma~\ref{lem deg} there exists a nontrivial bypass $\alpha \in \Hom(\g, \g')$ such that $\ov(\alpha)=\beta(\ov(\beta))$ and $\g_{\uv(\alpha)} \subset [\tg_{\uv(\beta)}(x), \tg_{\ov(\beta)}(z)]$; see Figure \ref{5-1-4} for an example.
There exists a nontrivial bypass $\wt{\alpha} \in \Hom(\tg, \tg')$ whose arc of attachment is disjoint from that of $\beta$; attaching the bypasses in different orders, we obtain a bypass $\beta' \in \Hom(\tg', \g')$ such that $\beta' \circ \wt{\alpha} = \alpha \circ \beta \in \Hom(\tg, \g')$.
Note that $\beta'$ is a trivial bypass if $\uv(\alpha)=\beta(\uv(\beta))$. We also use the bypass triangle $\g \xra{\alpha} \g' \xra{\alpha'} \g'' \xra{\alpha''} \g$.
\begin{figure}[ht]
\begin{overpic}
[scale=0.2]{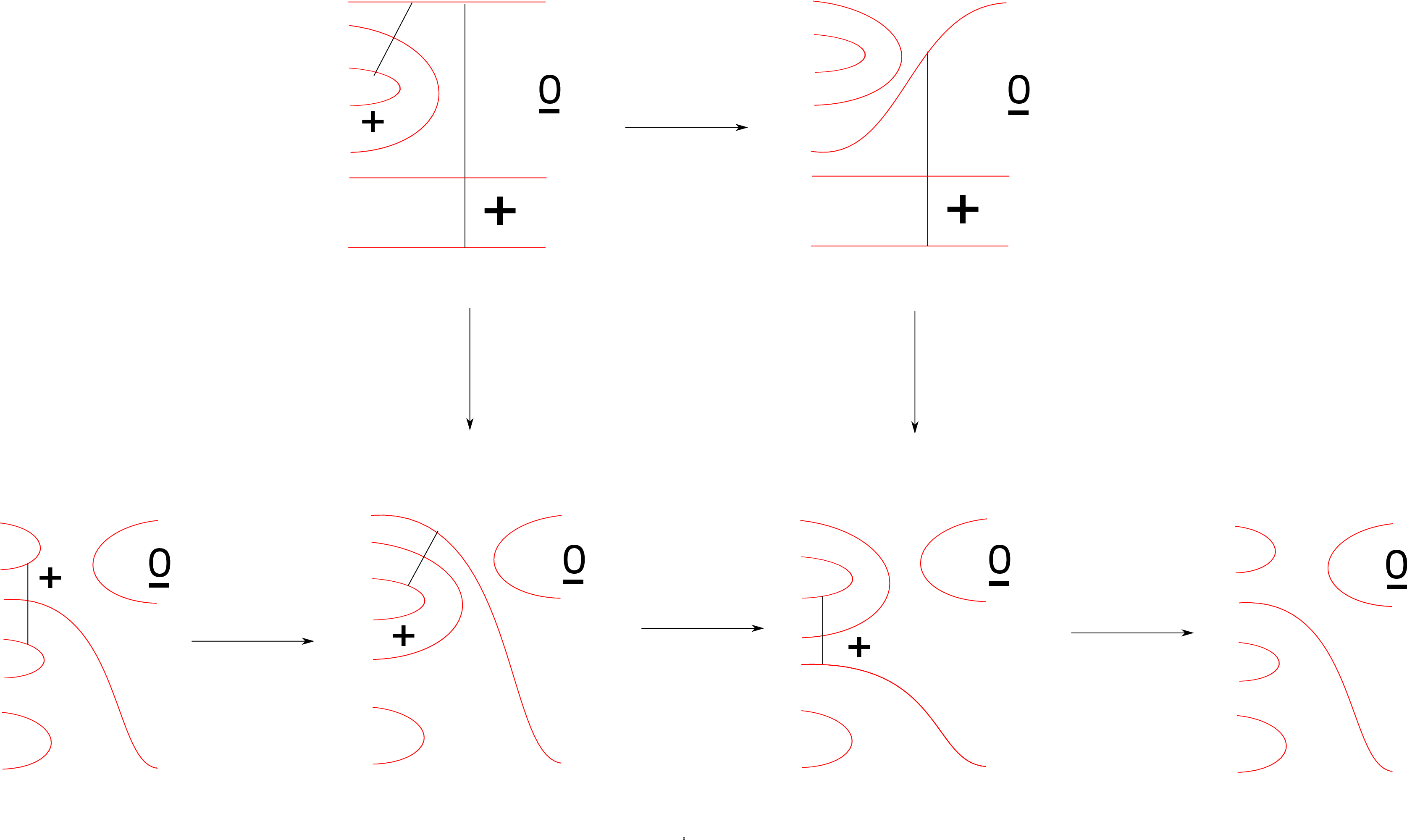}
\put(25,36){$\tg$}
\put(60,36){$\tg'$}
\put(48,52){$\wt{\alpha}$}
\put(0,0){$\g''$}
\put(30,0){$\g$}
\put(60,0){$\g'$}
\put(90,0){$\g''$}
\put(35,32){$\beta$}
\put(67,32){$\beta'$}
\put(48,17){$\alpha$}
\put(80,16){$\alpha'$}
\put(18,16){$\alpha''$}
\end{overpic}
\caption{The case $\ul{0} \in P_5(\beta)$ and $\deg(\cf(\beta))<0$.}
\label{5-1-4}
\end{figure}

By the definition of $\alpha$, we have $\ul{0} \in P_4(\alpha)$, which implies that $\ul{0} \in P_2(\alpha'')$ and $\ul{0} \in P_6(\alpha')$. Note that $\ul{0} \in P_5(\wt{\alpha}) \cup P_6(\wt{\alpha})$ and $\ul{0} \in P_5(\beta') \cup P_6(\beta')$. By Lemma \ref{lem deg}(1),(2),
\begin{align*}
\deg(\cf(\beta))&=\deg(\cf(\wt{\alpha}))+\deg(\cf(\beta'))-\deg(\cf(\alpha)) \\
&=\deg(\cf(\wt{\alpha}))+\deg(\cf(\beta'))+\deg(\cf(\alpha'))+\deg(\cf(\alpha''))-1 \\
&=\deg(\cf(\wt{\alpha}))+\deg(\cf(\beta'))+\deg(\cf(\alpha'))-1.
\end{align*}

Suppose by induction Equation \eqref{eq deg} holds for any $\wt{\beta}$ such that $\deg(\cf(\wt{\beta}))>\deg(\cf(\beta))$ and $\ul{0} \in P_5(\wt{\beta}) \cup P_6(\wt{\beta})$. In particular, it holds for $\wt{\beta} \in \{\wt{\alpha}, \beta', \alpha'\}$.  Then
\begin{align*}
[\alpha \circ \beta \circ \xi(\tg)]=[\beta' \circ \wt{\alpha} \circ \xi(\tg)]&=[\xi(\g')][\deg(\cf(\beta'))+\deg(\cf(\wt{\alpha}))]\\
&=[\xi(\g')][\deg(\cf(\alpha ))+\deg(\cf(\beta))].
\end{align*} By composing with $\alpha'' \circ \alpha'$, we obtain
\begin{align*}
[\beta \circ \xi(\tg)][1]&=[\alpha'' \circ \alpha' \circ \alpha \circ \beta \circ \xi(\tg)]\\
&=[\alpha'' \circ \alpha' \circ \xi(\g')][\deg(\cf(\alpha))+\deg(\cf(\beta))]\\
&=[\xi(\g')][\deg(\cf(\alpha'' ))+\deg(\cf(\alpha'))+\deg(\cf(\alpha ))+\deg(\cf(\beta))]\\
&=[\xi(\g')][1+\deg(\cf(\beta))],
\end{align*}
where the third line uses Equation \eqref{eq deg} for $\alpha''$ and $\alpha'$.  Equation~\eqref{eq deg} holds for $\alpha''$, since $\ul{0} \in P_2(\alpha'')$ was treated in Case 1, and for $\alpha'$ by the inductive hypothesis.
This proves Equation \eqref{eq deg} for $\beta$.

By induction on $\deg(\cf(\beta))$, we reduce to the case where $\deg(\cf(\beta))=0$.
Then, by the procedure used in Case 1, we reduce to the case where $l_{\tg}((0,\tg_{\ov(\beta)}(0)))=0$ and $\dim \uv(\beta) =0$.  A further reduction (details left to the reader) gets us to the case where $\beta \in \Hom(\tg, \g)$ and $\tg, \g$ are basic dividing sets.
Hence Equation \eqref{eq deg} holds since $[\xi(\g)]=[\beta \circ \xi(\tg)]$ by definition.

\s
\n{\bf Case 3.} $\ul{0} \in P_3(\beta)\cup P_4(\beta)$.  We have a bypass triangle $\tg \xra{\beta} \g \xra{\beta'} \g' \xra{\beta''} \tg$ such that $\ul{0} \in P_5(\beta')\cup P_6(\beta')$ and $\ul{0} \in P_1(\beta'')\cup P_2(\beta'')$.
Equation \eqref{eq deg} follows from Lemma \ref{lem deg}(2) and Cases 1 and 2.
\end{proof}

\subsection{Bypass triangles} \label{Sec triangle}

%The notion of bypass triangles imitates that of distinguished triangles in a triangulated category.
In this subsection we show that $\tf$ takes bypass triangles in $\tcne$ to distinguished triangles in $\tdne$.

\begin{prop} \label{prop triangle}
Let $(\g, [\xi]) \xra{\beta} (\g', [\beta \circ \xi]) \xra{\beta'} (\g'', [\beta' \circ \beta \circ \xi]) \xra{\beta''} (\g, [\xi][1])$ be a bypass triangle in $\tcne$.
Then its image under $\tf$ is a distinguished triangle in $\tdne$.
\end{prop}

\begin{proof}
%Up to an overall grading shift, the image under $\tf$ is determined by the bypasses.
We will omit the gradings in the objects and write $\tf(\g)$, etc.\  for simplicity. Let $\cf(\gamma'')$ denote $\cf(\gamma(\beta''))$ from Equation~\eqref{eqn: def gamma}.
Similarly, let $\cf(\gamma)$ and $\cf(\gamma')$ denote $\cf(\gamma(\beta))$ and $\cf(\gamma(\beta'))$, respectively.

In order to show that $ \mbox{cone}(\tf(\beta))$ and $\tf(\g'')$ are homotopy equivalent, we define the following two maps
$$\tf(\eta)= \tf(\gamma'') \oplus \tf(\beta'): \mbox{cone}(\tf(\beta)) = \tf(\g)\oplus \tf(\g')\ra \tf(\g''),$$
$$\tf(\epsilon)=\tf(\beta'')+\tf(\gamma'): \tf(\g'') \ra \mbox{cone}(\tf(\beta))=\tf(\g)\oplus \tf(\g').$$
They fit into the following diagram:
$$\xymatrix{
\tf(\g')  \ar[r]^{\tf(\beta')} & \tf(\g'')  \ar[r]^{\tf(\gamma')} \ar[dr]_{\tf(\beta'')} & \tf(\g')\\
\tf(\g) \ar[u]^{\tf(\beta)}  \ar[ur]_{\tf(\gamma'')} & & \tf(\g) \ar[u]_{\tf(\beta)}
}$$
where $\tf(\alpha)=\cf(\alpha)$ up to grading shifts for $\alpha \in \{\beta, \beta', \beta'', \gamma', \gamma''\}$. By the method of Lemma \ref{lem deg}, we can show that $\cf(\gamma)$, $\cf(\gamma')$, and $\cf(\gamma'')$ are homogeneous. By Lemma \ref{lem comp triangle}, $\tf(\eta)$ and $\tf(\epsilon)$ are chain maps and are therefore morphisms in $\tdne$.

We now show that they are homotopy inverses.  By Lemma \ref{lem ind triangle} and Equation~\eqref{eqn: def gamma},
\begin{align*}
\tf(\gamma) \circ \tf(\beta) + \tf(\beta'') \circ \tf(\gamma'')=&\op{id}_{\tf(\g)}, \\
\tf(\gamma') \circ \tf(\beta') + \tf(\beta) \circ \tf(\gamma)=&\op{id}_{\tf(\g')}, \\
\tf(\gamma'') \circ \tf(\beta'') + \tf(\beta') \circ \tf(\gamma')=&\op{id}_{\tf(\g'')}.
\end{align*}
We have $\tf(\eta) \circ \tf(\epsilon)=\op{id}_{\tf(\g'')}$ from the third equation.
It follows from the first two equations and Lemma \ref{lem ind triangle} that $\tf(\epsilon) \circ \tf(\eta)=\op{id}_{\mbox{cone}(\tf(\beta))}+ d\tf(\gamma)$, where $\tf(\gamma): \tf(\g') \ra \tf(\g)$.  Here we are also using $\tf(\gamma')\circ \tf(\gamma'')=0$ which is immediate from the definition.
\end{proof}

\section{$\tdne$ as a triangulated envelope of $\tcne$} \label{section: triangulated envelope}

The image of $\tfne$ generates $\tdne$ under taking iterated cones since all the $P(\g)$, $\g \in \bne$, are in the image. The goal of this section is to show that $\tfne$ is faithful, i.e.,
\begin{equation} \label{eq embed} \tag{F}
\tfne: \Hom_{\tcne}((\g,[\xi]), (\g',[\xi'])) \xra{\sim} \Hom_{\tdne}(\tf(\g,[\xi]), \tf(\g',[\xi'])).
\end{equation}
We first prove (\ref{eq embed}) in the following three basic cases: (i) $\g=\g'$; (ii) there exists a bypass $\g \xra{\beta} \g'$; and (iii) there exists a bypass $\g' \xra{\beta} \g$.  Using the calculations for the basic cases we show in Proposition \ref{prop serre d} that the Serre functors $\cstc$ of $\tcne$ and $\cstd$ of $\tdne$ commute with $\tfne$.
We finally prove the faithfulness in general by combining the results for the basic cases and the Serre functors.

\subsection{Basic cases}

\subsubsection{The case $\g=\g'$}

The goal is to prove (\ref{eq embed}) for $\g=\g'$. Since $\op{End}_{\cne}(\g)=\F\lan \op{id}_{\g} \ran$, it suffices to prove the following.

\begin{prop} \label{prop embed end}
For any $\g$ in $\cne$, $\op{End}_{\dne}(\cf(\g)) = \F\lan \op{id}_{\cf(\g)} \ran$.
\end{prop}

Before proving Proposition~\ref{prop embed end} we introduce some definitions.  Recall that $d_{\g}=\sum_{\mf{v} \in \tpv(\g)}d_{\mf{v}}$, where $d_{\mf{v}}$ is given in Equation~\eqref{eqn: def of dv}.
%$\sum\limits_{\mf{v} \in \vnb(\g)}d_{\mf{v}}$ since $d_{\mf{v}}=0$ for $\mf{v} \notin \vnb(\g)$.
Lemma \ref{lem dv} implies that:
\begin{gather} \label{eq double} \tag{D}
d_{\mf{v}}^2=0, \quad d_{\g}d_{\mf{v}}+d_{\mf{v}}d_{\g}=0,
\end{gather}
for any $\mv \in \tpv(\g)$.

For $\g,\g'$ in $\cne$, let $\op{Map}(\cf(\g),\cf(\g'))$ denote the space of $\rne$-module maps, where $\cf(\g),\cf(\g')$ are viewed as $\rne$-modules by ignoring the differentials.
%\begin{defn} \label{def dvw}
We then define maps
\begin{equation}\label{eqn: bidifferentials}
d_{\mw, \mf{v}}, d_{\es, \mf{v}}, d_{\mw, \es},d_{\g',\g}: \op{Map}(\cf(\g),\cf(\g')) \ra \op{Map}(\cf(\g),\cf(\g')),
\end{equation}
where $f \in \op{Map}(\cf(\g),\cf(\g'))$, $\mf{v} \in \oi(\g), \mw \in \oi(\g')$, and
\begin{align*}
d_{\mw, \mf{v}}f &=d_{\mf{w}}\circ f + f \circ d_{\mf{v}}, \\
d_{\es, \mf{v}}f&=f \circ d_{\mf{v}}, \\
d_{\mw, \es}f&=d_{\mf{w}}\circ f,  \\
d_{\g', \g}f&=d_{\g'} \circ f + f \circ d_{\g}.
\end{align*}

The lemma below follows from (\ref{eq double}).

\begin{lemma} \label{lem double complex}
If $d_0, d_1 \in \{d_{\mw, \mf{v}}, d_{\es, \mf{v}}, d_{\mw, \es}, d_{\g',\g}\}$, then
$$d_0^2=0, \qquad d_0d_1=d_1d_0,$$
in $\op{End}(\op{Map}(\cf(\g),\cf(\g')))$.
\end{lemma}

Note that $\Hom(\cf(\g),\cf(\g'))$ is the cohomology $H_{d_{\g',\g}}(\op{Map}(\cf(\g),\cf(\g')))$ by definition.

Recall that to compute $\op{End}_{\cne}(\g)$ on the topological side, we observe that $\#\gamma_{\g,\g}=\#\gamma_{\g',\g'}$, where $\g'$ is obtained from $\g$ by removing a boundary parallel component.  By repeating this reduction, we eventually obtain $\#\gamma_{\g,\g}=\#\gamma_{\g^0,\g^0}=1$, where $\g^0$ is the unique dividing set in $\cal{C}_{0,0}$.  To compute $\op{End}_{\dne}(\cf(\g))$ on the algebraic side we consider a similar reduction on
$$|\vnb(\g)|=|\{\mf{v} \in \tpv(\g) ~|~ l_{\gv}>0\}|.$$
The following example illustrates the idea of the reduction.

\begin{example} \label{ex end}
We compute $\op{End}(\cf(\g))$ for $\g$ from Example \ref{ex g4}.
Recall that $\vnb(\g)=\{(1),(1,1)\}$, where $\g_{(1,1)}=\{\ul{2},\ul{3}\}$ directly nests inside $\g_{(1)}=\{\ul{1},\ul{4}\}$ and no vector in $\vnb(\g)$ nests inside $\g_{(1,1)}$.
Let $\mw=(1)$ and $\mv=(1,1)$.
Let $\widehat{\g}^{\mf{v}}$ denote a dividing set which is obtained from $\g$ by replacing the component $\gv$ with $l_{\gv}+1$ boundary parallel components corresponding to the $l_{\gv}+1$ elements of $\gv$.
In particular, $\vnb(\widehat{\g}^{\mf{v}})=\{(1)\}$ and $\widehat{\g}^{\mf{v}}_{(1)}=\{\ul{1},\ul{4}\}$.

The complex $\mathcal{N}$ (given by \eqref{ex g4 cx}) for $\cf(\g)$ is
$$P(\ul{1},\ul{2}) \xra{d_{\mv}} P(\ul{1},\ul{3}) \xra{d_{\mw}} P(\ul{2},\ul{4}) \xra{d_{\mv}} P(\ul{3},\ul{4}),$$
where $d_{\mv}$ is given in Equation~\eqref{eqn: def of dv}.
As an $\F$-vector space $\op{Map}(\cf(\g),\cf(\g))$ is $7$-dimensional.
The differential $d_{\mv,\mv}$ acts on $\op{Map}(\cf(\g),\cf(\g))$ by:
$$\begin{array}{cccc}
d_{\mv,\mv}: & \op{Map}(\cf(\g),\cf(\g)) & \ra & \op{Map}(\cf(\g),\cf(\g)) \\
& \op{id}_{P(\ul{1},\ul{2})} & \mapsto & d_{\mv}\circ\op{id}_{P(\ul{1},\ul{2})}, \\
& \op{id}_{P(\ul{1},\ul{3})} & \mapsto & d_{\mv}\circ\op{id}_{P(\ul{1},\ul{2})}, \\
& \op{id}_{P(\ul{2},\ul{4})} & \mapsto & d_{\mv}\circ\op{id}_{P(\ul{2},\ul{4})}, \\
& \op{id}_{P(\ul{3},\ul{4})} & \mapsto & d_{\mv}\circ\op{id}_{P(\ul{2},\ul{4})}, \\
& d_{\mw}\circ\op{id}_{P(\ul{1},\ul{3})},~~ d_{\mv}\circ\op{id}_{P(\ul{1},\ul{2})},~~  d_{\mv}\circ\op{id}_{P(\ul{2},\ul{4})} & \mapsto & 0.
\end{array}
$$

There is a double complex $(\op{Map}(\cf(\g),\cf(\g)); d_{\mv,\mv}, d_{\g,\g}-d_{\mv,\mv})$.
Its first page $(H_{d_{\mv,\mv}}; d_{\g,\g}-d_{\mv,\mv})$ is given by:
$$\begin{array}{cccc}
d_{\g,\g}-d_{\mv,\mv}: & H_{d_{\mv,\mv}} & \ra & H_{d_{\mv,\mv}} \\
& [\op{id}_{P(\ul{1},\ul{2})}+\op{id}_{P(\ul{1},\ul{3})}] & \mapsto & [d_{\mw}\circ\op{id}_{P(\ul{1},\ul{3})}], \\
& [\op{id}_{P(\ul{2},\ul{4})}+\op{id}_{P(\ul{3},\ul{4})}] & \mapsto & [d_{\mw}\circ\op{id}_{P(\ul{1},\ul{3})}].
\end{array}
$$
It is easy to see that $(H_{d_{\mv,\mv}}; d_{\g,\g}-d_{\mv,\mv})$ is isomorphic to $(\op{Map}(\widehat{\g}^{\mf{v}}, \widehat{\g}^{\mf{v}}); d_{\widehat{\g}^{\mf{v}},\widehat{\g}^{\mf{v}}})$ whose cohomology is $\op{End}(\widehat{\g}^{\mf{v}})$.
By the spectral sequence associated to the double complex, $\op{End}(\widehat{\g}^{\mf{v}})$ converges to $\op{End}(\g)$.
\end{example}

Returning to the discussion of $\op{End}(\cf(\g))$ in general, suppose $\g \notin \bne$ and $\mf{v}\in \vnb(\g)$ such that $\nv(\mv, l_{\gv})=\es$, i.e., no vector in $\vnb(\g)$ directly nests inside $\mf{v}$.  (If $\g\notin\bne$, then such a vector always exists.)  We define $\widehat{\g}^{\mf{v}}$ as the dividing set in $\mathfrak{ob}(\cal{C}_{n,e-l_{\gv}})$ obtained from $\g$ by replacing the component $\gv$ with $l_{\gv}+1$ boundary parallel components corresponding to the $l_{\gv}+1$ elements of $\gv$. We will write $\wh{\g}$ for $\widehat{\g}^{\mf{v}}$ and $l$ for $l_{\gv}$ when $\mf{v}$ is understood.

There is a bijection
$$\wh{\mf{v}}:\vnb(\g) \backslash \{\mf{v}\} \xra{\sim} \vnb(\wh{\g})$$
satisfying $\whg_{\whv(\mf{w})}=\g_{\mf{w}}$ for $\mf{w} \in \vnb(\g) \backslash \{\mf{v}\}$.
There is also an induced map
$$\whv: OI(\g) \ra OI(\whg)$$
given by $\whv(\mf{i})_{\whv(\mf{w})}=\mi_{\mf{w}}$ for $\mf{w} \in \vnb(\g) \backslash \{\mf{v}\}$.  (It is also called $\whv$ by abuse of notation.)  The map $\whv: OI(\g) \ra OI(\whg)$ is surjective and is an $(l+1)$-to-$1$ map.
Given $\whi \in OI(\whg)$ and $0 \leq s \leq l$, define $\mf{i}^s \in \whv^{-1}(\whi)$ such that $(\mf{i}^s)_{\mf{v}}=s$.

%\begin{defn} \label{def red oi}

The following lemma relates $\op{End}(\cf(\g))$ and $\op{End}(\cf(\whg))$.

\begin{lemma} \label{lem embed end}
Suppose that $\mf{v} \in \vnb(\g)$ such that $\nv(\mv, l_{\gv})=\es$.
Then there is a finite double complex $(\op{Map}(\cf(\g),\cf(\g)); d_{\mv, \mv}, d_{\g,\g}-d_{\mv,\mv})$ whose first page $(H_{d_{\mv,\mf{v}}}; d_{\g,\g}-d_{\mv,\mf{v}})$ is isomorphic to the complex $(\op{Map}(\cf(\whg),\cf(\whg)); d_{\whg,\whg})$.
\end{lemma}

\begin{proof}
Since $d_{\mv, \mv}$ and $d_{\g,\g}$ are two commuting differentials by Lemma \ref{lem double complex},
$$(\op{Map}(\cf(\g),\cf(\g)); d_{\mv, \mv}, d_{\g,\g}-d_{\mv,\mv})$$ is a finite double complex.

Since $d_{\mf{u}}=0$ for $\mf{u} \notin \vnb(\g)$, the space $\op{Map}(\cf(\g),\cf(\g))$ has an $\F$-basis
$$\left\{\left.\prod_{t=1}^{k}d_{\mf{u}^t} \circ \op{id}_{\g(\mf{i})} ~\right|~ \mf{i} \in OI(\g), \mf{u}^t \in \vnb(\g), k\geq 0\right\},$$
where $\prod_{t=1}^{0}d_{\mf{u}^t} \circ \op{id}_{\g(\mf{i})}$ is understood to be $\op{id}_{\g(\mf{i})}$.
In this basis
\begin{itemize}
\item the composition is independent of the order of the $d_{\mf{u}^t}$ since they pairwise commute; and
\item each $d_{\mf{u}^t}$ appears at most once since $d_{\mf{u}^t}^2=0$ by Lemma \ref{lem dv}.
\end{itemize}

%For a nonzero generator $f=\prod_{t}d_{\mf{u}^t} \circ \op{id}_{\g(\mf{i})}$, its $d_{\mf{v}}$-degree is defined to be $1$ if $\mf{u}^t=\mv$ for some $t$; otherwise, its $d_{\mf{v}}$-degree is $0$. The map $d_{\mv,\mv}$ increases the $d_{\mf{v}}$-degree by $1$.

Let $\mw \in V(\g)$ such that $\mv$ directly nests inside $\mw$.  We consider the case $\mw \neq \ast$.  (When $\mw=\ast$, there is no $d_{\mw}$ and the same proof holds by setting $d_\mw=0$ and $d_{\widehat{\mv}(\mw)}=0$ below.)

We say that a generator $f$ is of Type (1) if it has the form:
$$f=\prod\limits_{t}d_{\mf{u}^t}  \circ \op{id}_{\g(\mf{i})}~\mbox{ or }~ \prod\limits_{t}d_{\mf{u}^t} \circ d_{\mf{v}} \circ \op{id}_{\g(\mf{i})}, ~\mbox{where}~  \mf{u}^t \neq \mf{w},\mv.$$
For each $\prod_{t}d_{\mf{u}^t}$, $\mf{u}^t \neq \mf{w},\mv$, there is a subcomplex of $(\op{Map}(\cf(\g),\cf(\g)); d_{\mv, \mf{v}})$:
$$\mbox{Type}~(1): \quad \prod\limits_{t}d_{\mf{u}^t} \circ \op{id}_{\g(\mf{i}^s)} \xra{d_{\mv, \mf{v}}} \left\{
\begin{array}{cl}
\prod\limits_{t}d_{\mf{u}^t} \circ d_{\mf{v}} \circ \op{id}_{\g(\mf{i}^1)} & s=0, \\
\prod\limits_{t}d_{\mf{u}^t} \circ d_{\mf{v}} \circ \op{id}_{\g(\mf{i}^s)}+\prod\limits_{t}d_{\mf{u}^t} \circ d_{\mf{v}} \circ \op{id}_{\g(\mf{i}^{s+1})} & 0<s<l, \\
\prod\limits_{t}d_{\mf{u}^t} \circ d_{\mf{v}} \circ \op{id}_{\g(\mf{i}^l)} & s=l.
\end{array}\right.$$
Such a subcomplex is said to be of Type (1). As an $\F$-vector space, this subcomplex has dimension $2l+1$ and any two subcomplexes of Type (1) are either equal or intersect trivially.

We say that a generator $f$ is of Type (2) if it has the form:
$$f=\prod\limits_{t}d_{\mf{u}^t} \circ d_{\mf{w}} \circ \op{id}_{\g(\mf{i})}, ~\mbox{where}~  \mf{u}^t \neq \mf{w},\mv.$$
If $f\not=0$, then $\mi_{\mf{v}}=0$, i.e., $\mf{i}=\mf{i}^0$ for some $\whi \in OI(\whg)$.  Each nonzero $f$ of Type (2) generates a $1$-dimensional subcomplex:
$$\mbox{Type}~(2): \quad f \xra{d_{\mv,\mv}} 0.$$

The complex $(\op{Map}(\cf(\g),\cf(\g)); d_{\mv, \mf{v}})$ is a direct sum of its subcomplexes of Types (1) and (2); note that there are no generators of type $\prod_{t}d_{\mf{u}^t} \circ d_{\mf{w}}\circ d_{\mv} \circ \op{id}_{\g(\mf{i})}$ since $d_{\mw}\circ d_{\mv}=0$.

We compute the cohomology $H_{d_{\mv, \mv}}$.
For a subcomplex of Type (1), $d_{\mv, \mf{v}}$ is a surjective map from $\F^{l+1}$ to $\F^{l}$ and a generator of $H_{d_{\mv, \mv}}$ is given by $\sum_{s=0}^{l}\prod_{t}d_{\mf{u}^t} \circ \op{id}_{\g(\mf{i}^s)}$.
For a subcomplex of Type (2), a generator of $H_{d_{\mv, \mv}}$ is given by $\prod_{t}d_{\mf{u}^t} \circ d_{\mf{w}} \circ \op{id}_{\g(\mf{i}^0)}$.

Define an $\F$-linear map:
$$
\begin{array}{cccc}
G: & H_{d_{\mv, \mv}}  & \ra  & \op{Map}(\cf(\whg),\cf(\whg)) \\
& \sum\limits_{s=0}^{l}\prod\limits_{t}d_{\mf{u}^t} \circ \op{id}_{\g(\mf{i}^s)} & \mapsto & \prod\limits_{t}d_{\whv(\mf{u}^t)} \circ \op{id}_{\whg(\whi)}, \\
& \prod\limits_{t}d_{\mf{u}^t} \circ d_{\mf{w}} \circ \op{id}_{\g(\mf{i}^0)}  & \mapsto & \prod\limits_{t}d_{\whv(\mf{u}^t)} \circ d_{\whv(\mf{w})} \circ \op{id}_{\whg(\whi)}.
\end{array}
$$
It is an isomorphism of $\F$-vector spaces since $\whv: \vnb(\g) \backslash \{\mf{v}\} \xra{\sim} \vnb(\wh{\g})$ is a bijection and $\whv: OI(\g) \ra OI(\whg)$ is surjective.

For any $\mf{u} \in \vnb(\g) \backslash \{\mf{v}\}$, we have $G(d_{\mf{u}}\circ f)=d_{\whv(\mf{u})} \circ G(f)$ and $G(f \circ d_{\mf{u}})=G(f) \circ d_{\whv(\mf{u})}$ for $f \in H_{d_{\mv, \mv}}$.
Since
$$d_{\g}-d_{\mf{v}}=\sum\limits_{\mf{u} \in \vnb(\g) \backslash \{\mf{v}\}}d_{\mf{u}},\quad \qquad d_{\whg}=\sum\limits_{\wh{\mf{u}} \in \vnb(\wh{\g})} d_{\wh{\mf{u}}},$$
it follows that $G$ commutes with $d_{\g,\g}-d_{\mv,\mf{v}}$ and $d_{\whg,\whg}$.  Hence the two complexes are isomorphic.
\end{proof}

\begin{proof}[Proof of Proposition \ref{prop embed end}]
Since $\cf(\g)$ is not isomorphic to the zero object of $\dne$, it follows that $\op{id}_{\cf(\g)} \in \op{End}(\cf(\g))$ is nonzero. It then remains to prove that $\dim(\op{End}(\cf(\g)))\leq 1$, which is proved by induction on $|\vnb(\g)|$.

If $|\vnb(\g)|=0$, then $\g \in \bne$. Hence $\cf(\g)=P(\g)$ and $\dim(\op{End}(P(\g)))\leq 1$.

If $|\vnb(\g)|>0$, there exists $\mf{v} \in \vnb(\g)$ such that $\nv(\mv, l_{\gv})=\es$.
By definition $|\vnb(\whg^{\mf{v}})|=|\vnb(\g)|-1.$
By Lemma \ref{lem embed end}, the cohomology of $(\op{Map}(\cf(\whg),\cf(\whg)); d_{\whg,\whg})$ converges to
the cohomology of $(\op{Map}(\cf(\g),\cf(\g)); d_{\g,\g})$.
Hence
$$\dim(\op{End}(\cf(\g)))\leq \dim(\op{End}(\cf(\whg)))$$
and $\dim(\op{End}(\cf(\g)))\leq 1$ by induction.
\end{proof}

\subsubsection{The case of a bypass}

Let $\beta \in \Hom(\g, \g')$ be a nontrivial bypass. Then $\Hom(\g,\g')=\F\lan \beta \ran$ and $\Hom(\g',\g)=0$.

\begin{prop} \label{prop embed bypass}
For any nontrivial bypass $\beta \in \Hom(\g, \g')$, we have
\be
\item $\Hom(\tf(\g,[\xi]), \tf(\g',[\wt{\xi'}])) = \F\lan \tf(\beta) \ran$, if $[\wt{\xi'}]=[\beta \circ \xi]$; otherwise, it is zero.
\item $\Hom(\cf(\g'), \cf(\g)) =0$.
\ee
\end{prop}

\begin{proof}
Consider a bypass triangle $(\g,[\xi]) \xra{\beta} (\g',[\xi']) \xra{\beta'} (\g'',[\xi'']) \xra{\beta''} (\g',[\xi][1])$ in $\tcne$.
By Proposition \ref{prop triangle}, it is mapped to a distinguished triangle in $\tdne$.
By applying the exact functor $\Hom_{\tdne}(\tf(\g,[\wt{\xi}]), -)$ to the distinguished triangle, we obtain a long exact sequence:
$$\Hom(\tf(\g,[\wt{\xi}]), \tf(\g,[\xi])) \xra{\tf(\beta)} \Hom(\tf(\g,[\wt{\xi}]), \tf(\g',[\xi'])) \xra{\tf(\beta')}\Hom(\tf(\g,[\wt{\xi}]), \tf(\g'',[\xi''])),$$
where the subscripts $\tdne$ are omitted.
Hence (1) is equivalent to $\Hom(\cf(\g), \cf(\g'')) =0$.

Similarly, by applying the exact functor $\Hom_{\tdne}(-,\tf(\g',[\wt{\xi'}]))$, one can see that (1) is equivalent to $\Hom(\cf(\g''), \cf(\g')) =0$.
By rotating the bypass triangle, the proposition is equivalent to any one of the following three statements:
$$(i) \Hom(\cf(\g'), \cf(\g)) =0; \quad (ii) \Hom(\cf(\g''), \cf(\g')) =0; \quad (iii) \Hom(\cf(\g), \cf(\g'')) =0.$$
Without loss of generality we can assume that $\ul{0} \in P_1(\beta) \cup P_2(\beta)$.  Proposition~\ref{prop embed bypass} is a consequence of the following lemma.
\end{proof}

\begin{lemma} \label{lemma embed bypass}
We have $\Hom(\cf(\g'), \cf(\g)) =0$ if $\Hom(\g, \g')$ is generated by a bypass $\beta$ and $\ul{0} \in P_1(\beta) \cup P_2(\beta)$.
\end{lemma}

\begin{proof}
There is a bypass triangle $\g \xra{\beta} \g' \xra{\beta'} \g''$ starting with $\beta$. Let us write $\uv, x, y, \mf{w}$ for $\uv(\beta'), x(\beta')$, $y(\beta'), \ov(\beta)$, as in Notation~\ref{notation: beta}; see Figure \ref{6-1}. Since $\ul{0} \in P_1(\beta) \cup P_2(\beta)=P_3(\beta') \cup P_4(\beta')$, we have $0< x \leq y=l_{\g'_{\uv}}$ and $\uv \neq \ast, \mf{w} \neq \ast$.    We write
$$\op{Map}(\cf(\g'), \cf(\g))=\bigoplus_{\mf{j}\in OI(\g'), \mf{i} \in OI(\g)}  \Hom(P(\g'(\mf{j}))[h(\mf{j})],  P(\g(\mf{i}))[h(\mf{i})]).$$

We prove the lemma by induction on $l_{\g'}((\g'_{\uv}(0), \g'_{\uv}(y)))$, defined in Equation~\eqref{eqn: l of A}.

\begin{figure}[ht]
\s
\begin{overpic}
[scale=0.2]{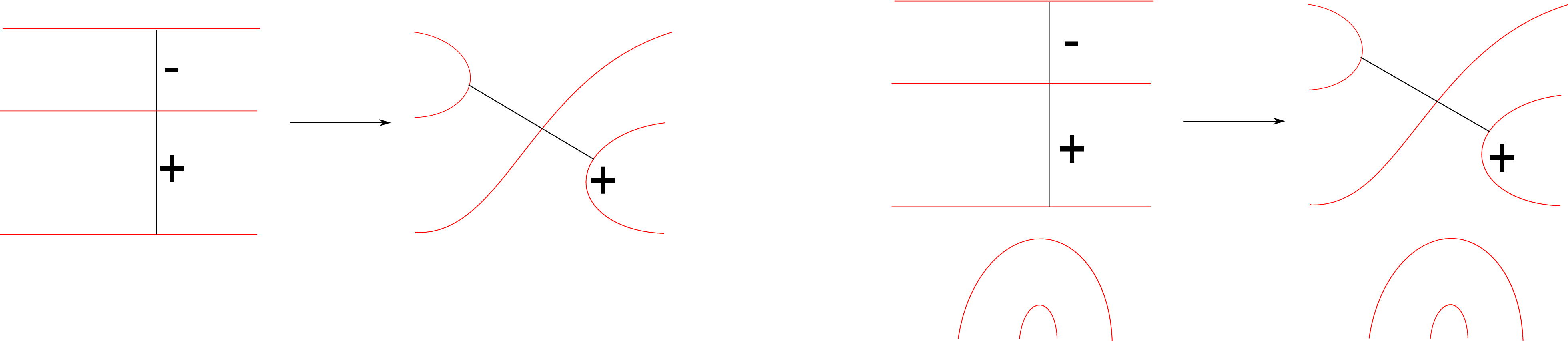}
\put(41,9){$\mf{w}$}
\put(98,10){$\mf{w}$}
\put(33,16){$\beta$}
\put(89,18){$\beta$}
\put(91,3){$\mf{v}$}
\put(64,3){$\mf{v'}$}
\put(82,0){$\g$}
\put(56,0){$\g'$}
\put(64,18){$\beta'$}
\put(60,10){$\uv$}
\put(7,16){$\beta'$}
\put(5,10){$\uv$}
\put(32,0){$\g$}
\put(9,0){$\g'$}
\put(6,22){${\scriptstyle P_4(\beta')}$}
\put(-3,17){${\scriptstyle P_3(\beta')}$}
\put(32,22){${\scriptstyle P_2(\beta)}$}
\put(22,17){${\scriptstyle P_1(\beta)}$}
\put(-1,8){${\scriptscriptstyle \ul{s+x}}$}
\put(-1,13){${\scriptscriptstyle \ul{s+y}}$}
\put(14,8){${\scriptscriptstyle \ul{s+x-1}}$}
\put(16,13){${\scriptscriptstyle \ul{s}}$}
\end{overpic}
\caption{The left-hand side describes Case 1 after removing boundary parallel components nesting inside $(\g'_{\uv}(0), \g'_{\uv}(y))$; the right-hand side depicts $\mf{v}, \mf{v'}$ in Case 2.}
\label{6-1}
\end{figure}

\s
\n{\bf Case 1.} If $l_{\g'}((\g'_{\uv}(0), \g'_{\uv}(y)))=0$, then no vector in $\vnb(\g')$ nests inside $\uv$. We view $\op{Map}(\cf(\g'), \cf(\g))$ as a finite double complex
$$(\op{Map}(\cf(\g'), \cf(\g)); d_{\es, \uv}, d_{\g,\g'}-d_{\es,\uv}),$$
whose first page $(H_{d_{\es,\uv}}; d_{\g,\g'}-d_{\es,\uv})$ converges to $H_{d_{\g,\g'}}=\Hom(\cf(\g'), \cf(\g))$. Here $d_{\es,\uv}f:=f \circ d_{\uv}$ for $f \in \op{Map}(\cf(\g'), \cf(\g))$.
The following claim implies the lemma for Case 1.

\begin{claim2}% \label{claim: mochi}
The cohomology of the complex $(\op{Map}(\cf(\g'), P(\g(\mf{i}))); d_{\es,\uv})$ is zero for any $\mf{i} \in OI(\g)$.
\end{claim2}

\begin{proof}[Proof of Claim]
We can ignore boundary parallel components $\gv' \subset (\g'_{\uv}(0), \g'_{\uv}(y))$ in the computation since those labels do not appear in either $P(\g'(\mf{j}))$ or $P(\g(\mf{i}))$ for any $\mf{j} \in OI(\g'), \mf{i} \in OI(\g)$.
Since we are assuming that $l_{\g'}((\g'_{\uv}(0), \g'_{\uv}(y)))=0$,
$$\g'_{\uv}(0)=\ul{s}, ~~\g'_{\uv}(x)=\ul{s+x}, ~~\g'_{\uv}(y)=\ul{s+y};$$
$$\g_{\mf{w}}(0)=\ul{s}, ~~ \g_{\mf{w}}(l_{\g_{\mf{w}}})=\ul{s+x-1}.$$

Let $f \in \op{Map}(P(\g'(\mf{j})), P(\g(\mf{i})))$ be a generator such that $f \circ d_{\uv}=0$.  If $\mf{j}_{\uv}<y$, then $f\circ d_{\uv}$ is a map $P(\g'(\mf{j'})) \ra P(\g(\mf{i}))$, where $\mf{j}=\uv|\mf{j'}$. We have
$$\Hom(\g'(\mf{j}), \g(\mf{i})) \neq 0, \quad \Hom(\g'(\mf{j'}), \g(\mf{i})) = 0; \qquad \g'(\mf{j'}) \xra{s+\mf{j}_{\uv}} \g'(\mf{j}).$$
By the tightness criterion (Proposition~\ref{prop: tightness criterion}), in order for $\Hom(\g'(\mf{j'}), \g(\mf{i})) = 0$,  $\Hom(\g'(\mf{j}), \g(\mf{i}))$ must factor through $\Hom(\g'(\mf{j}), \g'(\mf{j''}))$, where $\mf{j''}=\uv|\mf{j} \in OI(\g')$ such that $\g'(\mf{j}) \xra{s+\mf{j}_{\uv}-1} \g'(\mf{j''})$. Hence $\Hom(\g'(\mf{j''}), \g(\mf{i})) \neq 0$ is generated by $g$ such that $f=g \circ d_{\uv}$.

If $\mf{j}_{\uv}=y$ and $\Hom(\g'(\uv|\mf{j}), \g(\mf{i})) \neq 0$ is generated by $g$, then $f=g \circ d_{\uv}$. If $\mf{j}_{\uv}=y$ and $\Hom(\g'(\uv|\mf{j}), \g(\mf{i}))=0$, then $\{s,s+1,\dots,s+y-1\} \subset \g(\mf{i})_{\ast}$.
In particular $\g_{\mf{w}} \subset \g(\mf{i})_{\ast}$ which is not possible since $\mf{w} \neq \ast$.  This proves the claim.
\end{proof}

\s
\n{\bf Case 2.} If $l_{\g'}((\g'_{\uv}(0), \g'_{\uv}(y)))>0$, then there exists $\mf{v'} \in \vnb(\g')$ such that $\g'_{\mf{v'}} \subset (\g'_{\uv}(0), \g'_{\uv}(y))$ and no vector in $\vnb(\g')$ nests inside $\mf{v'}$ (i.e., $\mf{v'}$ is ``outermost''). There is a corresponding vector $\mf{v} \in \vnb(\g)$ such that $\gv=\g'_{\mf{v'}}$ and no vector in $\vnb(\g)$ nests inside $\mf{v}$; see the right-hand side of Figure \ref{6-1}.
Let $\whg, \wh{\g'}$ denote $\whg^{\mf{v}}, \wh{\g'}{}^{\mf{v'}}$.
There is a bypass triangle $\whg \xra{\wh{\beta}} \wh{\g'} \xra{\wh{\beta'}} \wh{\g''}$.
Let $\wh{\uv}, \wh{x}, \wh{y}$ denote $\uv(\wh{\beta'}), x(\wh{\beta'}), y(\wh{\beta'})$.
Then $$l_{\wh{\g'}}((\wh{\g'}_{\wh{\uv}}(0), \wh{\g'}_{\wh{\uv}}(\wh{y})))< l_{\g'}((\g'_{\uv}(0), \g'_{\uv}(y))).$$
The following claim implies the lemma for Case 2.

\begin{claim2}%\label{claim: mochi2}
There is a finite double complex $(\op{Map}(\cf(\g'),\cf(\g)); d_{\mf{v},\mf{v'}}, d_{\g,\g'}-d_{\mf{v},\mf{v'}}),$
where $d_{\mf{v},\mf{v'}}f=d_{\mf{v}}\circ f+f \circ d_{\mf{v'}}$, whose first page with respect to $d_{\mf{v},\mf{v'}}$ is isomorphic to $(\op{Map}(\cf(\wh{\g'}),\cf(\whg)); d_{\whg,\wh{\g'}})$.
\end{claim2}

\begin{proof}[Proof of Claim]
%$\op{Map}(\cf(\g'),\cf(\g))=\bigoplus_{\mf{j}, \mf{i}} \Hom(P(\g'(\mf{j})), P(\g(\mf{i})))$, where the direct sum is taken over $\mf{j} \in OI(\g'), \mf{i} \in OI(\g)$.

If $\mf{j}_{\uv} < x$, then there exists $\mf{i'} \in OI(\g)$ such that $\g(\mf{i'})=\g'(\mf{j})$. Such $\mf{j} \in OI(\g')$ is said to be of Type (1). Any $f \in \Hom(P(\g'(\mf{j})), P(\g(\mf{i})))$ can be written as $\prod_{t}d_{\mf{u}^t} \circ \op{id}_{\g'(\mf{j})}$ for $\mf{u}^t \in \vnb(\g)$.

If $\mf{j}_{\uv} \geq x$ and $\Hom(P(\g'(\mf{j})), P(\g(\mf{i}))) \neq 0$ for some $\mi$, then $\mf{j}_{\uv} = x$ and $\uv \in \sv(\mf{j})$ such that $(\uv|\mf{j})_{\uv}=x-1$. Such $\mf{j} \in OI(\g')$ is said to be of Type (2).  Then there exists $\mf{i'} \in OI(\g)$ such that $\g(\mf{i'})=\g'(\uv|\mf{j})$ and any $f \in \Hom(P(\g'(\mf{j})), P(\g(\mf{i})))$ can be written as $\prod_{t}d_{\mf{u}^t} \circ \op{id}_{\g'(\mf{j})} \circ d_{\uv}$ for $\mf{u}^t \in \vnb(\g)$.

Summarizing, $\op{Map}(\cf(\g'),\cf(\g))$ has an $\F$-basis:
$$\left\{\left.\prod\limits_{t}d_{\mf{u}^t} \circ \op{id}_{\g'(\mf{j})} ~\right|~ \mf{j} \in OI(\g')~\mbox{of Type (1)}, \mf{u}^t \in \vnb(\g)\right\}$$
$$\bigcup ~~ \left\{ \left.\prod\limits_{t}d_{\mf{u}^t} \circ \op{id}_{\g'(\mf{j})} \circ d_{\uv} ~\right|~ \mf{j} \in OI(\g')~\mbox{of Type (2)}, \mf{u}^t \in \vnb(\g)\right \}.$$
The rest of the proof is similar to that of Lemma \ref{lem embed end} and is left to the reader.
\end{proof}

This completes the proof of Lemma~\ref{lemma embed bypass}.
\end{proof}

\subsection{Serre functors of $\tdne$} \label{Sec serre d}

In this subsection we write $R=\rne$ for simplicity. According to \cite[Theorem 3.1]{Ke}, $\mf{D}^b(R)$ admits a Serre functor since $R$ has finite global dimension and the Serre functor is the left derived functor
$$M \ra DR \otimes_R^{\mathbb{L}}M,$$
where $M$ is a left $R$-module and $DR$ denotes the $R$-bimodule $\Hom_{\F}(R,\F)$. Note that this means that if $r_1,r_2,r\in R$ and $\phi\in \Hom_{\F}(R,\F)$, then $r_1\phi r_2(r)= \phi(r_2 r r_1)$.    Since $\tdne$ is equivalent to $\mf{D}^b(R)$, it admits an induced Serre functor which we denote by $\cstd$. For any projective $R$-module $P(\g)$, $\g \in \bne$, $\cstd(P(\g))$ is isomorphic to a projective resolution of the tensor product $DR \otimes_R P(\g)$.

%In order to compute $\cs(P(\g))$ we first examine the $R$-bimodule $DR$.
By definition, $DR$ has a dual $\F$-basis $\{[\g'|\g] ~|~ \Hom(\g,\g')\neq 0, \g,\g' \in \bne\}$, where the linear map $[\g'|\g]: R \ra \F$ sends the generator of $\Hom(\g,\g')$ to $1$ and other generators to zero.
As an $R$-bimodule, $DR$ has the defining relations:
$$(\g')[\g'|\g]=[\g'|\g](\g)=[\g'|\g].$$
$$(\tg'|\g')[\g'|\g]= \left\{
\begin{array}{cl}  [\tg'|\g] & \mbox{if}~ \Hom(\g, \tg')\neq0;
\\0 & \mbox{otherwise}.
\end{array}\right.$$
$$[\g'|\g](\g|\tg)= \left\{
\begin{array}{cl}  [\g'|\tg] & \mbox{if}~ \Hom(\tg, \g')\neq0;
\\0 & \mbox{otherwise}.
\end{array}\right.$$
Hence $DR \otimes_R P(\g)$ has an $\F$-basis $\{[\g'|\g] ~|~ \Hom(\g, \g')\neq 0, \g' \in \bne\}$.

For $\g \in \bne$, we compute $\cstd(P(\g)) \cong DR \otimes_R P(\g)$ in terms of $\csg$.
If $\ul{1} \in \g_{\ast}$ then $\vnb(\csg)=\es$.
If $\ul{1} \not\in \g_{\ast}$, then $\vnb(\csg)=\{\mf{w}\}, l_{\csg_{\mf{w}}}=e$ and $\csg_{\mw}=\{\g_\ast(1)-\ul{1}, \dots, \g_\ast(e)-\ul{1}, \ul{n}\}$, and we write
\begin{equation} \label{eqn: def s g i}
\csg^i=\csg(\mf{i}) \in \bne ~\mbox{ for } ~ 0 \leq i \leq e,
\end{equation}
where $\mf{i} \in OI(\csg)$ such that $\mi_\mf{w}=i$.

\begin{lemma} \label{lem serre d basic}
For any $\g \in \bne$, $\cstd(\tf(\g))=\cstd(P(\g))$ is isomorphic to $\cf(\csg)$ in $\tdne$.
\end{lemma}

\begin{proof}
The tensor product $DR \otimes_R P(\g)$ has an $\F$-basis $\{[\g'|\g] ~|~ \Hom(\g,\g')\neq 0, \g' \in \bne\}$.

If $\ul{1} \in \g_{\ast}$, then $\csg \in \bne$.
By Lemma \ref{lem serre c}, $\Hom(\g,\g')\neq0$ if and only if $\Hom(\g',\csg)\neq0$.
Hence $DR \otimes_R P(\g)$ is isomorphic to $P(\csg)$, i.e., $\cstd(\tf(\g))$ is isomorphic to $\cf(\csg)$.

Assume $\ul{1} \notin \g_{\ast}$ from now on.  Our proof makes repeated use of Proposition \ref{prop: tightness criterion} and Corollary \ref{cor stack}.
%Then $\csg_{\mf{w}}=\{\g_{\ast}(1)-\ul{1}, \dots, \g_{\ast}(e)-\ul{1},\ul{n}\}$ where $\vnb(\cs(\g))=\{\mf{w}\}$.
%% about $\Hom(\g,\g')$ for $\g, \g' \in \bne$.
Consider the complex $\cf(\csg)$:
$$P(\csg^e) \xra{\op{pr}_e} \cdots \xra{\op{pr}_2} P(\csg^{1}) \xra{\op{pr}_1} P(\csg^{0}).$$
We have % $\csg_{\ast}=\{0, \g_{\ast}(1), \dots, \g_{\ast}(e)\}$ and
$\csg^0_{\ast}=\{\ul{0}, \g_{\ast}(2)-\ul{1}, \dots, \g_{\ast}(e)-\ul{1}, \ul{n}\}$.
By Proposition \ref{prop: tightness criterion}, $\Hom(\g,\csg^0)\neq0$ so that $[\csg^0|\g] \in DR \otimes_R P(\g)$ exists.
Define a map of left $R$-modules $$\op{pr}_0: P(\csg^0) \ra DR \otimes_R P(\g)$$ by $\op{pr}_0(\csg^0)=[\csg^0|\g]$.
Moreover, the path from $\g$ to $\csg^0$ is the longest nonzero path starting from $\g$.
In other words, $\Hom(\g',\csg^0)\neq0$ if $\Hom(\g,\g')\neq 0$ for $\g' \in \bne$.
Then $\op{pr}_0(\g'|\csg^0)=[\g'|\g]$ for any generator $[\g'|\g]\in DR \otimes_R P(\g)$.
Hence $\op{pr}_0$ is a surjection.

\begin{claim2}% \label{claim isom}
$\op{Ker}(\op{pr}_0)=\op{Im}(\op{pr}_1)$.
\end{claim2}

\begin{proof}[Proof of Claim]
Since
$$\csg^1_{\ast}=\{\ul{0}, \g_{\ast}(1)-\ul{1}, \g_{\ast}(3)-\ul{1},\dots, \g_{\ast}(e)-\ul{1}, \ul{n}\},$$ it follows that $\Hom(\g, \csg^1)=0$ and $\op{Im}(\op{pr}_1) \subset \op{Ker}(\op{pr}_0)$.

To prove $\op{Ker}(\op{pr}_0)\subset\op{Im}(\op{pr}_1)$, it suffices to show that if $\Hom(\g,\g') = 0$ and $\Hom(\g',\csg^0)\neq0$ for $\g' \in \bne$, then $\Hom(\g',\csg^1)\neq0$.
Since $\Hom(\g',\csg^0)\neq0$ we have
$$\g_{\ast}(i)-\ul{1}=\csg^0_{\ast}(i-1) < \g'_{\ast}(i) \leq \csg^0_{\ast}(i)=\g_{\ast}(i+1)-\ul{1},$$
for $1<i\leq e$.  This implies that $\g'_{\ast}(1)<\g_{\ast}(1)$ since $\Hom(\g,\g') = 0$.
We have
$$\g'_{\ast}(1) \leq \g_{\ast}(1)-\ul{1}, \qquad  \csg^1_{\ast}(i-1) \leq \csg^0_{\ast}(i-1) < \g'_{\ast}(i) \leq \csg^0_{\ast}(i)=\csg^1_{\ast}(i),$$
for $1<i\leq e$. Hence $\Hom(\g',\csg^1)\neq 0$.
\end{proof}

The proofs of $\op{Ker}(\op{pr}_{i})=\op{Im}(\op{pr}_{i+1})$ for $0 < i \leq e$ are similar, where $\op{Im}(\op{pr}_{e+1})$ is understood to be $0$.
Hence $\cf(\csg)$ is a projective resolution of $DR \otimes_R P(\g)$.
\end{proof}

\begin{prop} \label{prop serre d}
The Serre functors $\cstc$ and $\cstd$ commute with $\tf$.
\end{prop}

\begin{proof} $\mbox{}$

\s\n
{\bf Step 1.}
We first show that $\cstc$ and $\cstd$ commute with $\tf$ on the level of objects, i.e.,
\begin{equation} \label{eqn: Serre commutes}
\cstd(\tf(\g,[\xi])) \cong \tf(\cstc(\g,[\xi])).
\end{equation}
We prove this by induction on $m(\g)=e+1-|\g_{\ast}|$.

Suppose $m(\g)=0$, i.e., $\g \in \bne$. It suffices to prove Equation~\eqref{eqn: Serre commutes}  for $\xi=\xi(\g)$ since both sides commute with shift functors. By Lemma \ref{lem serre d basic}, $$\cstd(\tf(\g,[\xi(\g)]))=\cstd(\cf(\g))\cong \cf(\csg) =\tf(\csg,[\xi(\csg)]).$$
By Definition \ref{def serre tc}, $\tf(\cstc(\g,[\xi(\g)]))=\tf(\csg,[\zeta(\g) \circ \xi(\g)])$, where $\zeta(\g)$ is the generator of $\Hom(\g,\csg)$.  It remains to show that
\begin{gather} \label{eq serre deg}
[\zeta(\g) \circ \xi(\g)]=[\xi(\csg)],
\end{gather}
for $\g \in \bne$. If $\csg \in \bne$, then $\zeta(\g)$ corresponds to a nonzero path from $\g$ to $\csg$ in $\qne$ and Equation \eqref{eq serre deg} follows from the definition of $[\xi(\csg)]$. If $\csg \notin \bne$, then consider $\ob(\csg) \in \Hom(\csg^0, \csg)$ as defined in Equation~\eqref{eqn: triangle with ob}.
By Lemma~\ref{lemma: serre c 2}, the composition $\Hom(\csg^0, \csg) \times \Hom(\g, \csg^0) \ra \Hom(\g, \csg)$ is nontrivial. Hence
$$[\xi(\csg)]=[\ob(\csg) \circ \xi(\csg^0)]=[\ob(\csg) \circ \xi_{\g,\csg^0} \circ \xi(\g)]=[\zeta(\g) \circ \xi(\g)],$$
where $\xi_{\g,\csg^0}$ is the generator of $\Hom(\g, \csg^0)$.

Suppose $m(\g)>0$. Consider the bypass triangle
$$(\g,[\xi][-1]) \xra{\beta(\g)} (\g', [\xi']) \xra{\beta} (\g'', [\xi'']) \xra{} (\g, [\xi])$$
in $\tcne$ which contains $\beta(\g)$ as defined in Definition \ref{def beta g}.
We have $m(\g'), m(\g'') < m(\g)$.
Let $\beta\in \Hom_{\tcne}((\g', [\xi']), (\g'', [\xi'']))$ be the second morphism in the triangle.
By Proposition \ref{prop embed bypass},
$$\Hom_{\tdne}(\tf(\g', [\xi']), \tf(\g'', [\xi'']))=\F\lan\tf(\beta)\ran.$$
Since $\cstc$ is the rotation endofunctor of $\tcne$, it maps bypass triangles to bypass triangles and
$$\Hom_{\tdne}(\tf(\cstc(\g', [\xi'])), \tf(\cstc(\g'', [\xi''])))=\F\lan\tf(\cstc(\beta))\ran.$$
Also, since $\cstd$ is an auto-equivalence of $\tdne$,
$$\Hom_{\tdne}(\cstd(\tf(\g', [\xi'])), \cstd(\tf(\g'', [\xi''])))=\F\lan\cstd(\tf(\beta))\ran.$$
Finally, since $\cstd(\tf(\g', [\xi'])) \cong \tf(\cstc(\g', [\xi']))$ and $\cstd(\tf(\g'', [\xi''])) \cong \tf(\cstc(\g'', [\xi'']))$ by induction, we have $\cstd(\tf(\beta))=\tf(\cstc(\beta))$.
Hence,
$$\cstd(\tf(\g, [\xi])) \cong \cstd(\op{Cone}(\tf(\beta)))=\op{Cone}(\cstd(\tf(\beta)))=\op{Cone}(\tf(\cstc(\beta)))\cong \tf(\cstc(\g, [\xi])),$$
where the first and last isomorphisms follow from Proposition \ref{prop triangle}.

\s
\n{\bf Step 2.} Since the morphisms of $\tcne$ are generated by bypasses, it suffices to prove that $\cstd(\tf(\beta)) \cong \tf(\cstc(\beta))$ for any bypass $\beta \in \Hom_{\tcne}((\g, [\xi]), (\g', [\xi']))$.  This in turn follows from observing that both are generators of
$$\Hom_{\tdne}(\tf(\cstc(\g, [\xi])), \tf(\cstc(\g', [\xi'])))=\Hom_{\tdne}(\cstd(\tf(\g, [\xi])), \cstd(\tf(\g', [\xi']))).$$
\vskip-.26in
\end{proof}

We prove the analogue of Lemma \ref{lem serre c deg} for the Serre functor $\cstd$ of $\tdne$.

\begin{prop} \label{prop serre d deg}
There is an isomorphism of endofunctors of $\tdne$: $\cstd^{n+1} \cong T^{e(n-e)}$.
\end{prop}

\begin{proof}
Since $\tdne$ is generated by the image of $\tf$ (and in particular the projectives $P(\g)$, $\g\in \bne$, and morphisms between projectives), it suffices to show that $\cstd^{n+1}(\tf(\g,[\xi]))=\tf(\g,[\xi])[e(n-e)]$ for any $(\g,[\xi])$.
By Lemma \ref{lem serre c deg} and Proposition \ref{prop serre d}, 
$$\cstd^{n+1}(\tf(\g,[\xi]))=\tf(\cstc^{n+1}(\g,[\xi]))=\tf(\g,[\xi][e(n-e)])=\tf(\g,[\xi])[e(n-e)].$$
\vskip-.26in
\end{proof}

\subsection{General cases}

Since $\Hom_{\cne}(\g, \g')$ is at most one-dimensional, $\tfne$ is faithful if and only if $\fne$ is faithful, i.e.,
\begin{equation} \label{eq embed cdf} \tag{F$'$}
\fne: \Hom_{\cne}(\g, \g') \xra{\sim} \Hom_{\dne}(\cf(\g), \cf(\g')).
\end{equation}
By Proposition \ref{prop serre d}, (\ref{eq embed cdf}) holds for $\g, \g'$ if and only if it holds for $\cs^k(\g), \cs^k(\g')$ for some $k$.

We prove Equation \eqref{eq embed cdf} for $\g, \g'$ in $\cne$ by induction on $n$.  If $\g$ and $\g'$ have a common boundary parallel component, then it is either a positive region or a negative region.

\s\n{\bf Case 1.} Suppose that $R_+(\g)$ and $R_+(\g')$ have a common boundary parallel component, i.e., there exist $\mf{v} \in V(\g)$ and $\mf{v'} \in V(\g')$ such that $\gv=\g'_{\mf{v'}}=\{\ul{t}\}$.
By applying the Serre functor $t+1$ times, we can assume that $\gv=\g'_{\mf{v'}}=\{\ul{n}\}$.
Let $\tg$ and $\tg'$ denote dividing sets in $\cal{C}_{n-1,e}$ obtained from $\g$ and $\g'$ by removing $\gv$ and $\g'_{\mf{v'}}$, respectively.

\s \n{\bf Case 2.} Suppose that $R_-(\g)$ and $R_-(\g')$ have a common boundary parallel component, i.e., there exist $\mf{v} \in V(\g)$ and $\mf{v'} \in V(\g')$ such that $\{\ul{t},\ul{t+1}\} \subset \gv,\g'_{\mf{v'}}$.
By applying the Serre functor $t+1$ times, we can assume that $\{\ul{0},\ul{n}\} \subset \g_{\ast}, \g'_{\ast}$.
Let $\tg$ and $\tg'$ denote dividing sets in $\cal{C}_{n-1,e-1}$ obtained from $\g$ and $\g'$ by removing $\ul{n}$ from $\g_{\ast}$ and $\g'_{\ast}$, respectively.

\begin{lemma} \label{lem embed ind}
{\em (F$'$)} holds for $\g, \g'$ if and only if it holds for $\tg, \tg'$ in both Cases 1 and 2.
\end{lemma}

\begin{proof}
It suffices to prove that there exist canonical isomorphisms:
\begin{gather}
\label{iso I}\Hom_{\cne}(\g, \g') \cong \Hom_{\cal{C}_{n-1,e}}(\tg, \tg'); \\
\label{iso II}\Hom_{\dne}(\cf(\g), \cf(\g')) \cong \Hom_{\cal{D}_{n-1,e}}(\cf(\tg), \cf(\tg')).
\end{gather}

The first isomorphism~\eqref{iso I} follows from observing that $\gamma_{\g,\g'}$ is isomorphic to $\gamma_{\tg,\tg'}$.

Consider two full sub-quivers $\qne'$ and $\qne''$ of $\qne$, where
$$V(\qne')=\{\g \in \bne ~|~ \ul{n} \notin \g_{\ast}\}, \quad V(\qne'')=\{\g \in \bne ~|~ \ul{n} \in \g_{\ast}\}.$$
There are two subalgebras $\rne'$ and $\rne''$ of $\rne$ which are generated by $\qne'$ and $\qne''$, respectively.
Let $\dne'$ and $\dne''$ be the corresponding full subcategories of $\dne$.
By Lemma \ref{lem rne}, $\rne'$ is canonically isomorphic to $R_{n-1,e}$, and $\rne''$ is canonically isomorphic to $R_{n-1,e-1}$.
The second isomorphism~\eqref{iso II} follows from compositions of functors:
$\cal{D}_{n-1,e} \xra{\sim} \dne' \hookrightarrow \dne$ and $\cal{D}_{n-1,e-1} \xra{\sim} \dne'' \hookrightarrow \dne$.
\end{proof}

Before proving (\ref{eq embed cdf}) in general, consider the special cases described in Figure \ref{6-2}. There are two boundary parallel components, one in $R_{\pm}(\g)$ and the other in $R_{\mp}(\g')$. The boundary parallel component of $\g'$ is obtained by rotating that of $\g$ through a counterclockwise angle of $\frac{\pi}{n+1}$.
We say that the pair $(\g,\g')$ is in {\em local annihilation position.}

\begin{figure}[ht]
\begin{overpic}
[scale=0.3]{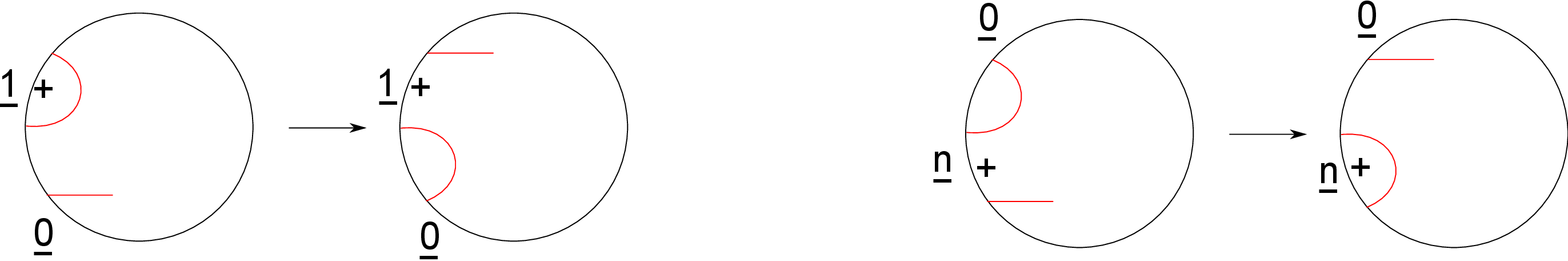}
\put(14,0){$\g$}
\put(38,0){$\g'$}
\put(74,0){$\g$}
\put(98,0){$\g'$}
\end{overpic}
\caption{A pair $(\g,\g')$ in local annihilation position, normalized using the Serre functor.}
\label{6-2}
\end{figure}

\begin{lemma} \label{lem embed local}
If the pair $(\g,\g')$ is in local annihilation position, then
$$\Hom_{\cne}(\g, \g')=0, \qquad \Hom_{\dne}(\cf(\g), \cf(\g'))=0.$$
\end{lemma}

\begin{proof}
Since the two boundary parallel components form a loop after edge rounding, $\#\gamma_{\g,\g'}>1$ and $\Hom_{\cne}(\g, \g')=0$.

By applying the Serre functor we are in one of the following two cases as in Figure \ref{6-2}:
\be
\item The boundary parallel components are in $R_+(\g)$ and $R_-(\g')$: there exists $\mf{v} \in V(\g)$ such that $\gv=\{\ul{1}\}$; and $\ul{1} \in \g'_{\ast}$.

\item The boundary parallel components are in $R_-(\g)$ and $R_+(\g')$: there exists $\mf{v'} \in V(\g')$ such that $\g'_{\mf{v'}}=\{\ul{n}\}$; and $\ul{n} \in \g_{\ast}$.
\ee
For any $\mf{i} \in OI(\g), \mf{j} \in OI(\g')$, $\ul{1} \notin \g(\mf{i})$ and $\ul{1} \in \g'(\mf{j})$ in the first case and $\ul{n} \in \g(\mf{i})$ and $\ul{n} \notin \g'(\mf{j})$ in the second case. In either case $\Hom(\g(\mf{i}), \g'(\mf{j}))=0$ by Proposition \ref{prop: tightness criterion}. Hence $\Hom_{\dne}(\cf(\g), \cf(\g'))=0$.
\end{proof}

We are finally in a position to complete the proof of Theorem \ref{thm embed}.

\begin{proof}[Proof of Theorem \ref{thm embed}]
We show that (\ref{eq embed cdf}) holds for any $\g, \g'$ by induction on $n$.  For any boundary parallel component of $\g'$, consider the neighborhood of the component in $\g'$ as on the right-hand side of Figure~\ref{6-3}: there are three endpoints $r,s,t$ of $\g'$ in clockwise order around $\bdry D^2$ and $\g'$ connects $r$ and $s$.  We may assume that $\g$ does not connect $r$ and $s$ since if $\g$ and $\g'$ have a common boundary parallel component then we can reduce $n$ by Lemma \ref{lem embed ind}; and that $\g$ does not connect $s$ and $t$ since if $(\g, \g')$ is in local annihilation position then we are done by Lemma \ref{lem embed local}. Hence there exists a nontrivial bypass triangle $\g \xra{\beta} \tg \xra{} \g^0 \xra{} \g$ such that $(\g^0, \g')$ is in local annihilation position; see Figure \ref{6-3}.

By applying exact functors $\Hom(-,\g')$ (this is exact by Lemma~\ref{lemma: exact functor}) and $\Hom(-,\cf(\g'))$, we have two isomorphisms:
$$\Hom(\tg, \g') \xra{\circ \beta} \Hom(\g, \g'), \quad \Hom(\cf(\tg), \cf(\g')) \xra{\circ \cf(\beta)} \Hom(\cf(\g), \cf(\g')),$$
since $\Hom(\g^0, \g')=0$ and $\Hom(\cf(\g^0), \cf(\g'))=0$ by Lemma \ref{lem embed local}.

\begin{figure}[ht]
\begin{overpic}
[scale=0.3]{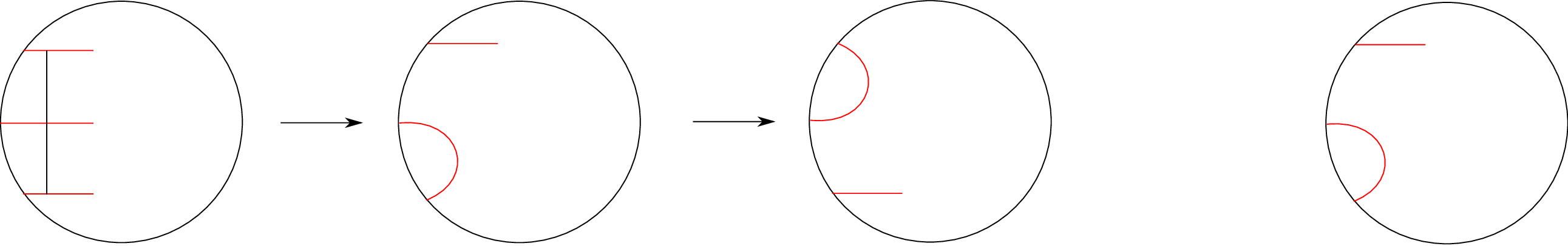}
\put(14,0){$\g$}
\put(40,0){$\tg$}
\put(66,0){$\g^0$}
\put(98,0){$\g'$}
\end{overpic}
\caption{}
\label{6-3}
\end{figure}

Since $\tg$ and $\g'$ have a common boundary parallel component we can reduce $n$ by Lemma \ref{lem embed ind}.   In the case where $n=2$, there are $5$ dividing sets in $\cal{C}_2$, and $\g$ and $\g'$ are either the same or in the unique bypass triangle in $\cal{C}_{2,1}$. The first part of Theorem \ref{thm embed} follows from Propositions \ref{prop embed end} and \ref{prop embed bypass}.  The assertion about exact triangles was the content of Proposition~\ref{prop triangle}.
\end{proof}

\newpage

\section*{Index of notation}

\s\n {\em Section~\ref{section: contact category}}

\s\n $\g:$ a dividing set

\n $\rg:$ the positive region of a convex surface $\Sigma$ with dividing set $\g$

\n $R_+(F):$ the positive region of $\bdry\Sigma$ with respect to the marked points $F\subset \bdry\Sigma$

\n $\delta=\delta_+ \cup \delta_-:$ the arc of a bypass attachment as a union of its positive and negative parts

\s\n {\em Section~\ref{section: contact category of disk}}

\s\n  $\cne:$ the (skeletal version of the) contact category of a disk

\n $\tcne:$ the universal cover of the (skeletal version of the) contact category of a disk

\n $\chi_+=\chi_+(\Gamma), \chi_-=\chi_-(\Gamma):$ the Euler characteristics of $R_+(\Gamma)$ and $R_-(\Gamma)$.

\n $\ul{s}:$ the label $s$ of a positive arc in $\bdry D^2$

\n $\cs, \cs_{\wt{\cal{C}}}:$ Serre functors of $\cne$ and $\tcne$

\s\n {\em Section~\ref{section: algebraic description}}

\s\n $\pi_0(\rg), \pi_0(R_-(\g)):$ the set of components of $\rg$ and $R_-(\g)$

\n $\vz:$ the set of vectors of positive integers

\n $\ast:$ the special element of $\vz$

\n $\Phi_{\g}: \pi_0(\rg) \ra \vz$ the assignment of components of $\rg$ by vectors in $\vz$

\n $V(\g)=\op{Im}(\Phi_{\g}):$ the set of vectors of $\g$

\n $\tpv(\g)=V(\g) \backslash \{\ast\}$

\n $\vnb(\g):$ the subset of non-boundary-parallel components of $\tpv(\g)$

\n $\bne:$ the set of basic dividing sets

\n $\g(\ul{s_1}, \dots, \ul{s_e}):$ the basic dividing set $\g \in \bne$ such that $\g_{\ast}=\{\ul{0}, \ul{s_1}, \dots, \ul{s_e}\}$.

\n $\gv:$ the $\mf{v}$-component of $\rg$, where $\mf{v}\in V(\g)=\op{Im}(\Phi_{\g})$; the set of labels contained in the 

$\mf{v}$-component of $\rg$

\n $\g_{\ast}:$ the based component of $\rg$ containing the label $\ul{0}$

\n $\gv(i):$ the $i$th element of $\gv$

\n $\l_{\gv}=|\gv|-1$

\n $\beta:$ a bypass attachment; later $\beta$ will also be a map $V(\g)\to V(\g')$ (cf.\ Equation~\ref{eqn: beta}) and a

map $II(\beta) \sqcup SI(\beta) \ra OI(\g')$ (cf.\ Definition~\ref{def beta ind})

\n $\uv(\beta), \ov(\beta): $ elements of $V(\g)$, cf.\ Notation~\ref{notation: beta}

\n $x(\beta), y(\beta):$ elements of $\{0,\dots,l_{\g_{\uv(\beta)}}\}$, cf.\ Notation~\ref{notation: beta}

\n $z(\beta):$ element of $\{0,\dots,l_{\g_{\ov(\beta)}}\}$, cf.\ Notation~\ref{notation: beta}

\n $\g_{\uv(\beta)}^l, \g_{\uv(\beta)}^r:$ left and right subsets of $\g_{\uv(\beta)}$, cf.\ Notation~\ref{notation: beta 2}

\s\n {\em Section~\ref{section: defn of algebra}}

\s\n $\qne:$ the quiver

\n $\rne:$ the $\F$-algebra

\n $\g \xra{i} \g':$ an arrow in the quiver $\qne$

\n $(\g), (\g|\g'):$ generators of the algebra $\rne$

\n $P(\g):$ left projective $\rne$-module corresponding to $\g\in \bne$

\n $\tdne:$ the homotopy category of bounded complexes of finitely projective $\rne$-modules

\n $\dne:$ ungraded version of $\tdne$

\s\n{\em Section~\ref{section: defn of functor}}

\s\n $\fne: \cne\to \dne$ and $\tfne: \tcne \to \tdne$ functors

\n $\beta(\g):$ the leftmost bypass on $\g$

\n $\mi=\lan \mi_{\mv} \ran:$ an omitting index with its $\mv$-entries

\n $\oi(\g):$ the set of omitting indices of $\g$

\n $\g(\mi):$ the basic dividing set corresponding to $\mi\in OI(\g)$

\n $c_{\mv}(i):$ the nesting degree of $i$ for $\mv\in V(\g)$

\n $h(\mi):$ the cohomological degree of $\mi\in OI(\g)$

\n $\nv(\mv, i):$ the set of nesting vectors inside $\mv$ up to $i$

\n $\dnv(\mv, i):$ the set of direct nesting vectors inside $\mv$ between $i-1$ and $i$

\n $\slv(\mi):$ the set of sliding vectors of $\mi$

\n $\shv(\mi):$ the set of shuffling vectors of $\mi$

\n $\sv(\mi):$ $\slv(\mi)\cup \shv(\mi)$

\n $\vi:$ $\mv$-modified omitting index, cf.\ Definition~\ref{def diff ind}

\n ${\frak c}|\mi:$ ${\frak c}$-modified omitting index, where ${\frak c}$ is a component of $\pi_0(R_-(\Gamma))$

\n $r(\mi, \mv):$ a nonzero element of $\rne$ corresponding to a path from $\g(\mi)$ to $\g(\vi)$ in $\qne$

\n $d(\mi, \mv): P(\g(\mi))\to P(\g(\vi))$ given by right multiplication by $r(\mi,\mv)$

\n $r(\mi,{\frak c})$, $d(\mi,{\frak c})$ defined analogously

\n $d_{\mv}, d_{\frak c}:$ components of differential $d=d_{\g}$ for $\cf(\g)$

\n $LSV(\beta):$ the set of left shuffling vectors of $\beta$

\n $II(\beta):$ the set of omitting indices of type (Id) for $\beta$

\n $SI(\beta):$ the set of omitting indices of type (Sh) for $\beta$

\n $t(\beta, \mi):$ a nonzero element of $\rne$ corresponding to a path from $\g(\mi)$ to $\g'(\bi)$ in $\qne$

\s \n{\em Section~\ref{section: functors on universal cover}}

\s\n $[\xi(\g)]:$ a homotopy class of $\g$

\n $\overline{\beta}(\g):$ a nontrivial bypass to $\g$

\n $l_{\g}(A):$ the sum of $l_{\gv}$ for $\gv \subset A$

\n $P_i(\beta):$ six parts of $\bdry D^2$ for $\beta$

\s \n{\em Section~\ref{section: triangulated envelope}}

\s\n $\whg^{\mv}:$ a dividing set associated to $\g$ and $\mv \in \vnb(\g)$; also written as $\whg$ if $\mv$ is understood

\n $\whv:$ bijection $\vnb(\g)\backslash \{\mv\} \xra{\sim} \vnb(\whg)$; also denotes the induced map $\whv: OI(\g) \ra OI(\whg)$

\n $\op{Map}(\cf(\g), \cf(\g')):$  $\rne$-module maps where $\cf(\g)$ and $\cf(\g')$ are viewed as $\rne$-modules

\n $d_{\mw, \mf{v}}f=d_{\mf{w}}\circ f + f \circ d_{\mf{v}},$ where $f \in \op{Map}(\cf(\g),\cf(\g'))$ and $\mf{v} \in \oi(\g), \mw \in \oi(\g')$

\n $d_{\es, \mf{v}}f=f \circ d_{\mf{v}}$,  $d_{\mw, \es}f=d_{\mf{w}}\circ f,$  $d_{\g', \g}f=d_{\g'} \circ f + f \circ d_{\g}$

\n  $\cs_{\wt{\cal{D}}}:$  Serre functor of $\tdne$

\n $DR:$ $R$-bimodule $\Hom_{\F}(R,\F)$

\n $[\g'|\g]:$ generators of the $R$-bimodule $DR$

\n $\csg^i:$ a basic dividing set representing $\csg$

\end{document}